%% file: frame.tex
\documentclass[10pt,oneside]{amsart}
\usepackage{macros}

\title{Chiral differential operators via\\quantization of the holomorphic $\sigma$-model}

\author{Vassily Gorbounov}
\address{University of Aberdeen,
Higher School of Economics, and the Laboratory of Algebraic Geometry and Homological Algebra at the Moscow Institute of Physics and Technology}
\email{ vgorb10@gmail.com}

\author{Owen Gwilliam}
\address{University of Massachusetts, Amherst}
\email{gwilliam@math.umass.edu}
\urladdr{http://people.math.umass.edu/~gwilliam/}

\author{Brian Williams}
\address{Northeastern University}
\email{brianwilliams.math@gmail.com}
\urladdr{https://web.northeastern.edu/brwilliams/}

\begin{document}
\maketitle

\begin{abstract}
The curved $\beta\gamma$ system is a nonlinear $\sigma$-model with a Riemann surface as the source and a complex manifold $X$ as the target. Its classical solutions pick out the holomorphic maps from the Riemann surface into $X$. Physical arguments identify its algebra of operators with a vertex algebra known as the chiral differential operators (CDO) of $X$. We verify these claims mathematically by constructing and quantizing rigorously this system using machinery developed by Kevin Costello and the second author, which combine renormalization, the Batalin-Vilkovisky formalism, and factorization algebras. Furthermore, we find that the factorization algebra of quantum observables of the curved $\beta\gamma$ system encodes the sheaf of chiral differential operators. In this sense our approach provides deformation quantization for vertex algebras. As in many approaches to deformation quantization, a key role is played by Gelfand-Kazhdan formal geometry. We begin by constructing a quantization of the $\beta\gamma$ system with an $n$-dimensional formal disk as the target. There is an obstruction to quantizing equivariantly with respect to the action of formal vector fields $\Vect$ on the target disk, and it is naturally identified with the first Pontryagin class in Gelfand-Fuks cohomology. Any trivialization of the obstruction cocycle thus yields an equivariant quantization with respect to an extension of $\Vect$ by $\hOmega^2_{cl}$, the closed 2-forms on the disk. By machinery mentioned above, we then naturally obtain a factorization algebra of quantum observables, which has an associated vertex algebra easily identified with the formal $\beta\gamma$ vertex algebra. Next, we introduce a version of Gelfand-Kazhdan formal geometry suitable for factorization algebras, and we verify that for a complex manifold $X$ with trivialized first Pontryagin class, the associated factorization algebra recovers the vertex algebra of CDOs of~$X$.\\

Le syst\`{e}me b\^{e}ta gamma incurv\'{e} est un mod\`{e}le sigma nonlin\'{e}aire de
source une surface de Riemann et de cible une vari\'{e}t\'{e} complexe $X$. Ses
solutions classiques sont donn\'{e}es par des cartes holomorphes de la
surface de Riemann dans $X$. Les arguments physiques identifient son
alg\`{e}bre d'op\'{e}rateurs avec une alg\`{e}bre vertex connue sous le nom
d'op\'{e}rateurs diff\'{e}rentiels chiraux (CDO) de $X$. Nous v\'{e}rifions ces
affirmations de mani\`{e}re math\'{e}matique en construisant et en quantifiant
rigoureusement ce syst\`{e}me en utilisant sur les techniques d\'{e}velopp\'{e}es
par Kevin Costello et le second auteur, combinant \`{a} la fois les outils
de la renormalisation, le formalisme de Batalin-Vilkovisky, et les
alg\`{e}bres \`{a} factorisation. En outre, nous prouvons que l'alg\`{e}bre \`{a}
factorisation des observables quantiques du syst\`{e}me b\^{e}ta gamma courb\'{e}
encode les gerbes d'op\'{e}rateurs diff\'{e}rentiels chiraux. En un sens,
notre approche fournit une quantification par d\'{e}formation pour les
alg\`{e}bres de vertex. Comme dans de nombreuses approches \`{a} la
quantification par d\'{e}formations, la g\'{e}om\'{e}trie formelle de
Gelfand-Kazhdan joue un r\^ole cl\'{e}. Nous commençons par construire une
quantification du syst\`{e}me b\^{e}ta gamma \`{a} valeur un disque formel de
dimension n. Il existe une obstruction \`{a} l'existence d'une
quantification qui soit \'{e}quivariante par rapport \`{a} l'action des champs
de vecteurs formels sur le disque; cette obstruction s'identifie
naturellement \`{a} la premi\`{e}re classe de Pontryagin de la cohomologie de
Gelfand-Fuks. Toute trivialisation du cocycle d'obstruction donne
ainsi une quantification \'{e}quivariante vis-\`{a}-vis d'une extension de
champs de vecteurs formels par les 2-formes ferm\'{e}es sur le disque.
D'apr\`{e}s les r\'{e}sultats cit\'{e} ci-dessus, nous en d\'{e}duisons
naturellement une alg\`{e}bre \`{a} factorisation d'observables quantiques, \`{a}
laquelle est associ\'{e}e une alg\`{e}bre de vertex qui s'identifie \`{a}
l'alg\`{e}bre vertex formelle de type beta gamma. Par ailleurs, nous
introduisons une version de la g\'{e}om\'{e}trie formelle de Gelfand-Kazhdan
adapt\'{e}e aux alg\`{e}bres \`{a} factorisation et nous v\'{e}rifions que, pour une
vari\'{e}t\'{e} complexe munie d'une trivialisation de sa premi\`{e}re classe de
Pontryagin, l'alg\`{e}bre \`{a} factorisation associ\'{e}e d\'{e}crit l'alg\`{e}bre vertex.
\end{abstract}


\input{intro}
\tableofcontents
\input{part1}
\input{part2}
\input{part3}

\input{appx}

\addtocontents{toc}{\protect\setcounter{tocdepth}{0}}
\addcontentsline{toc}{section}{References}
\bibliographystyle{amsalpha}
\bibliography{refs}

\end{document}

%% file: intro.tex
\part*{Introduction}

The curved $\beta\gamma$ system is an elegant nonlinear $\sigma$-model with a Riemann surface $\Sigma$ as the source and a complex manifold $X$ as the target. The equations of motion pick out the holomorphic maps $\Sigma \to X$. Thus, from a purely mathematical perspective, it is a compelling example to study 
because the classical theory naturally involves complex geometry and so must the quantization, although the meaning is less familiar. From a physical perspective, the curved $\beta\gamma$ system arises naturally as a close cousin of more central theories: it is a half-twist of the $(0,2)$-supersymmetric $\sigma$-model \cite{WittenCDO}, and it is also the chiral part of the infinite volume limit of the usual (non-supersymmetric) $\sigma$-model (see the appendix). In consequence, the curved $\beta\gamma$ system exhibits many features of these theories while enjoying the flavor of complex geometry, rather than super- or Riemannian geometry.

In mathematics, however, this theory first appeared in a hidden form in the work of Beilinson-Drinfeld and Malikov-Schechtman-Vaintrob \cite{BD,MSV}, and it was subsequently developed by many mathematicians (see \cite{KV,Cheung,Bressler} among much else). The {\em chiral differential operators} (CDOs) on a complex $n$-manifold $X$ are a sheaf of vertex algebras locally resembling a vertex algebra of $n$ free bosons, and the name indicates the analogy with the differential operators, a sheaf of associative algebras on $X$ locally resembling the Weyl algebra for $T^*\CC^n$. Unlike the situation for differential operators, which exist on any manifold $X$, such a sheaf of vertex algebras exists only if $\ch_2(X) = 0$ in $H^2(X, \Omega^2_{cl})$, and each choice of trivialization $\alpha$ of this characteristic class yields a different sheaf $\CDO_{X,\alpha}$. In other words, there is a gerbe over $X$ of vertex algebras \cite{GMS}. The appearance of this topological obstruction (essentially the first Pontryagin class, but non-integrally) was surprising, and even more surprising was that the character of this vertex algebra was the Witten genus of $X$, up to a constant depending only on the dimension of $X$ \cite{BorLib}. These results exhibited the now-familiar rich connections between conformal field theory, geometry, and topology, but arising from a mathematical process rather than a physical argument. 

Witten \cite{WittenCDO} explained how CDOs on $X$ arise as the perturbative piece of the chiral algebra of the curved $\beta\gamma$ system, by combining standard methods from physics and mathematics. (In elegant lectures on the curved $\beta\gamma$ system \cite{Nek}, with a view toward Berkovit's approach to the superstring, Nekrasov also explains this relationship.  Kapustin \cite{KapCDR} gave a similar treatment of the closely-related chiral de Rham complex.) This approach also gave a different understanding of the surprising connections with topology, in line with anomalies and elliptic genera as seen from physics. Let us emphasize that only the perturbative sector of the theory appears (i.e., one works near the constant maps from $\Sigma$ to $X$, ignoring the nonconstant holomorphic maps); the instanton corrections are more subtle and not captured just by CDOs (see \cite{KapOrlov} for a treatment of the instanton corrections for complex tori).

In this paper we construct mathematically the perturbative sector of the curved $\beta\gamma$ system via the approach to quantum field theory developed in \cite{CosBook, CG1, CG2}, thus providing a rigorous construction of the path integral for the curved $\beta\gamma$ system. That means we work in the homotopical framework for field theory known as the Batalin-Vilkovisky (BV) formalism, in conjunction with Feynman diagrams and renormalization methods. As a very brief gloss, the BV formalism amounts to deforming the classical action $S^\cl$ to a ``quantized action'' $S^\q = S^\cl + \hbar S_{(1)} + \hbar^2 S_{(2)} + \cdots$ satisfying a condition known as the quantum master equation. This quantized action $S^\q$ provides a formal substitute for the path integral; more precisely, it is a homological version of the integrand ``$\exp(-S(\phi)/\hbar) \mathcal{D}\phi$'' for the path integral. Indeed, given this quantized action $S^\q$, one can extract the algebraic relations that hold between the expected values of observables. Thus the quantum master equation encodes homologically the condition that the associated quantum integrand is well-posed. We find, for instance, that the curved $\beta\gamma$ system admits a quantized action satisfying the quantum master equation only if the target manifold $X$ has $\ch_2(X) = 0$, where $\ch_2(X)$ is a component of the Chern character. (Given our context, this condition is that the first Pontryagin class vanishes.) This condition was found in the earlier mathematical work by quite different methods.

One key feature of the framework in \cite{CG1, CG2} is that every BV theory yields a factorization algebra of observables. (We mean here the version of factorization algebras developed in \cite{CG1, CG2}, not the version of Beilinson and Drinfeld \cite{BD}.)
In our situation, the theory produces a factorization algebra living on the source manifold $\CC$, 
and the machinery of \cite{CG1, CG2} allows one to extract a vertex algebra from this factorization algebra.
Our main result is that this vertex algebra is the CDOs.
Thus, we show that in a wholly mathematical setting, one can start with the action functional for the curved $\beta\gamma$ system
and recover the sheaf $\CDO_{X,\alpha}$ of vertex algebras on $X$ via the algorithms of \cite{CosBook, CG1, CG2}.
To accomplish this, we develop machinery that ought to be useful in constructing nonlinear $\sigma$-models in the BV framework and allows one to analyze explicitly the resulting factorization algebra.

\begin{rmk*}
In a sense, the curved $\beta\gamma$ system is a perfect testing ground for the formalism of \cite{CosBook, CG1, CG2}: physical arguments about anomalies and moduli ought to be codified on the BV quantization side, and the consequences on the factorization algebra side ought to recover the vertex algebra constructions of \cite{MSV,GMS}. The work here shows that the formalism passes this test.
\end{rmk*}

Let us explain a little about our methods before stating our theorems precisely. The main technical challenge is to encode the nonlinear $\sigma$-model in a way so that the BV formalism of \cite{CosBook} applies. In \cite{WG2}, Costello introduces a sophisticated approach by which he recovers the anomalies and the Witten genus as partition function, but it seems difficult to relate CDOs directly to the factorization algebra of observables of his quantization. Instead, we use formal geometry {\it \`a la} Gelfand and Kazhdan \cite{GK}, as applied to the Poisson $\sigma$-model by Kontsevich \cite{KonDQ} and Cattaneo-Felder \cite{CF}.  The basic idea of Gelfand-Kazhdan formal geometry is that every $n$-manifold $X$ looks, very locally, like the formal $n$-disk, and so any representation $V$ of the formal vector fields and formal diffeomorphisms determines a vector bundle $\cV \to X$, by a sophisticated variant of the associated bundle construction. (Every tensor bundle arises in this way, for instance.) In particular, the Gelfand-Kazhdan version of characteristic classes for $V$ live in the Gelfand-Fuks cohomology $H^*_{\GF}(\Vect)$ and map to the usual characteristic classes for $\cV$. There is, for instance, a Gelfand-Fuks version of the Witten class for every tensor bundle.

Thus, we start with the $\beta\gamma$ system with target the formal $n$-disk $\widehat{D}^n = \rm{Spec}\,\CC[[t_1,\ldots,t_n]]$ and examine whether it quantizes \emph{equivariantly} with respect to the actions of formal vector fields $\Vect$ and formal diffeomorphisms on the formal $n$-disk. (These actions are compatible, so that we have a representation of a Harish-Chandra pair.) We call this theory the \emph{equivariant formal $\beta\gamma$ system of rank $n$}.

\begin{thm*}
The $\Vect$-equivariant formal $\beta\gamma$ system of rank $n$ has an anomaly given by a cocycle $\ch_2(\widehat{D}^n)$ in the Gelfand-Fuks  complex ${\rm C}^*_{\GF}(\Vect ; \widehat{\Omega}^2_{n,cl})$. This cocycle determines a Lie algebra extension $\TVect$ of $\Vect$. The cocycle is exact in ${\rm C}^*_{\GF}(\TVect ; \hOmega^2_{n,cl})$, and yields a $\TVect$-equivariant BV quantization, unique up to homotopy. The partition function of this theory over the moduli of elliptic curves is the formal Witten class in the Gelfand-Fuks  complex ${\rm C}^*_{\GF}(\Vect, \bigoplus_k \widehat{\Omega}_n^k[k])[[\hbar]]$.
\end{thm*}

Gelfand-Kazhdan formal geometry is used often in deformation quantization. See, for instance, the elegant treatment by Bezrukavnikov-Kaledin \cite{BK}. Here we develop a version suitable for vertex algebras and factorization algebras, which requires allowing homotopical actions of the Lie algebra $\Vect$. (Something like this appears already in \cite{BD,KV,Malikov2008}, but we need a method with the flavor of differential geometry and compatible with Feynman diagrammatics. It would be interesting to relate directly these different approaches.) In consequence, our equivariant theorem implies the following global version.
 
\begin{thm*}
Let $X$ be a complex manifold. The curved $\beta\gamma$ system admits a BV quantization if the characteristic class $\ch_2(X)$ vanishes, and each choice of trivialization $\alpha$ yields a BV quantization, unique up to homotopy. The associated factorization algebra on $X$ recovers the vertex algebra $\CDO_{X,\alpha}$ of chiral differential operators associated to the trivialization $\alpha$. Moreover, the partition function is the Witten class ${\rm Wit}(X)$ in $\bigoplus_k H^*(X, \Omega^k[k])$, where $\Omega^k$ here denotes the sheaf of holomorphic $k$-forms.
\end{thm*}

\begin{rmk*} 
In physics, ordinarily, the partition function refers to the full path integral.
What we are referring to is the path integral of the effective action on the zero modes in perturbation theory. 
\end{rmk*}

To identify $\CDO_{X,\alpha}$ as coming from the factorization algebra, we prove general statements relating factorization algebras for such chiral theories with vertex algebras. Indeed, our work shows how the elegant formulas uncovered in \cite{MSV,GMS} arise explicitly by canceling the anomalies that appear in setting the integrand of the quantum measure (or to use BV language, in finding a solution to the quantum master equation). Via the factorization algebra of observables, these BV manipulations become the computations that Witten and Nekrasov used in explaining why the curved $\beta\gamma$ system should recover the chiral differential operators.

\begin{rmk*}
We wish to emphasize that our central goal in this paper 
is not to provide yet another method for constructing sheaves of vertex algebras
or another understanding of the geometry behind the Pontryagin class as an anomaly.
Instead our goal is to offer an explanation for how CDOs appear as a quantization,
from a path integral perspective as rigorously encoded in the BV formalism.

One essential application of this perspective is the establishment of modularity for characters for general chiral conformal field theories. 
This manifest modularity for the characters of observables in the BV formalism is due to their explicit presentation via Feynman diagrams as integrals over the elliptic curve. 
\end{rmk*}

Our techniques for assembling BV theories in families --- and their factorization algebras in families --- apply to many $\sigma$-models already constructed , such as the topological $B$-model \cite{LiLi}, Rozansky-Witten theory \cite{CLL}, and topological quantum mechanics \cite{GG1, GLL}. They also allow us to recover quickly nearly all the usual variants on CDOs and structures therein, such as the chiral de Rham complex and the Virasoro actions, and we intend to explain that elsewhere. Other veins of research are also opened up, notably new approaches to quantum sheaf cohomology and to the curved $\beta\gamma$ system with higher-dimensional source manifold.

\subsection{Overview}

The paper is divided into three parts. Part I is devoted purely to the vertex algebra of chiral differential operators, Part II constructs the curved $\beta\gamma$ system as a BV field theory and analyzes its associated factorization algebra of observables, and Part III explains how to recover the vertex algebra from the factorization algebra. Each Part has its own introduction with a detailed overview of its contents. We emphasize that Parts I and II can be read independently; only in Part III are the two stories in explicit dialogue.

\subsection{Acknowledgements}

This work would not have been possible without the support of several organizations.
First, it was the open and stimulating atmosphere of the Max Planck Institute for Mathematics 
that made it so easy to begin our collaboration.
Moreover, it is through the MPIM's great generosity that we were able to continue
work and finish the paper during several visits by VG and BW.
Second, we benefited from the support and convivial setting of the Hausdorff Institute for Mathematics
and its Trimester Program ``Homotopy theory, manifolds, and field theories" during the summer of 2015.
Third, the Oberwolfach Workshop ``Factorization Algebras and Functorial Field Theories" in May 2016
allowed us all to gather in person and finish important discussions.
In addition, OG enjoyed support from the National Science Foundation as a postdoctoral fellow under Award DMS-1204826, 
and BW enjoyed support as a graduate student research fellow under Award DGE-1324585.  
Finally, this research was carried out, in part, within the HSE University Basic Research Program
and funded by the Russian Academic Excellence Project~'5-100'.

For OG there is a large cast of mathematicians whose questions, conversation, and interest 
have kept these issues alive and provided myriad useful insights that are now hard to enumerate in detail.
He thanks Kevin Costello for introducing him to the $\beta\gamma$ system in graduate school --- and for innumerable discussions since --- as well as Dan Berwick-Evans, Ryan Grady, and Yuan Shen
for grappling collaboratively with \cite{WG2} throughout that period.
Si Li's many insights and questions have shaped this work substantially.
Matt Szczesny's guidance at the Northwestern CDO Workshop was crucial; 
his subsequent encouragement is much appreciated.
OG would also like to thank Stephan Stolz and Peter Teichner for the still-running conversation about
conformal field theory that influences strongly his approach to the subject.
Finally, he thanks Andr\'e Henriques, John Francis, and Scott Carnahan for letting him eavesdrop as they chatted about CDOs over a decade ago.

BW feels fortunate to have stepped into this community early in his graduate work and has benefited from the support of many of the individuals mentioned above. 
First and foremost, he thanks his adviser Kevin Costello for guidance and Si Li for helping him to harness Feynman diagrams in the context of the BV formalism. 
He also thanks Ryan Grady, Matt Szczesny, and Stephan Stolz for invitations to talk about this project as well as valuable input on various aspects of it. 
In addition, numerous discussions with Dylan William Butson, Chris Elliott, and Philsang Yoo about perturbative QFT have informed his work. 
Finally, we would like to thank Matt Szczesny and James Ladouce for pointing out numerous typos and providing feedback on an earlier draft of this paper.

%% file: part1.tex
\newpage

\part*{Part I: Gelfand-Kazhdan descent and chiral differential
  operators}

The goal of this part is to provide a construction of chiral differential operators via the methods of Gelfand-Kazhdan formal geometry; this approach is a modest modification of an approach described for the chiral de Rham complex in \cite{MSV}. (In subsequent work we will provide a BV construction of chiral de Rham complex, along with a family of related sheaves of vertex algebras.) Our phrasing here aims to highlight the parallels with the next part, where we introduce a homotopy-coherent version of Gelfand-Kazhdan formal geometry that works nicely with the Batalin-Vilkovisky formalism and Feynman diagrammatics and thus allows us to construct a factorization algebra refining CDOs.

Recall that Gelfand-Kazhdan formal geometry is an approach to any ``natural'' construction in differential geometry, i.e., to constructions that apply uniformly to all manifolds of a given dimension (or with some other common, local geometric structure). The basic idea is that on any $n$-manifold, the immediate neighborhood of every point looks the same, and so if some construction can be described on any sufficiently small neighborhood and is equivariant for local diffeomorphisms, the construction should apply to every $n$-manifold. In other words, it is a kind of refinement of tensor calculus. To be more precise, in formal geometry, one works with a ``formal'' neighborhood of a point $p$ in $\RR^n$, namely the ``space'' whose algebra of functions is the $\infty$-jets of functions at $p$ ({\em aka} Taylor series at $p$ of functions). Let us denote this space by $\hD^n$.The relevant group of ``formal'' diffeomorphisms then means the $\infty$-jet at $p$ of diffeomorphisms that fix $p$ ({\em aka} Taylor series at $p$ of diffeomorphisms), which we will denote by $\Aut_n$. (Note that for every point $p$, the group is isomorphic.) Every $n$-manifold $X$ possesses a canonical {\em flat} principal $\Aut_n$-bundle $X^{coor}$ whose fiber over $p \in X$ is the space of formal coordinates centered at $p$ and is equipped with a flat connection valued in formal vector fields $\Vect$ (which is slightly larger than the Lie algebra of $\Aut_n$).  In the context of this paper we are interested in complex manifolds and there is a corresponding bundle of holomorphic formal coordinates. These are $\infty$-jets of biholomorphisms.

From this reasoning we see that every $\Aut_n$-representation $V$ that has a compatible action of $\Vect$ produces a flat vector bundle $\cV_X$ over each $n$-dimensional manifold $X$ whose horizontal sections typically encode familiar vector bundles. Such a representation is called a Harish-Chandra module. As an example, consider $\hO_n$, the functions on $\hD^n$, whose flat sections over $X$ are smooth functions on $X$ (holomorphic functions in the complex case). Or consider $\hT_n$, the vector fields on $\hD^n$, whose flat sections over $X$ are vector fields on $X$. This construction of a vector bundle on $X$ from an $\Aut_n$-representation is an example of Gelfand-Kazhdan descent. In light of this, it should be no surprise that there is a Gelfand-Kazhdan version of characteristic classes for these vector bundles that recovers the usual Chern classes.

\begin{rmk} 
The Gelfand-Kazhdan approach to formal geometry can also be applied to more interesting geometries. For example, symplectic, Poisson, or even Riemannian geometry can be encapsulated by the formalism.
\end{rmk}

Chiral differential operators, like differential operators, are easy to define locally on an $n$-manifold, using coordinates. The challenge is to glue these local descriptions to produce the global object. The vertex algebra $\hCDO_n$ of CDOs on a formal $n$-disk $\hD_n$ is well-known, but it is not equivariant for the Harish-Chandra pair of automorphisms of $\hD_n$. The failure to be equivariant is a characteristic class that globalizes to the first Pontryagin class, or $\ch_2(T_X)$, in Dolbeault cohomology. This class defines an extension of the Harish-Chandra pair, and $\hCDO_n$ is equivariant for this extension. Each choice of trivialization $\alpha$ of $\ch_2(T_X)$ encodes an extension $\Tilde{X}^{coor}_\alpha$ of $X^{coor}$ to a flat principal bundle for this extension of pairs. Hence, one can apply Gelfand-Kazhdan descent to $\hCDO_n$ along $\Tilde{X}^{coor}_\alpha$ to produce a vertex algebra, and it is precisely the chiral differential operators on $X$ associated to the trivialization $\alpha$.

Sections \ref{sec hc descent} and \ref{sec gk descent} of this part are devoted to articulating rigorously this machinery in a format convenient for our problem. As mentioned parenthetically, we need a slight enlargement of the theory of flat vector bundles involving Harish-Chandra pairs, which consist of a Lie group and a thickening of its Lie algebra.

Specifically, in Section \ref{sec gk descent} we formulate a version of Harish-Chandra descent that we call {\em Gelfand-Kazhdan descent} that is suitable for our purposes. In Section \ref{sec vertex alg} we recall well-known facts about the vertex algebra of affine chiral differential operators and extracting the relavent Harish-Chandra structures. Sections \ref{sec ext desc} is devoted to developing an extended version of Gelfand-Kazhdan descent that is applied to CDOs in Section ~\ref{sec vertex desc}.

We will extract some familiar properties and structures of the sheaf of chiral differential operators from the perspective of Gelfand-Kazhdan formal geometry. For instance, we show in Proposition \ref{prop conformal cdo} that chiral differential operators have the structure of a sheaf of {\em conformal} vertex algebras only if $X$ is Calabi-Yau (in addition to having $\ch_2(T_X) = 0$, of course). Moreover, we show how the Witten genus appears as the character of the sheaf of CDOs, which has already appeared in the works \cite{Cheung, BorLib}. 

\begin{rmk} 
We should emphasize here that Part I is not the truly novel aspect of this paper. As mentioned in \cite{MSV}, the standard arguments of Gelfand-Kazhdan formal geometry apply to the chiral de Rham complex, and they certainly knew that a minor extension of such formal geometry should allow one to construct CDOs. Our main goal in Part I is to explain these standard arguments and this extension carefully and systematically. We do this for two reasons: first, to allow a systematic comparison with the BV quantization in Part II, and second, to provide a general tool that ought to be applicable to constructing many more interesting vertex algebras. There are elegant machines for such purposes, thanks to \cite{BD, KV}, but we wanted a version closer to the concrete computations that most interest us.
\end{rmk}

\section{Flat vector bundles and Harish-Chandra descent} \label{sec hc descent}

This section is a discussion of the theory of vector bundles with flat connection arising from Harish-Chandra pairs. It establishes notation and terminology but can likely be used as a reference. In this section we have largely treated both smooth and holomorphic geometry, but throughout the rest of the paper we are concerned with the latter.

\subsection{Algebra of Harish-Chandra pairs}

\subsubsection{Harish-Chandra pairs}
All Lie algebras and Lie groups will be defined over $\CC$. For $G$ a Lie group, we use $\Lie(G)$ to denote its associated Lie algebra, which can be identified with the tangent space of the identity element. To start, we work with finite-dimensional groups and algebras, but we will eventually discuss certain infinite-dimensional examples.
 
\begin{dfn} A {\em Harish-Chandra pair} (or HC-pair) is a pair $(\fg, K)$ where $\fg$ is a Lie algebra and $K$ is a Lie group together with
\begin{itemize}
\item[(i)] an action of $K$ on $\fg$, $\rho_K : K \to {\rm Aut}(\fg)$
\item[(ii)] an injective Lie algebra map $i : {\rm Lie}(K) \hookrightarrow \fg$
\end{itemize}
such that the action of $\Lie(K)$ on $\fg$ induced by $\rho_K$,
\ben
{\rm Lie}(\rho_K) : {\rm Lie}(K) \to {\rm Der}(\fg),
\een
is the adjoint action induced from the embedding $i: {\rm Lie}(K) \hookrightarrow \fg$.
\end{dfn}

\begin{ex} If $G$ is a Lie group and $K$ is a closed subgroup, then
  the pair $({\rm Lie}(G), K)$ is a HC-pair. 
\end{ex}

\begin{dfn} 
A {\em morphism of Harish-Chandra pairs} $(\mathfrak{f}, f) : (\fg, K) \to (\fg',K')$ is 
\begin{itemize}
\item[(i)] a map of Lie algebras $\mathfrak{f} : \fg \to \fg'$ and
\item[(ii)] a map of Lie groups $f : K \to K'$
\end{itemize}
such that the diagram in Lie algebras
\ben
\xymatrix{
{\rm Lie}(K) \ar[r]^-{{\rm Lie}(f)} \ar[d]_-{i} & {\rm Lie}(K') \ar[d]_-{i'} \\
\fg \ar[r]^-{\mathfrak{f}} & \fg' 
}
\een
commutes. 
\end{dfn}

\subsubsection{Modules}

Fix a HC-pair $(\fg,K)$. In this section we set up the notion of a module for~$(\fg,K)$. Below, we discuss modules in the category of vector spaces, but the definition is easily generalized to $\CC$-linear symmetric monoidal categories.

\begin{dfn} \label{hcmod}
A {\em $(\fg,K)$-module} is a vector space $V$ together with 
\begin{itemize}
\item[(i)] a Lie algebra map $\rho_\fg : \fg \to {\rm End}(V)$ and
\item[(ii)] a Lie group map $\rho_K : K \to \GL(V)$ 
\end{itemize}
such that the composition
\ben
\xymatrix{
{\rm Lie}(K) \ar[r]^-{i} & \fg \ar[r]^-{\rho_\fg} & {\rm End}(V)
}
\een
equals ${\rm Lie}(\rho_K)$. 

A {\em morphism of $(\fg,K)$-modules} is a linear map intertwining the actions of $\fg$ and $K$.

Denote the category of $(\fg,K)$-modules by $\Mod_{(\fg,K)}$.
Denote by $\Mod_{(\fg,K)}^{fin}$ the full subcategory whose objects consist of modules whose underlying vector space is finite-dimensional.
\end{dfn}

\subsection{Bundles}

We will need the analog of a torsor for a pair $(\fg,K)$ over a manifold $X$.
Our definitions are structural and so apply equally well to both
smooth and complex manifolds. In the complex case we will need
the notion of a holomorphic $(\fg,K)$-torsor. 

When $X$ is complex, we use $\cO^{hol}(X)$ to denote the space of
holomorphic functions and $\cX^{hol}(X)$ to denote holomorphic vector fields, 
i.e., holomorphic sections of $T^{1,0}X$. 
We let $\Omega^k(X)$ denote the space of smooth $k$-forms and $d_{\dR}$ the de Rham differential. 
If $X$ is complex, then $\Omega^{k,l}(X)$ denotes the smooth $(k,l)$-forms according to the Hodge decomposition. 
We denote by $\Omega^k_{hol}(X)$ the space of holomorphic $k$-forms,
i.e., holomorphic sections of~$\Lambda^k T^{*,(1,0)}X$. Finally, when
we consider a differential graded vector space $(V,\d)$ we let $V^\#$
denote the underlying graded vector space $V$. For instance,
$\Omega^\#(X)$ denotes the graded vector space of differential forms on $X$. 

\begin{dfn} \label{gkbun}
A {\em $(\fg,K)$-principal bundle with flat connection} (or more concisely, flat $(\fg,K)$-bundle) over $X$ is 
\begin{itemize}
\item[(i)] a principal $K$-bundle $P \to X$ and
\item[(ii)] a $K$-invariant $\fg$-valued $1$-form on $P$, $\omega \in \Omega^1(P; \fg)$
\end{itemize}
such that 
\begin{itemize}
\item[(1)] for all $a \in {\rm Lie}(K)$, we have $\omega(\xi_a) = a$ where $\xi_a \in \cX(P)$ denotes the induced vector field and 
\item[(2)] $\omega$ satisfies the Maurer-Cartan equation
\ben
\d_{dR} \omega + \frac{1}{2}[\omega,\omega] = 0
\een
where the bracket is taken in the Lie algebra $\fg$. 
\end{itemize}
\end{dfn}

In particular, if $\fg = \Lie(K)$, then condition (1) encodes the usual notion of a connection on a principal $K$-bundle, while condition (2) says that the connection is {\em flat}, i.e., we have a principal bundle for the discrete group~$K^\delta$ underlying~$K$.

Recall that one can interpret a connection on a principal $K$-bundle as a splitting of the tangent bundle $TP$ into horizontal and vertical components in a $K$-equivariant way. 
Let $T_\pi P$ denote the vertical tangent bundle (i.e., the kernel of the projection map $TP \to \pi^* TX$);
note that $T_\pi P$ is canonically isomorphic to the trivial bundle $\ul{{\rm Lie}(K)}$ over $P$.
A connection $\omega \in \Omega^1(P,{\rm Lie}(K))$ then determines a splitting
\ben
TP = T_\pi P \oplus H_\omega
\een
where $H_\omega \subset TP$ is defined as $\ker \omega$. That is, $H_\omega|_{p}$ is the subspace of $T_p P$ consisting of all vectors $X_p$ such that $\omega(X_p) = 0$ so that $H_{\omega}|_p \cong T_{\pi(p)}X$.   

There is a similar interpretation for Harish-Chandra pairs.
The embedding $i: \Lie(K) \hookrightarrow \fg$ determines a map of trivial bundles $i_P: \ul{{\rm Lie}(K)} \to \ul{\fg}$ over $P$. Define $T_\fg P$ to be the pushout
\ben
\xymatrix{
\ul{{\rm Lie}(K)} \ar[r] \ar[d] & \ul{\fg} \ar[d] \\
TP \ar[r] & T_\fg P
}
\een
in bundles over $P$. Then a $K$-equivariant element $\omega \in \Omega^1(P; \fg)$ satisfying (ii) above is equivalent to a $K$-equivariant splitting
\ben
T_\fg P = \ul{\fg} \oplus H_{\omega}, 
\een
where $H_{\omega}|_p \cong T_{\pi(p)} X$. 

Note that if there is an inclusion of Lie groups $K \hookrightarrow G$ inducing $\Lie(K) \hookrightarrow \Lie(G) = \fg$, 
then this data is a flat $G$-bundle along with a reduction of structure group to a flat $K$-bundle.
This example indicates that the Harish-Chandra version is a useful replacement in the case where the map $i:\Lie(K) \to \fg$ does \emph{not} integrate to a map of Lie groups. 

\begin{ex} The most
important example is the case where $\fg = \Vect$, the Lie algebra
of formal vector fields, and $K = \GL_n$. In fact, $\Vect$ is not the
Lie algebra of any Lie group. The pair $(\Vect,\GL_n)$ is
fundamental for Gelfand-Kazhdan descent, defined in later sections.
\end{ex}

In the complex case it is natural to include the holomorphic structure. 

\begin{dfn}\label{holgk} 
Let $X$ be a complex manifold, $K$ a complex Lie group, and $\fg$ a complex Lie algebra. 
A {\em holomorphic} $(\fg,K)$-principal bundle with flat connection is a holomorphic principal
$K$-bundle $P \to X$ together with a $K$-invariant $\fg$-valued
holomorphic $1$-form $\omega \in \Omega^1_{hol}(P;\fg)$ such that 
\begin{itemize}
\item[(1)] for all $a \in \Lie(K)$, we have $\omega(\xi_a) = a$ where
  $\xi_a \in \cX^{hol}(P)$; 
\item[(2)] $\omega$ satisfies the Maurer-Cartan equation
\ben
\partial \omega + \frac{1}{2} [\omega, \omega] = 0 .
\een 
\end{itemize}
\end{dfn}

\begin{rmk} 
Since $\omega$ in Definition \ref{holgk} is assumed to be holomorphic, i.e., $\dbar \omega = 0$, 
the Maurer-Cartan equation is equivalent to 
\ben
\d_{dR} \omega + \frac{1}{2} [\omega, \omega] = 0
\een 
where $\d_{dR} = \partial + \dbar$ is the full de Rham differential decomposed via the complex structure on $P \to X$. 
Thus, a holomorphic $(\fg,K)$-principal bundle with flat connection $(P \to X, \omega)$ is equivalent to 
an ordinary $(\fg,K)$-principal bundle with flat connection (as in Definition \ref{gkbun}) such that 
the underlying $K$-bundle is holomorphic and $\omega$ is a $(1,0)$-form.
\end{rmk}

We now turn to maps between such structures.

\begin{dfn}
A {\em morphism of $(\fg,K)$-bundles} $(P \to X, \omega) \to (P' \to X', \omega')$ is a map of $K$-principal bundles
\ben
\xymatrix{
P \ar[r]^F \ar[d] & P' \ar[d] \\
X \ar[r]^f & X'
}
\een
such that $F^* \omega' = \omega$. 

Denote the category of flat $(\fg,K)$-bundles by $\Loc_{(\fg,K)}$. 
\end{dfn}

Note that there is a forgetful functor from  $\Loc_{(\fg,K)}$ to ${\rm
  Man}$, the category of manifolds which is either (a) smooth manifolds
with smooth maps or (b) complex manifolds with holomorphic maps. As flat bundles pull back along maps of the underlying manifolds,
we have that this functor is a cartesian fibration.

\subsection{Descent}

Recall the associated bundle construction: given a principal $K$-bundle $\pi : P \to X$ and a finite-dimensional $K$-representation $V$, form the vector bundle
\ben
V_X := P \times^K V
\een
over $X$. 
(People often use the notation $P\times_K V$ instead, but we wish to avoid potential confusion with the fibered product.)

One can view this construction as a pairing between the category of principal $K$-bundles and the category of finite-dimensional $K$-modules, i.e., a functor
\[
- \times^K - : {\rm Bun}_K^\op \times \Mod_K^{fin} \to {\rm VB}
\]
where ${\rm Bun}_K$ is the cartesian fibration whose fiber over a manifold $X$ is the category of $K$-principal bundles on $X$,
where ${\rm VB} \to {\rm Man}$ is the cartesian fibration whose fiber over $X$ is the category of vector bundles on $X$, and 
where $\Mod_K$ denotes the category of $K$-modules. This is a functor
between cartesian fibrations over ${\rm Man}$. This functor exhibits how natural the associated bundle construction is, and 
it can be used to produce natural characteristic classes for $K$-bundles.

In this section, we will produce an analogous functor of {\em Harish-Chandra descent}
\[
\desc: \Loc_{(\fg,K)}^\op \times \Mod_{(\fg,K)}^{fin} \to {\rm VB}_{flat},
\]
where ${\rm VB}_{flat}$ denotes the cartesian fibration whose fiber over a manifold $X$ is the category of flat vector bundles on $X$.
It says, in essence, that each $(\fg,K)$-bundle on $X$ produces a family of local systems on $X$, and 
these are natural under pullback of bundles.
Similarly, each $(\fg,K)$-module produces a functor from flat $(\fg,K)$-bundles to local systems over the site of all manifolds.

\begin{rmk}
The construction is often termed \emph{Harish-Chandra localization}
(see \cite{BK} \cite{BL}),
but this terminology occasionally led to possible ambiguities due to other uses of of the word ``localization,''
so we use ``descent.''
\end{rmk}

We will also describe the characteristic map, which is a natural transformation 
\[
{\rm char} : \clie^*(\fg,K; -) \Rightarrow \Omega^*(-,\desc(-)),
\]
where $\clie^*(\fg,K; -)$ denotes the relative Lie algebra cochains functor (it is independent of the bundle variable) and 
where $\Omega^*(-,\desc(-))$ denotes the de Rham complex of the flat bundle produced by $\desc$.
This natural transformation encodes the secondary characteristic classes of these flat bundles.

\subsubsection{Basic forms}

There is a model for the associated bundle construction that is useful for our purposes. 
Let $V$ be a finite-dimensional $K$-representation. 
Denote by $\ul{V}$ the trivial vector bundle on $P$ with fiber $V$. 
Sections of this bundle $\Gamma_P(V)$ have the structure of a $K$-representation by
\ben
A \cdot (f\tensor v) := (A \cdot f) \tensor (A \cdot v) \;\; , \;\; A \in K, \; f \in \cO(P)\; , v \in V .
\een
Every $K$-invariant section $f : P \to \ul{V}$ induces a section $s(f): X \to V_X$,
where the value of $s(f)$ at $x \in X$ is the $K$-equivalence class $[(p,f(p)]$, with $p \in \pi^{-1}(x) \cong K$.
That is, there is a natural map 
\ben
s : \Gamma_P(\ul{V})^K \to \Gamma_X(V_X) 
\een
and it is an isomorphism of $\cO(X)$-modules. A $K$-invariant section $f$ of $\ul{V} \to P$ also satisfies the infinitesimal version of invariance: 
\ben
(Y \cdot f)\tensor v + f \tensor {\rm Lie}(\rho)(Y) \cdot v = 0 
\een
for any $Y \in {\rm Lie}(K)$.

There is a similiar statement for differential forms with values in the bundle $V_X$. Let $\Omega^k(P ; \ul{V}) = \Omega^k(P) \tensor V$ denote the space of $k$-forms on $P$ with values in the trivial bundle $\ul{V}$. Given $\alpha \in \Omega^1(X ; V_X)$, its pull-back along the projection $\pi: P \to X$ is annihilated by any vertical vector field on $P$. In general, if $\alpha \in \Omega^k(X; V_X)$, then $i_Y(\pi^*\alpha) = 0$ for all $Y \in {\rm Lie}(K)$.

\begin{dfn} A $k$-form $\alpha \in \Omega^k(P; \ul{V})$ is called {\em basic} if 
\begin{itemize}
\item[(i)] it is $K$-invariant: $L_Y \alpha + \rho(Y) \cdot \alpha = 0 $ for all $Y \in {\rm Lie}(K)$ and
\item[(ii)] it vanishes on vertical vector fields: $i_Y \alpha = 0$ for all $Y \in {\rm Lie}(K)$. 
\end{itemize}
\end{dfn}

Denote the subspace of basic $k$-forms by $\Omega^k(P; \ul{V})_{bas}$. Just as with sections, there is a natural isomorphism
\ben
s : \Omega^k(P; \ul{V})_{bas} \xto{\cong} \Omega^k(X; V_X) 
\een
between basic $k$-forms and $k$-forms on $X$ with values in the associated bundle.
In fact, $\Omega^{\#}(P; \ul{V})_{bas}$ forms a graded submodule of $\Omega^{\#}(P; \ul{V})$ and the isomorphism $s$ extends to an isomorphism of graded modules $\Omega^{\#}(P; \ul{V})_{bas} \cong \Omega^{\#}(X; V_X)$.

It is manifest that this construction of basic forms is natural in maps of $(\fg,K)$-bundles: basic forms pull back to basic forms along maps of bundles.

\subsubsection{}

Fix a $(\fg,K)$-bundle $P \to X$ with connection one-form $\omega \in \Omega^1(P; \fg)$. Fix a $(\fg,K)$-module $V$ with action maps $\rho_K$ and $\rho_{\fg}$. The subalgebra of basic forms
\ben
\Omega^{\#}(P; \ul{V})_{bas} \subset \Omega^{\#}(P; \ul{V})
\een
only uses the data of the $K$-representation. The $\fg$-module structure induces an operator
\ben
\rho_\fg(\omega) : \Omega^k(P; \ul{V}) \to \Omega^{k+1}(P; \ul{V})
\een
for each $k$. Let $\nabla^{P,V}$ denote the operator
\ben
\nabla^{P,V} := \d_{dR,P} + \rho_\fg(\omega) : \Omega^k(P; \ul{V}) \to \Omega^{k+1}(P; \ul{V})
\een 
for each~$k$. 

A direct calculation verifies the following. 

\begin{lemma} 
The operator $\nabla^{P,V}$ is a differential on the submodule of basic forms.
Under the isomorphism $s: \Omega^{\#}(P; \ul{V})_{bas} \cong \Omega^{\#}(X;
V_X)$, the cochain complex $(\Omega^{\#}(P; \ul{V})_{bas}, \nabla^V)$ is a
dg module over $\Omega^*(X)$. 
\end{lemma}

\begin{dfn}\label{def desc}
The {\em associated flat vector bundle} to the flat $(\fg,K)$-bundle $P \to X$ and the finite-dimensional $(\fg,K)$-representation $V$ is 
\[
\desc((P\to X),V) := (P \times^K V, \nabla^{P,V}),
\]
namely the associated vector bundle on $X$ and its flat connection.
Its {\em de Rham complex} is
\[
\bdesc((P \to X), V) := \left(\Omega^*(P; \ul{V})_{bas}, \nabla^{P,V}\right),
\]
whose zeroeth cohomology is the horizontal sections of the local system.
\end{dfn}

As the construction of the flat connection $\nabla^{P,V}$ intertwines naturally with maps of $(\fg,K)$-bundles, 
we obtain the following functors.

\begin{dfn}
The {\em $(\fg,K)$-descent functor} is the functor 
\[
\desc: \Loc_{(\fg,K)}^\op \times \Mod^{fin}_{(\fg,K)} \to {\rm VB}_{flat}
\]
sending $(P \to X, V)$ to $(V_X,\nabla^{P,V})$.
There is a closely related functor
\[
\ddesc: \Loc_{(\fg,K)}^\op \times \Mod^{fin}_{(\fg,K)} \to \Mod_{\Omega^*}
\]
sending $(P \to X, V)$ to the de Rham complex of $\desc((P \to X,V))$. 
\end{dfn}

To every flat vector bundle we can associate a local system by taking
the horizontal sections. We denote by $\sdesc$ the
composition of the functor $\desc$ with taking horizontal
sections. Explicitly, $\sdesc$ is the zeroeth cohomology of the de
Rham complex of the flat vector bundle given by descent. In other words, it is
the zeroth cohomology of the complex~$\left(\Omega^*(P;  \ul{V})_{bas}, \nabla^{P,V}\right)$.

In the case of a holomorphic $(\fg, K)$-bundle with flat connection $(P \to X, \omega)$,
the $(0,1)$-component of the connection $\nabla^{P,V}$ agrees with the $\dbar$ operator
\ben
(\nabla^{P,V})^{0,1} = \dbar_P : \Omega^{0} (P; \ul{V})_{bas} \to \Omega^{0,1}(P;\ul{V})_{bas} .
\een
Hence the horizontal sections are also holomorphic. 

\begin{ex}
Let $G$ be a Lie group and let $K \subset G$ be a closed Lie subgroup. 
Then the $K$-principal bundle $G \to G/K$ has the natural structure of a $(\fg,K)$-principal bundle where $\fg = {\rm Lie}(G)$. 
When $G$ is compact, there is a quasi-isomorphism $\bdesc(G \to G/K, V) \simeq \clie^*(\fg, K ; V)$ for any $G$-representation $V$. 
This quasi-isomorphism is a relative version of the standard fact that the differential forms on $G$ are quasi-isomorphic to the absolute Chevalley-Eilenberg complex~$\clie^*(\fg)$. 
\end{ex}

\begin{ex} 
Let $K$ be a Lie group and let $\fk$ be its Lie algebra. 
Then $(\fk,K)$ is a HC-pair and we have an equivalence of categories
\ben
\Mod^{fin}_{(\fk,K)} \cong {\rm Rep}^{fin}_K .
\een
Let $P \to X$ be a principal $K$-bundle and $\omega \in \Omega^1(P;\fk)$ a {\em flat} connection (in the traditional sense). 
Then the functor
\ben
\ddesc((P,\omega),-) : {\rm Mod}^{fin}_{(\fk,K)} \to {\rm Mod}_{\Omega^*_X}
\een
is equivalent to the functor ${\rm Rep}_{K}^{fin} \to
\Mod_{\Omega^*_X}$ that sends a $K$-representation $V$ to the de Rham
complex of the associated bundle $V_X = P \times^K V$ equipped with
its induced flat connection, i.e., $V \mapsto \Omega^*(X ; V_X)$.
\end{ex}

\begin{rmk}
\label{descent for infinite dimensional things}
We have described these constructions for finite-dimensional representations,
but they make sense with an infinite-dimensional representation $V$ 
\emph{provided} one knows how smooth manifolds map into $V$.
Given that data, one knows how to write down functions (or differential forms) on $P$ with values in $V$.
In many examples, the vector space $V$ comes equipped with that information.
For instance, every locally convex topological vector space has it,
as do bornological or convenient vector spaces.
A systematic discussion of these issues can be found in \cite{KrieglMichor},
and an overview with close ties to the examples used here can be found in Appendix B of~\cite{CG1}.
\end{rmk}

\subsection{The characteristic map}

Recall that on a principal $K$-bundle $P \to X$ with connection one-form $\omega \in \Omega^1(P,\ul{\Lie(K)})$, 
the one-form provides a linear map $\omega^*: \Lie(K)^* \to \Omega^1(P)$.
If the connection is flat (i.e., satisfies the Maurer-Cartan equation), 
then $\omega^*$ extends to a map of commutative dg algebras
\[
\omega^*: \clie^*(\Lie(K)) \to \Omega^*(P),
\]
which provides some kind of characteristic classes for the flat $K$-bundle $P$.

We now adapt this construction to the Harish-Chandra setting.
In this case, the connection one-form $\omega$ lives in $\Omega^1(P, \ul{\fg})$ and as it is flat,
it provides a map of commutative dg algebras $\omega^*: \clie^*(\fg) \to \Omega^*(P)$.This map admits an important refinement: since $\omega$ is $K$-invariant, it induces a map
\[
\omega^*: \clie^*(\fg,K) \to \Omega^*(P)_{bas}.
\]
This construction extends to associated bundles, so that for $V$ a $(\fg,K)$-module, there is a map
\[
{\rm char}^{P,V} : \clie^*(\fg, K; V) \to \Omega^*(\desc((P,\omega),V)), 
\]
which provides some kind of characteristic classes for flat $(\fg,K)$-bundles.

As these constructions manifestly intertwine with pullbacks of bundles, we have the following.

\begin{dfn}
The {\em characteristic map} is the natural transformation
\[
{\rm char} : \clie^*(\fg,K; -) \Rightarrow \Omega^*(-, \desc(-,-))
\]
between the relative Lie algebra cohomology of a $(\fg,K)$-module and 
the de Rham complex of its associated local system along a flat $(\fg,K)$-bundle .
\end{dfn}

\section{Formal vector bundles and Gelfand-Kazhdan descent}\label{sec
  gk descent}

In this section we review the theory of Gelfand-Kazhdan formal geometry and its use in natural constructions in differential geometry,
organized in a manner somewhat different from the standard
approaches. We emphasize the role of the frame bundle and jet bundles.
We conclude with a treatment of the Atiyah class, which may be our only novel addition (although unsurprising) to the formalism.

We remark that from hereon we will work with complex manifolds and holomorphic vector bundles.
 
\subsection{A Harish-Chandra pair for the formal disk}

Let $\hO_n$ denote the algebra of formal power series 
\ben
\CC \llbracket t_1,\ldots,t_n \rrbracket,
\een 
which we view as ``functions on the formal $n$-disk $\hD^n$.'' 
It is filtered by powers of the maximal ideal $\fm_n = (t_1,\ldots,t_n)$, and it is the limit of the sequence of artinian algebras
\[
\cdots \to \hO_n/(t_1,\ldots,t_n)^k \to \cdots \hO_n/(t_1,\ldots,t_n)^2 \to \hO_n/(t_1,\ldots,t_n) \cong \CC.
\] 
One can use the associated adic topology to interpret many of our constructions, but we will not emphasize that perspective here.

We use $\Vect$ to denote the Lie algebra of derivations of $\hO_n$, which consists of first-order differential operators with formal power series coefficients:
\[
\Vect = \left\{ \sum_{i =1 }^n f_i \frac{\partial}{\partial t_i} \,:\, f_i \in \hO_n\right\}.
\]
The group $\GL_n$ also acts naturally on $\hO_n$: for $M \in \GL_n$ and $f \in \hO_n$,
\[
(M \cdot f)(t) = f (Mt),
\]
where on the right side we view $t$ as an element of $\CC^n$ and let $M$ act linearly.
In other words, we interpret $\GL_n$ as acting ``by diffeomorphisms'' on $\hD^n$ and then use the induced pullback action on functions on $\hD^n$.
The actions of both $\Vect$ and $\GL_n$ intertwine with multiplication of power series, 
since ``the pullback of a product of functions equals the product of the pullbacks.''

\subsubsection{Formal automorphisms}

Let $\Aut_n$ be the group of filtration-preserving automorphisms of the algebra $\hO_n$,
which we will see is a pro-algebraic group.
Explicitly, such an automorphism $\phi$ is a map of algebras that preserves the maximal ideal, 
so $\phi$ is specified by where it sends the generators $t_1$, \dots, $t_n$ of the algebra.
In other words, each $\phi \in \Aut_n$ consists of an $n$-tuple $(\phi_1,\ldots,\phi_n)$ 
such that each $\phi_i$ is in the maximal ideal generated by $(t_1,\ldots,t_n)$ and such that there exists an $n$-tuple $(\psi_1,\ldots,\psi_n)$ 
where the composite
\[
\psi_j(\phi_1(t),\ldots,\phi_n(t)) = t_j
\]
for every $j$ (and likewise with $\psi$ and $\phi$ reversed).
This second condition can be replaced by verifying that the Jacobian matrix
\[
Jac(\phi) = (\partial \phi_i/\partial t_j) \in {\rm Mat}_n(\hO_n)
\]
is invertible over $\hO_n$, by a version of the inverse function theorem.

Note that this group is far from being finite-dimensional, so it does not fit immediately into the setting of HC-pairs described above. 
It is, however, a {\em pro}-Lie group in the following way. 
As each $\phi \in \Aut_n$ preserves the filtration on $\hO_n$, it induces an automorphism of each partial quotient $\hO_n/\fm_n^k$.
Let $\Aut_{n,k}$ denote the image of $\Aut_n$ in $\Aut(\hO_n/\fm_n^{k+1})$; this group $\Aut_{n,k}$ is clearly a quotient of $\Aut_n$.
Note, for instance, that $\Aut_{n,1} = \GL_n$.
Explicitly, an element $\phi$ of ${\rm Aut}_{n,k}$ is the collection of $n$-tuples $(\phi_1,\ldots,\phi_n)$ 
such that each $\phi_i$ is an element of $\fm_n/\fm_n^{k+1}$ and such that the Jacobian matrix $Jac(\phi)$ is invertible in $\hO_n/\fm_n^{k+1}$.
The group ${\rm Aut}_{n,k}$ is manifestly a finite dimensional Lie group, as the quotient algebra is a finite-dimensional vector space. 
 
The group of automorphisms $\Aut_n$ is the pro-Lie group associated with the natural sequence of Lie groups
\ben
\cdots \to \Aut_{n,k} \to \Aut_{n,k-1} \to \cdots \to \Aut_{n,1} = \GL_n.
\een
Let $\Aut_n^+$ denote the kernel of the map $\Aut_n \to \GL_n$ so that we have a short exact sequence
\ben
1 \to \Aut_n^+ \to \Aut_n \to \GL_n \to 1 .
\een
In other words, for an element $\phi$ of $\Aut_n^+$, each component
$\phi_i$ is of the form $t_i + \cO(t^2)$. The group $\Aut_n^+$ is
pro-nilpotent, hence contractible. 

The Lie algebra of $\Aut_n$ is {\em not} the Lie algebra of formal
vector fields $\Vect$. A direct
calculation shows that the Lie algebra of $\Aut_n$ is the Lie algebra $\Vectz \subset \Vect$ of formal vector fields with zero constant coefficient (i.e., that vanish at the origin of $\hD^n$). 

Observe that the group $\GL_n$ acts on the Lie algebra $\Vect$ by the obvious linear ``changes of frame.''
The Lie algebra $\Lie({\GL_n}) = \fgl_n$ sits inside $\Vect$ as the linear vector fields
\ben
\left\{\sum_{i,j} a^j_i t_i \frac{\partial}{\partial t_j} \; : \; a^{i}_j \in \CC \right\}.
\een 
We record these compatibilities in the following statement.

\begin{lemma} 
The pair $(\Vect, \GL_n)$ form a Harish-Chandra pair.
\end{lemma}
\begin{proof} The only thing to check is that the differential of the
  action of $\GL_n$ corresponds with the adjoint action of $\fgl_n
  \subset \Vect$ on formal vector fields. This is by construction. 
\end{proof}

\subsection{The coordinate bundle}

In this section we review the central object in the Gelfand-Kazhdan
picture of formal geometry: the coordinate bundle.

\subsubsection{}

Given a complex manifold, its {\em coordinate space} $X^{coor}$ is the (infinite-dimensional) space parametrizing holomorphic formal coordinate systems of $X$. 
(It is a pro-complex manifold, as we'll see.) 
Explicitly, a point in $X^{coor}$ consists of a point $x \in X$ 
together with an isomorphism of completed commutative algebras 
\[
\phi: \sO_{X,x}^\wedge \to \CC[[t_1,\ldots,t_n]] = \hO_n,
\]
where $\sO_{X,x}^\wedge$ denotes the completion $\lim_{\leftarrow} \sO_{X,x}/\fm_x^k$ of the germ at $x$ of holomorphic functions with respect to powers of the ideal $\fm_x$ of functions vanishing at~$x$.  
Intuitively, $\phi$ corresponds to an embedding of the formal disk into $X$, sending the base point to~$x$.

There is a canonical projection map $\pi^{coor} : X^{coor} \to X$ by remembering only the underlying point in $X$. 
The group $\Aut_n$ acts on $X^{coor}$ by ``change of coordinates," 
i.e., by precomposing a  $\phi$ with an automorphism of the disk around the origin in $\CC^n$.
This action identifies $\pi^{coor}$ as a principal bundle for the pro-Lie group $\Aut_n$. 

One way to formalize these ideas is to realize $X^{coor}$ as a limit of finite-dimensional complex manifolds. 
Let $X_k^{coor}$ be the space consisting of points $(x, [\phi]_k)$, 
where $\phi$ is a formal holomorphic coordinate system, as above, and $[-]_k$ denotes the projection on $\CC[[[t_1,\ldots,t_n]]/(t_1,\ldots,t_n)^{k+1}$. 
Let $\pi_k^{coor} : X^{coor}_k \to X$ be the projection. 
By construction, the finite-dimensional complex Lie group $\Aut_{n,k}$ acts on the fibers of the projection freely and transitively 
so that $\pi_k^{coor}$ is a holomorphic principal $\Aut_{n,k}$-bundle. The bundle $X^{coor} \to X$ is the limit of the sequence of holomorphic principal bundles on X
\ben
\xymatrix{
\cdots \ar[r] & X^{coor}_k \ar[r] \ar[drrrr]_-{\pi_k^{coor}} & X^{coor}_{k-1} \ar[drrr]^-{\pi_{k-1}^{coor}} \ar[r] & \cdots \ar[r] & X_2^{coor} \ar[dr]^{\pi_2^{coor}} \ar[r] & X_1^{coor} \ar[d]^-{\pi_1^{coor}} \\ 
 & & & & & X .
}
\een
In particular, note that the $\GL_n = \Aut_{n,1}$-bundle $\pi_1^{coor} : X^{coor}_1 \to X$ is the frame bundle
\ben
\pi^{fr} : {\rm Fr}_X \to X,
\een
i.e., the principal bundle associated to the holomorphic tangent bundle of $X$.

\subsubsection{The Grothendieck connection} 

We can also realize the Lie algebra $\Vect$ as an inverse limit. 
Recall the filtration on $\Vect$ by powers of the maximal ideal $\fm_n$ of $\hO_n$. 
Let ${\rm W}_{n,k}$ denote the quotient $\Vect / \fm_n^{k+1} \Vect$. 
For instance, ${\rm W}_{n,1} = \mathfrak{aff}_n = \CC^n \ltimes \fgl_n$, the Lie algebra of affine transformations of $\CC^n$. 
We have $\Vect = \lim_{k \to \infty} {\rm W}_{n,k}$. 

The Lie algebra of $\Aut_{n,k}$ is
\[
{\rm W}_{n,k}^0 := \fm_n \cdot \Vectz /\fm_n^{k+1} \Vectz .
\]
That is, the Lie algebra of vector fields vanishing at zero modulo the $(k+1)$th power of the maximal ideal. Thus, the principal $\Aut_{n,k}$-bundle $X_{k}^{coor} \to X$ induces an exact sequence of holomorphic tangent spaces
\ben
{\rm W}_{n,k}^0 \to T_{(x,[\varphi]_k)}X^{coor} \to T_x X;
\een
by using $\varphi$, we obtain a canonical isomorphism of tangent spaces $\CC^n \cong T_0 \CC^n \cong T_x X$. Combining these observations, we obtain an isomorphism
\ben
{\rm W}_{n,k} \cong T_{(x,[\varphi]_k)} X^{coor}_k .
\een
In the limit $k \to \infty$ we obtain an isomorphism $\Vect \cong T_{(x,[\varphi]_\infty)} X^{coor}$ at each point.
These isomorphisms glue together to give the following. 

\begin{dfn}
Let
\ben
\theta : \Vect \to \cX^{hol}(X^{coor}) .
\een
denote the Lie algebra morphism encoding the canonical action of $\Vect$ on $X^{coor}$ by
holomorphic vector fields.
(See Section 5 of \cite{NT} and Section 3 of \cite{CF2} for further discussions of this kind of construction.)
\end{dfn}

The inverse of the map $\theta$ provides a connection one-form
\ben
\omega^{coor} \in \Omega^1_{hol}(X^{coor}; \Vect),
\een
which we call the {\em universal Grothendieck connection} on $X$. 
As $\theta$ is a Lie algebra homomorphism, $\omega^{coor}$ satisfies the Maurer-Cartan equation
\be\label{mc}
\partial \omega^{coor} + \frac{1}{2} [\omega^{coor},\omega^{coor}] = 0 .
\ee
Note that the proposition ensures that this connection is universal on all complex manifolds of dimension $n$ 
and indeed pulls back along local biholomorphisms.

\begin{rmk} 
Both the pair $(\Vect, \Aut_n)$ and the bundle $X^{coor} \to X$ together
with $\omega^{coor}$ do not fit in our model for general
Harish-Chandra descent above. 
They are, however, objects in a larger category of pro-Harish-Chandra pairs and pro-Harish-Chandra bundles, respectively. 
We do not develop this theory here, but it is inherent in the work of~\cite{BK}.  
Indeed, by working with well-behaved representations for the pair $(\Vect,\Aut_n)$, 
Gelfand, Kazhdan, and others use this universal construction to produce many of the natural constructions in differential geometry.
As we remarked earlier, it is a kind of refinement of tensor calculus.
\end{rmk}

\subsubsection{A Harish-Chandra structure on the frame bundle}

\def\Sect{{\rm Sect}}
\def\Fr{{\rm Fr}}
\def\Exp{{\rm Exp}}

Although the existence of the coordinate bundle
$X^{coor}$ is necessary in the remainder of this paper, it is convenient for us to use it in a rather
indirect way. Rather, we will work with the frame bundle ${\rm Fr}_X \to X$ equipped with the structure of a module for the Harish-Chandra pair $(\Vect, \GL_n)$. 
The $\Vect$-valued connection on $\Fr_X$ is induced from the Grothendieck connection above.

\begin{dfn}\label{fmlexp} 
Let $\Exp (X)$ denote the quotient $X^{coor} / \GL_n$. 
A $C^\infty$-section of $\Exp(X)$ over $X$ is called a {\em formal exponential}. 
\end{dfn}

\begin{rmk}
Following Section 4 of \cite{kapranov1999}, 
we can realize $\Exp (X)$ as the inverse limit of finite dimensional manifolds $\Exp_k (X) = X^{coor}_k / \GL_n$.
Moreover, we can equip $\Exp(X)$ with the structure of a principal $\Aut_n^+$-bundle over $X$ in the following way.
Consider the short exact sequence of pro-Lie groups
\ben
1 \to \Aut_n^+ \to \Aut_n \to \GL_n \to 1 .
\een
There is a splitting of this determined by the choice of coordinates on the formal disk which exhibits an isomorphism
\[
\Aut_n = \Aut_{n}^+ \rtimes \GL_n 
\] 
and a bijection of {\em sets} $q : \Aut_n^+ \xto{\cong} \Aut_n / \GL_n$. 

Further, there is an action of $\Aut_n$ on $\Aut_n^+$ defined by
\[
f \cdot p := Jac(f) \circ p \circ Jac(f)^{-1}
\]
which makes $q$ a $\Aut_n$-equivariant isomorphism. 
Using this isomorphism, we equip $X^{coor} / \GL_n$ with the desired $\Aut_n^+$ structure.
\end{rmk}

Note that $\Aut_n^+$ is contractible, and so smooth sections always exist. 
A formal exponential is useful because it equips the frame bundle with a $(\Vect,\GL_n)$-module structure, as follows.

\begin{prop} \label{gauge equiv}
A formal exponential $\sigma$ pulls back to a $\GL_n$-equivariant map $\tilde{\sigma} : \Fr_X \to X^{coor}$,
and hence equips $(\Fr_x, \sigma^* \omega^{coor})$ with the structure
of a principal $(\Vect,\GL_n)$-bundle with flat connection.
Moreover, any two choices of formal exponential determine $(\Vect,\GL_n)$-structures on $X$ that are gauge-equivalent. 
\end{prop}

For a full proof, see \cite{NT}, \cite{nest1995}, or \cite{khors} but the basic idea is easy to explain.

\begin{proof}[Sketch of proof]
The first assertion is tautological, since the data of a section is equivalent to such an equivariant map, but we explicate the underlying geometry.
A map $\rho : \Fr_X \to X^{coor}$ assigns to each pair  $(x, \mathbf{y}) \in \Fr_X$,
with $x \in X$ and $\mathbf{y} : \CC^n \xto{\cong} T_x X$ a linear frame,
an $\infty$-jet of a biholomorphism $\phi: \CC^n \to X$ such that $\phi(0) = x$ and $D\phi(0) = \mathbf{y}$.
Being $\GL_n$-equivariant ensures that these biholomorphisms are related by linear changes of coordinates on $\CC^n$.
In other words, a $\GL_n$-equivariant map $\tilde{\sigma}$ describes how each frame on $T_x X$ exponentiates to a formal coordinate system around $x$,
and so the associated section $\sigma$ assigns a formal exponential map $\sigma(x) \colon T_x X \to X$ to each point $x$ in $X$.
(Here we see the origin of the name ``formal exponential.'')

The second assertion would be immediate if $X^{coor}$ were a complex manifold, since the flat bundle structure would pull back,
so all issues are about carefully working with pro-manifolds.

The final assertion is also straightforward: the space of sections is contractible since $\Aut_n^+$ is contractible, 
so one can produce an explicit gauge equivalence.
\end{proof}
 
\begin{rmk} 
In \cite{willwacher} Willwacher provides a description of the space $\Exp(X)$ of {\em all} formal exponentials. He shows that it is isomorphic to the space of pairs $(\nabla_0, \Phi)$
where $\nabla_0$ is a torsion-free connection on $X$ for $T_X$ and $\Phi$ is a section of the bundle
\ben
\Fr_X \times_{\GL_n} {\rm W}_n^3
\een
where ${\rm W}_n^3 \subset \Vect$ is the subspace of formal vector fields whose coefficients are at least cubic. 
In particular, every torsion-free affine connection determines a formal exponential. The familiar case above that produces a formal coordinate from a connection corresponds to choosing the zero vector field. 
\end{rmk}

\begin{dfn}
A {\em Gelfand-Kazhdan structure} on the frame bundle $\Fr_X\to X$ of a complex manifold $X$ of dimension $n$ is a formal exponential $\sigma$, 
which makes $\Fr_X$ into a flat $(\Vect,\GL_n)$-bundle with connection one-form $\omega^\sigma$, 
the pullback of $\omega^{coor}$ along the $\GL_n$-equivariant lift $\tilde{\sigma} : \Fr_X \to X^{coor}$.
\end{dfn}

\begin{ex} 
Consider the case of an open subset $U \subset \CC^n$. 
There are thus natural holomorphic coordinates $\{z_1,\ldots,z_n\}$ on $U$. 
These coordinates provides a natural choice of a formal exponential. 
Moreover, with respect to the isomorphism
\ben
\Omega^1_{hol}(\Fr_U ; \Vect)^{\GL_n} \cong \Omega^1_{hol}(U ; \Vect),
\een
we find that the connection 1-form has the form
\ben
\omega^{coor} = \sum_{i=1}^n \d z_i \tensor \frac{\partial}{\partial t_i},
\een 
where the $\{t_i\}$ are the coordinates on the formal disk $\hD^n$.
\end{ex} 

A Gelfand-Kazhdan structure allows us to apply a version of Harish-Chandra descent, which will be a central tool in our work.

Although we developed Harish-Chandra descent on all flat $(\fg,K)$-bundles, 
it is natural here to restrict our attention to manifolds of the same dimension,
as the notions of coordinate and affine bundle are dimension-dependent.
Hence we replace the underlying category of all complex manifolds by a more restrictive setting.

\begin{dfn}
Let $\Hol_n$ denote the category whose objects are complex manifolds of dimension $n$ and whose morphisms are local biholomorphisms.
In other words, a map $f: X \to Y$ in $\Hol_n$ is a map of complex manifolds such that each point $x \in X$ admits a neighborhood $U$ on which $f|_U$ is biholomorphic with $f(U)$.
\end{dfn}

There is a natural inclusion functor $i : \Hol_n \to {\rm CplxMan}$ (not fully faithful) and the frame bundle $\Fr$ defines a section of the fibered category $i^*\VB$,
since the frame bundle pulls back along local biholomorphisms.
For similar reasons, the coordinate bundle is a pro-object in $i^*\VB$.

\begin{dfn}
Let $\GK_n$ denote the category fibered over $\Hol_n$ whose objects are a Gelfand-Kazhdan structure 
--- that is, a pair $(X, \sigma)$ of a complex $n$-manifold and a formal exponential ---
and whose morphisms are simply local biholomorphisms between the underlying manifolds.
\end{dfn}

Note that the projection functor from $\GK_n$ to $\Hol_n$ is an equivalence of categories, since the space of formal exponentials is affine.

\subsection{The category of formal vector bundles}

For most of our purposes, it is convenient and sufficient to work with a small category of $(\Vect,\GL_n)$-modules 
that is manifestly well-behaved and whose localizations appear throughout geometry in other guises, 
notably as $\infty$-jet bundles of vector bundles on complex manifolds.
(Although it would undoubtedly be useful, we will not develop here the general theory of modules for the Harish-Chandra pair $(\Vect,\GL_n)$, 
which would involve subtleties of pro-Lie algebras and their representations.)

We first start by describing the category of $(\Vect, \GL_n)$-modules
that correspond to modules over the structure sheaf of a manifold. Note that $\hO_n$ is the quintessential example of a commutative algebra object in the symmetric monoidal category of $(\Vect,\GL_n)$-modules, 
for any natural version of such a category. We consider modules that
have actions of both the pair and the algebra $\hO_n$ with obvious
compatibility restrictions. 

\begin{dfn} A {\em formal $\hO_n$-module} is a
  vector space $\cV$ equipped with
\begin{itemize}
\item[(i)] the structure of a $(\Vect, \GL_n)$-module;
\item[(ii)] the structure of a $\hO_n$-module;
\end{itemize}
such that 
\begin{itemize}
\item[(1)] for all $X \in \Vect$, $f \in \hO_n$ and $v \in \cV$ we
  have $X(f \cdot v) = X(f) \cdot v + f \cdot (X \cdot v)$;
\item[(2)] for all $A \in \GL_n$ we have $A (f \cdot v) = (A \cdot f) \cdot (A \cdot v)$,  where $A$ acts on $f$ by a linear change of frame.
\end{itemize}
A morphism of formal $\hO_n$-modules is a $\hO_n$-linear map of
$(\Vect, \GL_n)$-modules $f : \cV \to \cV'$. We denote this category
by $\Mod_{(\Vect, \GL_n)}^{\cO_n}$. 
\end{dfn}

It is useful to bear in mind that such an object is much like a vector bundle equipped with a flat connection,
due to the action of vector fields by derivations.
More properly this category behaves much like $D$-modules (i.e., modules over the ring $D$ of differential operators).
For instance, just as the category of $D$-modules is symmetric monoidal via tensor over $\cO$, we have the following result.

\begin{lemma}
The category $\Mod^{\cO_n}_{(\Vect, \GL_n)}$ is symmetric monoidal with respect to tensor over $\hO_n$.
\end{lemma} 

\begin{proof}
The category of $\hO_n$-modules is clearly symmetric monoidal by tensoring over $\hO_n$. We simply need to verify that the Harish-Chandra module structures extend in a natural way, but this is clear.
\end{proof}

We will often restrict ourselves to considering Harish-Chandra modules as above that are free as underlying $\hO_n$-modules. 
Indeed, let
\ben
\VB_n \subset \Mod_{(\Vect, \GL_n)}^{\cO_n}
\een
be the full subcategory spanned by objects that are free and finitely generated as underlying $\hO_n$-modules,
so we refer to this category as {\em formal vector bundles}.

The category of formal $\hO_n$-modules has a natural symmetric monoidal structure by tensor product over~$\hO$. The Harish-Chandra action is extended by
\[
X \cdot (s \otimes t) = (X s) \otimes t + s \otimes (Xt). 
\]
This should not look surprising; it is the same formula for tensoring
$D$-modules over~$\cO$. 

The internal hom $\Hom_{\hO}(\cV,\cW)$ also provides a vector bundle on the formal disk, 
where the Harish-Chandra action is extended by
\[
(X \cdot \phi)(v) = X \cdot (\phi(v)) - \phi(X\cdot v). 
\]
Observe that for any $D$-module $M$, we have an isomorphism
\[
\Hom_{D}(\hO, M) \cong \Hom_{\Vect}(\CC, M)
\]
since a map of $D$-modules out of $\hO$ is determined by where it sends the constant function~1. 
Hence we find that there is a quasi-isomorphism 
\[
\RR\Hom_{D}(\hO, \cV) \simeq \clie^*(\Vect ; \cV),
\]
or more accurately a zig-zag of quasi-isomorphisms. Here
$\clie^*(\Vect ; \cV)$ is the continuous cohomology of $\Vect$ with
coefficients in $\cV$. This is known as the {\em Gelfand-Fuks}
cohomology of $\cV$ and is what we use for the remainder of the
paper. 

This relationship extends to the $\GL_n$-equivariant setting as well, giving us the following result.

\begin{lemma}
There is a quasi-isomorphism
\[
\clie^*(\Vect,\GL_n; \cV) \simeq \RR \Hom_D(\hO,\cV)^{\GL_n-{\rm eq}},
\]
where the superscript $\GL_n-{\rm eq}$ denotes the $\GL_n$-equivariant maps.
\end{lemma}

\begin{rmk}
One amusing way to understand this category is as Harish-Chandra descent to the formal $n$-disk itself. 
Consider the frame bundle $\widehat{\Fr} = \hD^n \times \GL_n \to \hD^n$ of the formal $n$-disk itself, 
which possesses a natural flat connection via the Maurer-Cartan form $\omega_{MC}$ on $\GL_n$. 
Let $\rho: \GL_n \to \GL(V)$ be a finite-dimensional representation. 
Then the subcomplex of $\Omega^*(\widehat{\Fr})\otimes V$ given by the basic forms is isomorphic to
\[
\left(\Omega^*(\hD^n) \otimes V, \d_{dR} + \rho(\omega_{MC}) \right).
\]
This equips the associated bundle $\widehat{\Fr}\times^{\GL_n} V$ with a flat connection and 
hence makes its sheaf of sections a $D$-module on the formal disk.
\end{rmk}

Many of the important $\hO_n$-modules we will consider simply come from linear tensor representations of $\GL_n$. 
Given a finite-dimensional $\GL_n$-representation $V$, we construct a $\hO_n$-module $\cV \in \VB_n$ as follows. 

Consider the decreasing filtration of $\Vect$ by vanishing order of jets 
\ben
\cdots \subset \fm^{2}_n \cdot {\rm W}_{n} \subset \fm^1_n \cdot {\rm W}_n \subset {\rm W}_n .
\een 
The induced map $\fm_n^1 \cdot \Vect \to \fm_n^1 \cdot \Vect / \fm_n^2
\cdot \Vect \cong \fgl_n$ allows us to restrict $V$ to a $(\fm_n^1 \cdot
\Vect)$-module. 
We  then coinduce this module along the inclusion $\fm^1 \cdot \Vect
\subset \Vect$ to get a $\Vect$-module $\cV = \Hom_{\fm_n^1 \cdot \Vect}(U(\Vect),V)$. 
There is an induced action of $\GL_n$ on $\cV$. Indeed, as a $\GL_n$-representation one has $\cV \cong \hO_n \tensor_{\CC} V$.
Moreover, this action is compatible with the $\Vect$-module structure, so that $\cV$ is actually a $(\Vect, \GL_n)$-module. 
Thus, the construction provides a functor  from $\Rep_{\GL_n}$ to
$\VB_n$.

\begin{dfn} 
We denote by $\Tens_n$ the image of finite-dimensional $\GL_n$-representations in $\VB_n$ along this functor. 
We call it the category of {\em formal tensor fields}.
\end{dfn}

As mentioned $\hO_n$ is an example, associated to the trivial one-dimensional $\GL_n$ representation.
Another key example is $\hT_n$, the vector fields on the formal disk, which is associated to the defining $\GL_n$ representation $\CC^n$; 
it is simply the adjoint representation of $\Vect$.
Other examples include $\hOmega^1_n$, the 1-forms on the formal disk; it
is the correct version of the coadjoint representation, and more
generally the space of $k$-forms on the formal disk $\hOmega^k_n$. 

The category $\Tens_n$ can be interpreted in two other ways, as we will see in subsequent work.
\begin{enumerate}
\item They are the $\infty$-jet bundles of tensor bundles: for a finite-dimensional $\GL_n$-representation, 
construct its associated vector bundle along the frame bundle and take its $\infty$-jets.
\item They are the flat vector bundles of finite-rank on the formal $n$-disk that are equivariant with respect to automorphisms of the disk. 
In other words, they are $\GL_n$-equivariant $D$-modules whose underlying $\hO$-module is finite-rank and free.
\end{enumerate}
It should be no surprise that given a Gelfand-Kazhdan structure on the frame bundle of a non-formal $n$-manifold $X$, 
a formal tensor field descends to the $\infty$-jet bundle of the corresponding tensor bundle on $X$. 
The flat connection on this descent bundle is, of course, the Grothendieck connection on this $\infty$-jet bundle. 
(For some discussion, see section 1.3, pages 12-14, of \cite{Fuks}.)

Note that the subcategories 
\ben
\Tens_n \hookrightarrow \VB_n
\hookrightarrow \Mod_{(\Vect, \GL_n)}^{\cO_n}
\een
inherit the symmetric monoidal structure constructed above. 

\subsection{Gelfand-Kazhdan descent}

We will focus on defining descent for the category $\VB_n$ of formal vector
bundles. 

Fix an $n$-dimensional manifold $X$.
The main result of this section is that the associated bundle construction along the frame bundle $\Fr_X$,
\[
\begin{array}{cccc}
\Fr_X \times^{\GL_n} - :&  \Rep(\GL_n)^{fin} & \to &\VB(X)\\
& V & \mapsto & \Fr_X \times^{\GL_n} V
\end{array},
\]
which builds a tensor bundle from a $\GL_n$ representation, arises from Harish-Chandra descent for $(\Vect,\GL_n)$. 
This result allows us to equip tensor bundles with interesting structures (e.g., a vertex algebra structure) by working $(\Vect,\GL_n)$-equivariantly on the formal $n$-disk.
In other words, it reduces the problem of making a universal
construction on all $n$-manifolds to the problem of making an
equivariant construction on the formal $n$-disk,
since the descent procedure automates extension from the formal to the global.

Note that every formal vector bundle $\cV \in \WGLCAT$ is naturally filtered via a filtration inherited from $\hO_n$. 
Explicitly, we see that $\cV$ is the limit of the sequence of finite-dimensional vector spaces
\[
\cdots \to \hO_n/\fm_n^k \otimes V \to \cdots \to \hO_n/\fm_n \otimes V \cong V
\]
where $V$ is the underlying $\GL_n$-representation.
Each quotient $\hO_n/\fm_n^k \otimes V$ is a module over $\Aut_{n,k}$, and 
hence determines a vector bundle on $X$ by the associated bundle construction along $X^{coor}_k$.
In this way, $\cV$ produces a natural sequence of vector bundles on $X$ and thus a pro-vector bundle on $X$.

Given a formal exponential $\sigma$ on $X$, we obtain a $\GL_n$-equivariant map from $\Fr_X$ to $X^{coor}_k$ for every $k$,
by composing the projection map $X^{coor} \to X_k^{coor}$ with the $\GL_n$-equivariant map from $\Fr_X$ to $X^{coor}$.

\begin{dfn}
{\em Gelfand-Kazhdan descent} is the functor
\[
\desc_\GK: \GK_n^\op \times \WGLCAT \to \Pro(\VB)_{flat}
\]
sending $(\Fr_X,\sigma)$ --- a frame bundle with formal exponential
--- and a formal vector bundle $\cV$ 
to the pro-vector bundle $\Fr_X \times^{\GL_n} \cV$ with flat connection induced by the Grothendieck connection.
\end{dfn}

This functor is, in essence, Harish-Chandra descent, but in a slightly exotic context.
It has several nice properties.

\begin{lemma}\label{prop lax}
For any choice of Gelfand-Kazhdan structure $(\Fr_X,\sigma)$, the descent functor $\desc_\GK((\Fr_X,\sigma),-)$ is lax symmetric monoidal.
\end{lemma}

\begin{proof}
For every $\cV,\cW$ in $\WGLCAT$, we have natural maps
\[
(\Omega^*(\Fr_X) \otimes \cV)_{basic} \otimes (\Omega^*(\Fr_X) \otimes \cW)_{basic} \to (\Omega^*(\Fr_X) \otimes (\cV \otimes \cW))_{basic} \to (\Omega^*(\Fr_X) \otimes (\cV \otimes_{\hO_n} \cW))_{basic}
\]
and the composition provides the natural transformation producing the lax symmetric monoidal structure.
\end{proof}

In particular, we observe that the de Rham complex of $\desc_\GK((\Fr_X,\sigma),\hO_n)$ is a commutative algebra object in $\Omega^*(X)$-modules. 
As every object of $\WGLCAT$ is an $\hO_n$-module and the morphisms are $\hO_n$-linear, 
we find that descent actually factors through the category of $\desc_\GK((\Fr_X,\sigma),\hO_n)$-modules. 
In sum, we have the following.

\begin{lemma}
The descent functor $\desc_\GK((\Fr_X,\sigma),-)$ factors as a composite
\[
\VB_n \xto{\widetilde{\desc}_\GK((\Fr_X,\sigma),-)} \Mod_{\desc_\GK((\Fr_X,\sigma),\hO_n)} \xto{\txt{forget}} \VB_{flat}(X)
\]
and the functor $\widetilde{\desc}_\GK((\Fr_X,\sigma),-)$ is symmetric monoidal.
\end{lemma}

As before, we let $\sdesc_{\GK}$ denote the associated local system
obtained from $\desc_{\GK}$ by taking horizontal sections. This
functor is well-known: it recovers the tensor bundles on $X$.

If $E \to X$ is a holomorphic vector bundle on $X$ we denote by
$J_{hol}^\infty(E)$ the holomorphic $\infty$-jet bundle of $E$. If
$E_0$ is the fiber of $E$ over a point $x \in X$, then the fiber of
this pro-vector bundle over $x$ can be identified with
\ben
J_{hol}^\infty (E)|_{x} \cong E_0 \times \CC \ll t_1,\ldots,t_n\rr .
\een
This pro-vector bundle has a canonical flat connection.

\begin{prop}
For $\cV \in \VB_n$ corresponding to the $\GL_n$-representation $V$,
there is a natural isomorphism of flat pro-vector bundles
\[
\desc_\GK((\Fr(X),\omega^\sigma),\cV) \cong J_{hol}^\infty(\Fr_X
\times^{\GL_n} V)
\]
In other words, the functor of descent along the frame bundle is
naturally isomorphic to the functor of taking $\infty$-jets of the associated bundle construction.
\end{prop} 

As a corollary, we see that the associated sheaf of flat sections is
\ben
\sdesc_{\GK}(\omega^\sigma, \cV) \cong \Gamma_{hol}(\Fr_X
\times^{\GL_n} V)
\een
where $\Gamma_{hol}(-)$ denotes the space of holomorphic sections. 

In other words, Gelfand-Kazhdan descent produces every tensor bundle. 
For example, for the defining representation $V = \CC^n$ of $\GL_n$, we have $\cV =\hT_n$, 
i.e., the vector fields on the formal disk viewed as the adjoint representation of  $\Vect$. 
Under Gelfand-Kazhdan descent, it produces the tangent bundle ${\rm T}$ on~$\Hol_n$.

\subsection{Formal characteristic classes}

\subsubsection{Recollection}

In \cite{atiyah}, Atiyah examined the obstruction --- which now bears his name --- to equipping a holomorphic vector bundle with a holomorphic connection from several perspectives. To start, as he does, we take a very structural approach. He begins by constructing the following sequence of vector bundles (see Theorem~1).

\begin{dfn}
Let $G$ be a complex Lie group. Let $E \to X$ be a holomorphic vector
bundle on a complex manifold and $\cE$ its sheaf of sections. The {\em Atiyah sequence} of $E$ is the
exact sequence holomorphic vector bundles given by
\[
0 \to E \tensor T^* X \to J^1(E) \to E \to 0,
\]
where $J^1(E)$ the bundle of {\em first-order} jets of $E$
The {\em Atiyah class} is the element $\At(E) \in {\rm H}^1(X, \Omega^1_X
\tensor \End_{\cO_X} (\cE))$ associated to the extension above. 
\end{dfn}

\begin{rmk}
Taking linear duals we see tha above short exact sequence is
equivalent to one of the form
\ben
0 \to \End (E) \to {\rm A}(E) \to T X \to 0
\een
where ${\rm A}(E)$ is the so-called {\em Atiyah bundle} associated to $E$. 

We should remark that the sheaf $\cA(E)$ of holomorphic sections of the Atiyah bundle ${\rm A}(E)$ is a Lie algebra by borrowing the Lie bracket on vector fields.
By inspection, the Atiyah sequence of sheaves (by taking sections) is a sequence of Lie algebras; 
 in fact, $\cA(E)$ is a central example of a Lie algebroid, as the quotient map to vector fields $\cT_X$ on $X$ is an anchor map.
\end{rmk}

Atiyah also examined how this sequence relates to the Chern theory of connections.

\begin{prop} 
A {\em holomorphic connection} on $E$ is a splitting of the Atiyah sequence (as holomorphic vector bundles).
\end{prop}

Atiyah's first main result in the paper is the following.

\begin{prop}[Theorem 2, \cite{atiyah}]
A holomorphic connection exists on $E$ if and only if the Atiyah class $\At(E)$ vanishes.
\end{prop}

He observes immediately after this statement that the construction is
functorial in maps of bundles. Later, he finds a direct connection
between the Atiyah class and the curvature of a smooth connection. A
smooth connections always exists (i.e., the sequence splits as smooth
vector bundles, not necessarily holomorphically), and one is free to
choose a connection such that the local 1-form only has
Dolbeault type $(1,0)$, i.e., is an element in $\Omega^{1,0}(X; \End(E))$. In that case, the $(1,1)$-component
$\Theta^{1,1}$ of the curvature $\Theta$ is a 1-cocycle in the
Dolbeault complex $(\Omega^{1,*}(X ; \End(E)), \overline{\partial})$ for $\End(E)$ and its cohomology class $[\Theta^{1,1}]$ is the Atiyah class $\At(E)$. In consequence, Atiyah deduces the following.

\begin{prop}
For $X$ a compact K\"ahler manifold, the $k$th Chern class $c_k(E)$ of $E$ is given by the cohomology class of $(2\pi i)^{-k} S_k(\At(E))$, 
where $S_k$ is the $k$th elementary symmetric polynomial, and hence only depends on the Atiyah class.
\end{prop}

This assertion follows from the degeneracy of the Hodge-to-de Rham
spectral sequence. More generally, the term $(2\pi i)^{-k}
S_k(\At(E))$ agrees with the image of the $k$th Chern class in the
Hodge cohomology $H^k(X ; \Omega^k_{hol})$.

The functoriality of the Atiyah class means that it makes sense not just on a fixed complex manifold, but also on the larger sites $\Hol_n$ and $\GK_n$. 
We thus immediately obtain from Atiyah the following notion.

\begin{dfn}
For each $V \in \vb(\Hol_n)$, the {\em Atiyah class} $\At(V)$ is the equivalence class of the extension of the tangent bundle $T$ by $\End(V)$ given by the Atiyah sequence.
\end{dfn}

Moreover, we have the following.

\begin{lemma}
The cohomology class of $(2\pi i)^{-k} S_k(\At(V))$ provides a section
of the sheaf $H^k(X ; \Omega^k_{hol})$. On any compact K\"ahler manifold, it agrees with $c_k(V)$.
\end{lemma}

\subsubsection{The formal Atiyah class}

We now wish to show that Gelfand-Kazhdan descent sends an exact sequence in $\VB_{\hc}$ to an exact sequence in $\vb(\GK_n)$ (and hence in $\vb(\Hol_n)$). 
It will then remain to verify that for each tensor bundle on $\Hol_n$, 
there is an exact sequence over the formal $n$-disk that descends to the Atiyah sequence for that tensor bundle.

We will use the notation $\desc_\GK(\cV)$ to denote the functor $\desc_\GK(-,\cV): \GK_n^\op \to \Pro(\vb)_{flat}$, 
since we want to focus on the sheaf on $\GK_n$ (or $\Hol_n$) defined
by each formal vector bundle~$\cV$. Taking flat sections we get an
$\cO$-module $\sdesc_{\GK}(\cV)$ which is locally free of finite
rank and so determines an object in $\vb(\GK_n)$. 

\begin{lemma}
If $$\cA \to \cB \to \cC$$ is an exact sequence in $\vb_{\hc}$, then 
$$\sdesc_\GK(\cA) \to \sdesc_\GK(\cB) \to \sdesc_\GK(\cC)$$ 
is exact in $\vb(\GK_n)$.
\end{lemma}

\begin{proof}
A sequence of vector bundles is exact if and only if the associated
sequence of $\cO$-modules is exact (i.e., the sheaves of sections of
the vector bundles). But a sequence of sheaves is exact if and only if
it is exact stalkwise. Observe that there is only one point at which
to compute a stalk in the site $\Hol_n$, since every point $x \in X$
has a small neighborhood isomorphic to a small neighborhood of $0 \in
\CC^n$. As we are working in an analytic setting, the stalk of a
$\cO$-module at a point $x$ injects into the $\infty$-jet at
$x$. Hence, it suffices to verifying the exactness of the sequence of
$\infty$-jets. Hence, we consider the $\infty$-jet at $0 \in \CC^n$ of
the sequence $\desc_\GK(A) \to \desc_\GK(B) \to \desc_\GK(C)$. But
this sequence is simply $A \to B \to C$, which is exact by
hypothesis.
\end{proof}

\begin{cor}
There is a canonical map from $\Ext^1_{\hc}(\cB,\cA)$ to $\Ext^1_{\GK_n}(\sdesc_\GK(\cB), \sdesc_\GK(\cA))$.
\end{cor}

In particular, once we produce the $\hc$-Atiyah sequence for a formal tensor field $\cV$, 
we will have a very local model for the Atiyah class living in $\clie^*(\Vect,\GL_n; \hOmega^1_n \otimes_{\hO_n} \End_{\hO_n}(\cV))$.

\subsubsection{The formal Atiyah sequence} \label{sec gk cc}

Let $\cV$ be a formal vector bundle. 
We will now construct the ``formal'' Atiyah sequence associated to $\cV$.  
First, we need to define the $(\Vect, \GL_n)$-module of {\em first order jets} of $\cV$. 
Let's begin by recalling the construction of jets in ordinary geometry.

If $X$ is a manifold, we have the diagonal embedding $\Delta : X \hookrightarrow X \times X$. 
Correspondingly, there is the ideal sheaf $\cI_\Delta$ on $X \times X$ of functions vanishing along the diagonal. 
Let $X^{(k)}$ be the ringed space $(X, \cO_{X \times X}/\cI_\Delta^k)$ 
describing the $k$th order neighborhood of the diagonal in $X \times X$. 
Let $\Delta^{(k)} : X^{(k)} \to X \times X$ denote the natural map of ringed spaces.
The projections $\pi_1, \pi_2 : X \times X \to X$ compose with $\Delta^{(k)}$ 
to define maps $\pi^{(k)}_1, \pi_2^{(k)} : X^{(k)} \to X$. 
Given an $\cO_X$-module $\cV$, 
``push-and-pull'' along these projections,
\ben
J^k_X(\cV) = (\pi_1^{(k)})_* (\pi_2^{(k)})^* \cV,
\een
defines the $\cO_X$-module of $k$th order jets of~$\cV$.

There is a natural adaptation in the formal case. 
The diagonal map corresponds to an algebra map $\Delta^* : \hO_{2n} \to \hO_n$.
Fix coordinatizations $\hO_n = \CC \ll t_1,\ldots,t_n \rr$ and $\hO_{2n} = \CC \ll t'_1,\ldots,t_n', t_1'', \ldots,t_n'' \rr.$ 
Then the map is given by $\Delta^*(t'_i) = \Delta^*(t_i'') = t_i$. 

Let $\hI_n = \ker(\Delta^*) \subset \hO_{2n}$ be the ideal given by the kernel of $\Delta^*$. 
For each $k$ there is a quotient map
\ben
\Delta^{(k)*}: \hO_{2n} \to \hO_{2n} / \hI_n^{k+1} ,
\een
The projection maps have the form
\ben
\pi_1^{(k)*}, \pi_2^{(k)*}  : \hO_n \to \hO_{2n} / \hI_n^{k+1},
\een
which in coordinates are $\pi_1^*(t_i) = t'_i$ and $\pi_2^*(t_i)
=~t''_i$. 

\begin{dfn} 
Let $\cV$ be a formal vector bundle on $\hD^n$.
Consider the $\hO_{2n} / \hI_n^{k+1}$-module
$\cV \tensor_{\hO_n} \left(\hO_{2n} / \hI_n^{k+1}\right)$,
where the tensor product uses the $\hO_n$-module structure on the
quotient $\hO_{2n} / \hI_n^{k+1}$ coming from the map $\pi_2^{(k)*}$. 
We define the {\em $k$th order formal jets of $\cV$}, denoted $J^k(\cV)$, 
as the restriction of this $\hO_{2n}/\hI_n^{k+1}$-module 
to a $\hO_n$-module using the map $\pi_1^{(k)*} : \hO_n \to \hO_{2n} / \hI_n^{k+1}$. 
\end{dfn}

\begin{lemma} For any $\cV \in \VB_n$ the $k$th order formal jets
  $J^k(\cV)$ is an element of $\VB_n$. 
\end{lemma}
\begin{proof}
For $\cV$ in $\VB_n$ there is an induced action of $(\Vect, \GL_n)$ on
the tensor product $\cV \tensor_{\hO_n} \hO_{2n} /
\hI_{n}^{k+1}$. For fixed $k$ we see that $\hO_{2n} / \hI_n^{k+1}$ is
finite rank as a $\hO_n$ module. Thus it
is immediate that this module satisfies the conditions of a formal
vector bundle.
\end{proof}

As a $\CC$-linear vector space we have $J^1(\cV) = \cV \oplus (\cV \tensor_{\hO_n} \hOmega^1_n)$. 
For $f \in \hO_n$ and $(v, \beta) \in \cV \oplus (\cV \tensor \hOmega^1_n)$, 
the $\hO_n$-module structure is given by
\ben
f \cdot (v, \beta) = (f v, (f \beta + v \tensor \d f)).
\een 
(This formula is the formal version of Atiyah's description in Section 4 of \cite{atiyah},
where he uses the notation~$\mathcal{D}$.) The following is proved in
exact analogy as in the non-formal case which can also be found in
Section 4 of \cite{atiyah}, for instance. 

\begin{prop}\label{1jet2} 
For any $\cV \in \VB_{\hc}$, the $\hO_n$-module $J^1(\cV)$ has a compatible action of the pair $(\Vect, \GL_n)$ and hence determines an object in $\VB_{\hc}$. 
Moreover, it sits in a short exact sequence of formal vector bundles 
\be\label{formalatiyah1}
\cV \tensor \hOmega^1_n \to J^1 (\cV) \to \cV .
\ee
Finally, the Gelfand-Kazhdan descent of this short exact sequence is isomorphic to the Atiyah sequence
\ben
\sdesc_{\GK}(\cV) \tensor \Omega^1_{hol} \to J^1 \sdesc_{\GK} (\cV) \to \sdesc_{\GK}(\cV) .
\een
In particular, $J^1 \desc_{\GK}(\cV) = \desc_{\GK}(J^1 \cV)$.
\end{prop}

We henceforth call the sequence (\ref{formalatiyah1}) {\em the formal Atiyah sequence} for $\cV$. 

\begin{rmk} 
Note that $J^1(\cV)$ is an element of the category $\VB_n$ but it is {\em not} a formal tensor field. 
That is, it does not come from a linear representation of $\GL_n$ via coinduction. 
\end{rmk}

\begin{rmk} 
A choice of a formal coordinate defines a splitting of the first-order jet sequence as $\hO_n$-modules. 
If we write $\cV = \hO_n \tensor_\CC \cV$, then one defines 
\ben
j^1 : \cV \to J^1 \cV \;\; , \;\; f \tensor_\CC v \mapsto (f \tensor_\CC v, (1 \tensor_\CC v) \tensor_{\cO} \d f) .
\een
It is a map of $\hO_n$-modules, and it splits the obvious projection $J^1(\cV) \to \cV$. 
We stress, however, that it is {\em not} a splitting of $\Vect$-modules. 
We will soon see that this is reflected by the existence of a certain characteristic class in Gelfand-Fuks cohomology. 
\end{rmk}

Note the following corollary, which follows from the identification 
$$\Ext^1(\cV \tensor_{\hO_n} \hOmega_{n}^1, \cV) \cong \clie^1(\Vect,\GL_n; \hOmega^1_n \otimes_{\hO_n} \End_{\hO_n}(\cV))$$ 
and from the observation that an exact sequence in $\vb(\hD^n)$ maps to an exact sequence in $\vb(\GK_n)$.

\begin{cor}
There is a cocycle $\At^\GF(\cV) \in \clie^1(\Vect,\GL_n; \hOmega^1_n \otimes_{\hO_n} \End_{\hO_n}(\cV))$ representing the Atiyah class $\At(\desc_\GK(\cV))$. 
\end{cor}

We call this cocycle the Gelfand-Fuks-Atiyah class of $\cV$ since it
descends to the ordinary Atiyah class for $\desc(\cV)$ as a sheaf of
$\cO$-modules. 

\begin{dfn}
The {\em Gelfand-Fuks-Chern character} is the formal sum $\ch^\GF(\cV) = \sum_{k \geq 0} \ch^\GF_k(\cV)$, 
where the $k$th component
\ben
{\rm ch}_k^\GF(\cV) := \frac{1}{(-2 \pi i)^k k!} {\rm Tr}({\At}^\GF(\cV)^k)
\een
lives in $\clie^k(\Vect,\GL_n; \hOmega^k_n)$.
\end{dfn}

It is a direct calculation to see that $\ch^{\GF}_k(\cV)$ is closed for
the differential on formal differential forms, 
i.e., it lifts to an element in $\clie^k(\Vect,\GL_n; \hOmega^k_{n,cl})$.

\subsubsection{An explicit formula}

In this section we provide an explicit description of the Gelfand-Fuks-Atiyah class  
\ben
\At^{\rm GF}(\cV) \in \clie^1(\Vect, \GL_n ; \hOmega^1_n
\tensor_{\hO_n} \End_{\hO}(\cV)) .
\een 
of a formal vector bundle $\cV$. 

By definition, any formal vector bundle has the form $\cV = \hO_n \tensor V$, 
with $V$ a finite-dimensional vector space.
We view $V$ as the ``constant sections'' in $\cV$ by the inclusion $i: v \mapsto 1 \otimes v$.
This map then determines a connection on $\cV$:
we define a $\CC$-linear map $\nabla: \cV \to \hOmega^1_n \otimes_{\hO_n} \cV$
by saying that for any $f \in \hO_n$ and $v \in V$,
\[
\nabla(f v) = \d_{dR}(f) v,
\]
where $\d_{dR} : \hO_n \to \hOmega^1_n$ denote the de Rham
differential on functions. This connection appeared earlier when we
defined the splitting of the jet sequence $j^1 = 1 \oplus \nabla$. 

The connection $\nabla$ determines an element in $\clie^1(\Vect ;
\hOmega^1_n \tensor_{\hO} \End_{\hO}(\cV))$, as follows. Let 
\ben
\rho_{\cV} : \Vect \tensor \cV \to \cV
\een
denote the action of formal vector fields and consider the composition
\ben
\Vect \tensor V \xto{\id \tensor i} \Vect \tensor \cV \xto{\rho_{\cV}} \cV \xto{\nabla} \hOmega^1_n \tensor_{\hO} \cV .
\een 
Since $V$ spans $\cV$ over $\hO_n$, this composite map determines a $\CC$-linear map
\[
\alpha_{\cV,\nabla}: \Vect \to \hOmega^1_n \tensor_{\hO} \End_{\hO}(\cV)
\]
by
\[
\alpha_{\cV,\nabla}(X)(fv) = f \nabla( \rho_\cV(X)(i(v))),
\]
with $f \in \hO_n$ and $v \in V$.

\begin{prop} \label{atiyahprop1} 
Let $\cV$ be a formal vector bundle. 
Then $\alpha_{\cV,\nabla}$ is a representative for the Gelfand-Fuks-Atiyah class~$\At^{\rm GF}(\cV)$. 
\end{prop}

\begin{proof}
We begin by recalling some general facts about the Gelfand-Fuks-Atiyah class as an
extension class of an exact sequence of modules. Viewing $\hO_n$ as functions on the formal $n$-disk, we can ask about the jets of such functions.
A choice of formal coordinates corresponds to an identification $\hO_n \cong \CC[[t_1,\ldots,t_n]]$,
and that choice provides a trivialization of the jet bundles by providing a preferred frame.
This frame identifies, for instance, $J^1$ with $\hO_n \oplus \hOmega^1_n$,
and the1-jet of a formal function $f$ can be understood as~$(f, \d_{dR}f)$.

For a formal vector bundle $\cV = \hO_n \otimes V$, something similar happens after choosing coordinates.
We have $J^1(\cV) \cong \cV \oplus \hOmega^1_n \otimes_{\hO_n} \cV$ and
the 1-jet of an element of $\cV$ can be written as
\ben
\begin{array}{cccc}
j^1 : & \cV & \to & J^1(\cV)   \\
& f v & \mapsto & (f  v, \d_{dR}(f) v ) .
\end{array}
\een 
where $f \in \hO_n$ and $v \in V$. 
The projection onto the second summand is precisely the connection $\nabla$ on $\cV$ 
determined by $\cV = \hO_n \otimes V$, the defining decomposition.

The Gelfand-Fuks-Atiyah class is the failure for this map $\nabla$ to be a map of $\Vect$-modules. 
Indeed, $\nabla$ determines a map of graded vector spaces
\ben
1 \tensor \nabla : \clie^\#(\Vect ; \cV) \to \clie^\#(\Vect ;\hOmega^1_n
\tensor_{\hO} \cV) .
\een
Let $\d_{\cV}$ denote the differential on $\clie^*(\Vect; \cV)$ and
$\d_{\Omega^1 \tensor \cV}$ denote the differential on $\clie^*(\Vect
; \hOmega^1_n \tensor_{\hOmega} \cV)$. The failure for $1 \tensor \nabla$ is precisely the difference
\be\label{difference}
(1 \tensor \nabla) \circ \d_{\cV} - \d_{\Omega^1 \tensor \cV} \circ (1 \tensor
\nabla).
\ee
This difference is $\clie^\#(\Vect)$ linear and can hence be
thought of as a cocycle of degree one in $\clie^*(\Vect ; \hOmega^1
\tensor_{\hO} \End_{\hO} (\cV))$. This is the representative for the Atiyah
class. 

We proceed to compute this difference. The differential $\d_{\cV}$ splits as $\d_{\Vect}
\tensor 1_\cV
+ \d'$ where $\d_{\Vect}$ is the differential on the complex
$\clie^*(\Vect)$ and $\d'$ encodes the action of $\Vect$ on
$\cV$. Likewise, the differential $\d_{\Omega^1 \tensor \cV}$ splits
as $\d_{\Vect} \tensor 1_{\Omega^1 \tensor \cV} + \d_{\Omega^1}
\tensor 1_V + 1_{\Omega^1} \tensor \d '$ where $\d_{\Omega^1}$ is the differential on the complex $\clie^*(\Vect ;
\hOmega^1_n)$. 

The de Rham differential clearly commutes with the
action of vector fields so that $(1 \tensor \d_{dR}) \circ
(\d_{\cO}\tensor 1) = (\d_\Vect + \d_{\Omega^1})\circ(1 \tensor
\d_{dR})$ so that the the difference in (\ref{difference}) reduces to 
\ben
(1 \tensor \nabla) \circ \d' - (1_{\Omega^1} \tensor \d') \circ (1
\tensor \nabla) .
\een
By definition $\d'$ is the piece of the Chevalley-Eilenberg
differential that encodes the action of $\Vect$ on $\cV$, so if we
evaluate on an element of the form $1 \in v \in \clie^0(\Vect ; V)
\subset \clie^0(\Vect ; \cV)$ the only term that survives is the GF 1-cocycle
\ben
X \mapsto \nabla \d'(1 \tensor v)(X) = \nabla (\rho_\cV(X) (v)) .
\een
as desired. 
\end{proof}

\begin{cor} 
On the formal vector bundle $\hT_n$ encoding formal vector fields, 
fix the $\hO_n$-basis by $\{\partial_j\}$ and the $\hO_n$-dual basis of one-forms by $\{\d t^j\}$. 
The explicit representative for the Atiyah class is given by the Gelfand-Fuks 1-cocycle 
\ben
f^i \partial_i\mapsto - \d_{dR} (\partial_j f^i) (\d t^j
\tensor \partial_i)
\een
taking values in $\hOmega^1_n \tensor_{\hO_n} \End_{\hO}(\hT_n)$.
\end{cor}

\begin{proof} 
We must compute the action of vector fields on $\hO_n$-basis elements of $\hT_n$. 
We fix formal coordinates $\{t_j\}$ and let $\{\partial_j\}$ be the associated constant formal vector fields. 
Then the structure map is given by the Lie derivative $\rho_{\hT} (f^i \partial_i , \partial_j ) = - \partial_j f^i$. 
The formula for the cocycle follows from the Proposition. 
\end{proof}
 
We can use this result to explicitly compute the cocycles representing the Gelfand-Kazhdan Chern characters. 
For instance, we have the following formulas that will be useful in later sections.

\begin{cor}
The second component $\ch_2^{\rm GF}(\hT_n)$ of the universal Chern character is the cocycle
\ben
{\rm Tr}({\rm At}^\GF(\hT_n)^{\wedge 2}): (f^i \partial_i, g^j \partial_j) \mapsto - \d_{dR}(\partial_j f^i) \wedge \d_{dR}(\partial_i g^j)
\een 
in $\clie^2(\Vect,\GL_n; \hOmega_n^2)$. 
As the de Rham differential $\d_{dR} : \hOmega^1_n \to \hOmega^2_n$ is $\Vect$-equivariant, 
there is an element $\alpha$ in $\clie^2(\Vect,\GL_n; \hOmega_n^1)$ such that
\[
\ch_2^{\rm GF}(\hT_n) = \d_{dR} \alpha
\]
where 
\[
\alpha: (f^i \partial_i, g^j \partial_j) \mapsto \partial_j f^i \wedge \d_{dR}(\partial_i g^j).
\]
Moreover, as $\ch_2$ is closed for the differential $\partial$, 
it lifts to a cocycle in $\clie^2(\Vect, \GL_n ; \hOmega^2_{n,cl})$.
\end{cor}

\subsubsection{Extended pair}
The $2$-cocycle $\ch_2^{\GF}(\hT_n)$ determines an extension Lie algebras of $\Vect$ by the abelian Lie algebra $\hOmega^2_{n,cl}$
\ben
0 \to \hOmega^2_{n,cl} \to \TVect \to \Vect \to 0 .
\een

We have already discussed the pair $(\Vect, \GL_n)$. We will need that
the above extension of Lie algebras fits in to a Harish-Chandra pair
as well. The action of $\GL_n$ extends to an action on $\TVect$ where
we declare the action of $\GL_n$ on closed two-forms to be the natural
one via linear formal automorphisms.

\begin{lemma} 
The pair $(\TVect, \GL_n)$ form a Harish-Chandra pair and fits into an extension of pairs
\ben
0 \to \hOmega^2_{n,cl} \to (\TVect, \GL_n) \to (\Vect, \GL_n) \to 0
\een
which is determined by the cocycle $\ch_2^{\GF}(\hT_n)$. 
\end{lemma}

One might be worried as to why there is only a non-trivial extension
of the Lie algebra in the pair. The choice of a coordinate determines
an embedding of linear automorphisms $\GL_n$ into formal automorphisms
$\Aut_n$. The extension of formal automorphisms $\Aut_n$ defined by
the group two-cocycle $\ch_2^\GF(\hT_n)$ is trivial when restricted to
$\GL_n$ so that it does not get extended.

\section{Harish-Chandra structure on CDOs} \label{sec vertex alg}

In this section, we first recall the definition of chiral differential
operators on affine space $\CC^n$; this object always exists and there
is no obstruction to defining it. Then we formulate a construction of chiral differential
operators on more general complex manifolds based on the theory of
Gelfand-Kazhdan descent developed in the previous section. The key
element of this formulation is the Harish-Chandra module structure for
formal vector fields and automorphisms, much of which has been studied
in the literature on vertex algebras. The two main results we extract is
Theorem \ref{GMS1} which shows how formal automorphisms act, and Theorem
\ref{MSV1} which shows how formal vector fields act. We find these
actions to be compatible and deduce the structure of a module 

\subsection{Recollections on vertex algebras}

\subsubsection{Recollections}

We briefly recall the definition of a vertex algebra and some other notions associated to vertex algebras. 
Our main references are \cite{BZF} and \cite{Kac}.

\begin{dfn}
A {\em vertex algebra} is the following data:
\begin{itemize}
\item[(i)] a vector space $V$ over $\CC$ (the \emph{state space}); 
\item[(ii)] a nonzero vector $\left|0 \> \in V$ (the \emph{vacuum vector}); 
\item[(iii)] a linear map $T : V \to V$ (the \emph{translation operator});
\item[(iv)] a linear map $Y (-; z) : V \to \End(V) \ll z^{\pm} \rr$ (the \emph{vertex operator});
\end{itemize}
subject to the following conditions:
\begin{itemize}
\item[(1)] For $v \in V$, let 
\[
Y(v;z) = \sum_{n \in \ZZ} v_{(n)} z^{-n-1}
\]
in $\End (V) \ll z^{\pm} \rr$. 
(We call the endomorphisms $v_{(n)}$ the \emph{Fourier modes} of $Y(v;z)$.)
Then for each $w \in V$ there exists some $N \in \ZZ$ such that $v_{(j)} w = 0$ for all $j > N$.
\item[(2)] $Y(\left|0\>; z) = \id_V$ and $Y(v;z) \left|0\> \in v + z V \ll z
  \rr$ for all $v \in V$. 
\item[(3)] For every $v$, $[T, Y(v;z)] = \partial_z Y(v; z)$, and $T\left|0\> = 0$.
\item[(4)] For any pair $v,v'\in V$, there exists $N \in \ZZ_{\geq 0}$ such that $(z-w)^N [Y(v;z), Y(v';w)] = 0$.
\end{itemize}
\end{dfn}

\begin{rmk}
Alternatively, one can  formulate the definition of a vertex algebra in terms of the Fourier modes $v_{(n)}$. 
Indeed, our definition above provides a family of bilinear operations
\[
\begin{array}{cccc}
(-)_{(n)} (-) : & V \times V & \to & V \\
& (v,w) & \mapsto & v_{(n)}w
\end{array}
\]
These operations satisfy algebraic conditions coming from conditions (1)-(4) above. 
For instance, see~\cite{Kac}. 
\end{rmk}

We will be interested in, and take advantange of, vertex algebras with the additional structure of a $\ZZ_{\geq 0}$-grading. 
This grading is not cohomological in nature and does not follow the Koszul sign rule.
We call it the \emph{conformal dimension grading}.

\begin{dfn}
A vertex algebra as above is {\em $\ZZ_{\geq 0}$-graded} if 
the underlying state space $V$ is a $\ZZ_{\geq 0}$-graded vector space $V = \bigoplus_{N \in \ZZ_{\geq 0}} V^{(N)}$
such that 
\begin{enumerate}
\item[(1)] the vacuum $\left|0\>$ has dimension zero, 
\item[(2)] the translation operator $T$ is a dimension $1$ map, and 
\item[(3)] for $v \in V^{(N)}$ the dimension of the endomorphism $v_{(m)}$ is $-m + N -1$. 
\end{enumerate}
\end{dfn}

Condition (3) ensures that  if $v \in V^{(N)}$ and $w \in V^{(M)}$, 
then $v_{(m)} w \in V_{-m + N + M-1}$.

\subsubsection{Actions on vertex algebras}

We now discuss what it means for a Harish-Chandra pair to act on a
vertex algebra. 

It is clear how to define an action of a Lie group on a vertex
algebra $V$. Indeed, if $K$ is a Lie group then by an action of $K$ on
$V$ is a group homomorphism
\ben
\rho_K : K \to \Aut_{{\rm VA}}(V)
\een
where $\Aut_{{\rm VA}}(V)$ are the vertex algebra automorphisms. That is,
maps of vertex algebras $V \to V$ whose underlying $\CC$-linear map is
invertible.  

To define the action of a Lie algebra on $V$ we first recall what a
vertex algebra derivation is. It is the data of a linear map $D : V
\to V$ such that for all $v \in V$ one has
\ben
Y(D v ; z) = [D, Y(v ; z)] .
\een 
The set of all derivations forms a Lie algebra which we denote
$\Der_{{\rm VA}}(V)$. An action of a Lie algebra $\fg$ on $V$ is the data
of a homomorphism
\ben
\rho_\fg : \fg \to \Der_{{\rm VA}}(V) .
\een 

It also makes sense to talk about vertex algebras that have actions by
apair $(\fg,K)$. Indeed, a $(\fg,K)$-action on a vertex algebra $V$ is
a $(\fg,K)$-action is given by actions of $\fg$ and $K$ as above such
that we have the obvious compatibility. 

The underlying vector space of a vertex algebra is almost always infinite dimensional, however,
and so does not immediately fit into our definition of a module of a Harish-Chandra pair from Section 1. 
We sidestep this issue by focusing on vertex algebras are graded by conformal dimension where the conformal dimension $N$ space $V^{(N)}$ is finite dimensional for each $N$
(in our case, finite rank over $\hO_n$) 
so that we have a well-behaved category of modules. 
From here on, we will assume the following definition of an action on
a vertex algebra. 

\begin{dfn}\label{graction} 
An \emph{action of a Harish-Chandra pair $(\fg,K)$ on a $\ZZ_{\geq 0}$-graded vertex algebra $V = \bigoplus_{N} V^{(N)}$} 
is a collection of $(\fg,K)$-actions $(\rho^{(N)}_{\fg}, \rho^{(N)}_{K})$ on the underlying fixed conformal dimension spaces $V^{(N)}$ such that:
\begin{itemize}
\item[(1)] for each $x \in \fg$ the induced map $\oplus_N \rho^{(N)}_\fg(x)$ is a vertex algebra derivation for $V$, and
\item[(2)] for each $A \in K$ the induced map $\oplus_{N} \rho^{(N)}_K(A)$ is a vertex algebra automorphism for $V$.
\end{itemize}
\end{dfn}

\subsection{The $\beta\gamma$-vertex algebra}

One of the main objects that we will focus on is the vertex algebra of
chiral differential operators on $\CC^n$. In the physics literature
\cite{WittenCDO}, \cite{Nek}, \cite{polchinski} it is typically called the $n$-dimensional $\beta\gamma$ vertex algebra. 

\begin{dfn}
Let $\CDO_n$ denote the vertex algebra of {\em chiral differential operators} for $\CC^n$. The underlying vector space is
\[
\CC [b_l^j,c_m^j]_{1 \leq j \leq n, l < 0, m \leq 0},
\]
the translation operator $T$ is
\begin{align*} 
 b_m^j &\mapsto -m b_{m-1}^j, \\
 c_m^j &\mapsto -(m-1) c_{m-1}^j,
\end{align*}
and the vertex operator is
\[
Y(b^j_{-1},z) = \sum_{m < 0} b_m^j z^{-1-m} + \sum_{m \ge 0} \frac{\partial}{\partial c_{-m}^j} z^{-1-m} 
\]
and 
\[
Y(c_{0}^j,z) = \sum_{m \le 0} c^j_m z^{-m} - \sum_{m > 0}  \frac{\partial}{\partial b_{-m}^j} z^{-m}.
\]
These determine a vertex algebra by a reconstruction theorem (see, e.g., Theorem 2.3.11 of \cite{BZF}).
\end{dfn}

This vertex algebra $\CDO_n$ is a $\ZZ_{\geq 0}$-graded. We specify this by
saying that $c_0^j$ has conformal dimension $0$ and $b_{-1}^j$ has
conformal dimension $1$. Denote by $\CDO_n^{(N)}$ the conformal dimension
$N$ subspace. 

\begin{rmk}
\label{rmk:notations}
The generators of this vertex algebra are given a variety of symbols:
\begin{enumerate}
\item in \cite{MSV}, they use $a^j_{m}$ for what we call $b^j_m$ and $b^j_{m}$ for our~$c^j_m$;
\item in Chapter 11 Section 3.6 of \cite{BZF}, they use $a^j_{-m}$ for what we call $b^j_m$ and ${a^{\ast}}^j_{-m}$ for our~$c^j_m$.
\end{enumerate}
\end{rmk}

It is important to note how $\CDO_n$ relates to the geometry of $\CC^n$, and some subtleties thereof.
The underlying vector space of $\CDO_n$ is isomorphic with the underlying vector space of the commutative algebra $\CC [b_{l}^i, c_{m}^j]$, 
where $l < 0$, $m \leq 0$, and $i,j = 1,\ldots,d$. 
The map
\[
\begin{array}{cccc}
\label{poly}
\tau : &\cO_n = \CC [t^1,\ldots, t^n] &\to& \CDO^0_n\\ 
& t_i & \mapsto & c_0^i .
\end{array}
\]
identifies the commutative algebra $\cO_n$ of functions on $\CC^n$ with the conformal dimension zero subspace $\CDO_n^{(0)} \subset \CDO_n$, 
which consists of polynomials in the variables $c_0^1,\ldots, c_0^{n}$. 
In other words, for any polynomial function $f$ on $\CC^n$, we substitute $c^i_0$ for $t_i$ in the polynomial $f$.
We will express this, somewhat abusively, as $\tau(f) = f(c)$.
This relationship shows that one can view the underlying vector space of $\CDO_n$ as an $\cO_n$-module.

One might hope that this map $\tau$ is a map of commutative algebras,
by using the $(-1)$-Fourier mode as a bilinear operation
\ben
(-)_{(-1)} (-) : \CDO_n^{(0)} \times \CDO_n^{(0)} \to \CDO_n^{(0)} 
\een
to equip $\CDO_n^{(0)}$ with a commutative algebra structure.
The issue is that this product is not associative!
Hence, we emphasize that $\CDO_n$ {\em as a vertex algebra} is not $\cO_n$-linear. 

\begin{rmk}
In \cite{BorVAMS} Borisov develops a clean formalization of this situation by introducing the notion of an $R$-loop module over a commutative algebra $R$,
clarifying how, in our example above, the $\cO_n$-module structure interacts with the vertex algebra structure on $\CDO_n$.
This notion of loop module also leads to quasi-loop-coherent sheaves and then to sheaves of vertex algebras.
Borisov used it to globalize the chiral de Rham complex to complex manifolds, 
and suitably modified, it should apply to chiral differential operators too.
\end{rmk}

\subsubsection{Completion}

We use this module structure to complete along powers of $\{t^1,\ldots,t^n\}$. 
That is, we base change $\CDO_n$ to a module for $\hO_n$:
\ben
\hCDO_n := \hO_n \tensor_{\cO_n} \CDO_n.
\een
In Theorem 3.1 of \cite{MSV} it is shown that this module obtains a
vertex algebra structure by extending that on $\CDO_n$. The critical
step is showing that the vertex operator
\ben
Y(-,z) : \hCDO_n \to \End(\hCDO_n)\llbracket z,z^{-1}\rrbracket 
\een
is well defined. Every power series $f \in \hO_n$ is a limit of polynomials $\{f_k\} \subset \cO_n$. According to the inclusion (\ref{poly}), every $f_k \in \cO_n$ determines a field
\ben
f_k(z) := f_k(c_0^1(z),\ldots, c_0^n(z)) \in \End(\CDO_n)\llbracket z, z^{-1} \rrbracket .
\een 
The result of \cite{MSV} is that the limit of $\{f_k(z)\}$ determines a field
\ben
f(z) := f(c_0^1(z),\ldots,c_0^n(z)) \in \End(\hCDO_n)\llbracket z, z^{-1} \rrbracket .
\een
Note that $\hCDO_n$ is still a $\ZZ$-graded vertex algebra, inherited from the conformal dimension $\ZZ$-grading on $\CDO_n$. 
Indeed, for each $N \in \ZZ$ we have $\hCDO_n^{(N)} = \hO_n \tensor_{\cO_n} \CDO_n^{(N)}$.

\subsubsection{}

The primary complication in gluing chiral differential operators $\CDO_n$ to a sheaf on a general manifold 
is that the group of automorphisms of the disk do {\em not} act as automorphisms of the vertex algebra. 
This problem appears for the formal disk as well. 
The group of formal automorphisms $\Aut_n$ do not act on $\hCDO_n$, 
as we will see explicitly at the level of formal vector fields, in way preserving the vertex operator. 

If we restrict ourselves to {\em linear} automorphisms of the disk, however, we find that there is no such problem. 
Indeed we can explicitly describe the action of the Lie group $\GL_n$ 
by vertex algebra automorphisms on $\CDO_n$ and $\hCDO_n$  as follows. 
Denote by $\mathbf{b}_m$ the $n$-tuple $(b_m^1,\ldots,b_m^n)$ considered as a vector in $\CC^n$ and
$\mathbf{c}_m$ as the vector $(c_m^1,\ldots, c_m^n)$. 
Given $A \in \GL_n$, the action of $A$ is specified by 
\begin{align}
A \cdot {\bf c}_{0} & = A {\bf c}_0 \label{GLact1} \\
A \cdot {\bf b}_{-1} & = (A^T)^{-1} {\bf b}_{-1} \label{GLact2}
\end{align} 
where on the right-hand side we understand matrix
multiplication. 
Clearly this action preserves the $\ZZ_{\geq 0}$-grading.

\subsection{The classical limit}
While the $\beta\gamma$ vertex algebra does not carry an action of
formal automorphisms or formal vector fields, its ``classical limit''
does. For this reason, descending the classical vertex algebra is much
simpler and the formalism of Gelfand-Kazhdan descent from Section
\ref{sec gk descent} directly applies. (For an alternative approach to
this see \cite{Malikov2008}) First, we discuss what we mean by the
classical limit of $\CDO_n$.

For each fixed conformal dimension $N \in \ZZ$, there is a filtration on the subspace $\CDO_n^{(N)}$ that we now describe. First we set up some notation.

Let $L = \left((i_1,l_1),\ldots,(i_p,l_p)\right) \in
\left(\{1,\ldots,n\} \times \ZZ_{<0}\right)^p$, $M =
\left((j_1,m_1),\ldots,m_q)\right) \in \left(\{1,\ldots,n\} \times
  \ZZ_{\leq 0}\right)^q$ be multi-indices of length $p$ and $q$ respectively. Define
\bestar
b_L &:= & b_{l_1}^{i_1} \cdots b_{l_p}^{i_p} \\
c_M & := & c_{m_1}^{j_1} \cdots c_{m_q}^{j_q} .
\eestar
Then as a vector space, the subspace $\CDO_n^{(N)}$ of conformal dimension $N$ is generated by monomials of the form $c_M b_L$ where
\ben
|L| +|M| = (l_1 + \cdots l_p)  + (m_1 + \cdots m_q) = - N .
\een
Define the subspace $F^k \CDO_n^{(N)}$ as the $\CC$-linear span of all
elements of the form $c_M b_L$ such that $p \leq k$ where $p$ is the
length of the multi-index $L$ as above. This construction also
provides a filtration on $\hCDO_n^{(N)}$.

\def\Gr{{\rm Gr}}

\begin{prop} 
The associated graded 
\ben
{\rm Gr} \, \CDO_n := \bigoplus_{N \in \ZZ} \; {\rm Gr}(\CDO_n^{(N)}) = \bigoplus_N \left(\bigoplus_k F^k \CDO_n^{(N)}/ F^{k-1} \CDO_n^{(N)} \right)
\een
has the structure of a $\ZZ_{\geq 0}$-graded Poisson vertex algebra, as does ${\rm Gr} \, \hCDO_n$. 
\end{prop}

Roughly, a vertex algebra can be thought of as an integer family of
products. A Poisson vertex algebra is essentially a commutative vertex algebra
together with an integer family of Lie brackets that act on the underlying
commutative vertex algebra by derivations. For the precise
definition of a Poisson vertex algebra see Chapter 16 Section 2 of \cite{BZF}. The fact that $\Gr \, \CDO_n$ is a Poisson vertex algebra follows from the well-known fact. 

\begin{prop}[Chapter 16 of \cite{BZF} ]
If $V$ is a filtered vertex algebra such that $\Gr \; V$ is a commutative vertex algebra,
then $\Gr \; V$ carries a canonical structure of a Poisson vertex algebra.
\end{prop}

\begin{rmk} 
The associated graded $\Gr\, \CDO_n$ can be thought of as a {\em classical limit} of the vertex algebra $\CDO_n$. 
We can introduce a deformation parameter $\hbar$ by modifying the definition of the vertex operator to
\[
Y(b^j_{-1},z) = \sum_{n < 0} b_n^j z^{-1-n} + \hbar \sum_{n \ge 0} \frac{\partial}{\partial c_{-n}^j} z^{-1-n} 
\]
and 
\[
Y(c_{0}^j,z) = \sum_{n \le 0} c^j_n z^{-n} - \hbar \sum_{n > 0}  \frac{\partial}{\partial b_{-n}^j} z^{-n}.
\]
These formulas define a vertex algebra $\CDO_{n,\hbar}$ over the ring $\CC_\hbar = \CC [ \hbar]$ whose specialization $\hbar = 1$ agrees with $\CDO_n$. 
Moreover, when we specialize $\hbar = 0$ we get the Poisson vertex algebra above. 
It is called the ``classical'' $\beta\gamma$ vertex algebra.
\end{rmk}

For each conformal dimension $N$, we can thus identify the associated graded ${\rm Gr}(\CDO_n^{(N)})$ with a direct sum of symmetric powers of tensor modules on $\CC^n$. Under this identification, the Lie algebra of polynomial vector fields 
\ben
{\rm W}^{\rm poly}_n = \CC[t^1,\ldots,t^n]\{\partial_1,\ldots,\partial_n\}
\een
acts via Lie derivative on ${\rm Gr} \; \CDO_n$. Expressing the Lie algebra of
{\em formal} vector fields as $\Vect = \hO_n \tensor_{\cO_n} {\rm
  W}_n^{\rm poly}$, we find that $\Vect$ acts on the vertex algebra
${\rm Gr} \, \hCDO_n$ via derivations.

\begin{prop} \label{grcdo}
The construction in the preceding paragraph defines the structure of a $(\Vect, \GL_n)$-module on the $\ZZ_{\geq 0}$-graded vertex algebra ${\rm Gr} \; \hCDO_n$ preserving the family of brackets definining the Poisson vertex algebra structure. Moreover, this action is compatible with the $\hO_n$-module structure. 
\end{prop}

\begin{proof} 
The associated graded can be written as
\[
\Gr \; \hCDO_n \cong
\underset{{0\leq k}}{\bigotimes} \,\Sym_{\hO_n}(\hOmega^1_n)\,
\otimes \, \underset{{0<
    l}}{\bigotimes} \,\Sym_{\hO_n}(\hT_n) .
\]
The action by $\Vect$ and $\GL_n$ is by Lie derivative and changes of
linear
frame on the respective tensor
bundles appearing in the large decomposition above. 
\end{proof}

\subsubsection{Conformal structure}\label{sec conformal structure}

The vertex algebra $\CDO_n$, and its completion $\hCDO_n$, has the additional structure of a {\em
  conformal vertex algebra} of central charge equal to twice the dimension
$2n$. This means that $\hCDO_n$ receives a map
from the Virasoro vertex algebra, ${\rm Vir}_{c=2n}$, of central
charge $c = 2n$. The Virasoro vertex algebra is the $\ZZ_{\geq 0}$-graded with underlying vector space ${\bf
  Vir}_c = \CC [L_k, C]$ where $k \leq -2$ and generating field given
by 
\ben
Y(L_{-2}, z) = \sum_{k \in \ZZ} L_k z^{-k-2} .
\een
where the conformal dimension of $L_{-2}$ is $2$. A conformal vector for $\hCDO_n$ is defined by
\ben
L_{-2} := \sum_{i=1}^n b^i_{-1} T c^i_0 = \sum_{i=1}^n b^i_{-1}
c^i_{-1} \in \hCDO_n^{(2)}.
\een

\begin{rmk} 
If $V$ is a conformal vertex algebra and $L_0 \in V$ is the zero Fourier mode of the Virasoro field, 
then for every $v \in V$, one has $L_0 v = N(v) v$ for some $N(v) \in \ZZ$. 
If $N(v) \in \ZZ_{\geq 0}$ for all $v$, 
we see that $L_0$ determines the structure of a $\ZZ_{\geq 0}$-graded vertex algebra on $V$, 
where $V^{(N)}$ denotes the $N$-eigenspace of the $L_0$ operator. 
This fact motivates the use of the term ``conformal dimension'' for a $\ZZ_{\geq 0}$-graded vertex algebra.
\end{rmk}

\subsection{Harish-Chandra structure on $\CDO_n$}\label{sec hc cdo}

As opposed to the classical limit the vertex algebra $\hCDO_n$ is {\em not} a module for the pair $(\Vect, \GL_n)$. The main result of this section is to show that there is an extension of this Harish-Chandra pair that does act on the vertex algebra. 
This is largely based on the work of \cite{MSV} and \cite{GMS}, as
well as \cite{GMS2}, and we summarize their results below.
 
\subsubsection{Extension of vector fields}

On the Lie algebra side, the extension of Lie algebras that acts on
$\hCDO_n$ is precisely the extension $\TVect$ of formal vector fields
by $\hOmega^2_{n,cl}$ defined by the Gelfand-Fucks second component of
the Chern character defined in Section \ref{sec gk descent}. We now recall the
construction in \cite{MSV} that describes how this extension acts. We can state the main result as follows. 

\begin{thm}\label{MSV1} [Section 5.1 of \cite{MSV}]
There is map of Lie algebras
\ben
\rho : \TVect \hookrightarrow {\rm Der}_{\rm VA}(\hCDO_n) 
\een
of the extended Lie algebra $\TVect$ into derivations of vertex algebra of chiral differential operators on the formal $n$-disk. 
In particular, $\hCDO_n$ is a $\TVect$-module.
\end{thm}

First, we describe how one embeds the vector space of formal vector fields inside of chiral differential operators on $\hD^n$. We have already described how to map a formal power series $f(t_1,\ldots,t_n)$ to an element $f(c_0^{1},\ldots, c_0^{n}) \in \hCDO_n^{(0)}$. This puts the structure of an $\hO_n$-module on $\hCDO_n^{(N)}$ for each $N$. Note that as $\hO_n$-modules we have a splitting $\Vect = \hO_n \tensor \CC \{\partial_1,\ldots,\partial_n\}$. We define 
\[
\begin{array}{cccc}
\tau_{\rm W} : & \Vect &\to& \hCDO_n^{(1)}, \\
& f(t) \partial_j & \mapsto & \tau(f) b_{-1}^j = f(c) b_{-1}^j.
\end{array}
\]
In other words, we substitute $c^i_0$ for $t_i$ in the power series $f$ and replace $\partial_j$ by~$b_{-1}^j$.

The subspace of vectors of conformal dimension one $\hCDO_n^{(1)}$ acts on the vertex algebra through left multiplication by its zero Fourier mode
\ben
(-)_{(0)}(-) : \CDO_n^{(1)} \times \CDO_n \to \CDO_n
\een
In fact, for a fixed $a \in \hCDO_n^{(1)}$ the endomorphism $a_{(0)}$ is a derivation of the vertex algebra. The composite map of taking the zero mode after $\tau_{\rm W}$ thus produces a linear map
\[
\begin{array}{cccc}
\rho_{\rm W} :  &\Vect &\to& {\rm Der}_{{\rm VA}} (\hCDO_n),\\ 
&f(t) \partial_j& \mapsto & (\tau(f) b_{-1}^j)_{(0)}=(f(c) b_{-1}^j)_{(0)} .
\end{array}
\]
Moreover, for any $a \in \hCDO_n^{(1)}$, 
the derivation $a_{(0)} : \hCDO_n \to \hCDO_n$ preserves the $\ZZ$-grading 
and so defines a map $a_{(0)} : \hCDO^{(N)}_n \to \hCDO^{(N)}_n$ for each $N \in \ZZ$. 
A quick calculation verifies that this map is {\em not} a map of Lie algebras. 
This issue is remedied this by introducing an extension of Lie algebras, as we will see shortly.

\subsubsection{}
We introduce the space of $1$-forms $\hOmega^1_n$ on the formal disk. Considered as an abelian Lie algebra this acts on the vertex algebra $\hCDO_n$ as we now describe.

The de Rham differential $\d_{dR} : \hO_n \to \hOmega_n^1$ has an
interpretation in the vertex algebra $\hCDO_n$ as the translation operator $T$ defining the vertex algebra structure. 
Indeed, we define
\[
\begin{array}{cccc}
\tau_{\Omega^1} : &\hOmega_n^1 &\to &\hCDO_n^{(1)},\\
& f(t) \d t_j & \mapsto & \tau(f) T(c_0^j) = f(c) T(c_0^j).
\end{array}
\]
As $\hOmega^1_n$ is abelian, the map $\tau_{\Omega^1}$ automatically determines a Lie algebra representation of $\Omega^1_n$ on $\CDO_n$ via the Lie algebra homomorphism
\[
\begin{array}{cccc}
\rho_{\Omega^1} : &\hOmega_n^1 &\to& {\rm Der}_{{\rm VA}}(\hCDO_n), \\ 
&\omega &\mapsto& \tau_{\Omega^1}(\omega)_{(0)} .
\end{array}
\]
It is clear that the action by an exact one-form is zero, so $\rho_{\Omega^1}$ factors as
\ben
\hOmega_n^1 \to \hOmega_n^1 / \d \hO_n \cong \hOmega^2_{n,cl} \xto{\rho_{\Omega^2_{cl}}}  {\rm Der}_{{\rm VA}}(\hCDO_n),
\een
where we have identified $\hOmega_n^1 / \d \hO_n \cong \hOmega^2_{n,cl}$ via the de Rham differential. 
The map $\rho_{\Omega^2_{cl}}$ is the desired action by closed two-forms. 

We can explicitly describe the action by a closed two-form $\omega$ as follows.
Let $\alpha = \alpha_i (t) \d t^i$ be a one-form such that $\d \alpha =  \omega$.
Then
\ben
\rho_{\Omega^2_{cl}} (\omega) = - \left(\partial_j \alpha_i (c) T(c_0^j) T (c_0^{i}) + \alpha_i(c) T (c_0^{i})^2 \right)_{(1)} .
\een

\subsubsection{}

Consider the linear subspace 
\[
\Bar{\rm W}_n := {\rm Im}(\rho_{\rm W}) \oplus {\rm Im}(\pi_{\Omega^2_{cl}}) \subset {\rm Der}_{{\rm VA}}(\hCDO_n).
\] 
A direct calculation shows that $\Bar{\rm W}_n$ is actually a sub-Lie algebra of the vertex algebra derivations. 
It is immediate that ${\rm Im}(\rho_{\Omega^2_{cl}})$ is an ideal in $\Bar{\rm W}_n$ and the quotient is isomorphic to $\Vect$.
Thus, $\Bar{\rm W}_n$ sits in a short exact sequence
\ben
0 \to \hOmega^2_{n,cl} \to \Bar{\rm W}_n \to \Vect \to 0 .
\een
The 2-cocycle determining this extension is
\ben
\alpha_{MSV}(f^i \partial_i, g^j \partial_j) = - \d_{dR}(\partial_j f^i) \wedge \d_{dR}(\partial_i g^j) .
\een 
This cocycle is precisely the cocycle $\ch_2^{\GF}(\hT_n)$ determining the extension $\TVect$, so that we have $\TVect \cong \Bar{\rm W}_n$. 

\begin{rmk} 
In \cite{MSV} the connection to the Gelfand-Fuks Chern character is not present, 
though our cocycle agrees on the nose with the vertex algebra calculation. 
\end{rmk}

We have thus constructed a map of Lie algebras 
\[
\Tilde{\rho}_{\rm W} = (\rho_{\rm W}, \rho_{\Omega^2_{cl}}) : \TVect \cong \Bar{\rm W}_n \to {\rm Der}_{{\rm VA}}(\hCDO_n),
\]
as desired.

\subsubsection{}

We have already described the action of $\GL_n$ on $\hCDO_n$ in
(\ref{GLact1}) and (\ref{GLact2}). Combining the results highlighted
in the section above we obtain the following. This action is
compatible with the action of $\Vect$ just constructed. In conclusion,
we can summarize the above as follows. 

\begin{prop}
\label{hc str on cdo}
The pair $(\TVect, \GL_n)$ acts on the $\ZZ_{\geq 0}$-graded vertex algebra $\hCDO_n$. Moreover, this action is compatible with the $\hO_n$-module structure. 
\end{prop}

\subsection{Formal automorphisms}

The above construction of the action of the pair $(\TVect, \GL_n)$ on
the vertex algebra of CDOs will be sufficient for our purposes. In
this section we review the main result of \cite{GMS} which constructs
an action of an extension of {\em all} formal automorphsims $\Tilde{\Aut}_n \to
\Aut_n$ on CDOs. This action is compatible with our construction
above.   

\begin{thm} 
\label{GMS1}
Let $\Aut_{VA}(\hCDO_n)$ denote the group of automorphisms of the
vertex algebra $\hCDO_n$. There is a subgroup 
\ben
\Tilde{\rm Aut}_n \hookrightarrow {\rm Aut}_{\rm VA}(\hCDO_n)
\een
that fits in a short exact sequence of groups
\be\label{grpses}
0 \to \Hat{\Omega}^2_{n,cl} \to \Tilde{\rm Aut}_n \to {\rm Aut}_n \to 1 .
\ee
\end{thm}

In \cite{GMS}, this subgroup $\Tilde{\rm Aut}_n \subset {\rm Aut}_{{\rm VA}}(\hCDO_n)$ is characterized as the ``natural'' vertex algebra automorphisms.
We will outline their argument and attempt to explain the sense of ``natural'' here.

First, as $\hCDO_n$ is $\ZZ$-graded by conformal dimension, it is reasonable to restrict to dimension-preserving automorphisms,
which will be determined by where they send the generators.
As discussed above, the generators are in dimensions 0 and 1:
the dimension 0 component can be identified with $\hO_n$ --- functions on the formal $n$-disk --- 
and the dimension 1 component with $\hOmega^1 \oplus \hT$ --- one-forms and vector fields on the formal $n$-disk.
We view the dimension 1 component as 2-step filtered, with  $\hOmega^1$ as the submodule.

Before worrying about the vertex algebra structure, let us consider dimension-preserving maps of the space of generators.
This group is $\Aut(\hO_n) \times \Aut(\hOmega^1 \oplus \hT)$.
Following \cite{GMS}, we restrict our attention to an important subgroup.
On the dimension 0 component, they only consider the subgroup $\Aut_n$.
Note that every element of $\Aut_n$ acts on the dimension 1 component, since they are tensor fields,
so there is a natural map $\Aut_n \to \Aut(\hOmega^1 \oplus \hT)$.
On the dimension 1 component, they restrict to automorphisms whose action respects the filtration 
and whose associated graded action on $\Gr\, \hCDO_n$ is simply the action induced by the underlying automorphism on the dimension 0 component.
In other words, such an automorphism $\phi$ is ``triangular'':
it consists of a term $\phi_0 \in \Aut_n$ and of an $\hO_n$-linear map $\phi_1: \hT \to \hOmega^1$,
and 
\[
\phi(f, \omega, X) = (\phi_0 \cdot f, \phi_0 \cdot \omega + \phi_1(X), \phi_0\cdot X) \in \hO_n \oplus \hOmega^1 \oplus \hT.
\]
Let us use $\Aut^{GMS}_n$ to denote this group considered in \cite{GMS}.
The underlying set is isomorphic to the product 
\ben
{\rm Aut}_n \times {\rm Mat}_n(\ccO_n),
\een
by using the natural isomorphism 
\[
\hOmega^1_n \widehat{\otimes}_{\hO_n} \hOmega^1_n \cong {\rm Mat}_n(\ccO_n).
\]
But this group $\Aut^{GMS}_n$ has an interesting group structure because of how $\Aut_n$ acts on the dimension 1 component.
In fact, it has the structure of a semi-direct product ${\rm Aut}_n \ltimes_{pb} (\hOmega^1_n)^{\tensor 2}$, where the pull-back action is as above.

By definition, the group $\Tilde{\rm Aut}_n$ is the subgroup of $\Aut^{GMS}_n$ consisting of vertex algebra automorphisms. 
In other words, we pick out the dimension-preserving automorphisms of generators that intertwine with the vertex operator and so on.

In \cite{GMS} it is shown that the composition 
\[
\widetilde{\Aut}_n \to\Aut^{GMS}_n \to \Aut_n
\]
is surjective and that its kernel is isomorphic to closed 2-forms. That is, one has a map of extensions
\ben
\xymatrix{
\Hat{\Omega}^2_{n,cl} \ar[r] \ar[d] & \widetilde{\Aut}_n \ar[r] \ar[d] & \Aut_n \ar[d] \\
\Hat{\Omega}_n^1 \tensor \Hat{\Omega}^1_n \ar[r] & \Aut^{GMS}_n \ar[r] & {\rm Aut}_n.
}
\een 
This identifies the relevant short exact sequence (\ref{grpses}). 

\subsubsection{An explicit formula for the cocycle} 
\label{sec GMScocycle}

In this section we describe an explicit group 2-cocycle
\ben
\Tilde{\alpha}_{GMS} \in {\rm C}_{\rm Grp}^2(\Aut_n ; \hOmega^2_{n,cl}) .
\een
describing the extension (\ref{grpses}). First, we elaborate on what
we mean by a group 2-cocycle. 

We use the van Est model for smooth group cohomology and denote the cochains by ${\rm C}_{\rm Grp}^*$.
(See Chapter 3 of \cite{Fuks} for more discussion.) 
Given a Lie group $G$ and $M$ a representation,
let ${\rm C}_{\rm Grp}^k(G ; M)$ denote the space of smooth functions $C^\infty(G^k,M)$.
(Typically we have in mind a finite-dimensional representation,
but it is well-defined for any vector space $M$ such that smooth maps $C^\infty(G^k,M)$ is defined.)
The differential $\d_{Grp}$ is defined by
\ben
(\d_{\rm Grp} \alpha)(g_1,\ldots, g_{k+1}) = 
g_1 \alpha(g_2, \ldots, g_{k+1}) + 
\sum_{i = 1}^{k} (-1)^i \alpha(g_1,\ldots,g_{i}g_{i+1}, \ldots, g_{k+1}) +
(-1)^{k+1} \alpha(g_1,\ldots,g_k) .
\een 
When $M$ is itself a cochain complex with differential $\d_M$, 
we naturally obtain a double complex.
Let ${\rm C}_{\rm Grp}^*(G; M)$ denote the associated total complex, combining the differential $\d_{\rm Grp}$ and~$\d_M$. 

Note that  a 2-cocycle $\alpha$ of ${\rm C}_{\rm Grp}^*(G; M)$ determines an extension
\ben
0 \to M \to \Tilde{G}_\alpha \to G \to 0
\een
where the group structure on $\Tilde{G}_\alpha$ is defined by
\[
(g_1, m_1) \cdot (g_2, m_2) = (g_1 g_2, m_1 + g_1m_2 + \alpha(g_1,g_2)),
\]
in the standard way.

We now proceed to write down a formula for the cocycle associated to the extension (\ref{grpses}). 
Much of the argumentation below is implicit in Section 6 of \cite{GMS2} (in the context of the closely related
\v{C}ech approach to CDOs) and also in \cite{GMS}, and we refer the reader to these sources for more details.  

Let $\hOmega^{\geq 2}_n$ denote a truncation of the de Rham complex: it is the total complex of the double complex
\[
\hOmega^{2}_n \xto{\d_{dR}} \hOmega^3_n \xto{\d_{dR}}\cdots \xto{\d_{dR}} \hOmega^n_n.
\]
There is a natural action of $\Aut_n$ on this complex, as Cartan's
formula for the action of vector fields on differential forms
intertwines with the de Rham differential. 

\def\D{{\rm D}}

First, we write down a 2-cocycle $\alpha_{GMS} \in {\rm C}_{\rm Grp}^*(\Aut_n ; \hOmega^{\geq 2}_n)$,
following~\cite{GMS}.
Given an element $f \in \Aut_n$, we will use $\D f$ to denote its Jacobian.
We give an explicit formula for $\alpha_{GMS}$ via a pair of maps $(\alpha_2, \alpha_3)$, where 
\ben
\begin{array}{cccc}
\alpha_2 : & \Aut_n \times \Aut_n & \to & \hOmega^{2}_n \\ 
& (f_1,f_2) & \mapsto&\tr\left((\D f_1)^{-1} \d_{dR} (\D f_1) (\d_{dR}(\D f_2) (\D f_2)^{-1})\right)
\end{array}
\een
and
\ben
\begin{array}{cccc}
\alpha_3 : & \Aut_n & \to & \hOmega^3_n \\ 
& f & \mapsto & \frac{1}{3} \tr\left( ((\D f)^{-1}\d_{dR} \D f)^3\right) .
\end{array}
\een
This cochain is of degree 2 and has no terms of type $\Omega^k$ for $k \geq 4$. 
One
immediately checks that this is a cocycle. That is, 
\begin{align*}
\d_{dR} \alpha_2 (f_1,f_2) & =  (\d_{\rm Grp} \alpha_3) (f_1,f_2),  \\
\d_{dR} \alpha_3 (f_1) & = 0 , \\
(\d_{\rm Grp} \alpha_2)(f_1,f_2,f_3) & = 0 ,
\end{align*}
for all $f_1,f_2,f_3 \in \Aut_n$. 
The last two equations are immediate by computation. 
The first equation follows from the relation
\be\label{PW1}
\alpha_3(f_2 \circ f_1) = \alpha_3(f_1) + f_1^* \alpha_3(f_2) - \d_{dR} \alpha_2(f_1,f_2), 
\ee
which is an instance of the Polyakov-Wiegmann identity. The Jacobian of the composition $f_2
\circ f_1$ is given by $\D(f_2 \circ f_1) = f_1^* (\D f_2) \D f_1$ as matrix-valued formal power series. 
Thus, for instance, we have
\ben
\d_{dR} \left(\D(f_2 \circ f_1)\right) = f_1^*(\d_{dR} \D f_2) \D f_1 + f_1^* \D f_2
\d_{dR} \D f_1 . 
\een 
Let $\D f = Jac(f_2 \circ f_1)$ so that $\alpha_3(f_2 \circ f_1) =
\frac{1}{3} \Tr\left(((\D f)^{-1} \d_{dR} \D f)^3\right)$. Plugging in the formula for the
Jacobian we compute
\bestar
\frac{1}{3} \left((\D f)^{-1} \d_{dR} \D f\right)^3 & = & \frac{1}{3} \left((\D f_1)^{-1} f_1^*((\D f_2)^{-1} \d_{dR}
\D f_2) \D f_1\right)^3 + \frac{1}{3} \left((\D f_1)^{-1} \d_{dR} \D f_1\right)^3 +
\{{\rm cross \; terms}\} .
\eestar
Taking the trace of both sides we see that the first two terms return
the first two terms of Equation (\ref{PW1}). In a similar way, a
direct (albeit tedious) calculation shows that the cross terms agree with $\d_{dR}
\alpha(f_1,f_2)$. 

By the formal Poincar\'{e} lemma we know that the inclusion 
$\hOmega^{2}_{n,cl} \hookrightarrow \hOmega^{\geq 2}_n$ is a quasi-isomorphism.
Moreover, this quasi-isomorphism is clearly $\Aut_n$-equivariant 
so that we have a resulting quasi-isomorphism of complexes 
\ben
{\rm C}_{\rm Grp}^*(\Aut_n ; \hOmega^2_{n,cl}) \to {\rm C}_{\rm Grp}^*(\Aut_n ; \hOmega^{\geq 2}_n) .
\een 
A lift $\Tilde{\alpha}_{GMS} \in \clie^2(\Aut_n ; \hOmega^2_{n,cl})$
of the cocycle $\alpha_{GMS} = (\alpha_2,\alpha_3)$ under this
quasi-isomorphism is a representative for the group extension~(\ref{grpses}). 

We can obtain an explicit formula as follows. 
Since $\d_{dR} \alpha_3(f) = 0$ for all $f$, 
the formal Poincar\'e lemma assures the existence of a map
$\mu : \Aut_n \to \hOmega^2_n$
such that $\d_{dR} \mu  = \alpha_3$. 
We define the 2-cocycle
\ben
\Tilde{\alpha}_{GMS}(f_1,f_2) = \alpha_2(f_1,f_2) + \mu(f_1) +
f_1^*\mu(f_2) - \mu (f_2 \circ f_1) .
\een 
Via the Polyakov-Wiegmann identity (\ref{PW1}), this element is closed and determines a
2-cocycle in ${\rm C}_{\rm Grp}^2(\Aut_n ; \hOmega^2_{n,cl})$. 

\subsubsection{} 

We discuss how the construction of $\Tilde{\Aut}_n$ and its action on $\hCDO_n$
from Proposition \ref{GMS1} is
compatible with the action of $(\TVect,\GL_n)$ on $\hCDO_n$ that we constructed in
Proposition \ref{hc str on cdo}. First, we see that the group cocycle $\alpha_{GMS}$ is
compatible with the cocycle $\ch_2(\hT_n)$ defining $\TVect$. 

Given any Lie group and $G$-representation $M$,
the derivative at the identity of $G$ (and its products $G^k$) determines a cochain map
\ben
D_1 : {\rm C}_{\rm Grp}^*(G ; M) \to \clie^*(\fg ; M),
\een 
where we view $M$ as a $\fg = \Lie(G)$-module on the right hand side. 
Explicitly, given a $k$-cochain $\alpha$ of $G$ we define
\ben
(D_1 \alpha)(x_1,\ldots,x_k) = \frac{\d}{\d t}  \left. \left(
  \alpha(x_1(t), \ldots,x_k(t)\right) \right|_{t= 0}
\een
where $x_i(t)$ is the flow on $G$ determined by $x_i \in \fg$. 

The Lie algebra of formal automorphisms of the $n$-disk is identified
with the subalgebra $\Vectz \subset \Vect$ consisting of formal vector
fields that vanish at the origin. Thus, there is a map of vector
spaces 
\be\label{lie1}
D_1 : {\rm C}_{\rm Grp}^2({\rm Aut}_n; \Hat{\Omega}^2_{n,cl}) \to \clie^2(\Vectz ; \Hat{\Omega}^2_{n,cl}) 
\ee
induced by taking the tangent space at the identity.

\begin{prop} 
The image of $\Tilde{\alpha}_{GMS}$ under the map (\ref{lie1}) is equal to 
the restriction of $\ch_2^{\rm GF}(\hT_n)$ to formal vector fields that vanish at the origin. 
\end{prop}

This proposition shows that $(\TVect, \Tilde{\rm Aut}_n)$ is a Harish-Chandra
pair extending the pair $(\Vect, \Aut_n)$. Combined with Theorem
\ref{GMS1} of \cite{GMS} we see that $(\TVect, \Tilde{\Aut}_n)$ acts on the vertex algebra
$\hCDO_n$. This action is compatible with the action of the pair
$(\TVect, \GL_n)$ we have constructed from Proposition \ref{hc str on cdo} in the following way. 

There is a natural map $p : \Tilde{\Aut}_n \to \GL_n$ that takes a formal automorphism together with a closed two-form and
maps it to the linear piece of the 1-jet of the automorphism. This is
clearly equivariant for the action of vector fields so that we have an
induced map of pairs $p : (\TVect, \TAut_n) \to (\TVect, \GL_n)$. 
The choice of a formal coordinate determines a splitting $s : \GL_n \to \Aut_n$ and hence a map of pairs 
$s : (\TVect, \GL_n)\to (\TVect, \TAut_n)$.
The action of $(\TVect, \GL_n)$ on $\hCDO_n$ constructed in
Proposition \ref{hc str on cdo} is the restriction along the map $s$ of the action by
$(\TVect, \TAut_n)$ constructed in~\cite{GMS}. 

\subsection{The conformal structure for the equivariant vertex algebra}

We have already seen that the $\beta\gamma$ vertex algebra is
conformal so that there is a map of vertex algebras $\Phi_{\rm Vir} :
{\rm Vir}_{c=n} \to \hCDO_n$. This map is {\em not} equivariant for
the action of the extended Lie algebra $\TVect$ (where we equip ${\bf
  Vir}_{c=n}$ with a trivial $\TVect$ action). We will see that the
failure for this map to be a map of $\TVect$-modules is measured by a
certain Gelfand-Kazhdan characteristic class. 

The map of vertex algebras $\Phi_{\rm Vir}$ is completely determined
by where it sends the Virasoro generator, which we called $L_{-2} \in
\hCDO_n^{(2)}$. Since ${\rm Vir}_{c=n}$ has the trivial $\TVect$
module structure, we see we see that $\Phi_{\rm Vir}$ is map of $\TVect$-modules if and only if $\Tilde{X} \cdot L_{-2}$ is zero for all $\Tilde{X}$ in $\TVect$. An immediate calculation shows that closed two-forms act on $L_{-2}$ by zero, thus it suffices to look at $X \cdot L_{-2}$ for $X \in \Vect$. 

Given any element $a \in \hCDO_n^{(2)}$ we obtain a linear map given
by the second Fourier mode
\ben
a_{(2)} : \hCDO_n^{(1)} \to \hCDO_n^{(0)}. 
\een 
For $X \in \Vect$, the element $X \cdot L_{-2} \in \hCDO_n^{(2)}$ thus determines a map $(X \cdot L_{-2})_{(2)} : \hCDO_n^{(1)} \to \hCDO_n^{(0)}$. 

Finally, recall that we have described a map of $\Vect$-modules $\tau_{\Omega^1} : \hOmega^1_n \to \hCDO_n^{(1)}$. It's cokernel is identified with $\hT_n$. That is, there is a short exact sequence of $\Vect$-modules
\ben
0 \to \hOmega^1_n \to \hCDO_n^{(1)} \to \hT_n \to 0 .
\een 

\begin{prop}\label{prop c1 conformal} For each $X \in \Vect$ the linear map $(X \cdot L_{-2})_{(2)} : \hCDO_n^{(1)} \to \hO_n$ factors through the quotient $\hT_n$ 
\ben
\xymatrix{
\hCDO_n^{(1)} \ar[d] \ar[r] & \hO_n \\
\hT_n \ar@{.>}[ur]_{\alpha(X)} &  .
}
\een 
and hence determines an $\hO_n$-linear map $\alpha(X) : \hT_n \to \hO_n$ as in the diagram. Moreover, the assignment $X \mapsto \alpha({X})$ defines a cocycle in $\clie^1(\Vect ; \hOmega^1_n)$ and is cohomologous to the Gelfand-Fuks-Chern class $c^\GF_1(\hT_n) \in \clie^1(\Vect ; \hOmega^1_n)$. 
\end{prop}

\begin{proof}
The fact that $(X \cdot L_{-2})_{(2)}$ factors through $\hT_n$ follows from the following short calculation. 
\begin{lemma} For any $c \in \hCDO_n^{(0)}$ we have $(L_{-2})_{(2)} (T c)$. Similarly, for $X \in \Vect$ one has $(X \cdot L_{-2})_{(2)}(T c) = 0$. 
\end{lemma}
\begin{proof} Set $L = L_{-2}$. Since $T$ is a derivation we have $T(L_{(2)}c) = L_{(2)}(T a) + (T L)_{(2)} a$. Thus $L_{(2)}(T c) = T(L_{(2)} c) - (T L)_{(2)} c$. For conformal dimension reasons we have $L_{(2)} c = 0$, thus $L_{(2)}(Tc) = -(TL)_{(2)} c = 2 L_{(1)} c$, again since $T$ is a derivation. The element $L$ is a Virasoro vector, thus $L_{(1)} = T$, so that $L_{(1)} c = Tc = 0$, since $c$ is of degree zero. Similarly, for $X \in \Vect$, we have $(X \cdot L)_{(2)} (Tc) = L_{(2)}(X \cdot Tc) = L_{(2)}(T X \cdot c))$ as $X$ is a derivation of the vertex algebra.
\end{proof}

We thus obtain a linear map $\alpha : \Vect \to \hOmega^1_n$. We verify that this is equal to $c_1^{\rm GF}(\hT_n)$. The formula for this Chern class is given by
\ben
c_1^{\rm GF}(\hT_n) (X) = \d_{dR}(\partial_i f_i)
\een
where $X = f_i \partial_i \in \Vect$. We utilize the following Borcherds identity for how Fourier modes compose
\ben
(a_{(l)} b)_{(m)} c = \sum_j (-1)^j \begin{pmatrix} l \\ j \end{pmatrix} \left(a_{(l-j)} b_{(m+j)} c - (-1)^l b_{(l+m-j)} a_{(j)} c \right) .
\een
First, we simplify $X \cdot L_{-2} = (f_i(c) b_{-1}^i)_{(0)} (b^k_{-1} T c_0^k)$. Since $x_{(0)} (T c_0^k) = 0$ for any $x$, we see 
\ben
X \cdot L_{-2} = \left((f_i (c) b_{-1}^j)_{(0)} (b_{-1}^k) \right) T c_0^k .
\een
By the Borcherd's formula this simplifies to $(- (b_{-1}^k)_{(0)} (f_i(c))  b_{-1}^j) T c_0^k = - (\partial_k f_i)(c)T c_0^k  b_{-1}^i$. We compute the value of $(X \cdot L_{-2})_{(2)}$ on the generators $b_{-1}^j$. There is only one term in the Borcherd's expansion and it is of the form
\begin{align*}
(X \cdot L_{-2})_{(2)} (b_{-1}^j) & = \left(b_{-1}^i\right)_{(1)} \left( (\partial_k f_i (c) T c_0^k)_{(0)} b_{-1}^j \right) \\ & = - (b_{-1}^i)_{(1)} \left((b_{-1}^j)_{(0)} (\partial_k f_i (c)) T c_0^k \right) \\ & = - \delta^{ki} \partial_j \partial_k f_i .
\end{align*}
Thus $\alpha(X) = c_1^{\GF}(\hT_n)(X)$ and the proof is complete. 
\end{proof}

\subsection{The character of a graded vertex algebra} \label{sec vert character}

In this section we define and compute the ``local character" of the vertex algebra $\hCDO_n$. It will globalize, under Gelfand-Kazhdan descent, to the character of the sheaf of chiral differential operators on a complex manifold $X$. 

\begin{dfn} Let $V$ be a $\ZZ_{\geq 0}$-graded vertex algebra.
The {\em graded character} of $V$ is the following $q$-expansion
\be\label{char1}
\chi (V) := \sum_{N} q^{N} \left(\dim V^{(N)}\right) \in \CC[[q]] .
\ee
\end{dfn}

\begin{rmk} 
When $V$ is a {\em conformal} vertex algebra, there is a slight variant of the graded character that involves the central charge $c$ of $V$. 
If $L_0$ is the zero mode of the Virasoro vector in the conformal vertex algebra, 
the character is defined by ${\rm char}(V) := {\rm Tr}_{V} q^{L_0 - c/24}$. 
The relationship to the graded character we defined in Equation~(\ref{char1}) is given by $q^{-c/24} \chi(V) =  {\rm char}(V) \in q^{-c/24} \CC[[q]]$. 
The reason for this extra factor of $q$ is that $\chi(V)$ has nicer modular properties. 
For more about this modularity, and motivation for the the definition of the character, see~\cite{Zhu}. 
\end{rmk}

We wish to define the graded character of a vertex algebra with an action of a Harish-Chandra pair $(\fg, K)$. 
Suppose $\fg$ acts on a  $\ZZ_{\geq 0}$-graded vertex algebra $V$ by grading-preserving derivations. 
Then, each weight space $V^{(N)}$ is a module for $\fg$. 
The character of the vertex algebra will be a $q$-expansion of equivariant characters of the individual spaces of fixed conformal dimension $V^{(N)}$. 
Thus, it suffices to define what we mean by the character of a $(\fg,K)$-module (in vector spaces). 

For simplicity we work just with the Lie algebra $\fg$. The generalization to a module for the pair $(\fg, K)$ is a small extension of this. For any $\fg$-module $W$, with action $\rho : \fg \to \End(W)$, its Chern character is given by $\ch^\fg(W) = {\rm Tr}\left(\exp(\rho(X)) \right) \in \Sym(\fg^\vee)$. 
Since the trace is conjugation invariant the character determines an element in the Hochschild homology of the algebra $\clie^*(\fg)$:
\[
{\rm ch}^\fg(W) \in {\rm HH}_0\left( \clie^*(\fg) \right) \cong
\cSym (\fg^\vee)^\fg .
\]
There is a way to express this character at the cochain level. For this, it is useful to have an interpretation of the character in terms of Lie algebra cohomology, which will coincide with the Gelfand-Fuks-Chern characters in the case of $(\fg, K) = (\Vect, \GL_n)$. 

Let ${\rm Hoch}_*(-)$ denote the complex of Hochschild chains, computing Hochschild homology. The Hochschild-Rosenberg-Kostant theorem for the commutative ring $R$ posits a quasi-isomorphism of cochain complexes 
\ben
{\rm Hoch}_*\left(R\right) \simeq \Omega^{-*}_{R}
\een
where $\Omega^{-*}_{R}$ is the regraded de Rham complex of the commutative ring $R$. In the case that $R = \clie^*(\fg)$ this quasi-isomorphism takes the form
\ben
{\rm Hoch}_*\left(\clie^*(\fg)\right) \simeq \clie^*\left(\fg ; \bigoplus_{k \geq 0} \Sym^k (\fg^\vee)[k] \right) . 
\een
The definition of the Atiyah class of a $\fg$-module $W$ can be found in \cite{GG1}. This class is an element $\At^\fg(W) \in \Omega^1_{B \fg} \tensor \End(W)$ gives a Chern-Weil description in Lie algebra cohomology of the Chern character above: 
\ben
\ch^\fg(W) = {\rm Tr}\left(\exp \left(\frac{1}{2\pi i} \At^{\fg}(W) \right) \right) .
\een 
We will encounter the Atiyah class later, in Part II. Other characteristic classes also admit a description in terms of this Atiyah class. For instance, the Todd class of the $\fg$-module $W$ is defined to be the determinant of a certain formal series involving the Atiyah class: 
\ben
{\rm Td}^\fg (W) = {\rm det} \left(\frac{1 - e^{- {\rm At}(W)}}{{\rm At}(W)} \right) .
\een 

The {\em Euler class} of the $\fg$-module $W$ is defined to be
\ben
\chi^{\fg} (W) := {\rm Td}^\fg(\fg[1]) \cdot \ch^\fg(W) \in \clie^*\left(\fg ; \bigoplus_{k \geq 0} \Sym^k (\fg^\vee)[k] \right) .
\een
If $W$ is a module for the Harish-Chandra pair $(\fg,K)$ the same construction defines the Euler class in relative Lie algebra cochains $\chi^{(\fg,K)}(W) \in \clie^*\left(\fg,K ; \bigoplus_{k \geq 0} \Sym^k (\fg^\vee)[k]\right)$. 

We now return to the case of a vertex algebra. 

\begin{dfn} 
If a pair $(\fg,K)$ acts on a $\ZZ_{\geq 0}$-graded vertex algebra $V$, 
the $(\fg,K)$-{\em equivariant graded character} of $V$ is the $q$-expansion
\ben
{\rm char}^{(\fg,K)}(V) := \sum_{N \geq 0} q^N \chi^{(\fg,K)} (V^{(N)})
\een 
in~$\clie^*\left(\fg,K ; \bigoplus_{k \geq 0} \Sym^k (\fg^\vee)[k]\right)[[q]]$.
\end{dfn}

\subsubsection{} 

We now turn to computing the character for the main example, 
the $(\TVect,\GL_n)$-equivariant vertex algebra $\hCDO_n$. 

There is one subtlety: the conformal weight spaces of $\hCDO_n$ (and $\CDO_n$) are not finite dimensional. 
They are, however, finite rank over the ring $\hO_n$, 
and so, whenever we count dimensions or take duals, we will do so in a $\hO_n$-linear way.

With this modification, the equivariant graded character of $\hCDO_n$ as a module for $(\TVect, \GL_n)$ will be an element 
\[
\chi^{(\Vect, \GL_n)} (\hCDO_n) \in \clie^*(\TVect , \GL_n ; \hOmega^{-*}_n)[[q]]. 
\]
Here, $\hOmega^{-*}_n$ is the regraded de Rham complex $\bigoplus_{k \geq 0} \hOmega^k_n [k]$ on the formal disk.  

Something important happens here: the tilde has vanished on $\Vect$ so that the character is the image of an element from~$\clie^*(\Vect , \GL_n ; \hOmega^{-*}_n)$.
To justify this location for the character, we use the following argument.

Let $p : \Tilde{\fg} \to \fg$ be a morphism of Lie algebras, and let $\mathfrak{k}$ be its kernel. 
We say that a finite-dimensional $\Tilde{\fg}$-module $V$ is {\em off-diagonal} for $p$ if there is a filtration 
\ben
0 = F^{-1} V \subset F^0 V \subset F^1 V \subset \cdots \subset F^N V = V
\een
such that for all $i$, $\mathfrak{k} \cdot F^i V \subset F^{j} V$ for some $j < i$. 
There is an elementary fact about traces of such modules.

\begin{lemma}
If $V$ is an off-diagonal module for the Lie algebra map $p : \Tilde{\fg} \to \fg$ and $V$ has finite dimension, 
then ${\rm tr}(\exp(x)) = \tr(\exp(p(x)))$ for all $x \in \Tilde{\fg}$. 
\end{lemma}

To see this, choose a filtration for $V$ exhibiting the off-diagonal action, and a pick a basis for $V$ compatible with this filtration. 
In terms of this basis, each element of $\frak{k}$ acts by a matrix that is strictly upper triangular (i.e., off-diagonally), and hence does not contribute to the trace. 

Similarly, we have the following.

\begin{lemma}\label{lem: offdiag}
Let $(\fg, K)$ and $(\Tilde{\fg}, K)$ be Harish-Chandra pairs with $(p, \id) : (\fg, K) \to (\Tilde{\fg}, K)$ a morphism of pairs. 
If $V$ is a finite-dimensional off-diagonal $(\Tilde{\fg},K)$-module,
then $\ch^\fg(V)$ is in the image of the map
\[
p^* : \clie^*\left(\fg,K ; \bigoplus_{k \geq 0} \Sym^k (\fg^\vee)[k]\right) \to  \clie^*\left(\Tilde{\fg},K ; \bigoplus_{k \geq 0} \Sym^k (\fg^\vee)[k]\right) .
\]
\end{lemma}

In our case, we take the map of Lie algebras $p : \TVect \to \Vect$ and the $\TVect$-module $\hCDO_n$.
Note that we will exhibit a filtration of infinite length.
 
Consider the basis of $\hCDO_n$ given by products of elements $c^i_m, b^j_l$. 
Define the filtered subspace $F^q \hCDO_n$ to be the subspace spanned by elements of the form
\ben
c_{m_1}^{i_1} \cdots c_{m_k}^{i_k} b_{l_1}^{j_1} \cdots b_{l_q}^{j_q} .
\een 
It is a quick computation to verify that the action of $\ker(p)=\hOmega^2_{n,cl}$ on the CDOs is off-diagonal for this filtration.

A slight modification of Lemma \ref{lem: offdiag}, which applies to the $q$-graded situation, implies the following. 

\begin{cor}\label{cor formal char cdo} 
The $(\TVect, \GL_n)$-equivariant graded character of the vertex algebra $\hCDO_n$ is the image of an element 
\ben
\chi^{(\Vect, \GL_n)}(\hCDO_n) \in \clie^*(\Vect , \GL_n ; \hOmega^{-*}_n) [[q]]
\een
along the pull-back $\clie^*(\Vect , \GL_n ; \hOmega^{-*}_n) \to \clie^*(\TVect , \GL_n ; \hOmega^{-*}_n).$
\end{cor}

\def\Td{{\rm Td}}

For the pair $(\fg,K) = (\Vect, \GL_n)$ the Chern character in the previous section coincides with the Gelfand-Fuks-Chern character $\ch^{\GF}(W)$ for any formal vector bundle $W$. 
Set $\Td^{\GF} := \Td^{(\Vect, \GL_n)}$. 

\begin{prop} \label{prop local character} 
The $(\Vect, \GL_n)$-equivariant graded character of $\hCDO_n$ is given by
\ben
\chi^{(\Vect, \GL_n)}(\hCDO_n) = \Td^{\GF} \cdot \ch^{\rm GF} \left(\bigotimes_{l \geq 1} \Sym_{q^l} (\hOmega^1_n \oplus \hT_n) \right)
\een
as an cocycle in $\clie^*\left(\Vect , \GL_n ; \hOmega^{-*}_n\right)[[q]]$.
\end{prop}

\begin{rmk}
Here we use the notation 
\[
\Sym_{q^l}(V) = \bigoplus_{k \geq 0} q^{kl} \Sym^k(V),
\] 
so that
\[
\ch\left(\Sym_{q^l}(V)\right) = \sum_{k \geq 0} q^{kl} \ch\left(\Sym^k(V)\right).
\]
\end{rmk}

\begin{proof}
The conformal dimension zero subspace of $\hCDO_n$ is identified with $\hO_n$ and the conformal dimension one subspace is identified with $\hOmega^1_n \oplus \hT_n$ (all as $\TVect$-modules). The full associated graded of CDOs is given by
\[
\Gr \; \hCDO_n \cong
\underset{{0\leq k}}{\bigotimes} \,\Sym_{\hO_n}(\hOmega^1_n)\,
\otimes \, \underset{{0<
    l}}{\bigotimes} \,\Sym_{\hO_n}(\hT_n) .
\]
Putting this all together we find
\begin{align*}
{\rm char}^{(\TVect, \GL_n)} (\hCDO_n) & = \sum_{N \geq 0} q^N {\rm
                                         ch}^{\GF}\left(\hCDO_n^{(N)}\right)
  \\ & = \ch^{\rm GF} \left(\bigotimes_{N \geq 1} \Sym_{q^N} (\hT_n \oplus
       \hOmega^1_n) \right) 
\end{align*}
as desired.
\end{proof}

\section{Extended Gelfand-Kazhdan descent} \label{sec ext desc}

Our construction of descent in Section \ref{sec gk descent} uses the Harish-Chandra pair $(\Vect, \GL_n)$. 
We have seen, however, that this pair is not appropriate 
if we wish to describe descent for the vertex algebra of chiral differential operators. 
In this section we develop the theory of descent for the pair $(\TVect, \GL_n)$, 
which does act on the vertex algebra, as we saw in the preceding section. 

\subsection{The extended bundle}
The central object in the construction of Gelfand-Kazhdan descent is the coordinated bundle $X^{coor}$. 
This space is a principal bundle for the group of formal automorphisms. 
Using a Gelfand-Kazhdan structure, we obtain from $X^{coor}$ a $\Vect$-valued flat connection on the frame bundle $\Fr_X$. 
In this section, we construct and classify lifts of the bundle $X^{coor}$ 
to an ``extended'' coordinate bundle $\Tilde{X}^{coor}$ on which the extension $\TVect$ acts transitively. 
Together with the choice of an {\em extended} Gelfand-Kazhdan structure (defined in Section \ref{sec ext descent2}), 
this extended bundle will give us the data of a holomorphic $(\TVect, \GL_n)$-bundle with flat connection on the frame bundle of $X$.

\subsubsection{} \label{otherway}
 
The data of a flat $\TVect$-valued connection on $\Fr_X$ is a 1-form
\ben
\Tilde{\omega} \in \Omega^{1,0} (\Fr_X ; \TVect)
\een
satisfying the Maurer-Cartan equation
\ben
\d_{dR} \Tilde{\omega} + \frac{1}{2} [\Tilde{\omega}, \Tilde{\omega}]
= 0
\een 
where $[-,-]$ is the Lie bracket for $\TVect$ extended to the de Rham complex. 
A crucial issue here is that such a structure on the frame bundle does not always exist.

We have already seen that the Gelfand-Fuks-Chern character
$\ch^\GF (\hT_n)$ maps to the ordinary Chern character of a complex
$n$-manifold under the characteristic map 
\ben
{\rm char}_\sigma : \bigoplus_k {\rm H}^k(\Vect , \GL_n ; \hOmega^k_{n,cl}) \to
\bigoplus_k {\rm H}^k(X ; \Omega^k_{X, cl}) 
\een  
associated to a Gelfand-Kazhdan structure~$(X, \sigma)$.
Assuming we have an extension  $\Tilde{X}^{coor}$,
the image of $\Tilde{\omega}$ under the quotient map 
$\Omega^1(\Fr_X ; \TVect) \to \Omega^1(\Fr_X ; \Vect)$ 
is the connection one-form $\omega_\sigma$ defined by the Gelfand-Kazhdan structure. 
Thus, the restriction of the second component of the Chern character $\ch_2^{\GF}(\hT_n)$ to
an element in $\clie^2(\TVect ; \hOmega^2_{n,cl})$ still maps to the
ordinary Chern character $\ch_2(T_X)$ using the characteristic map for
the flat connection~$\Tilde{\omega}$. 

The point here is that in $\clie^2(\TVect ; \hOmega^2_{n,cl})$, 
the element $\ch_2(\hT_n)$ is cohomologically {\em trivial}. 
That is, there
is an element $\alpha_n$ such that $\d_{\rm Lie} \alpha_n = \ch_2(\hT_n)$ where
$\d_{\rm Lie}$ is the differential on $\clie^*(\TVect ;
\hOmega^2_{n,cl})$. By naturality of descent, we see that the image of
$\alpha_n$ under the characteristic map is a trivialization for
$\ch_2(T_X)$. We conclude that lifts exists only if the second
component of the Chern character of
the manifold is trivial. Moreover, we wish to classify such lifts.

\begin{thm}\label{extbundle} 
Fix a Gelfand-Kazhdan structure $\sigma$ on X. 
Then there is a bijection between  lifts of the $(\Vect, \GL_n)$-bundle $(\Fr_X, \omega_\sigma)$ to a $(\TVect, \GL_n)$-bundle 
and trivializations of $\ch_2(T_X) \in H^2(X ; \Omega^{2}_{cl,X})$.
Moreover, if $\ch_2(T_X) = 0$, such lifts are a torsor for $H^1(X; \Omega^{2}_{cl,X})$. 
\end{thm}

Our proof is based on the Dolbeault model for the Chern character, and
throughout this section we will work with Dolbeault representatives for the Atiyah class. 
This approach is well studied and an overview can be found in \cite{atiyah} and \cite{kapranov1999}, 
but we will briefly review the requisite background.

Fix a complex K\"ahler manifold $X$ and a holomorphic vector bundle $E$. 
Also, let $\nabla$ be a smooth connection of type $(1,0)$ on $X$ for a
holomorphic vector bundle $E$. That is, an operator
\ben
\nabla : \sE \to \Omega^{1,0}(X) \tensor \sE .
\een
Let $\nabla ' = \nabla + \dbar$, then $\nabla'$ is an ordinary
connection for $E$. The curvature of $\nabla '$ splits as
\ben
F_{\nabla'} = F_{\nabla '}^{2,0} + F_{\nabla '}^{1,1} \in \Omega^{2,0}(X
; \End(E)) \oplus \Omega^{1,1}(X ; \End(E)) .
\een 
According to the Dolbeault isomorphism 
$H^{p,q}_{\dbar} (X ; E) \cong H^q(X ;  \Omega^p_X \tensor \sE)$, 
one has the following fact about the $(1,1)$-component of the curvature. 

\begin{prop} (Proposition 4 in \cite{atiyah})
The $(1,1)$-form $F_{\nabla'}^{1,1}$ is $\dbar$-closed  and is independent, in Dolbeault cohomology, of the choice of $\nabla$. 
Moreover, the cohomology class $[F_{\nabla'}^{(1,1)}]_{\dbar} \in H^{1,1}(X ; \End(E))$ is a Dolbeault representative for the Atiyah class ${\rm At}(E) \in H^1(X ; \Omega^{1,hol}_X \tensor_{\cO} \End(\sE))$.
\end{prop}

As a corollary, we see that $\Tr \left((F_{\nabla'}^{(1,1)})^k\right)$ is closed for both $\partial$ and $\dbar$. 
Moreover, this $(k,k)$-form is a Dolbeault representative for the
$k$th component of the Chern character $\ch_k(E)$. In particular, trivializations for $\ch_2(T_X)$, as in the theorem, 
are equivalent to $\dbar$-trivializations of the element
$\Tr\left( (F_{\nabla '}^{(1,1)})^2 \right) \in \Omega^{2,2}(X)$. 

\subsubsection{Warm-up: Chern-Simons forms on $\CC^n$}

\def\CS{{\rm CS}}
\def\Tr{{\rm Tr}}

Let us consider an open subset $U \subset \CC^n$ and a hermitian vector
bundle $E$ on $U$. We fix a trivialization $E = U \times E_0$ with $E_0$
equipped with a hermitian inner product. 

In this situation, there is a unique connection on $E$ 
that preserves the hermitian inner product compatible with the complex structure. 
With respect to the trivialization, it takes the form
\ben
\d_{dR} + A 
\een
where $A \in \Omega^{1,0}(U ; \End(E_0))$. 
(This connection is usually called the Chern connection.)  
The curvature of the connection is of type $(1,1)$, and it has the form
\ben
F_A = F^{(1,1)}_A = \dbar A 
\een
and lives in $\Omega^{1,1}(U ; \End(E_0))$.

Consider the $(k,k)$-form $\Theta^{(k)}_A := \Tr (F_A^k)$. 
This form is a local representative for the $k$th Chern character. 
For the following calculations, it is convenient to introduce the following complex. 
Define $\Omega^{\geq 2, *} (U)$ to be the complex
\ben
\Omega^{2,*}_{\dbar} (U) \xto{\partial} \Omega^{3,*}_{\dbar}(U) \xto{\partial} \cdots
\een
where $\Omega^{p,*}_{\dbar} (U)$ is the Dolbeault complex of
$(p,*)$-forms with differential $\dbar$. In this complex the degree of a form of type
$(k,l)$ is $k+l - 2$. Equivalently, $\Omega^{\geq 2, *}$ is the total complex of the
double complex $(\Omega^{\geq 2, *}, \partial, \dbar)$. 

There is an obvious embedding
\ben
\Omega^{2,hol}_{cl}(U) \hookrightarrow \Omega^{\geq 2, *} (U)
\een 
where $\Omega^{2,hol}_{cl}(U)$ is concentrated in degree zero. This is a quasi-isomorphism by using Poincar\'e lemma for the operators
$\partial$ and $\dbar$ for the open set
$U$ together with the obvious spectral sequence. (Note that the left hand side is concentrated in cohomological
degree zero)

A direct calculation shows that $\Theta^{(k)}_A$ is both $\partial$ and
$\dbar$-closed. 
In fact, we will use a
preferred one given by the Chern-Simons functional. 

In the case $k = 2$ we evaluate the usual Chern-Simons functional on the Chern connection $A$: consider the $3$-form
\begin{align*}
\CS (A) & = {\rm Tr}\left(A \wedge \d A + \frac{2}{3} A \wedge A \wedge A \right) \\ & = {\rm Tr}\left(A \wedge \dbar A + A \wedge \partial A + \frac{2}{3} A \wedge A \wedge A \right) \\
& = {\rm Tr}\left(A \wedge \dbar A - \frac{1}{3} A\wedge A \wedge A\right)
\end{align*}
using $\partial A + A \wedge A = 0$. 
Note that $\CS (A)$ is an element in $\Omega^{\geq 2, *} (U)$ of cohomological degree one. 
By construction $\d\CS(A) =\Theta^{(2)}_A$, 
where $\d$ is the total differential on the complex. 

We are interested in how the Chern-Simons form interacts with other trivializations of $\Theta^{(2)}_A$. 
Let us fix another trivialization $\alpha \in \Omega^{\geq 2, *} (U)$ of $\Theta^{(2)}_A$ 
such that $\d \alpha = \Theta^{(2)}_A$. 
Notice that the element $\alpha - \CS(A)$ is a closed element of degree one in the complex $\Omega^{\geq 2, *} (U)$. 
Thus, there exists an element $\beta \in \Omega^{\geq 2, *} (U)$ of cohomological degree zero, i.e., a $(2,0)$-form
such that
\ben
\d \beta = \alpha - \CS (A) .
\een
The ambiguity in choosing such a $\beta$ is precisely the cohomology
of the complex which we already determined to be
$\Omega^{2,hol}_{cl}(U)$. That is, if $\omega$ is a closed holomorphic
two-form then $\beta + \omega$ satisfies 
\ben
\d (\beta + \omega) = \d \beta = \alpha - \CS (A) .
\een
More precisely, given a trivialization $\alpha$ the space of all such $\beta$ is a torsor for
$\Omega^{2,hol}_{cl}(U)$. 

Before we proceed to the formal situation, and the construction of the extended coordinated bundle, 
we need to understand how all of the trivializations above change as we make a gauge transformation. 

Suppose that our holomorphic vector bundle $E$ is $TU$, the holomorphic tangent bundle. 
Given a biholomorphism $f : U \to U$, we obtain a gauge transformation of $A$ to the new connection
\ben
f \cdot A := g^{-1} A g + g^{-1} \partial g,
\een
where $g = Jac(f)$ is the Jacobian of $f$. 

\begin{lemma}\label{cs formula} 
There is a $(2,0)$-form $\rho$ depending on $f$ and $A$ such that
\ben
\CS(A) - \CS(f \cdot A) = \d \rho.
\een 
\end{lemma}

\begin{proof} 
For the existence of such a $\rho$, it suffices
to show that the difference $\CS(A) - \CS(f \cdot A)$ is closed. Indeed, under a gauge
transformation the Chern-Simons
functional becomes
\ben
\CS(f \cdot A) = \CS(A) + \d \Tr(g^{-1} \partial g \wedge A) + \frac{1}{3}
\Tr\left( (g^{-1} \partial g)^3\right) .
\een
Now $\Tr((g^{-1} \partial g)^3)$ is both $\partial$ and $\dbar$
closed, so the result follows.
\end{proof}

\begin{rmk} 
The $2$-form $\rho$ is only unique up to a holomorphic closed
$2$-form. We will need to fix one in the next section when we define
the extended bundle.
\end{rmk}

\subsubsection{Formal coordinates}

There is a completely formal version of the above trivializations, and
we will use it to construct the bundle $X^{coor}_\alpha$ extending the
ordinary coordinate bundle.

Let $\varphi$ be a formal holomorphic coordinate around a point $x \in X$. 
In the construction of the coordinate bundle, 
we viewed a formal holomorphic coordinate as a map $\varphi : \hD^n \to X$ where $\hD^n$ is the holomorphic formal disk. In this section we view this coordinate as a ``holomorphic" map $\varphi : \hD^n_{\CC} \to X$ 
where $\hD^n_{\CC}$ denotes the {\em complex} formal disk in the sense that its ring of functions is
\ben
\cO(\hD^n_{\CC}) = \CC \ll t_1,\ldots, t_n, \Bar{t}_1,\ldots, \Bar{t}_n \rr  .
\een
Similarly to the non-formal case, we denote the full de Rham complex by
\ben
\hOmega^{*,*}_{n} := \left(\CC \ll t_1,\ldots, t_n, \Bar{t}_1,\ldots, \Bar{t}_n \rr \tensor \CC[\d t_i, \d \Bar{t}_j], \d_{dR} \right)
\een
where $\d t_i$, $\d \Bar{t}_j$ are placed in cohomological degree one. 
In this section, to stress holomorphic dependence, we denote by $\hOmega^{k,hol}_{n,cl}$ 
the space of holomorphic closed $k$-forms on $\hD^n$, 
i.e., $\partial$-closed $k$-forms depending only on the formal variables~$\{t_i\}$.  

\begin{notation} 
In this section we will denote the full de Rham differential by
\ben
\d_{dR} : \hOmega^{*,*}_n \to \hOmega^{*,*}_n
\een
and write $\d_{dR} = \partial + \dbar$ where $\partial$, $\dbar$ are
the formal Dolbeault operators. 
\end{notation}

We define the truncated de Rham complex $\hOmega^{\geq 2, *}_n$ to be
\ben
\hOmega^{2, *} \xto{\partial} \hOmega^{3,*} \xto{\partial} \cdots  .
\een 
Its differential will be
denoted by $\d$. Note that we still have a quasi-isomorphism at the
formal level
\ben
\hOmega^{2,hol}_{n,cl} \xto{\simeq} \hOmega^{\geq 2, *}_n
\een 
by the formal Poincar\'{e} lemma. 

Fix a K\"ahler manifold $X$ and equip the holomorphic tangent bundle
with the associated Chern connection $\nabla$. Let us also fix a global trivialization
$\alpha$ of the second component of the Chern character of $T_X$. 

Pulling back to the formal disk via the coordinate $\varphi : \hD^n_{\CC} \to X$, 
we can write the connection in the form $\d_{dR} + A_\varphi$,
where $A _\varphi \in \hOmega^{1,0}_{n} \tensor \End(\CC^n)$ is the formal connection one-form. 
Just as above, the degree two element
\ben
\Hat{\Theta}^{(2)}_{A_\varphi} = \Tr((\dbar A_\varphi )^2) \in
\hOmega^{\geq 2, *}_n 
\een 
is a representing form for $\Theta_{\nabla}^{(2)}$ on the formal disk.
Note that this element is both $\partial$ and $\dbar$-closed. 
Let $\Hat{\CS}(A) \in \hOmega^{\geq 2, *}_n$ be the corresponding Chern-Simons form on the formal disk. 

\begin{rmk}
\label{rmk: closed2}
It is here that we see the explicit appearance of closed $2$-forms, or really its natural resolution. 
\end{rmk}

For each formal coordinate $\varphi$, 
the trivialization $\alpha$ of $\Theta^{(2)}_X$ determines a formal trivialization 
$\Hat{\alpha}_\varphi \in \hOmega^{\geq 2, *}_n$ satisfying  $\d \Hat{\alpha}_\varphi =~\Hat{\Theta}^{(2)}_{A_\varphi}$. 
Just as above, the difference $\Hat{\alpha}_\varphi - \Hat{\CS} (A_\varphi)$ is $\d$-closed and hence
there exists a $\beta_\varphi \in \hOmega^{2,0}_n$ such that $\d \beta = \Hat{\alpha}_\varphi -~\Hat{\CS} (A_\varphi)$. 

\begin{dfn} 
The \emph{extended coordinate bundle} $X^{coor}_\alpha$ is the set of pairs
\ben
(\varphi, \beta_\varphi)
\een
where $\varphi : \hD^n_{\CC} \to X$ is a formal coordinate
and $\beta_\varphi \in \hOmega^{2,0}_{n}$ satisfies 
\ben
\d \beta_\varphi = \Hat{\alpha}_{\varphi} - \Hat{\CS}(A_\varphi) 
\een
in the cochain complex $\hOmega^{\geq 2, *}_n$.
\end{dfn}

\subsubsection{Defining the bundle}

We have just defined the {\em set} corresponding to the extended
bundle. We now show that it is a principal bundle on $X$ for the group
$\Tilde{\Aut}_n$ lifting the coordinate bundle $X^{coor}$. 

Before we define the action of $\TAut_n$ we make the following
observations. Given a formal coordinate $\varphi$ and an automorphism
$f \in \Aut_n$ we obtain a new formal coordiante $f^* \varphi =
\varphi \circ f$. If $A_\varphi$ is the connection one-form
corresponding to $\varphi$ then $A_{f^* \varphi}$ is given by the
gauge transformation
\ben
A_{f^*\varphi} = g^{-1} A_\varphi g + g^{-1} \partial g .
\een
where $g = Jac(f) \in \GL_n(\hO_n)$ is the Jacobian. Just as in the
proof of Lemma \ref{cs formula} we have
\ben
\Hat{\CS}(A_\varphi) - \Hat{\CS}(A_{f^*\varphi}) = \d \Tr(g^{-1} \partial g \wedge A) + \frac{1}{3}
\Tr\left( (g^{-1} \partial g)^3\right)  .
\een
The $3$-form $\Hat{\chi}^{WZW}(f) := \frac{1}{3}
\Tr\left( (g^{-1} \partial g)^3\right)$ is $\partial$-closed, and
hence we may choose a non-unique cobounding two-form. Explicitly, the
choice of a formal coordinate determines a homotopy
\ben
h : \hOmega^{k,hol}_n \to \hOmega^{k-1,hol}_n
\een
and we define $\Hat{\mu}_f := h (\Hat{\chi}^{WZW}(f))$. Note that $\mu_f$
does {\em not} depend on the coordinate $\varphi$. Finally, let 
\ben
\Hat{\rho}_{f,\varphi} := \Tr(g^{-1} \partial g \wedge A_\varphi) + \mu_f,
\een 
which lies in $\hOmega^{2,0}_n$.

Recall that the group $\Tilde{\Aut}_n$ consist of pairs $(f, \omega)$
with $f \in \Aut_n$ an automorphism of the holomorphic formal disk
and with $\omega \in \hOmega^2_{n,cl}$. For a pair $(\varphi,
\beta_\varphi)$ as in the definition above, define
\be\label{extended coords1}
f \cdot (\varphi, \beta_\varphi) := (f^* \varphi, f^* \beta_\varphi + \Hat{\rho}_{f,\varphi}) .
\ee
and
\be\label{extended coords2}
\omega \cdot (\varphi, \beta_\varphi) := (\varphi , \beta_\varphi + \omega) .
\ee
Here $f^* \varphi = \varphi \circ f$ is precomposition with the
automorphism $f$, i.e., change of coordinates, and $f^* \beta_\varphi$ is
the pull-back of forms. 

\begin{prop}\label{lift1} 
Equations (\ref{extended coords1}) and (\ref{extended coords2}) define an action of $\Tilde{\Aut}_n$ on $X^{coor}_\alpha$. Moreover, it induces the structure of a $\Tilde{\Aut}_n$-principal bundle $\pi^{coor}_\alpha :X^{coor}_\alpha \to X$ lifting the $\Aut_n$-principal bundle $\pi^{coor}:~X^{coor}~\to~X$.
\end{prop}

\begin{rmk} 
Note that the choice of $\Hat{\rho}_{f, \varphi}$ is only unique up to a closed holomorphic $2$-form on the formal disk. That is, for each $\eta \in \hOmega^2_{n,cl}$ we get a different action of $\Tilde{\Aut}_n$ defined by
\ben
f \cdot (\varphi, \beta_\varphi) := (f^*\varphi, f^* \beta_\varphi + \Hat{\rho}_{f,\varphi} + \eta) .
\een
This action is equivalent to the original action. Indeed, denote $\Tilde{X}^{coor}_\alpha$ with this new action determined by $\eta$ by $\Tilde{X}^{coor}_{\alpha,\eta}$. For any two closed $2$-forms $\eta,\eta'$ we define
\[
\begin{array}{cccc}
\Phi_{\eta,\eta'} : & \Tilde{X}^{coor}_{\alpha,\eta} & \to& \Tilde{X}^{coor}_{\alpha,\eta'}\\
&(\varphi, \beta_\varphi)& \mapsto &(\varphi, \beta_\varphi + \eta - \eta')
\end{array}.
\]
Then $\Phi_{\eta,\eta'}$ is a map of $\Tilde{\Aut}_n$-spaces. In fact, it is an isomorphism with inverse given by $\Phi_{\eta',\eta}$. Hence we have an isomorphism of principal $\Tilde{\Aut}_{n}$-bundles.
\end{rmk}

The proof of the proposition is a direct calculation. First, we show that the map is well defined at the level of sets. That is, for any $f$ we must show that $f \cdot (\varphi, \beta_\varphi) \in \Tilde{X}^{coor}_\alpha$. We have 
\begin{align*}
\d(f^* \beta_\varphi + \Hat{\rho}_{f, A_\varphi}) 
& =  f^* \d \beta_\varphi + \d\Hat{\rho}_{f, A_\varphi} \\ 
& =  f^*(\varphi^* \alpha - \Hat{CS}(A_{\varphi})) + (f^* \Hat{\CS}(A_\varphi) - \Hat{\CS}(f^* A_{\varphi})) \\ 
& =  f^*\varphi^* \alpha - \Hat{\CS}(f^*A_{\varphi}) .
\end{align*}
Thus $f^* \beta_\varphi + \Hat{\rho}_{f,\varphi}$ trivializes the difference of the Chern-Simons functional associated to $f^*A_{\varphi}$ and the original trivialization as desired. 

It remains to see that we have an action by $\Tilde{\Aut}_n$. 
It suffices, in fact, to show that for any $f_1,f_2 \in \Aut_n \subset~\Tilde{\Aut}_n$,
\be\label{action1}
f_1 \cdot \left(f_2 \cdot(\varphi, \beta_\varphi)\right) = (f_2 \circ f_1) \cdot \left(f_2 \cdot(\varphi, \beta_\varphi)\right) + (\varphi, \beta_\varphi + \alpha_{GMS}(f,g)) ,
\ee
where $\Tilde{\alpha}_{GMS}$ is the defining cocycle for the extension
(\ref{grpses}) defined in Section \ref{sec GMScocycle}. 

Expanding the left-hand side, we have
\ben
\left(f_1^* f_2^*\varphi, f_1^*f_2^* \beta_\varphi + f_1^* \Hat{\rho}_{f_2,\varphi} + \Hat{\rho}_{f_1, f_2^* \varphi}\right) .
\een
The last term $\Hat{\rho}_{f_1, f_2^* \varphi}$ has the following meaning. 
Choose any (macroscopic) automorphism $\Tilde{f}_2 : \CC^n \to \CC^n$ whose $\infty$-jet class is $f_2$, 
and look at the element $\Hat{\rho}_{f_1, \Tilde{f}_2^* \varphi}$. 
Since $\Hat{\rho}_{f, \psi}$ only depends on the power series expansion of $\psi$, 
this element is well defined and does not depend on the lift~$\Tilde{f}_2$.

Now the right-hand side of (\ref{action1}) is
\ben
\left((f_2 \circ f_1)^* \varphi, (f_2 \circ f_1)^* \beta_\varphi + \rho_{f_2 \circ f_1, A_{\varphi}} + \Tilde{\alpha}_{GMS} (f_1,f_2)\right) .
\een 
Thus, to verify we have an action and finish the proof of Proposition \ref{lift1}, 
it suffices to prove the following.

\begin{lemma}
The cocycle $\Tilde{\alpha}_{GMS}$ satisfies
\be\label{rhoeqn}
\Tilde{\alpha}_{GMS}(f_1,f_2) = \Hat{\rho}_{f_1,f_2^* \varphi} + f_1^*
\Hat{\rho}_{f_2, A_\varphi} - \Hat{\rho}_{f_2 \circ f_1, A_\varphi}  .
\ee
for any $f,g$ in $\Aut_n$. 
\end{lemma}
\begin{proof}
We recall the formula for the GMS 2-cocycle from Section \ref{sec GMScocycle}. 
In the notation from that section it reads
\ben
\Tilde{\alpha}_{GMS}(f_1,f_2) = \alpha_2(f_1,f_2) + \mu_{f_1} + f_1^* \mu_{f_2} -
\mu_{f_2 \circ f_1} .
\een
(We use $\partial$ this time and not $\d_{dR}$ to stress that it is
the holomorphic differential.) We expand the right-hand side of
Equation~(\ref{rhoeqn}):
\begin{align*}
\tr((\D f_1)^{-1} \partial \D f_2 A_{f_2^* \varphi}) 
+ f_1^*\tr((\D f_2)^{-1} \partial \D f_2 A_\varphi) &
- \tr( (f_1^*\D f_2 \D f_1)^{-1} \partial(f_1^*\D f_2 \D f_1) A_\varphi) 
\\+ \mu_{f_1} 
+ f_1^*\mu_{f_2} 
- \mu_{f_2 \circ f_1} .
\end{align*} 
We have used the fact that the Jacobian of $f_2 \circ f_1$ is given by
the product $f_1^*\D f_2 \D f_1$. Finally, to complete the proof we notice
that the first three terms in the above formula simplify to 
\[
\alpha_2(f_1,f_2) = \tr \left(g_1^{-1} \partial \D f_1 \wedge
  f_1^*(\partial \D f_2 (\D f_2)^{-1}) \right)
\]
and so we are done.
\end{proof}

\subsubsection{Proof of Theorem \ref{extbundle}} \label{sec ext descent2}
 
In this section we prove the theorem. 
We will use the data of an extended coordinate bundle to construct a Gelfand-Kazhdan structure for the frame bundle
$\Fr_X \to X$, with a connection one-form valued in the
extension~$\TVect$. 

Clearly, the action of $\TAut_n$ on the set of pairs $(\varphi,
\beta_\varphi)$ lifts the action of $\Aut_n$ on formal
coordinates $\varphi : \hD^n \to X$. This observation, together with the compatibility of the cocycle $\Tilde{\alpha}_{GMS}$ and the Gelfand-Fuks-Atiyah
cocycle $\ch^{\GF}_2(\hT_n)$ defining the extension $\TVect \to
\Vect$, allows us to
summarize the construction of previous section as follows.

\begin{prop} 
For each trivialization $\alpha$ of $\Theta^{(2)}_X$
there exists a transitive action of $\TVect$ on $\Tilde{X}^{coor}_\alpha$ that lifts the action of $\Vect$ on $X^{coor}$. 
That is, there is a map of Lie algebras
\ben
\Tilde{\theta}_\alpha : \TVect \to \cX(\Tilde{X}^{coor}_\alpha)
\een
such that for each $(x, \varphi, \beta_\varphi) \in \Tilde{X}^{coor}_\alpha$, 
the induced map $\Tilde{\theta}(x) : \TVect \to T_{(x, \varphi, \beta_\varphi)} \Tilde{X}^{coor}_\alpha$ is an isomorphism and the diagram
\ben
\xymatrix{
\TVect \ar[d] \ar[r]^-{\Tilde{\theta}_\alpha (x)} & T_{ (x, \varphi, \beta_\varphi)} \Tilde{X}^{coor}_\alpha \ar[d] \\ 
\Vect \ar[r]^-{\theta(x)} & T_{(x, \varphi)} X^{coor} .
}
\een
commutes.
\end{prop}

The inverse of $\Tilde{\theta}_\alpha$ defines a connection one-form
$\Tilde{\omega}_\alpha \in \Omega^1(\Tilde{X}^{coor}_\alpha ; \TVect) .$
Now, $\Tilde{X}_\alpha^{coor}$ is an $\hOmega^2_{cl,n}$-torsor over
$X^{coor}$ and so there exists a $\Aut_n$-equivariant smooth section
$\sigma_{\Omega^2} : X^{coor} \to \Tilde{X}^{coor}_\alpha.$
Note that this section is {\em not} unique, but its choice will not matter in the end 
(much as in the case of an ordinary Gelfand-Kazhdan structure). 
Given such a section we have an induced map
\ben
\Omega^1(\Tilde{X}^{coor}_\alpha ; \TVect) \xto{\sigma_{\Omega^2}^*}
\Omega^1(X^{coor} ; \TVect) \xto{p} \Omega^1(X^{coor} ; \Vect) ,
\een
where $p : \TVect \to \Vect$ is the projection. Under this composition, the 1-form $\Tilde{\omega}_\alpha^{coor}$ maps
to the Grothendieck connection 1-form $\omega^{coor} \in
\Omega^1(X^{coor} ; \Vect)$. 

Now, we would like to apply the theory of Gelfand-Kazhdan descent to this
situation. Recall that in the case of the pair $(\Vect, \GL_n)$, a
Gelfand-Kazhdan structure amounted to choosing a formal
exponential. That is, a $\GL_n$-equivariant splitting $\sigma : \Fr_X \to X^{coor}$
of the projection $\pi^{coor} : X^{coor} \to~\Fr_X$. 

Fixing a section $\sigma_{\Omega^2}$ of the $\Omega^2_{n,cl}$-torsor
over $X^{coor}$ as above,  we can compose with the canonical section
$\sigma_{\Omega^2} : X^{coor} \to \Tilde{X}_\alpha^{coor}$ of
$\Tilde{X}^{coor}_\alpha$ over $X^{coor}$ to get the section
$\sigma_{\Omega^2} \circ \sigma$. This composite defines the connection one-form
\ben
\Tilde{\omega}_{\sigma, \sigma_{\Omega^2}}^\alpha = (\sigma_{\Omega^2} \circ \sigma)^*
\omega^{coor} = \sigma^* \sigma_{\Omega^2}^* \omega^{coor}
\een
living in~$\Omega^1(\Fr_X; \TVect)$.

\begin{dfn} 
An {\em extended Gelfand-Kazhdan structure} on $X$ is a triple $(\alpha, \sigma, \sigma_{\Omega^2})$ where
\begin{itemize}
\item[(i)] $\alpha$ is a trivialization for the second component of the
Chern character of $X$;
\item[(ii)] $\sigma$ is a Gelfand-Kazhdan structure on $X$; and
\item[(iii)] $\sigma_{\Omega^2}$ is an $\Aut_n$-equivariant smooth splitting of
  $\Tilde{X}^{coor}_\alpha \to X^{coor}$ .
\end{itemize}
\end{dfn}

The construction in the above paragraph shows that the data of an
extended Gelfand-Kazhdan structure on $X$ determines a holomorphic
$(\TVect, \GL_n)$-bundle on $\Fr_X \to X$ with flat connection
one-form given by~$\Tilde{\omega}_{\sigma, \sigma_{\Omega^2}}^\alpha$.

The same argument as in the non-extended case (see Section \ref{gauge equiv}) gives the following.

\begin{lemma} 
Fix a Gelfand-Kazhdan structure $\sigma$. 
Let $\sigma^1_{\Omega^2}$ and $\sigma^1_{\Omega^2}$ be two smooth splittings of $\Tilde{X}^{coor}_\alpha \to X^{coor}$. 
Then the induced connection one-forms $\omega^{\alpha}_{\sigma, \sigma_{\Omega^2}^1}$ and $\omega^\alpha_{\sigma,\sigma_{\Omega^2}^2}$ are gauge equivalent.
\end{lemma}

To finish the proof of Theorem \ref{extbundle}, we must go the other way: 
given a lift $(\Fr_X, \Tilde{\omega})$ of the $(\Vect, \GL_n)$-bundle $(\Fr_X, \omega_\sigma)$, 
we must produce a trivialization. 
This construction is outlined above in Section \ref{otherway}. 
It is a direct calculation to show that these two constructions are
inverse to each other. 

Before we define extended descent, we discuss the class of modules that we wish to consider.

\subsection{Extended modules}

We have defined the category of ``vector bundles" on the formal disk $\VB_n$. 
These Harish-Chandra modules were especially well behaved from the point of view of Gelfand-Kazhdan descent. 
In this section we consider an analogue of this category of modules for the pair $(\TVect, \GL_n)$. 
These modules will be objects that descend along the extended bundle~$(\Fr_X, \Tilde{\omega}^\alpha_{\sigma})$. 

Since $\TVect$ is an extension of Lie algebras, 
it has a two-step filtration 
\ben
F^1 \TVect = \TVect \supset F^0 \TVect = \hOmega^2_{cl} .
\een 
The associated graded of this filtration is the Lie algebra $\Vect
\oplus \hOmega^2_{cl}$. 

Let $\Mod_{(\TVect, \GL_n)}^{\rm fil}$ denote the category of filtered modules for the pair $(\TVect,\GL_n)$, 
using the filtration above. 
Given any such module $\cV$, we can consider its associated graded ${\rm Gr} \; \cV$. 
This associated graded forgets down to a graded module for the Lie algebra $\Vect$. 
Since $\cV$ also has a compatible $\GL_n$-action,
the associated graded has the structure of a graded $(\Vect, \GL_n)$-module. 
There is thus a functor 
\ben
\Gr: \Mod_{(\TVect, \GL_n)}^{\rm fil} \to \Mod_{(\Vect, \GL_n)}^{\ZZ/2}
\een
given by taking the associated graded for the two-step filtration. 
Here, $\Mod_{(\Vect, \GL_n)}^{\ZZ/2}$ is the category of $\ZZ/2$-graded vector spaces 
together with a grading-preserving action of the pair~$(\Vect, \GL_n)$. 

Similarly, there is a full sub-category $\VB_n^{\ZZ/2} \subset \Mod_{(\Vect, \GL_n)}^{\ZZ/2}$ 
consisting of those $(\Vect,\GL_n)$-modules that are also elements in $\VB_n$ by forgetting the grading. 

\begin{dfn} 
The category $\Tilde{\VB}_n$ of {\em filtered $(\TVect, \GL_n)$-vector bundles} 
is the pull-back
\ben
\xymatrix{
\Tilde{\VB}_n\ar[r] \ar[d] & \VB^{\ZZ/2}_n \ar[d] \\
\Mod_{(\TVect, \GL_n)}^{\rm fil} \ar[r] & \Mod^{\ZZ/2}_{(\Vect, \GL_n)} 
}
\een 
of categories.
\end{dfn}

Explicitly, an object of $\Tilde{\VB}_n$ is a $\ZZ/2$-graded $\hO_n$-module 
that is free and finite rank together with a compatible action of $(\TVect, \GL_n)$ that respects the two-step filtration of~$\TVect$.

\subsection{Extended descent}

We are now in a place to define the extended Gelfand-Kazhdan descent
functor for modules as in the previous section. 

Define the category $\Tilde{\Hol}_n$ to have objects consisting of pairs $(X, \alpha)$,
where $X$ is a complex manifold of dimension $n$ and $\alpha$ is a
trivialization of its second component of the Chern character $\ch_2(T_X)$. Morphisms
are defined to be local biholomorphisms that pull-back
trivializations. For instance, if $(X,\alpha_X)$ and $(Y,\alpha_Y)$
are objects and $f : X \to Y$ is a local biholmorphism, we require $f^*
\alpha_Y = \alpha_X$. We let $\Tilde{\GK}_n$ denote the category fibered over $\Tilde{\Hol}_n$ whose objects over $(X,\alpha)$ are extended Gelfand-Kazhdan structures $(X,\alpha, \sigma, \sigma_{\Omega^2})$.

\begin{dfn} The {\em extended Gelfand-Kazhdan descent} is the functor
\ben
\Tilde{\desc}_{\GK} : \Tilde{\GK}_n^{op} \times \Tilde{\VB}_n
\to {\rm Pro}(\VB(X)_{flat})
\een
sending an extended Gelfand-Kazhdan structure $(X, \alpha, \sigma, \sigma_{\Omega^2})$ and an extended module $\cV \in \Tilde{\VB}_n$ to the pro-vector bundle $\Fr_X \times^{\GL_n} \cV$ with flat connection induced from $\omega_{\alpha, \sigma, \sigma_\Omega^2}$.
\end{dfn} 

Let $\Tilde{\bdesc}_{\GK}(X, \sigma, \sigma_{\Omega^2},\alpha; \cV)$ denote
the corresponding $\Omega^*(X)$-module. Since different
choices of sections $\sigma$ and $\sigma_{\Omega^2}$ determine gauge
equivalent connections the resulting sheaf of flat sections is independent of such choices and
we will denote the sheaf by $\sdesc_{\GK}(X, \alpha; \cV)$. 

\section{Descent for vertex algebras} \label{sec vertex desc}

In this section we discuss Gelfand-Kazhdan descent for vertex
algebras. Namely, we show how the structure of a vertex algebra that
has a compatible action of a pair $(\fg, K)$ descends to a sheaf of
vertex algebras on complex manifolds via the functors we have already
constructed. Of course, the most important cases will be the pairs
$(\Vect, \GL_n)$ and its extension $(\TVect, \GL_n)$. 

For another approach for constructing sheaves of vertex algebras on manifolds, 
see \cite{Malikov2008},  although the case of extended descent is not covered there. 

For a \v{C}ech style approach to constructing the sheaf of vertex
algebras given by CDOs, see~\cite{GMS2}. 

\subsection{General descent} 
\label{gendescent}

We will define descent for vertex algebras in a
similar way as in the general setting of Harish-Chandra descent. For this
to make sense, we need to first say what we mean by a vertex algebra
in the differentially graded setting. 

\begin{dfn} 
A {\em dg vertex algebra} is a $\ZZ$-graded vertex algebra $V$ 
together with a vertex algebra derivation $\d : V \to V$ of degree 1 such that
\begin{itemize}
\item[(i)] $\d^2 = 0$ and
\item[(ii)] the structure maps $Y(-; z) : V \to \End(V) \ll z^\pm \rr$ have cohomological degree zero.
\end{itemize}
Moreover, if $V$ has the additional structure of a $\ZZ_{\geq 0}$-graded vertex algebra 
(by what we call the dimension grading), 
a {\em dg $\ZZ_{\geq 0}$-graded vertex algebra} is a dg vertex algebra such that
$\d$ preserves the dimension grading. 
\end{dfn}

Consider now a torsor $P \to X$ for a Harish-Chandra pair $(\fg, K)$. 
If $V$ is a vertex algebra on which $K$ acts via vertex algebra automorphisms, 
then invariants for this group action will be a sub-vertex algebra. 
Likewise, if $\fg$ acts on $V$ via vertex algebra derivations,
then the induced connection $\nabla^{P,V} = \d_{dR} + \rho_\fg(\omega)$ also acts by vertex algebra derivations on $\Omega^*(P) \tensor V$. 
If we choose actions that are compatible (as well as $\ZZ_{\geq 0}$-graded) we
obtain the following. 

\begin{prop} 
If $(\fg,K)$ acts on the vertex algebra then $\bdesc((P, \omega, V))$
has the structure of a dg vertex algebra in $\Omega^*(X)$-modules. 
If $(\fg, K)$ acts on the $\ZZ_{\geq 0}$-graded vertex algebra $V$,
then $\bdesc((P, \omega), V)$ has the structure of a dg $\ZZ_{\geq 0}$-graded vertex algebra in dg $\Omega^*(X)$-modules.
\end{prop}

This result implies that the sheaf of flat sections $\sdesc((P, \omega), V)$
has the structure of a sheaf of $\ZZ_{\geq 0}$-graded vertex
algebras. 

\subsection{Formal vertex algebras}

We now develop what we mean by vertex algebras in the category of
formal vector bundles. The vertex algebras we are interested in are not finite dimensional,
so are ill-behaved in the context of doing ordinary Harish-Chandra
descent. The graded pieces, however, are finite dimensional over $\hO_n$, so we are in a similar context of Gelfand-Kazhdan descent as in Section 2.

Recall, that the category of formal vector bundles (or formal vector bundles on the formal $n$-disk) $\VB_n$ consists of
$\hO_n$-modules together with a compatible structure of a $(\Vect,
\GL_n)$-module. The category of formal vertex algebras we consider is a modest generalization of the category of formal vector bundles
$\VB_n$. 

\begin{dfn}
A {\em Gelfand-Kazhdan vertex algebra} is a $\ZZ_{\geq 0}$-graded vertex
algebra $\cV$ together with an
action of $(\Vect, \GL_n)$ as in Definition \ref{graction} such that
for each $N \geq 0$ one has a $\GL_n$-equivariant identification
\be\label{GK vert}
\cV^{(N)} = \hO_n \tensor_\CC V^{(N)}
\ee
where $V^{(N)}$ is a finite dimensional $\GL_n$-representation. A {\em morphism} of Gelfand-Kazhdan vertex algebras is a $(\Vect, \GL_n)$-equivariant morphism of $\ZZ_{\geq 0}$-graded vertex algebras.
We denote this category by~${\rm Vert}_n$. 
\end{dfn}

Thus, a Gelfand-Kazhdan vertex algebra is a vertex algebra in the category of
Harish-Chandra modules $\Mod_{(\Vect, \GL_n)}$ together with some
finiteness property. 

\begin{lemma}\label{lemma GK vert gr} The vertex algebra $\Gr \; \hCDO_n$ has the structure of a formal vertex
  algebra. 
\end{lemma}
\begin{proof}
We have already seen that $\Gr \; \hCDO_n$ has an action of the pair $(\Vect, \GL_n)$. Moreover, from the explicit formula $\Gr \; \hCDO_n = \Hat{\tensor}_{0 < k} \Hat{\Sym}_{\hO_n} (\hOmega^1_n) \Hat{\tensor} \Hat{\tensor}_{0 \leq l} \Hat{\Sym}_{\hO_n} (\hT_n)$ shows that the spaces of fixed conformal dimension are finite sum of tensor products of the $(\Vect,\GL_n)$ modules $\hOmega^1_n$ and $\hT_n$. Thus, we can write each space of conformal dimension $N$ in the presentation of Equation (\ref{GK vert}). 
\end{proof}

\begin{dfn} 
An {\em extended Gelfand-Kazhdan vertex algebra} is a $\ZZ_{\geq 0}$-graded vertex algebra together with an action of $(\TVect, \GL_n)$ as in Definition \ref{graction} such that for each $N \geq 0$ one has a $\GL_n$-equivariant identification
\be\label{GK vert 2}
\cV^{(N)} = \hO_n \tensor_\CC V^{(N)}
\ee 
where $V^{(N)}$ is a finite dimensional $\GL_n$-representation. A {\em morphism} of a Gelfand-Kazhdan vertex algebra is a $(\TVect, \GL_n)$-equivariant morphisms of $\ZZ_{\geq 0}$-graded vertex algebras. We denote this category by $\Tilde{\rm Vert}_{n}$. 
\end{dfn}

The category ${\rm Vert}_n$ is a full subcategory of vertex algebras in $\Mod^{\ZZ/2}_{(\TVect, \GL_n)}$ consisting of those objects that satisfy the finiteness constraint above. 

\begin{lemma} 
The vertex algebra $\hCDO_n$ is an extended Gelfand-Kazhdan vertex algebra.
\end{lemma}

\begin{proof} We have already seen that $\hCDO_n$ has the structure of a $(\TVect, \GL_n)$-module. The 
same argument as in Lemma \ref{lemma GK vert gr} shows that the spaces of fixed conformal dimension can be expressed as in Equation (\ref{GK vert 2}). 
\end{proof}

\begin{rmk} 
As noted following Remark \ref{rmk:notations}, we do {\em not} require that a Gelfand-Kazhdan vertex algebra $\cV$
be $\hO_n$-linear.
The underlying vector space of $\cV$ is an $\hO_n$-module but  vertex algebra operations do not preserve that action.
\end{rmk}

\subsection{Descending Gelfand-Kazhdan vertex algebras} 

We show how Section \ref{gendescent} carries over to
Gelfand-Kazhdan descent for the categories of equivariant vertex
algebras defined in the previous section. We will perform both an
extended and non extended version of descent.

For a Gelfand-Kazhdan vertex algebra $\cV$ we define the sheaf $\sdesc_{\GK} (\cV)$ of
vertex algebras on the category $\GK_n$, and hence on the category
${\rm Hol}_n$. For now, let's fix a Gelfand-Kazhdan structure $\sigma$ on $X$. It will be
evident that all constructions are still functorial in this parameter.  

For each $N \geq 0$ we have a decomposition $\cV^{(N)} = \hO_n
\tensor_\CC V^{(N)}$ and hence a filtration on $\cV^{(N)}$ coming
from the vanishing order of jets. Thus, applying the same
construction as in Section~\ref{sec gk descent}, 
we obtain the pro-vector
bundle $\Fr_X \times^{\GL_n} \cV^{(N)}$. 
Since the action of $(\Vect,
\GL_n)$ preserves the $\ZZ_{\geq 0}$ grading, 
it is a pro-vector bundle equipped with a flat connection and hence we can define the
$\Omega^*(X)$-module 
\ben 
\bdesc(\sigma, \cV^{(N)}) = \left( \left(\Omega^*(\Fr_X) \tensor
    \cV^{(N)}\right)_{bas}, \d_{dR} + \omega_\sigma\right)  .
\een 
We now sum over all dimension spaces to obtain the $\Omega^*(X)$-module
\ben
\bdesc(\sigma, \cV) := \bigoplus_{N \geq 0} \bdesc(\sigma, \cV^{(N)})
\een

\begin{lemma} 
For any Gelfand-Kazhdan vertex algebra $\cV$ and Gelfand-Kazhdan structure
$(X,\sigma)$, the $\Omega^*(X)$-module $\bdesc(\sigma, \cV)$ has the
structure of a $\ZZ_{\geq 0}$-graded dg vertex algebra over
$\Omega^*(X)$. 
\end{lemma}

Thus, we obtain a sheaf of $\ZZ_{\geq 0}$-graded vertex algebras
$\sdesc(\sigma, \cV)$ by
taking flat sections. It is clear the construction is natural in the choice of a GK
structure so that we obtain a sheaf $\sdesc(\cV)$ of vertex algebras
on the category $\Hol_n$. 

\subsubsection{Classical limit of the sheaf of CDOs} 

In the example of the Gelfand-Kazhdan vertex algebra $\Gr \; \hCDO_n$ we denote the descent object by
\ben
\Gr \; \CDO_X := \sdesc((X,\sigma) ; \Gr \; \hCDO_n) .
\een
This is a sheaf of vertex algebras defined on any complex
manifold. As we remarked above, functoriality of the construction
implies that we have a sheaf of vertex algebra $\Gr \; \CDO$ defined
on the {\em category} ${\rm Hol}_n$.

Moreover, as the action of the pair $(\Vect, \GL_n)$
preserves the Poisson structure. This shows that $\Gr \; \CDO$ is actually a sheaf of Poisson
vertex algebras. 

\subsubsection{Extended descent for vertex algebras}

The construction for extended formal vertex algebras is similar, this
time we use the bundle of extended coordinates constructed in Section
\ref{sec ext desc}. 

Let us fix an extended Gelfand-Kazhdan structure $(X, \alpha,
\sigma,\sigma_{\Omega^2})$, that we simply denote by $\Tilde{\sigma}$,
and an extended Gelfand-Kazhdan vertex algebra
$\cV$. 

By construction, each dimension space $\cV^{(N)} = \hO_n \tensor_\CC
V^{(N)}$ has an action of $(\TVect, \GL_n)$ and hence we can form the
pro-vector bundle $\Fr_X \times^{\GL_n} \cV^{(N)}$ that is equipped
with a flat connection. The de Rham complex is the
$\Omega^*(X)$-module
\ben
\Tilde{\bdesc}(\Tilde{\sigma}, \cV^{(N)}) =
\left(\left(\Omega^*(\Fr_X) \tensor \cV^{(N)}\right)_{bas}, \d_{dR} +
  \Tilde{\omega}^\alpha_{\sigma, \sigma_{\Omega^2}}\right) .
\een 
Again, by summing over spaces of fixed conformal dimension we obtain the
$\Omega^*(X)$-module
\ben
\Tilde{\bdesc}(\Tilde{\sigma}, \cV) = \bigoplus_{N \geq 0}
\Tilde{\bdesc}(\Tilde{\sigma}, \cV^{(N)}) .
\een

The same proof as above carries over with minor modifications to show.

\begin{lemma} For any extended Gelfand-Kazhdan structure
  $\Tilde{\sigma}$ and extended Gelfand-Kazhdan vertex algebra $\cV$ the
  $\Omega^*(X)$-module $\bdesc(\Tilde{\sigma}, \cV)$ is a $\ZZ_{\geq
    0}$-graded dg vertex algebra over $\Omega^*(X)$. 
\end{lemma}

We obtain a sheaf of $\ZZ_{\geq 0}$-graded vertex algebras by taking
flat sections that we denote $\Tilde{\sdesc}(\Tilde{\sigma}, \cV)$. Again,
the construction is natural in the extended Gelfand-Kazhdan structure
so we obtain a sheaf of $\ZZ_{\geq 0}$-graded vertex algebras
$\Tilde{\sdesc}(\cV)$ on the category $\Tilde{\Hol}_n$. 

\subsubsection{The sheaf of CDOs}

We are finally able to to define the central object of study in this
work. 

\begin{dfn} Let $X$ be a complex manifold together with a
  trivialization $\alpha$ of $\ch_2(T_X)$. The {\em sheaf of chiral
    differential operators} on $X$ is the sheaf of vertex algebras
\ben
\CDO_{X,\alpha} := \Tilde{\sdesc}_{\GK} (X, \alpha ; \hCDO_n) .
\een
\end{dfn} 

\begin{rmk} 
The descent functor $\Tilde{\desc}_{\GK}$ depends on the
  choice of an extended Gelfand-Kazhdan structure and not just a
  trivialization of $\ch_2$. But, as we have already mentioned, the
  sheaf of flat sections does not depend on such a choice so we omit
  it from the notation. 
\end{rmk}

This definition of chiral differential operators via Gelfand-Kazhdan
formal geometry is similar in spirit to the formulation of the chiral
de Rham complex in \cite{MSV}. There, one defines the sheaf in a
similar way as above though with the non-extended pair $(\Vect,
\GL_n)$ (this pair indeed acts on the affine chiral de Rham vertex
algebra). We hope that the above constructions reflect systematically
how one can handle descent for objects that require extending the
usual action of formal automorphisms and derivations on the formal
disk. 

The compatibility of the Gelfand-Fuks-Chern class $\ch_2(\hT_n)$ and
the group cocycle $\alpha_{GMS}$ 
shows how our definition of chiral differential operators is related
to the original definition given in~\cite{GMS}.

\subsubsection{The conformal structure}

We address the conformal structure for the sheaf of chiral differential operators. We have already notes that the vertex algebra $\hCDO_n$ has the structure of a conformal vertex algebra of charge $c = 2n$. This conformal structure, however, is not compatible with the action of $\TVect$ on CDO's on the formal disk. Indeed, Proposition \ref{prop c1 conformal} implies that the obstruction for these structures to be compatible is the first Gelfand-Fuks-Chern class~$c_1^{\GF}(\hT_n)$. 

The conformal vector $L_{-2} \in \hCDO_n^{(2)}$ is preserved, however, by the action of $\GL_n$. Thus, the map of vertex algebras $\Phi : {\rm Vir}_{c=2n} \to \hCDO_n$, encoding the conformal structure, determines a map of graded $\Omega^\#_X$-modules 
\ben
\Phi : \Omega^\#(X) \tensor {\rm Vir}_{c=2n} \to \left(\Omega^\#(\Fr_X) \tensor \hCDO_n\right)_{bas} .
\een 
Now, the action of $\TVect$ on ${\rm Vir}_{c=2n}$ is trivial. So, when we equip the left-hand side with the differential $\d_{dR} + \Tilde{\omega}_{\sigma, \sigma_{\Omega^2}}$ coming from a fixed extended Gelfand-Kazhdan structure we obtain the constant $\Omega^*_X$-module $\Omega^*_X \tensor {\rm Vir}_{c=n}$. Thus, the sheaf obtained via descent $\sdesc(X ; {\rm Vir}_{c=2n}) = \ul{\rm Vir}_{c=2n}$ is just the constant sheaf. 

The right hand side also has a natural differential $\d_{dR} + \Tilde{\omega}_{\sigma, \sigma_{\Omega^2}}$ coming from a fixed extended Gelfand-Kazhdan structure making it a $\Omega^*_X$-module. The calculation of Proposition \ref{prop c1 conformal} implies that the failure for $\Phi$ to be a map of $\Omega^*_X$-modules is the image of the cocycle $c^\GF_1(\hT_n)$ under the characteristic map of the Gelfand-Kazhdan structure. This is precisely the usual first Chern class $c_1(T_X)$. We have arrived at the following. 

\begin{prop}\label{prop conformal cdo} 
Let $\alpha$ be a trivialization of $\ch_2(T_X)$ and suppose that $c_1(T_X) = 0$ in $H^1(X ; \Omega^1_X)$. Then there exists a map of sheaves of vertex algebras on $X$
\ben
\Phi : \ul{\rm Vir}_{c=2n} \to \CDO_{X,\alpha} . 
\een
In other words, in the case that $c_1(T_X) = 0$ the sheaf of chiral differential operators has a global Virasoro vector. 
\end{prop}

\subsubsection{The Witten genus}

It is well-known \cite{BorLib, Cheung} that the character of the sheaf of chiral differential operators equals, up to a factor, the Witten genus of the complex manifold. In this section, we remark on how to recover this fact using the construction of chiral differential operators via Gelfand-Kazhdan descent.

Recall, in Section \ref{sec vert character} we have defined the formal graded character of a vertex algebra. For $\cV = \oplus_{N \geq 0} \cV^{(N)}$ a Gelfand-Kazhdan vertex algebra, it is the element 
\ben
\chi^{(\Vect, \GL_n)} (\cV) = \sum_{N \geq 0} q^N \left(\Td^\GF \cdot \ch^{\GF}(\cV^{(N)})\right) \in \clie(\Vect ; \hOmega_n^{-*})[[q]] .
\een 

Given any sheaf of $\ZZ_{\geq 0}$-graded vertex algebras $\cV_X$ on a manifold $X$, one defines the character as follows. Note that the sheaf cohomology $H^*(X ; \cV_X)$ has the structure of a graded vertex algebra (that is, a differential graded vertex algebra with zero differential). In particular, it is a $\ZZ/2$-graded vertex algebra, with even part equal to $H^{ev}(X ; \cV_X)$ and odd part equal to $H^{odd}(X; \cV_X)$. The character of $\cV_X$ is the super character of $H^*(X ; \cV_X)$. That is,
\ben
\chi (\cV_X) := \sum_{N \geq 0} q^N \left(\dim(H^{ev}(X ; \cV_X^{(N)})) - \dim(H^{odd}(X ; \cV_X^{(N)})) \right) .
\een

\begin{lemma}\label{lemma local to global char} Fix a Gelfand-Kazhdan structure $(X,\sigma)$ and let $\cV$ be a Gelfand-Kazhdan vertex algebra one has
\ben
\chi\left(\sdesc(X ; \cV)\right) = \int_X {\rm char}_\sigma \left( \chi^{(\Vect, \GL_n)} (\cV)\right)  \in \CC [[q]]
\een
where ${\rm char}_\sigma : H^*_{\rm Lie} (\Vect, \GL_n ; \hOmega^{-*}_n) \to H^*(\Omega^{-*}_X)$ is the characteristic map associated to the Gelfand-Kazhdan structure extended $q$-linearly. 
\end{lemma}

\begin{proof} As a consequence of Grothendieck-Riemann-Roch for sheaves on $X$, we have
\ben
\chi\left(\sdesc(X ; \cV)\right) =  \int_X \sum_{N \geq 0} q^N \Td_X \cdot \ch(\desc(X ; \cV^{(N)}))  .
\een 
The integrand on the right-hand side is precisely the image of the class $\chi^{(\Vect, \GL_n)}(\cV)$ under the characteristic map associated to the Gelfand-Kazhdan structure. 
\end{proof}

Similarly, if $\cV$ is an extended Gelfand-Kazhdan vertex algebra and $\Tilde{\sigma} = (X,\alpha, \sigma, \sigma_{\Omega^2})$ is an extended Gelfand-Kazhdan structure then one has
\ben
\chi\left(\Tilde{\sdesc}(X ; \cV)\right) = \int_X \Tilde{\rm char}_{\Tilde{\sigma}} \left( \chi^{(\TVect, \GL_n)} (\cV)\right)  \in \CC [[q]]
\een
where $\Tilde{\rm char}(-)$ is the extended characteristic map $H_{\rm Lie}^*(\TVect, \GL_n ; \hOmega^{-*}_n) \to H^*(\Omega^{-*}_X)$ extended $q$-linearly.  

As a corollary, we recover the appearance of the Witten genus as the character of chiral differential operators on $X$. Recall, the {\em Witten class} of a manifold $X$ with ${\rm ch}_2(T_X) = 0$, is defined (see \cite{WittenGenus1, WittenGenus2}) as the following $q$-series valued in differential forms
\ben
{\rm Wit}(X,q) = \Hat{\rm A}_X \cdot \ch\left(\bigotimes_{l \geq 1} \Sym_{q^l} (\Omega^1_X \oplus T_X) \right) \left(\prod_{k \geq 1} (1-q^k)\right)^{2n} \in \Omega^{-*}_X [[q]]
\een
where $\Hat{\rm A}_X$ is the A-hat class of the tangent bundle of $X$. The {\em Witten genus} is obtained as the integral $\int_X {\rm Wit}(X,q)$ and is the $q$-expansion of a modular form. As an immediate consequence of our calculation in Proposition \ref{prop local character}, we obtain the well-known relation of the character and the Witten genus. 

\begin{prop}\label{prop char cdo} Let $\alpha$ be a trivialization of $\ch_2(T_X)$. The graded character of $\CDO_{X,\alpha}$ satisfies
\ben
\chi(\CDO_{X,\alpha}) = \left(\prod_{k \geq 1} (1-q^k) \right)^{-2n} \int_X e^{c_1(T_X)/2} {\rm Wit}(X, q) .
\een
\end{prop}
\begin{proof}
We have identified, in Corollary \ref{cor formal char cdo}, the $(\TVect, \GL_n)$-equivariant graded character of $\hCDO_n$ with the image of the class
\ben
\Td^{\GF} \cdot \ch^\GF \left( \bigotimes_{l \geq 1} \Sym_{q^l} (\hOmega^1_n \oplus \hT_n) \right) \in\clie^*(\Vect , \GL_n ; \hOmega^{-*}_n)[[q]]
\een
under the map $\clie^*(\Vect , \GL_n ; \hOmega^{-*}_n)[[q]] \to \clie^*(\TVect , \GL_n ; \hOmega^{-*}_n)[[q]]$. The image of this class under the characteristic map of the extended Gelfand-Kazhdan structure is $\Td_X \cdot \ch \left( \bigotimes_{l \geq 1} \Sym_{q^l} (\Omega^1_X \oplus T_X) \right)$. Thus, by Lemma \ref{lemma local to global char} we see that the graded character of $\CDO_{X,\alpha}$ is 
\ben
\chi(\CDO_{X,\alpha}) = \int_X \Td_X \cdot \ch \left( \bigotimes_{l \geq 1} \Sym_{q^l} (\Omega^1_X \oplus T_X) \right) .
\een
Finally, note that $\Td_X = e^{c_1(T_X)/2} \Hat{\rm A}_X$. 
\end{proof} 

\begin{rmk} 
We have already pointed out that in the case that $c_1(T_X) = 0$, 
the sheaf $\CDO_{X,\alpha}$ is a sheaf of conformal vertex algebras. 
Recall that the honest character of a vertex algebra is related to the graded character via ${\rm char}(V) = q^{-c/24} \chi(V)$, 
where $c$ is the central charge. 
Thus, in this case we have the following expression for the character of chiral differential operators:
\ben
{\rm char} (\CDO_{X,\alpha}) = \eta(q)^{-2n} \int_X {\rm Wit}(X, q) 
\een
where $\eta(q) = q^{1/24} \prod_{k \geq 1} (1-q^k)$ is the Dedekind $\eta$-function.
\end{rmk}

%% file: part2.tex
\newpage

\part*{Part II: The curved $\beta\gamma$ system and its factorization algebra}

\section{Overview}

The curved $\beta\gamma$ system is an elegant nonlinear $\sigma$-model,
attractive for both mathematical and physical reasons.
The source is a Riemann surface $S$ and the target is any complex manifold $X$.
The fields are $\gamma: S \to X$ a smooth function and 
$\beta \in \Omega^{1,0}(S,\gamma^*T^*_X)$ a $(1,0)$-form on $S$ 
with values in the pullback along $\gamma$ of the holomorphic cotangent bundle of $X$.
The action functional is 
\[\int_S \langle \beta \wedge \dbar \gamma \rangle, \]
where the brackets indicate that one uses the fiberwise evaluation pairing between the tangent bundle $T_X$ and the cotangent bundle $T^*_X$.
The equations of motion for this action are then
\[ 
\dbar \gamma = 0 = \dbar \beta. 
\]
In other words, a solution is a holomorphic map $\gamma$ from $S$ to $X$ along with a holomorphic 1-form $\beta$ on $S$ with values in the pullback along $\gamma$ of the cotangent bundle. 
(When $S$ admits a nowhere-vanishing holomorphic 1-form, such as when $S$ is an elliptic curve,
one can view the classical $\beta\gamma$ system as picking out holomorphic maps from $S$ to~$T^*X$.)

The quantization of the classical $\beta\gamma$ system is a chiral conformal field theory whose chiral algebra is the CDOs of $X$,
as explained by Witten \cite{WittenCDO} and Nekrasov \cite{Nek}.
To be more precise, they explain that the ``perturbative'' sector of the 
theory --- i.e., working around the constant solutions --- admits a quantization for $S = \CC$ only if $\ch_2(T_X)$ vanishes and that each choice of trivialization produces a quantization.
(To extend to arbitrary Riemann surfaces, one needs $c_1(T_X) = 0$ as well.)
They also interpret this theory as the half-twisted form of a $(0,2)$-supersymmetric $\sigma$-model.

Our goal in this part is to describe and quantize the curved $\beta\gamma$ system 
using the renormalization and BV machinery of \cite{CosBook} in combination with Gelfand-Kazhdan descent.
Applying the main theorem of \cite{CG2}, we then obtain a factorization algebra of quantum observables 
and extract from it a vertex algebra, which is the CDOs.
In other words, we develop mathematically the physical arguments and results in \cite{WittenCDO} and \cite{Nek}.
Along the way, we will see how aspects of those physical arguments, such as the anomalies,
appear in this BV formalism.

\begin{rmk}
In \cite{WG1,WG2} Costello already used this machinery to quantize the curved $\beta\gamma$ system,
but he uses a formalism of $L_\infty$ spaces rather than Gelfand-Kazhdan descent.
This approach does not lend itself as easily to direct comparison with CDOs, so far as we can tell,
and so we pursued another approach to encoding the target space, 
which is more explicitly analogous to techniques used by Kontsevich, Cattaneo-Felder, and others.
Strictly speaking, we do not rely upon Costello's results --- notably the $L_\infty$ space formalism --- and we show that our approach recovers his results when the target spaces (e.g., complex manifolds) 
are treatable with Gelfand-Kazhdan descent.
In practice, though, we borrow and re-purpose several lemmas, and 
we clearly take our inspiration from his work.
\end{rmk}

A key idea in our approach, which we learned from Costello's work, is to encode the $\sigma$-model as a gauge theory. 
This alternative presentation of the $\beta\gamma$ system, with the formal $n$-disk $\hD^n$ as the target,
makes it amenable to Gelfand-Kazhdan descent.
Our approach thus breaks into the following steps:
\begin{enumerate}
\item write the classical BV theory of the $\beta\gamma$ system as a gauge theory with a natural action of the Harish-Chandra pair $(\Vect,\GL_n)$,
\item analyze the obstruction ({\em aka} anomaly) to quantizing this gauge theory equivariantly with respect to the $(\Vect,\GL_n)$-action,
\item construct an Harish-Chandra extension of $(\Vect,\GL_n)$ via the obstruction and show that there is an equivariant quantization for this extended pair, and
\item describe the bundle of factorization algebras obtained by
  Gelfand-Kazhdan descent, for this extended pair, applied to the
  factorization algebra of quantum observables with target the formal $n$-disk.
\end{enumerate}
The strong parallels with the CDO story, as articulated in Part I, should be apparent here: 
in both cases, the classical situation works nicely with usual Gelfand-Kazhdan descent, but the quantum situation requires an extended version.
Indeed, the primary changes are that we replace vertex algebras with factorization algebras and that we use the BV formalism to produce the quantization, rather than a vertex algebra version of canonical quantization.
Both changes require a heavy use of homological machinery, and so it should be no surprise that we must allow homotopical actions of the Harish-Chandra pair $(\Vect,\GL_n)$ on cochain complexes and thus develop a homotopical version of Gelfand-Kazhdan descent.

Throughout this part, we work in the formalism developed in \cite{CosBook,CG1, CG2}
and refer to them liberally, not aiming to be self-contained here.
Nonetheless, we recall essential ideas and notations along the way and give detailed citations.

\section{A brief overview of derived deformation theory and $\L8$ algebras}
\label{sec DDT}

Throughout this part, we will use some homological constructions 
that are not wholly standard knowledge and 
can seem rather sophisticated upon first acquaintance.
The actual manipulations are straightforward and 
amount to exploiting several ways of writing maps between completed symmetric algebras, 
these ways being equivalent but distinct in flavor.
The reader familiar with $\L8$ algebras, Maurer-Cartan elements, twisting cocycles,
and so on, can safely skip this section.
For others, it will at least identify the tricks outside the complicated context in which we use them.
Our treatment is succinct and casual, and we cite \cite{LV, Hinich, LurieSAG, CG2}
for detailed treatments.

There is one important idea, and not just manipulation, connected with these constructions:
every dg Lie algebra describes a ``formal space'' (in some sense a moduli space parametrizing deformations of something), 
and conversely every formal space is described by some dg Lie algebra.
This idea is attributed to Deligne, Drinfeld, Quillen, Schlessinger-Stasheff, and others, 
and thanks to Lurie \cite{LurieSAG} and Pridham \cite{Pridham}, 
it has a precise incarnation in derived algebraic geometry, 
which provides a suitably sophisticated notion of ``space.''

Here we only need the following dictionary between a formal moduli space $\frak{X}$ and its associated dg Lie algebra $\fg_{\frak{X}}$:
\begin{itemize}
\item the dg algebra of functions $\cO(\frak{X})$ on $\frak{X}$ corresponds to $\clies(\fg_{\frak{X}})$, the Chevalley-Eilenberg cochains,
\item the dg coalgebra of distributions on $\frak{X}$ corresponds to ${\rm C}^{\Lie}_*(\fg_{\frak{X}})$, the Chevalley-Eilenberg chains,~and
\item the dg Lie algebra of vector fields on $\frak{X}$ corresponds to $\clies(\fg_{\frak{X}},\fg_{\frak{X}}[-1])$.
\end{itemize}
For us, the Chevalley-Eilenberg chains $\cliel_*(\fg)$ has underlying graded vector space $\Sym(\fg[1])$,
equipped with the standard coproduct where $\Delta(x) = x \otimes 1 + 1 \otimes x$ for $x \in \fg[1]$, 
and the differential $\d_{\cliel_*}$ is a degree one coderivation determined by  
\[
\d_{\cliel_*}(xy) = (\d_\fg x) y \pm x(\d_\fg y) + [x,y]
\]
for any $x,y \in \fg[1]$.
The Chevalley-Eilenberg cochains $\clies(\fg)$ is the linear dual, so the underlying graded algebra is the {\em completed}
symmetric algebra $\cSym(\fg^\vee[-1])$.
(One must be careful about duals with infinite-dimensional vector spaces. 
In practice our examples will be tamed by a topology and will mean the continuous linear dual.)
The last identification, for vector fields, might seem strange until one computes that 
the cochain complex of derivations of the algebra $\clies(\fg_{\frak{X}})$ has underlying graded Lie algebra 
$\cSym(\fg^\vee[-1]) \otimes \fg[1]$ with the bracket the usual Lie bracket for vector fields with power series coefficients.

Let us introduce a toy example that plays an important role for us. 

\begin{dfn}\label{dfn fgn}
Let $\fgn$ denote the dg Lie algebra $\CC^n[-1]$,
which consists of a copy of $\CC^n$ in cohomological degree 1
and hence has zero differential and zero bracket.
\end{dfn}

Under the dictionary we find
$$\clies(\fg_n) = \cSym(\fg_n^*[-1]) \cong \CC[[t_1,\ldots,t_n]] = \hO_n,$$
so that $\fg_n$ should encode the formal $n$-disk $\hD^n$.
Under the dictionary, we also find an isomorphism of vector fields,
$$\clies(\fg_n,\fg_n[-1]) = \cSym(\fg_n^*[-1]) \otimes \CC^n \cong \bigoplus_{j=1}^n \CC[[t_1,\ldots,t_n]] \frac{\partial}{\partial t_j} = \Vect,$$
which will be useful for us.

Given the dictionary, it is not unreasonable to imagine enlarging both sides a bit, 
by allowing $n$-ary brackets (not just binary brackets) on the Lie side
and by allowing arbitrary (co)derivations on the (co)commutative (co)algebra side.
On the Lie side, such objects are called $\L8$ algebras,
but we use the following definition, which has the dictionary built into it.

\begin{dfn}
An {\em $\L8$ algebra} $\fg$ is a graded vector space $V$ along with
a degree 1 coderivation $Q$ on the coaugmented cocommutative coalgebra
$\Sym(V[1])$ that preserves the coaugmentation and squares to zero.
Its {\em Chevalley-Eilenberg chains} $\clls(\fg)$ is the dg cocommutative
coalgebra $(\Sym(V[1]),Q)$.
\end{dfn}

A coderivation $Q$ is determined by how it maps to cogenerators, 
so in this case it is determined by the ``Taylor components''
\[
Q_n: \Sym^n(V[1]) \to V[1],
\]
which encode the $n$-ary brackets
\[
\ell_n^\fg: (\Lambda^n V)[n-1] \cong \Sym^n(V[1])[-1] \xto{Q_n[-1]} V
\]
after shifting.
A dg Lie algebra gives an $\L8$ algebra in which $\ell_n = 0$ for $n > 2$.
Thus, $\clls$ is a direct generalization of $\cliels$, recovering it on dg Lie algebras.

\begin{dfn}
Let $\fg$ and $\fg'$ be $\L8$ algebras.
A {\em map of $\L8$ algebras} $f : \fg \l8to \fg'$ means a map of coaugmented dg cocommutative algebras
\[
f: \clls(\fg) \to \clls(\fg').
\]
Note that every strict map of dg Lie algebra yields an $\L8$ map by applying the functor $\cliel_*$.
\end{dfn}

We use $\l8to$ to emphasize that $f$ is {\em not} a cochain map from $\fg$ to $\fg'$.
This notion of $\L8$ map allows for a succinct way of describing a map between dg Lie algebras up to coherent homotopy.

This notion also leads to a homotopy coherent version of a representation.

\begin{dfn}
For $M$ a dg vector space, an {\em $\L8$ action of $\fg$ on $M$} means a map of $\L8$ algebras
$\rho: \fg \l8to \End(M)$, where $\End(M)$ denotes the dg vector space of graded endomorphisms of $M$
with the commutator bracket. We also say $\rho$ makes $M$ a {\em $\L8$-representation} or {\em $\L8$-module}
for~$\fg$.
\end{dfn}

We unravel this definition as follows.
A map of $\L8$ algebras $\rho: \fg \l8to \End(M)$ is determined by the composite 
\[
\pi_{\Sym^1} \circ \rho: \Sym(\fg[1]) \to \End(M)[1] = \Sym^1(\End(M)[1]),
\]
as any coalgebra map is determined by how it maps to cogenerators.
Thus, we obtain a sequence of maps
\[
\rho_n: \Sym^n(\fg[1]) \otimes M \to M
\]
of degree $2-n$, which describe more concretely how $\fg$ acts on elements of~$M$.
This version of the data makes it manifest how to define the {\em Chevalley-Eilenberg chains of $M$},
$\clls(\fg,M)$, which generalizes the Lie algebra homology of a representation of a Lie algebra
and which thus encodes the coinvariants of the representation~$M$.

An $\L8$ algebra $\fg$ also possesses a {\em Chevalley-Eilenberg cochains} $\clus(\fg)$, 
which is the dg completed commutative algebra $(\cSym(V^\vee[-1]),Q^\vee)$.
When $\fg$ is a dg Lie algebra, this definition $\clus(\fg)$ recovers the usual cochains $\clies(\fg)$.
(One must be careful about what one means by the graded linear dual $V^\vee$
if $V$ is not finite-dimensional in each cohomological degree.
In practice our infinite-dimensional vector spaces are tamed by a topology.)
For each representation $M$, we also have the {\em Chevalley-Eilenberg cochains of $M$},
$\clus(\fg,M)$, which generalizes the Lie algebra cohomology of a representation of a Lie algebra
and which thus encodes the invariants of the representation~$M$.
We will use $\Der(\fg)$ to denote the $\clus(\fg,\fg[1])$, 
as it encodes the vector fields (or derivations) of $\fg$ viewed as a formal space.

It is often convenient to describe a map of $\L8$ algebras $f: \fg \l8to \fg'$ in two other ways:
\begin{enumerate}
\item a map of augmented dg commutative algebras $f^*: \clus(\fg') \to \clus(\fg)$~or
\item a Maurer-Cartan element $\alpha_f$ in the $\L8$ algebra $\clus(\fg) \otimes \fg'$.
\end{enumerate}
Let us explain what we mean in the second case.

First, observe that the tensor product $A \otimes \fg$ of a dg commutative algebra $A$ 
and $\L8$ algebra $\fg$ obtains a natural $\L8$ structure where
\[
\ell_n^{A \otimes \fg}(a_1 \otimes x_1, \ldots, a_n \otimes x_n) = \pm (a_1 \cdots a_n) \otimes \ell_n^{\fg}(x_1,\ldots,x_n).
\]
In other words, we use the commutative product of $A$ to multiply the $A$-components 
and we use the $\L8$ structure on $\fg$ to bracket the $\fg$-components.
This definition is just the extension to $\L8$ algebras of the familiar construction 
with commutative algebras and Lie algebras 
(e.g., recall why the sections of the adjoint bundle of a principal $G$-bundle form a Lie algebra).
Second, a Maurer-Cartan element of an $\L8$ algebra $\fg$ is a degree one element $\alpha$ such that
\[
\sum_{n \geq 1} \frac{1}{n!} \ell^{\fg}_n(\alpha, \cdots,\alpha) = 0.
\]
When $\fg$ is a dg Lie algebra, this recovers the standard definition.
(In principle, this infinite sum is ill-defined, but we always work in situations where only a finite sum appears.
Alternatively, one needs to introduce some mechanism to make the infinite sum well-defined, such as with a topology.)
Finally, a map of $\L8$ algebras $f$ is determined by the composite 
\[
\pi_{\Sym^1} \circ f: \Sym(\fg[1]) \to \fg'[1] = \Sym^1(\fg'[1]),
\]
as any coalgebra map is determined by how it maps to cogenerators.
This composite provides an element $\alpha_f \in \cSym(\fg^\vee[-1]) \otimes \fg'[1]$,
and the condition that $f$ intertwines the differentials is equivalent 
to the Maurer-Cartan equation on~$\alpha_f$.

We introduce one final bit of notation, since we use it repeatedly below.
Let ${\rm C}^*_{\rm Lie, red}(\fg)$ denote the {\em reduced} cochains:
we remove the constant terms (i.e., the span of the unit element).
That is, the underlying graded vector space~is
\[
{\rm C}^\sharp_{\rm Lie,red}(\fg) = \cSym^{>0}(\fg^*[-1]),
\]
namely, the functions that vanish at the base point of the formal moduli space encoded by~$\fg$.

\section{The formal $\beta\gamma$ system}
\label{sec formal}

We now turn to the case where the target is the formal $n$-disk $\hD^n$,
which was formulate as a classical BV theory in the style of a gauge theory.
This encoding allows a concise description of how diffeomorphisms on the target act on the theory,
and thence a description as a $\Vect$-equivariant classical BV theory.

\subsection{The free $\beta\gamma$ system as a BV theory}

We briefly recall how to encode the free $\beta\gamma$ system --- where the target $X$ is the affine space $\AA^n$ --- as a BV theory, following \cite{GwThesis,CG1}.
(Note that the name is due to the traditional choice of letters to denote the fields.)

\begin{dfn} 
The rank $n$ {\em free $\beta\gamma$ system} on a Riemann surface $S$ has fields 
\ben
\Omega^{0,*}(S, \CC^n) \oplus \Omega^{1,*}(S, \CC^n),
\een
concentrated in cohomological degrees 0 and 1.  
We denote by $\gamma = (\gamma_1,\ldots,\gamma_n)$ a section of $\Omega^{0,*}(S, \CC^n)$, and 
we denote by $\beta = (\beta_1,\ldots,\beta_n)$ a section of $\Omega^{1,*}(S, \CC^n)$.
The shifted pairing is ``wedge and integrate'':
\be\label{pairing}
\langle \gamma + \beta,\gamma'+\beta'\rangle = \sum_{i=1}^n \int_S \gamma_i \wedge \beta_i' + \beta_i \wedge \gamma_i'.
\ee
The action functional is
\ben
S_{\text{free}}(\gamma, \beta) = \langle \beta, \dbar \gamma \rangle = \sum_{i=1}^n \int_S  \beta_i \wedge \dbar \gamma_i .
\een
The equations of motion are thus 
\[
\dbar \gamma_i = 0 = \dbar \beta_i
\]
for $i = 1,\ldots,n$.
\end{dfn}

There is an action of the general linear group $\GL_n$ on the space of
fields of the free $\beta\gamma$ system. Explicitly, for a field of
the form $(f \tensor v, g \tensor \lambda) \in \Omega^{0,*}(S ;
\CC^n) \tensor \Omega^{1,*}(S ; \CC^n)$ we define for $A \in \GL_n$
\ben
A \cdot (f \tensor v, g \tensor \lambda) = (f \tensor A v, g \tensor
(A^{-1})^T \lambda) .
\een
That is, we view $\GL_n$ acting on $\Omega^{0,*}(S ; \CC^n)$ through
the defining representation on $\CC^n$ and on $\Omega^{1,*}(S,
\CC^n)$ through the coadjoint representation on $\CC^n$. By
construction this action preserves $S_{free}$. 

\subsection{The formal $\beta\gamma$ system}

We now turn to the case where the target is the formal $n$-disk $\hD^n$,
which is closely related to the free case we just described.

Let $S$ denote a Riemann surface, and let $\fg_n$ be the abelian Lie algebra from Definition~\ref{dfn fgn}.
The dg Lie algebra $$\fg_n^S := \Omega^{0,*}(S,\fg_n)$$ plays a central role for us.
It is abelian but has a nontrivial differential via $\dbar$.
The Maurer-Cartan equation of this dg Lie algebra is $\dbar(\gamma) = 0$, where $\gamma: S \to \CC^n$ is a smooth function;
in other words, a solution is simply a holomorphic map from $S$ to $\CC^n$.
Under the dictionary, this dg Lie algebra $\fg^S_n$ encodes a formal moduli space that describes how to deform the constant function with value $0$ to a holomorphic functions.
Note that this Maurer-Cartan equation is precisely the Euler-Lagrange equation for $\gamma$ in the free $\beta\gamma$ system,
and the deformations describe the formal neighborhood of the constant zero map among all holomorphic functions.

To describe the $\beta$ fields as well, we simply enlarge the dg Lie algebra to its ``double''
\[
\DD \fg_n^S = \Omega^{0,*}(S, \fg_n) \oplus \Omega^{1,*}(S, \fg_n^\vee[-2]).
\]
Note that the shifts mean that in cohomological degree one, we have $\Omega^{0,0} \otimes \CC^n \oplus \Omega^{1,0} \otimes (\CC^n)^*$, 
and in cohomological degree two, we have $\Omega^{0,1} \otimes \CC^n \oplus \Omega^{1,1} \otimes (\CC^n)^*$.
The Lie bracket is still trivial, and the differential is $\dbar$ in both complexes. 
If $\beta$ denotes an element in $\Omega^{1,0} \otimes (\CC^n)^*$,
then the Maurer-Cartan equation is $\dbar(\beta) = 0$, which is precisely the Euler-Lagrange equation in the $\beta\gamma$ system.
Hence the dg Lie algebra $\DD \fg_n^S$ encodes, in some sense, the free $\beta\gamma$ system.
To be more precise, it encodes the $\beta\gamma$ system with the formal $n$-disk $\hD^n$ as the target,
since this dg Lie algebra describes deformations of the constant map to the origin.

\begin{rmk}
This holomorphic abelian gauge theory is simply a holomorphic version of BF theory, where the Lie algebra is now in a shifted degree.
\end{rmk}

Note that under this correspondence, the BV bracket for the BV theory encoding the $\beta\gamma$ system 
corresponds to the linear pairing on $\DD \fg^S_n$ arising from the evaluation pairing on $\fgn$. 
Explicitly, for $\gamma, \gamma' \in \Omega^{0,*}_c(S, \fg_n)$ and $\beta, \beta' \in \Omega^{1,*}_c(S, \fg_n^\vee[-2])$, 
consider the pairing
\be\label{pairing2}
\langle \gamma + \beta, \gamma' + \beta' \rangle = \int_S \ev_{\fgn}(\gamma \wedge \beta') + \ev_{\fgn}(\beta \wedge \gamma'),
\ee
where $\ev$ denotes the evaluation pairing between $\fgn$ and $\fgn^\vee$ and 
where $\ev_{\fgn}(\gamma \wedge \beta')$ denotes the composite of taking the wedge product of the Dolbeault components 
and the evaluation pairing of the Lie algebra components.
This pairing is invariant under the Lie bracket and has cohomological degree $-3$.
(This shift, in conjunction with the shift in Chevalley-Eilenberg cochains, ensures that one obtains a shifted Poisson bracket of degree 1,
as needed for a classical BV theory.)

Just as in the non-formal case the group $\GL_n$ acts on $\DD
\fg_n^S$. 

\begin{lemma}\label{GLaction} The group $\GL_n$ acts on the dg Lie algebra $\DD
  \fg_n^S$ in a way that preserves the pairing $\<-,-\>$. 
\end{lemma}
\begin{proof} 
The action of $\GL_n$ is induced by the defining representation on
$\fg_n [1] = \CC^n$ and the coadjoint action on $\fg_n^\vee [-2] =
(\CC^n)^*$. 
\end{proof}

\subsection{The $\Vect$ action on $\fgn^S$ and on $\DD\fgn^S$}\label{classicalvectaction}

We have just seen that the formal $\beta\gamma$ system is equivariant
for the group $\GL_n$. There is a richer equivariance coming from
non-linear automorphisms of the formal disk that we now wish to
describe. 

First, consider the global curved $\beta \gamma$-system with source
$S$ and target $X$. Explicitly, the fields consist of pairs of a map $\gamma : S
\to X$ together with a section $\beta \in \Gamma(K_S \tensor
\gamma^*(T^*X))$. The action is, as in the flat case, $\int_S \beta
\wedge \dbar \gamma$. 

Biholomorphisms act on the $\gamma$ fields in the
obvious way: given a biholomorphism $\phi : X \to X$ we obtain a new
field via composition $\phi \circ \gamma : S \to
X$. Now, a biholomorphism induces an action of sections on any tensor
bundle. In particular, on sections of $T^*X$ the biholomorphism $\phi$
acts by the inverse Jacobian $Jac(\phi)^{-1}$. Thus, we have an action
of the biholomorphism $\phi$ on the $\beta$ fields given by
$Jac(\gamma^*\phi)^{-1}$. Thus, the action on the pair $(\gamma,
\beta)$ is given by
\ben
\phi \cdot (\gamma, \beta) = (\phi \circ \gamma, Jac(\phi)^{-1} \beta) 
\een
where $Jac(\phi)^{-1} \beta$ is a section of $(\phi \circ
\gamma)^*T^*X \tensor K_S$. Since $\phi$ is holomorphic we have $$\dbar
(\gamma \circ \phi) = Jac(\phi) \cdot \dbar \gamma .$$ It follows that biholomorphisms
are a symmetry of the classical theory. 

The dg Lie algebras we introduced above describing the formal
$\beta\gamma$ system arise via a general method for producing dg Lie algebras:
given a dg Lie algebra $\fg$ and a commutative dg algebra $A$, the tensor product $A \otimes \fg$ has a natural 
dg Lie algebra structure where the differential is
\[
\d( a \otimes X) = (\d_A a) \otimes X + (-1)^{a} a \otimes \d_\fg X
\]
and the bracket is
\[
[a \otimes X, a' \otimes X'] = (-1)^{Xa'} (aa') \otimes [X,X'].
\]
Above, we took $A$ to be the Dolbeault complex $\Omega^{0,*}(S)$.

Now, if another Lie algebra $\fh$ acts on $\fg$, there is a natural extension to an action of $\fh$ on $A \otimes \fg$ by simply leaving the $A$-term alone.
We want to use an $\L8$ version of this procedure to equip $\fg_n^S$ and $\DD \fg_n^S$ with an $\L8$ action of $\Vect$,
extending the $\L8$ action of $\Vect$ on $\fg_n$.
This $\L8$ action is something familiar in physics, just expressed compactly via our dictionary.
For a $\sigma$-model with target $X$, a diffeomorphism of $X$ acts on the space of maps into $X$. 
If the diffeomorphism preserves structure relevant to the $\sigma$-model (e.g., a metric or complex structure), 
then the diffeomorphism acts on the space of solutions to the Euler-Lagrange equations of the theory.
This $\L8$ action encodes how formal diffeomorphisms of the target formal disk $\hD^n$ act on the formal moduli space of solutions to the equations of motion for the $\beta\gamma$ system. 

Let us provide an explicit description of this $\L8$ action in order to make the extension manifest.
Denote the generators of $\fg_n$ by $\{\xi_1,\ldots,\xi_n\}$ and the dual generators of $\fg_n^\vee$ by $\{t_1, \ldots,t_n\}$. 
Hence we have
\[
\clies(\fg_n) = \cSym(\fg^\vee_n[-1]) = \CC[[t_1,\ldots,t_n]],
\]
as already mentioned. Moreover we have a natural map
\be\label{algact}
\rho_W : \Vect \to \Der(\clies(\fg_n)), \;\; f(t_i) \partial_j \mapsto f(t_i) \xi_j .
\ee
Expressed as an $\L8$ action of $\Vect$ on $\fg_n$, it is given by a sequence of maps
\ben
\ell^{\rm W}_m : \Vect \otimes \fgn^{\otimes m} \to \fgn
\een
of cohomological degree $1-m$, where $m$ ranges over all non-negative integers.
These maps are simply the ``Taylor components'' of $\rho_W$.
For instance, the vector field $X = t_1^{m_1} \cdots t_n^{m_n} \partial_j \in \Vect$ acts by zero for any $m \neq m_1 + \cdots + m_n$, 
and for $m = m_1 + \cdots + m_n$, 
\ben
\ell^{\rm W}_m\left(X, (\xi_1^{\otimes m_1} \otimes \cdots \otimes \xi_n^{\otimes m_n}) \right) 
= \ell_m^{\rm W} \left((t_1^{m_1} \cdots t_n^{m_n} \partial_j) \otimes \xi_1^{\otimes m_1}\otimes \cdots \otimes \xi_n^{\otimes m_n} \right) 
= \xi_j 
\een 
and vanishes on any other basis element $\fg_n^{\otimes m}$.

With these formulas in hand, we can equip $A \otimes \fgn$ with an $\L8$ action of $\Vect$.
Here the sequence of maps is
\ben
\ell^{W,A}_m : \Vect \otimes (A \otimes \fgn)^{\otimes m} \to A \otimes \fgn
\een
with
\[
\ell^{W,A}_m(X, (a_1 \otimes x_1)\otimes \cdots \otimes (a_m \otimes x_m)) = \pm (a_1\cdots a_m) \otimes \ell^{\rm W}_m(X,x_1 \otimes \cdots \otimes x_m),
\]
where the sign is determined by Koszul's rule.
Equivalently, we can encode the $\L8$ action in a Lie algebra map
\[
\rho_{W,A}: \Vect \to \clies(A \otimes \fgn, A \otimes \fgn[-1]),
\]
which assembles the $\ell^A_m$ maps into a ``Taylor series.''
If we set $A$ to be $\Omega^{0,*}(S)$, then we obtain an $\L8$ action of $\Vect$ on $\fgn^S$.
A lift of this action to an $\L8$ action of $\Vect$ on $\DD\fgn^S$ is uniquely determined by the requirement that the action preserve the degree -3 pairing.

Diffeomorphisms of a manifold $X$ naturally determine symplectomorphisms of the cotangent bundle $T^*X$,
given simply by the associated map of vector bundles.
Thus diffeomorphisms also act naturally on the space of maps from $S$ into $T^*X$.
We can use the $\L8$ language to provide a concise description of 
this action of formal diffeomorphisms of the disk $\hD^n$ on the fields of the formal $\beta\gamma$ system.

The action of $\Vect$ on the Lie algebra $\fgn$ induces an action of $\Vect$ on the dual space $\fg_n^\vee$ via the evaluation pairing:
\ben
\<X \cdot v, w\>_{\fgn} = \<v, X \cdot w\>_{\fgn}
\een
for all $v + w \in \fgn \oplus \fgn^\vee$. 
This action is linear in the sense that brackets
\ben
\Vect \otimes (\fg_n^\vee)^{\otimes m} \to \fg_n^\vee
\een
are zero for $m > 1$. 
We can extend this action to the dg vector space $\Omega^{1,*}(S, \fg_n^\vee)$ and hence to $\DD \fg_n^S$;
we denote this $\L8$-action by
\ben
\DD \rho_W : \Vect \to \Der(\DD \fg_n^S).
\een
Since $\Vect$ preserves the dual pairing on $\fg_n$ and $\fg_n^\vee$, 
it is immediate that it preserves the invariant pairing of degree $-3$ on $\DD \fg_n^S$. 
We summarize these observations in the following.

\begin{lemma}
The classical BV theory of the formal $\beta\gamma$ system is equivariant with respect to $\Vect$:
the action of $\Vect$ on the fields preserves the shifted pairing on the fields and the action functional.
In other words, the $\L8$ action of $\Vect$ on $\DD\fgn^S$ determined by the canonical action of $\Vect$ on $\fgn$ preserves the shifted pairing and the differential.
\end{lemma}

\subsection{A Noether current and the obstruction-deformation complex}
\label{sec obsdef}

It thus should be no surprise that we can also use a local functional to express this action of infinitesimal diffeomorphisms. The explicit formula is quite simple and is just the natural formula from physics written in terms of formal power series. (See equation (\ref{eqn noether}) below.) To formulate this result, we recall now some useful notation. We will also see how we can obtain the usual Noether current for the symmetry by vector fields from this local functional.  

\subsubsection{Recollections on local functionals}

A systematic exposition of local functionals and deformation complexes can be found in \cite{CosBook}, particularly in chapter 5,
but here we provide a brief summary with our theory as a running example. 

Let $\cE$ be a dg Lie (or $L_\infty$) algebra associated to a classical BV field theory on $S$. 
The underlying graded vector space of $\cE$ consists of the smooth sections of a certain graded vector bundle $E$ on $S$;
we call such sections the {\em fields} of the field theory.
The key example here is the dg Lie algebra $\DD \fg_n^S$ for the formal $\beta\gamma$ system,
whose fields are $\gamma$ and~$\beta$.

A {\em Lagrangian density} is a functional on $\cE$ that takes values in smooth densities on $S$ and depends polynomially (or as a power series) on the fields and their derivatives.
This dependence is local: if $\cL(\gamma)$ is a Lagrangian density evaluated on the field $\gamma$, 
its value at a point $x \in S$ only depends on the $\infty$-jet ({\em aka} Taylor expansion) of $\gamma$ at $x$.
A Lagrangian density $\cL$ then determines a {\em local} functional on fields by integrating over $S$.
More precisely, one obtains a functional on compactly-supported fields, since integration is always well-defined so long as the domain is compact.
As an example, let  $\cE = \DD\fg_1^\CC$, the rank 1 formal $\beta\gamma$ system on $S = \CC$.
The Lagrangian density $\cL(\gamma,\beta) = \beta \wedge \dbar \gamma$
has local functional given by the action functional.

Note that since total derivatives with compact support have trivial integral,
two Lagrangian densities that differ by total derivatives determine the same local functional.
Thus, we  define the dg vector space of local functionals on the classical field theory $\cE$ by
\ben
\cloc^*(\cE) = \Dens_S \tensor_{D_S} {\rm C}_{\rm Lie,red}^*(J^\infty(E)) . 
\een
Here $D_S$ denotes the ring of differential operators on $S$, 
$\Dens_S$ denotes the smooth densities on $S$ equipped with its natural right action by $D_S$, 
and $J^\infty E$ denotes the sheaf of $\infty$-jets of smooth sections of the bundle $E$. 
Since $\cE$ is a sheaf of Lie algebras, $J^\infty(E)$ is a sheaf of Lie algebras in $D_S$-modules,
and we are computing the Lie algebra cochains in the category of $D_S$-modules.
Such cochains should be viewed as functions on the $\infty$-jets of fields with values in functions on $S$.
Hence, this tensor product produces {\em densities} on $S$ that are power series in the jets of fields.
Moreover, taking this tensor product over $D_S$ encodes the relation that total derivatives vanish.
Note that we've eliminated the constant functions on jets of sections.

This description of $\cloc^*(\cE)$ is quite abstract, but 
by restricting to compactly supported fields, we can provide a more concrete description of the situation. 
A local functional is a sum of functionals of the form
\[
\int_{x \in S} D_1\gamma(x) \wedge \cdots \wedge D_j\gamma(x) \wedge D_{j+1}\beta(x) \wedge\cdots \wedge D_k\beta(x)\wedge \d \mu,
\]
where the $D_i$ are differential operators and where $\d\mu$ is a smooth form.
This functional is {\em homogeneous of order $k$}.
(Such a functional can have any cohomological degree.)
As we saw, a prototypical example of a local functional is the action functional itself, which is quadratic and degree~zero.

\begin{dfn}
Let $\Def_n = \cloc^*(\DD \fg_n^S)$ denote the cochain complex of local functionals on $\DD \fg_n^S$.
Elements consist of formal sums $I = \sum_{k > 0} I_k$ where each $I_k$ is a local functional that is homogeneous of order $k$. We call $\Def_n$ the {\em obstruction-deformation complex} for the formal $\beta\gamma$ system.
\end{dfn}

The deformation complex is, in fact, a subcomplex of the Chevalley-Eilenberg continuous cochains on the dg Lie algebra $\DD \fg_n^S$. 
(Essentially, we mean the cochains as a dg Lie algebra in topological vector spaces, but 
see Section \ref{sec functional analysis} for more thorough discussion of this point.)
A cochain in the deformation complex of homogeneous degree $k$ is a distribution supported, by definition, along the small diagonal $X \hookrightarrow X^k$.

The complex $\Def_n$ has a shifted Lie algebra structure arising from the BV bracket $\{-,-\}$,
which is determined by the shifted pairing between the fields $\gamma$ and $\beta$ of Equation (\ref{pairing2}).
In essence, we view the pairing on fields as encoding a shifted symplectic structure on the space of fields,
and hence functionals should have a shifted Poisson bracket.
From this perspective $\gamma$ and $\beta$ are conjugate variables (like $q$ and $p$ in traditional, unshifted symplectic geometry).
Given two local functionals $F$ and $G$, one computes $\{F,G\}$ by pairing a $\gamma$ input from $F$ with a $\beta$ input from $G$, and {\it vice versa} (much as $\{q,p\} = 1$).
With respect to this bracket, the differential on $\Def_n$ is then $\{S_{\text{free}},-\}$.
(Note that local functionals do not form a commutative algebra, however, since the product of two such is no longer local!
The bracket does provide a shifted Lie bracket, just not a Poisson bracket.)

\subsubsection{The action of vector fields}

We now describe the action of formal vector fields on the classical theory using local functionals.
Verifying the lemma is a direct computation using the definitions.

\begin{lemma}\label{Noether}
The map $\DD \rho_W : \Vect \to \Der (\DD \fg_n^S)$ describing the $\L8$-action of $\Vect$ on the dg Lie algebra $\DD \fg_n^S$ has a lift
\ben
\xymatrix{
&\Def_n [-1] \ar[d]^-{\{-,-\}} \\
\Vect \ar[r]_-{\DD \rho_W} \ar[ur]^{I^{\rm W}} & \Der(\DD \fg_n^S)
}
\een
with $I^{\rm W}$ a local functional.
Explicitly, given a formal vector field 
\[
X = \sum_{j = 1}^n \sum_{\bm = (m_1,\ldots,m_n) \in \NN^n} a_{j, \bm} t_1^{m_1} \cdots t_n^{m_n} \partial_j,
\] 
the local functional 
\be\label{eqn noether}
I^{\rm W}_X(\gamma, \beta) = \sum_{j = 1}^n \sum_{\bm \in \NN^n} a_{j, \bm} \int_S  \gamma^{\wedge m_1}_1 \wedge \cdots \wedge \gamma^{\wedge m_n}_n \wedge \beta_j
\ee
satisfies $\{I^{\rm W}_X,-\} = \DD\rho_W(X)$.
\end{lemma}

\begin{rmk}
When restricted to linear vector fields, the action of $\Vect$ on $\beta\gamma$ system with target $\hD^n$ 
agrees with the action of $\GL_n$ described in Lemma \ref{GLaction}. In this sense, we have described an action of the Harish-Chandra pair $(\Vect, \GL_n)$ on the classical $\beta\gamma$ system. 
This theory can thus be treated by Gelfand-Kazhdan formal geometry.
We develop this reasoning more fully in Section \ref{sec comparison}. 
In particular, in the next section we will show that this theory descends to the classical curved $\beta\gamma$ system where the target is a complex manifold $X$; more precisely, we will identify this theory with the theory defined by Costello in \cite{WG2}.
\end{rmk}

As explained in Section \ref{sec DDT}, we can identify this map
$I^{\rm W}$ with a Maurer-Cartan element $\clies(\Vect,
{\rm C}^*_{\rm loc}(\DD\fgn^S))$. Explicitly, this identification means that
\be\label{MC1}
(\d_{\Vect} \tensor 1 + 1 \tensor \dbar)I^{\rm W} + \frac{1}{2} \{I^{\rm W}, I^{\rm W}\} =
0 .
\ee
In fact, given our formula for the local functional, it is natural to view $I^{\rm W}$ as a function of $X \in \Vect$ and the fields $\gamma$ and $\beta$.
We thus have the following cochain complex, which plays a crucial role in studying the formal $\beta\gamma$ system as a $\Vect$-equivariant BV theory.

\subsubsection{The relation to the formal Atiyah class} \label{sec atiyah 2}

In this section we describe how the local functional $I^{\rm W}$, which encodes the action of formal vector fields on the classical theory, is related to the Gelfand-Fuks-Atiyah class from Section~\ref{sec gk descent}. 

We have already discussed the action of $\Vect$ on the dg Lie algebra $\fg_n$ and its
dual $\fg_n^\vee$ and how this determines an action on the dg Lie
algebra $\DD \fg_n^S$ which is encoded by the Maurer-Cartan element $I^{\rm W} \in \clie^*(\Vect) \tensor \cloc^*(\DD \fg_n^S)[-1]$. 

Fix $S = \CC$ and use the natural framing of the tangent bundle by $\partial_z$ to write
$\Omega^{0,*}(S) = C^\infty(\CC) \tensor \CC[\d \zbar]$. Similarly,
$\Omega^{1,*}(\CC) = C^\infty(\CC)[\d \zbar] \d z$. Using this
notation, we find a decomposition
\be\label{decomposition}
\DD \fg_n^S = C^\infty(\CC) \tensor \left((\fg_n \oplus \fg_n \,\d \zbar) \oplus (\fg_n^\vee[-2] \oplus \fg_n^\vee[-2] \, \d \zbar)
  \, \d z\right)
\ee
as the tensor product of a commutative algebra $C^\infty(\CC)$ and a graded Lie algebra.
(The differential on the dg Lie algebra does not respect this decomposition.)

It will be convenient to analyze $I^{\rm W}$ in terms of this decomposition.
To be more precise, we consider the local functional $I^{\rm W}_X$ for each formal vector field $X$ of the form $a^i \partial_i$, where the coefficient $a^i \in \hO_n$ is a homogeneous polynomial. 
Observe then that $I^{{\rm W}}_X$ is itself a homogeneous local functional of the form
\ben
I^{{\rm W}}_X : \Sym^{k+1}(\DD \fg_n^S) \to \CC 
\een 
where $a^i$ has polynomial degree~$k$.
Using the decomposition (\ref{decomposition}), 
we can write $I^{{\rm W}}_X$ as $I^{{\rm W}, an}_X I^{{\rm W},alg}_X$, 
a product of an analytic factor times a algebraic factor with $I^{{\rm W},an}_X \in \Sym^{k+1}(C^\infty(\CC)^\vee)$ and with
\ben
I^{{\rm W}, alg}_X \in \clie^*\left((\fg_n \oplus \fg_n \,\d \zbar) \oplus (\fg_n^\vee[-2] \oplus \fg_n^\vee[-2] \, \d \zbar) \, \d z\right)
\een 
Moreover, $I^{\rm W}$ is linear in the inputs $\fg_n \, \d \zbar$ and $(\fg_n^\vee[-2] \oplus \fg_n^\vee[-2] \, \d \zbar)\,\d z$,
as there must be precisely one $\d z$ and $\d \zbar$ for the integral (\ref{eqn noether}) to be nonzero.
Thus, we see that the algebraic factor is an element in
\ben
I^{{\rm W},alg}_X \in \clie^*\left(\fg_n ; (\fg_n \, \d \zbar)^\vee[-1] \tensor \fg_n[1]\right) .
\een 
For the rest of this section, we suppress $\d \zbar$ from the notation and identify the right-hand side with $\hO_n \tensor \End(T_0)$ where $T_0 = \fg_n[1] = \CC^n$ is the space of constant vector fields.

The formal de Rham differential $\d_{dR} : \hO_n \to \hOmega^1_n$ determines a map
\ben
\d_{dR} \tensor 1 : \hO_n \tensor_\CC \End(T_0) 
\to \hOmega^1_n \tensor_\CC \End(T_0) .
\een
which is reminiscent of equipping a vector bundle with a connection after specifying a global frame.

We have, as a corollary of Proposition \ref{atiyahprop1}, the
following relationship of the functional $I^{{\rm W}, alg}$ to the
Gelfand-Fuks-Atiyah class. 

\begin{cor} \label{atiyah and IW}
For each $X \in \Vect$ of homogenous degree $k$,
\ben
(\d_{dR} \tensor 1) I^{{\rm W}, alg}_X =  \At^{\rm GF}(\hT_n)(X) \in \hOmega^1_n \tensor_{\hO} \End(\hT_n)
\een 
where  $\At^{\rm GF}(\hT_n)$ is the Gelfand-Fuks-Atiyah class of the formal
vector bundle $\hT_n$.
\end{cor}

\begin{proof}
We can think of $X \mapsto I^{{\rm W}, alg}_X$ as a linear map
\ben
I^{{\rm W},alg} : \Vect \to \hO_n \tensor_\CC \End(T_0) ,
\een
or equivalently as a linear map $I^{{\rm W}, alg} : \Vect \tensor T_0 \to \hT_n$. 
This map is, in fact, the restriction of the action $\rho_{\hT_n}$ of $\Vect$ on the formal tangent bundle $\hT_n$ to the space $T_0 = \CC^n$, the space of constant sections of $\hT_n$.
That is, $\rho_{\hT_n}|_{\Vect \tensor T_0} = I^{{\rm W}, alg}$. Proposition \ref{atiyahprop1} then implies that $(\d_{dR} \tensor 1) I^{{\rm W},alg}$ is a representative for the Gelfand-Fuks-Atiyah class.
\end{proof}

\subsubsection{Equivariant deformation complex} \label{sec eq def cplx}

We can now make the following deformation complex that controls $\Vect$-equivariant deformations of the classical theory. 

\begin{dfn}
The {\em $\Vect$-equivariant obstruction-deformation complex} is 
the graded vector space $\cSym(\Vect^\vee[-1]) \otimes {\rm C}^\sharp_{\rm loc}(\DD\fgn^S)$ 
equipped with the differential $\d_{\Vect} + \dbar + \{I^{\rm W},-\}$, 
where $\d_{\Vect}$ denotes the differential from $\clies(\Vect)$
and $\dbar$ denotes the differential from ${\rm C}^*_{\rm loc}(\DD\fgn^S)$.
We use $\Def^{\rm W}_n$ to denote this complex.
\end{dfn}

In other words, this complex is the tensor product $\clies(\Vect) \otimes {\rm C}^*_{\rm loc}(\DD\fgn^S)$
twisted by $I^{\rm W}$ as the twisting cochain.
It encodes succinctly how the formal vector fields $\Vect$ act on the local functionals of the field theory.
Its role in the equivariant BV formalism is analogous to the role of the non-equivariant obstruction-deformation complex in the BV formalism:
\begin{itemize}
\item first-order deformations of the formal $\beta\gamma$ system as a $\Vect$-equivariant classical BV theory live in the zeroth cohomology~and
\item the obstruction to equivariant BV quantization modulo $\hbar^2$ lives in the first cohomology.
\end{itemize}
Hence it behooves us to compute its cohomology.
We will find a particularly nice answer after further constraining the problem.

There are two further symmetries of this theory that we will exploit.
First, there is a natural scaling action of $\CC^\times$ on the fibers of the cotangent bundle (as on any vector bundle)
that scales the $\beta$ fields of the $\beta\gamma$ system.
The action functional has ``weight one'' with respect to this scaling action.
In our setting there is thus an action of $\CC^\times$ on $\DD \fgn^S$ given by scaling the $\beta$ fields. 
Second, we restrict now to the Riemann surface is $S = \CC$ and 
note that affine linear automorphisms ${\rm Aff}(\CC) = \CC \ltimes \CC^\times$ preserve the action functional
of the $\beta\gamma$ system.
We are only interested in the subcomplex of $\Def^{W}_n$ consisting of local functionals that are weight zero under the scaling action and invariant under the $\Aff(\CC)$ action.
Then we have the subcomplex
\ben
(\Def_n)^{\CC^\times \times {\rm Aff}(\CC)} \subset \Def_n
\een
and its equivariant version $(\Def^{\rm W}_n)^{\CC^\times \times {\rm
    Aff}(\CC)} \subset \Def_n^{\rm W}$.

\begin{prop}\label{eqdef} 
There is a quasi-isomorphism of $\Vect$-modules
\ben
J : \hOmega^2_{n,cl}[1]  \xto{\simeq} (\Def_n)^{\CC^\times \times {\rm Aff}(\CC)} .
\een
Applying the functor $\clie^*(\Vect ; -)$, we obtain a quasi-isomorphism
\be\label{trans}
J^{\rm W} : \clie^*(\Vect, \hOmega^2_{n,cl}[1]  ) \xto{\simeq} (\Def_n^{\rm W})^{\CC^\times \times {\rm Aff}(\CC)}.
\ee
\end{prop}

The proof of this result is in Section \ref{seceqdef}, but first we will have to describe the map $J$ in the above proposition,
which is the subject of the next section.

This quasi-isomorphism $J$ is, in fact, $(\Vect, \GL_n)$-equivariant.
Let us note an important consequence of this proposition.

\begin{cor}\label{gerbe of obsdef}
The Gelfand-Kazhdan descent along a complex manifold $X$ 
of the $(\Vect, \GL_n)$-module $(\Def_n)^{\CC^\times \times {\rm Aff}(\CC)}$ 
returns a sheaf of dg vector spaces that is quasi-isomorphic to the sheaf~$\Omega^2_{X,cl}[1]$. 
\end{cor}

In particular we have the following description over a general manifold:
\begin{itemize}
\item[(1)] the space of anomalies of the theory over $X$ is
$H^2(X, \Omega^2_{X,cl})$,
\item[(2)] the space of deformations over $X$ is $H^1(X, \Omega^2_{X,cl})$ and
\item[(3)] the space of automorphisms over $X$ is $H^0(X, \Omega^2_{n,cl})$.
\end{itemize}
This description matches precisely with the study of deformations of the curved
$\beta\gamma$ system as in~\cite{WittenCDO,Nek}. 

\subsection{Closed two-forms as local functionals} \label{sec j functional}

We have already seen how vector fields yield local functionals of the formal $\beta\gamma$-system 
and thus give it the structure of an equivariant BV theory. 
In this section we will show how closed two-forms yield
local functionals of $\gamma$, i.e., only of the subspace of fields $\Omega^{0,*}(S ;\fg_n[1])$. 
That is, we define a linear map
\ben
J : \hOmega^2_{n,cl} \to \Cloc^*(\fg_n^S)
\een 
and use $J_\omega$ to denote the image of $\omega$. 
This map will exhibit the quasi-isomorphism of Proposition \ref{eqdef}.

\begin{rmk}
This map has the following geometric interpretation. 
On the formal disk, every closed two-form $\omega$ is exact, 
so that $\omega = \d \theta$ for some $\theta \in \hOmega^1_n$. 
Use the field $\gamma : S \to \hD^n$ to pull back this one-form to the one-form $\gamma^* \theta$ on $S$.
We interpret this one-form as a {\em current}; 
we can integrate it around any closed one-cycle in $S$ to get a function of $\gamma$. 
We denote this current by $\Tilde{J}_\theta$, where $\Tilde{J}_\theta(\gamma)= \gamma^* \theta.$
By Stokes theorem, this current vanishes if $\theta$ is exact, 
so the local functional only depends, in fact, on the corresponding closed two-form $\d \theta$.
Hence we write $J_{\d\theta} = \Tilde{J}_\theta$. 
\end{rmk}

\subsubsection{Defining $J$}
\label{defining J}

Although pulling back forms is easy, 
we wish to rewrite this construction in terms of $\fg_n$
and hence we need to describe pullback under Koszul duality.
Thus, to define $J$, we need to introduce a few constructions. 

First, there is an assignment  
\ben
(-)^S : \hO_n \to {\rm Hom}_{\CC} \left( \Sym(\Omega^{0,*}_S \tensor \fg_n[1]), \Omega^{0,*}_S\right) 
\een
that promotes a function on the formal $n$-disk to a function on the formal moduli space $\fg_n^S$ with values in holomorphic functions on $S$.
It goes as follows.
Given an input $f \in \hO_n$, let $f_k$ denote its homogeneous component of degree $k$. 
View $f_k$ as a linear map $f_k : \Sym^k(\fg_n[1]) \to \CC$. 
We then define 
\ben
\begin{array}{cccc}
f^S_k : & \Sym^k(\Omega^{0,*}_S \tensor \fg_n[1]) & \to & \Omega^{0,*}_S \\ 
&(\gamma_1 \tensor \xi_1) \cdots (\gamma_k \tensor \xi_k) & \mapsto &
(\gamma_1 \wedge \cdots \wedge \gamma_k) f(\xi_1, \cdots, \xi_k) 
\end{array}
\een
Extend to non-homogenous elements by linearity so that $f^S = \sum_k f^S_k$. 

Similarly, a one-form on the formal disk $\theta \in \hOmega^1_n = \clie^*(\fg_n ;\fg_n^\vee[-1])$ 
encodes a linear map $\theta : \Sym(\fg_n[1]) \to \fg_n^\vee[-1]$. 
Let $\theta_k : \Sym^k(\fg_n[1]) \to \fg_n^\vee[-1]$ be its homogenous component of degree $k$. 
As above, there is a natural linear map
\ben
\begin{array}{cccc}
\theta^S_k :  & \Sym^k(\Omega^{0,*}_S \tensor \fg_n[1]) & \to & \Omega^{0,*}_S \tensor
\fg_n^\vee[-1] \\ 
&(\gamma_1 \tensor \xi_1) \cdots (\gamma_k \tensor \xi_k) & \mapsto &
(\gamma_1 \wedge \cdots \wedge \gamma_k) \tensor \theta_k(\xi_1, \cdots, \xi_k) 
\end{array}
\een
Let $\theta^S = \sum_k \theta^S_k$, as above.

Each one-form $\theta$ thus determines a local function $\Tilde{J}_\theta \in \cloc^*(\fg_n^S)$ by the formula
\ben
\Tilde{J}_\theta (\gamma) = \sum_k \int_S \<\theta^S_k \left(\gamma^{\tensor k}\right), \partial \gamma \>_{\fgn}.
\een 
Explicitly, if $\theta = t_1^{m_1} \cdots t_n^{m_n} \d t_j$ is
monomial one-form, then we have
\ben
\Tilde{J}_\theta (\gamma) = \int_S \gamma_1^{m_1} \wedge \cdots \wedge
\gamma_n^{m_n} \wedge \partial \gamma_j .
\een 
For shorthand notation, we will write $\Tilde{J}_{\theta} = \int_S
\<\theta^S(\gamma), \partial \gamma\>_{\fg_n}$ where the sum over
homogenous components is implicit. 

We tie up the properties of the functional $\Tilde{J}$ in the following proposition, proved below. 

\begin{prop} \label{jmap} 
The assignment $\theta \mapsto \Tilde{J}_\theta$ satisfies:  
\begin{itemize}
\item[(1)] For all $\theta$, the local functional $\Tilde{J}_\theta$ is $\dbar$-closed
  inside $\Def_n$ and lies in the subcomplex $(\Def_n)^{\CC^\times \times {\rm
    Aff}(\CC)}$.
\item[(2)] The assignment $\theta \mapsto \Tilde{J}_\theta$ is
  $\Vect$-equivariant. That is, $\Tilde{J}_{L_X \theta} = X \cdot
  J_\theta$ where $X \cdot (-)$ denotes the action of vector fields on
  functionals and $L_X$ is the Lie derivative.
\item[(3)] The functional $\Tilde{J}_\theta$ is identically zero if $\theta$
  is an exact one-form. 
\end{itemize}
Thus, $\Tilde{J}$ descends to a $\Vect$-equivariant map
\ben
J : \hOmega^2_{n,cl} [1] \to (\Def_n)^{\CC^\times \times {\rm Aff}(\CC)} 
\een
that we denote $\omega \mapsto J_\omega$. 
Here $J_\omega =\Tilde{J}_{\theta}$, where $\theta$ is any one-form satisfying $\d \theta = \omega$. 
\end{prop}

\subsubsection{Understanding $J$}

The formula for the functional $\Tilde{J}_\theta$ is best understood 
as integration over $S$ after applying an operator $\mathbf{J}$ valued in densities. 
We continue to describe everything via the homogeneous components $\theta_k$ of $\theta$.

First, for each homogeneous degree $k$, consider the composition
\ben
\begin{array}{ccc}
\Sym^k(\Omega^{0,*}_S \tensor \fgn[1] ) \tensor (\Omega^{0,*}_S \tensor
\fg[1]) & \xto{1 \tensor \partial} & \Sym(\Omega^{0,*}_S \tensor \fgn[1])
\tensor (\Omega^{1,*}_S \tensor \fg) \\ & \xto{\theta_k^S \tensor 1} &
(\Omega^{0,*}_S \tensor \fg_n^\vee[-1]) \tensor (\Omega^{1,*}_S
\tensor \fg_n[1]) \\ & \xto{\<-,-\>_\fg} & \Omega^{1,*} .
\end{array}
\een 
Here, $\<-,-\>_\fg$ is the evaluation pairing between $\fg_n[1]$ and $\fg_n^\vee[-1]$. 
We then symmetrize the composite to obtain the $(k+1)$th homogenous component of $\mathbf{J}_\theta$:
\[
(\mathbf{J}_\theta)_{k+1} : \Sym^{k+1}(\Omega^{0,*}_S \tensor \fg_n) \to \Omega^{1,*}_S. 
\]
In this notation, we have $\Tilde{J}_\theta = \int_S \mathbf{J}_\theta$. 

Before proving the main result, we make the following simple observations about the functional $\mathbf{J}$. 

\begin{lemma} \label{easy} 
For $f \in \hO_n$ and $\theta \in \hOmega^1_n$,
\begin{itemize}
\item[(1)] $\mathbf{J}_{f \theta} = f^S \wedge
  \mathbf{J}_\theta$ and
\item[(2)] $\mathbf{J}_{\d_{dR} f} = \partial \circ f^S$. 
\end{itemize}
\end{lemma} 

\begin{proof}
For simplicity, suppose $f$ is of homogenous degree $k$ and $\theta$ of homogenous degree $l$.  
Then $f \theta$ defines a linear map 
\[
\begin{array}{ccccc}
\Sym^{k+l}(\fg_n[1]) & \to & \Sym^{k}(\fg_n[1]) \tensor \Sym^l(\fg_n[1]) & \to & \fg_n^\vee[-1]\\
\xi_1,\ldots,\xi_k \xi_1',\ldots,\xi_l' & \mapsto & (\xi_1,\ldots,\xi_k) \otimes (\xi_1',\ldots,\xi_l') & \mapsto & f(\xi_1,\ldots,\xi_k) \theta(\xi_1',\ldots,\xi_l')
\end{array}.
\]
Thus, $(f \theta)^S = f^S \theta^S$, from which (1) follows.

We now show (2). Consider the special case of a linear functional $\tau : \fgn[1] \to \CC$, viewed as linear element of $\hO_n$.
The one-form $\d_{dR}(\tau)$ corresponds to the very simple functional $\Sym^0(\fgn[1]) \to \fg_n^\vee[-1]$ sending $1 \mapsto \tau$. 
Thus, $\mathbf{J}_{\d_{dR} \tau} = \partial(\tau^S)$. 
To see (2) in general, we note that both the left and right hand sides are derivations with respect to the product of functions. 
Indeed, if $f,g \in \clie^*(\fgn)$, then $\partial((fg)^S)) = \partial(f^S g^S) = \partial(f^S)\wedge g^S + f^S \partial(g^S)$. 
\end{proof}

\begin{proof}[Proof of Proposition \ref{jmap}]
Observe that the functional $\mathbf{J}_\theta$ is described by
applying a constant coefficient holomorphic differential operator to
the fields $\gamma$. Thus $\mathbf{J}_\theta$ is clearly holomorphic
and invariant under affine linear transformations. It follows that
$\Tilde{J}_\theta$ is holomorphic, that is $\dbar \Tilde{J}_\theta = 0$, and hence it is
closed in $\Def_n$. This proves (1). 
8

The formula in (2) in Lemma \ref{easy} implies (3) since the integral
of a $\partial$-exact form is zero. Hence $\Tilde{J}$ defines a map
$J : \hOmega^2_{n,cl} \to \cloc^*(\fg_n^S)$. Explicitly, given a
closed two-form $\omega$ with $\d_{dR} \theta = \omega$ we have $J_\omega =
\Tilde{J}_\theta$. This proves (3). 

Finally, we show (2). We have seen in our discussion of the Noether
current that the action of a formal vector field $X$ on the
deformation complex is through the BV bracket with $I^{\rm
  W}_X$. Thus, we must show for all one-forms $\theta$ that
$\Tilde{J}_{L_X \theta} = \{I^{\rm W}_X, \Tilde{J}_\theta\}$. For
simplicity, suppose $X = \partial_i$, a constant vector field. Then, if we choose a homogenous one-form
$\theta = t_1^{m_1} \ldots t_n^{m_n} \d t_j$ then
\ben
L_X \theta =  m_i t_1^{m_1} \cdots t_k^{m_j -1} \cdots
t_n^{m_n} \d t_j . 
\een 
Now, to compute $\{I_X, \Tilde{\theta}\}$. The functional $I_X$ has a
single $\beta_i$ input that pairs with a single $\gamma_i$ input from the
functional $\Tilde{J}_\theta$. There are $m_i + \delta_{ij}$ such
$\gamma_i$ inputs, the $\delta_{ij}$ coming from the factor $\partial
\gamma_j$ in the definition of $\Tilde{J}_\theta$. So,
we obtain
\ben
\{I_X, \Tilde{J}_\theta\} (\gamma) = m_i \int_S \gamma_1^{m_1} \wedge
\cdots \wedge \gamma_j^{m_i-1} \wedge \cdots \gamma^{m_n}_n \partial \gamma_j
+ \delta_{ij}\int_S \partial (\gamma_1^{m_1} \wedge \cdots \wedge \gamma_n^{m_n}) .
\een
The first term is $\Tilde{J}_{L_X\theta}$. Being the integral of a
total derivative the second term vanishes, so we are done. The case of
a general formal vector field $X$ is similar. Indeed, suppose $X$ is
homogeneous of the form $X= t_1^{k_1} \cdots
t_n^{k_n} \partial_i$. Then for $\theta$ as above we have
\ben
L_X \theta = m_i t_1^{k_1 + m_1} \cdots t_i^{k_i + m_i - 1}
  \cdots t_n^{k_n + m_n} \d t_j + \delta_{ij} t_1^{m_1} \cdots t_n^{m_n} \d(t_1^{k_1} \cdots
t_n^{k_n}) .
\een
On the other hand, we compute directly
\bestar
\{I_X, \Tilde{J}_\theta\} & = & m_i \int_S \gamma_1^{k_1+m_1} \wedge \cdots
\wedge \gamma_i^{k_i + m_i - 1} \wedge\cdots \wedge \gamma_{n}^{k_n +
  m_n} \partial \gamma_j \\ & - & \delta_{ij} \int_S\partial(\gamma_1^{m_1}
\wedge \cdots \wedge \gamma_n^{m_n})\gamma_1^{k_1} \wedge
\cdots \wedge \gamma_n^{k_n}  .
\eestar
The first line comes from pairing the $\beta_j$ input from the
functional $I_X^{\rm W}$
with the $\gamma_i^{m_i}$ input from $\Tilde{J}_\theta$. The next term
comes from pairing the $\beta_i$ input with the $\partial \gamma_j$
input from $\Tilde{J}_\theta$ (there is a sign from integrating by
parts). Integration by parts again returns $\Tilde{J}_{L_X \theta}$ as
desired.
\end{proof}

\begin{proof}[Proof of Proposition \ref{eqdef}]
\label{seceqdef}
We have just seen that $J : \hOmega^2_{n,cl} \to \Def_n$ is
$\Vect$-equivariant. To complete the proof it suffices to show that we have a $\Vect$-equivariant equivalence $(\Def_n)^{\CC^\times \times {\rm
    Aff}(\CC)} \simeq \hOmega^2_{n,cl}[1]$. With $\fgn$ as the choice of the $\L8$ algebra $\fg$, 
this equivalence appears as Proposition 15.1.1 in \cite{WG2},
whose proof we will review in order to keep track of the $\Vect$-action. 

First, observe that by restricting to weight zero local functionals under the scaling action,  
we only consider functionals that are independent of $\beta$. 
This constraint implies that 
\[
(\Def_n)^{\CC^\times \times {\rm Aff}(\CC)} \cong {\rm C}^*_{\rm loc}(\Omega^{0,*}(\CC, \fgn))^{{\rm Aff}(\CC)},
\] 
since we can write any such functional as a wedge product of $\beta$ with a nontrivial Lagrangian in $\gamma$.
Following Chapter 5, Section 6 of \cite{CosBook}, we exploit a
description of translation invariant local functionals via $D$-modules:
\be\label{eqn jet}
\Cloc^*\left(\Omega^{0,*}(\CC, \fg_{\hD^n})\right)^\CC 
\cong \CC \, \d z \, \d \zbar \otimes_{D_{\CC}} \Cred^*\left({\rm Jet}_0(\Omega^{0,*}(\CC ; \fgn))\right),
\ee
where $\text{Jet}_0$ denotes jets of sections at zero of~$\Omega^{0,*}(\CC ; \fg_n)$. 

Using $z$ for the holomorphic coordinate on $\CC$, we have
\ben
{\rm Jet}_0(\Omega^{0,*}(\CC, \fgn)) \cong \fgn \llbracket z, \zbar, \d \zbar \rrbracket.
\een 
and thus the identification (\ref{eqn jet}) is manifestly $\Vect$-equivariant. 
It follows that we have a $\Vect$-equivariant quasi-isomorphism
\ben
(\Def)^{\CC^\times \times \CC} \simeq \CC\, \d z \wedge \d \zbar \otimes_{\CC[\partial_z, \partial_{\zbar}]} \Cred^*(\fgn\llbracket z , \zbar, \d \zbar \rrbracket)
\een
where on the left-hand side we are taking ivariants with respect to $\CC^\times \times \CC \subset \CC^\times \times (\CC \ltimes \CC^\times) = \CC^\times \times {\rm Aff}(\CC)$. So, we only need to compute the $\CC^\times$-invariants of the right-hand side. Here $\CC^\times$ acts by scaling space-time.

The quasi-isomorphism of dg Lie algebras
\ben
(\fgn \llbracket z
\rrbracket,0) \xto{\simeq} \left(\fgn \llbracket z,
  \zbar, \d \zbar \rrbracket, \dbar\right)
\een
is obviously $\Vect$-equivariant. Finally, Costello's calculation implies that (in the case that $\fg = \fgn$) we have
\ben
\left(\CC\, \d z \wedge \d \zbar \otimes^{\mathbb{L}}_{\CC[\partial_z, \partial_{\zbar}]} \Cred^*(\fgn\llbracket z \rrbracket)\right) \simeq (\CC \to\hOmega^0_n \to \hOmega^1_n)[3].
\een
(That means the cochain complex on the right hand side starts with $\CC$ in degree $-3$.)
Moreover, the right-hand side is quasi-isomorphic via the de Rham differential to
\ben
\left(\hOmega^2_n [1] \to \hOmega^3_n [0] \to \cdots\right) \simeq \hOmega^2_{n,cl} [1].
\een
This identification is clearly $\Vect$-equivariant.
\end{proof}

\subsection{Holomorphic vector fields on the source} \label{sec hol vf}

We digress momentarily to describe another important symmetry present in the $\beta\gamma$ system:
the holomorphic $\sigma$-model possesses a natural symmetry of the Lie algebra of holomorphic vector fields $T_S = T^{1,0}_S$, 
much as the usual $\sigma$-model is conformal as a classical field theory.  
We will formulate this symmetry on the formal $\beta\gamma$ system. 

It is convenient for us to work with the Dolbeault resolution of holomorphic vector fields: 
define the dg Lie algebra 
\ben
\cT_S = \Omega^{0,*}(S ; T_S) 
\een
with differential given by $\dbar$ and Lie bracket given by the
extension of the Lie bracket of vector fields to $(0,*)$-forms. 
There is an action of $\cT_S$ on the Dolbeault complex $\Omega^{0,*}(S; \CC)^{\oplus n}$ given by Lie derivative of
$(0,*)$-forms:
\[
\xi \cdot (\gamma \tensor v) = (\cL_\xi \gamma) \tensor v.
\]
(We use the script $\cL$ to denote the Lie derivative with respect to
vector fields on the source, to avoid confusing it with $L_X$, 
the Lie derivative of vector fields on the target.) 
This action extends to an action of $\cT_S$ on the ``double'' dg Lie algebra $\DD \fg^S_n$
so that it preserves the shifted pairing between $\gamma$ and $\beta$ fields.

This action can be encoded by a local functional, and hence we obtain a $\hT_S$-equivariant field theory.

\begin{lemma} The map of dg Lie algebras $\cL : \cT_S \to \Der(\DD
  \fg_n^S)$, sending a holomorphic vector field $\xi$ on $S$ to the
  derivation $\cL_\xi$, describes an
  action of holomorphic vector fields on the rank $n$ free
  $\beta\gamma$ system. Moreover, it has a lift to a map of dg Lie algebras
\[
\begin{array}{cccc}
I^{\cT} :& \cT_S & \to & \Def_n[-1] \\
& \xi & \mapsto & \<\beta, \cL_\xi \gamma\> 
\end{array}
\]
along the map determined by the BV bracket $\{-,-\} : \Def_n[-1] \to
\Der(\DD \fg_n^S)$. 
\end{lemma}

\subsubsection{}

We wish to describe the equivariant obstruction-deformation complex
for the action of $\cT_S$. The functional $I^{\cT}$ endows the direct
sum $\cT_S \oplus \DD\fg_n^S$ with the structure of a local Lie
algebra. 
By definition, this equivariant obstruction-deformation complex
is given by the local cochains of this local Lie algebra
\ben
\Def_n^\cT = \cloc^*\left(\cT_S \ltimes \DD \fg_n^S\right) .
\een
where $\Def_n$ is the deformation complex for the formal $\beta\gamma$
system defined earlier. 
As $\cloc^*$ always involves taking the reduced Lie algebra cochains, 
there is a useful splitting of the deformation complex
\ben
\Def_n^{\cT} \cong \cloc^*(\cT_\CC) \oplus \cloc^*(\cT_\CC ; \Def_n),
\een
where $\Def_n$ is the deformation complex for the free $\beta\gamma$
system. 

For any complex manifold $Y$, the complex $\cloc^*(\cT_Y)$ has an
interpretation in terms of the diagonal cohomology of $Y$, studied by
\cite{Losik}. In the case of $Y = \CC$ it has a simple interpretation
in terms of Gelfand-Fuks cohomology.

\begin{prop}[Proposition 5.3 of \cite{bw_vir}] 
The cohomology of $\clie^*(\cT_\CC)$ is concentrated in degree one and is isomorphic to $H^3_{\rm Lie}({\rm W}_1)$. 
Hence, $H^*(\cloc^*(\cT_\CC)) \cong \CC [-1].$
\end{prop}

An explicit generator for the cohomology is given by the local cocycle 
\[
\begin{array}{lccc}
\omega^{\GF} : & \cT_{\CC} \times \cT_{\CC} & \to & \Omega^{1,1}(\CC) \\
& (\alpha \tensor \partial_z, \beta \tensor \partial_z) & \mapsto & \frac{1}{2\pi}\frac{1}{12} \left(
  \partial_z^3 \alpha^0 \beta^{0,1} +
  \partial_z^3\alpha^{0,1} \beta^0
\right) \d^2 z 
\end{array}
\]
where $\alpha = \alpha^0 + \alpha^{0,1} \d \zbar$ and similarly for $\beta$. 
(Note that we can integrate the density on the right whenever $\alpha$ or $\beta$ is compactly-supported,
but otherwise it is not integrable.  
This situation is just like that with action functionals, where the Lagrangian density is the important information rather than the purported function.)

\begin{prop} 
There exists a map of dg Lie algebras
\ben
(\omega^{\GF}, J,K) : \CC [-2] \oplus \hOmega^1_{n}[-1] \oplus
\hOmega^2_{n,cl} \to \Def_n^\cT [-1]
\een
sending $(1, \omega, \eta)$ to $(\omega^{\GF}, J_\omega, K_\eta)$ where
\begin{itemize}
\item $\omega^{\GF} \in \cloc^*(\cT_\CC) \subset \Def_n^{\cT}$
  represents the generator of the Gelfand-Fuks cohomology $H^3_{\rm
    Lie}({\rm W}_1) \cong \CC$; 
\item for every $\mu \in \hOmega^2_{n,cl}$, the functional $J_\mu
  \in \Def_n \subset \Def_n^\cT$ is the one defined in Section~\ref{sec j functional}; and
\item for $\eta \in \hOmega^1_n$, $K_\eta$ is the cocycle in $\cloc^1(\cT_\CC ; \Def_n)$ defined by  
\ben
K_\eta(\xi, \gamma) = \int_\CC \partial_z \xi^0
\<\eta^S(\gamma), \partial \gamma\>_{\fg_n} + \int_\CC \partial_z
\xi^{0,1} \d \zbar \<\eta^S(\gamma), \partial \gamma\>_{\fg_n} 
\een
where $\xi = \xi^0 \partial_z + \xi^{0,1} \d
  \zbar \,\partial_z$ is an element of $\cT_\CC$.
\end{itemize}
Moreover, this map is equivariant for the action of formal vector
fields~$\Vect$. 
\end{prop}

\begin{proof}
The assignment $\CC \to \cloc^*(\cT_\CC)$ sending $1 \mapsto \omega^{\GF}$ is tautologically
$\Vect$-equivariant. Moreover, we have already shown that the assignment
$J : \hOmega^2_{n,cl}[1] \to \Def_n$ is $\Vect$-equivariant. 

Thus, it suffices to show that $K : \hOmega^1 \to \cloc^*(\cT_\CC ;
\Def_n)$ is $\Vect$-equivariant. It suffices to check that
for all $\eta \in \hOmega^1_n$ and $X \in \Vect$, 
\ben
K_{L_X \eta} = \{I_X^{\rm W}, K_\eta\} .
\een  
This computation is parallel to the calculation in the proof of Proposition \ref{jmap}. 
\end{proof}

As a corollary we obtain a map of cochain complexes upon
applying the functor $\clie^*(\Vect ; -)$:
\ben
(\omega^{GF}, K, J) : \CC[-1] \oplus \clie^*\left( \Vect ; \Omega^1_{n} \oplus
\Omega^2_{n,cl}[1] \right) \to \clie^*(\Vect ; \Def_n^\cT) .
\een 
The complex $\clie^*(\Vect ; \Def_n^\cT)$ controls equivariant
deformations for both the Lie algebra $\Vect$ {\em and} $\cT_\CC$. The
map $(\omega^{\GF},K,J)$ will allow us to identify elements of the
deformation complex with ordinary characteristic classes. 

\section{Equivariant BV quantization of the formal $\beta\gamma$
  system}\label{sec equiv bv}

The free $\beta\gamma$ system is a free BV theory and hence admits a natural quantization.
(See Chapter 6 of \cite{GwThesis} for an extensive development.)  
This quantization is easily modified to encompass the formal $\beta\gamma$ system,
but here we want to quantize \emph{equivariantly} with respect to the action of $\Vect$.
We will find that there is an obstruction to quantizing equivariantly, 
given by the Gelfand-Fuks Chern class $\ch_2^{GF}(\hT_n)$ defined in Section \ref{sec gk cc}. 
This obstruction is a very local avatar of the anomaly described by Witten and Nekrasov \cite{WittenCDO,Nek},
and we will see in Part III that it corresponds in a very precise way to the obstruction to constructing CDOs
found by \cite{MSV, GMS} as described in Section \ref{sec vertex alg}.

There is, however, an equivariant quantization for a natural action of $\TVect$, 
the extension of formal vector fields $\Vect$ by closed two-forms
$\hOmega^2_{cl}$ introduced in Section \ref{sec gk descent}. 
In fact, we will see that the space of closed two-forms is precisely the space of deformations for the $\beta\gamma$ system. 
We construct this quantization explicitly using Feynman diagrams and, in later sections, explain when and how it descends to complex manifolds.

Most of this section is devoted to formulating precisely what equivariant BV quantization means and then proving the following result.

\begin{thm} 
\label{bvq} 
There is a unique (up to contractible choice) $\TVect$-equivariant and $\CC^\times \times \Aff(\CC)$-invariant quantization 
of the $\beta\gamma$ system on $\CC$ with target $\hD^n$.
\end{thm}

By an $\Aff(\CC)$-invariant quantization, we mean one that is invariant with respect to the action of affine symmetries of the complex line (i.e., translation and dilation by complex numbers). 
The $\CC^\times$-symmetry condition says that the quantization has weight one with respect to scaling the $\beta$ fields, which can viewed as scaling the cotangent fibers of $T^*\hD^n$.
(See the discussion preceding Proposition~\ref{eqdef}.)

\begin{rmk}
\label{rmk on si's work}
Subsequent to the writing of this paper, 
Si Li developed general technology that should imply this theorem~\cite{LiQME}.
(His results, as stated, do not explicitly encompass equivariant quantization but it is clear that they must  extend to our situation.)
In brief, he exhibits a map from holomorphic ({\em aka} chiral) deformations of a free holomorphic field theory of $\beta\gamma$-type to the modes dg Lie algebra of the associated dg vertex algebra,
and he shows that a deformation satisfies the quantum master equation if and only if it satisfies the Maurer-Cartan equation in the modes Lie algebra.
Hence, one can take our interaction term $I$ and check if its associated Lie algebra element $\oint \d z\, I$ satisfies the Maurer-Cartan equation.
This approach is easier in the sense that computing in the modes Lie algebra is purely algebraic in nature,
and does not involve Feynman diagram computations,
precisely because all the Feynman diagrams are hidden in Li's proof, where he does that analysis once and for all.

In this paper, we wish to be self-contained and hence explicitly solve the equivariant quantum master equation and explicitly describe the Feynman diagrammatics.
As a matter of taste, we also wanted to show clearly that the vertex algebra constructions are completely separate from quantum field theoretic constructions --- hence, the division between Parts I and II --- and then to exhibit  in Part III that these independent approaches produce equivalent answers.
To use Li's theorem would intertwine the QFT construction with the vertex algebra construction,
and it might misleadingly suggest we need results from the theory of CDOs to construct our factorization algebra.

On the other hand, Li's work gives a systematic explanation for {\em why} such identifications hold between vertex algebra and QFT constructions.
It sets the stage for proving theorems like ours in a very broad range of examples.

Finally, we remark that Li's theorem essentially reverses the logic of this paper.
He relies on essentially the same logic as Part III but oriented to state a different result.
As we explain in Part III, a free $\beta\gamma$-type field theory yields a factorization algebra that recovers a dg vertex algebra, which is the one ``expected'' from physics.
Any holomorphic deformation of the theory must yield a deformation of the dg vertex algebra, 
and so there must be a map of dg Lie algebras.
We examine what deformation of the vertex algebra arises from our preferred deformation of the theory,
but we exhibit general results that would let one describe other deformations as well.
By contrast, Li shows that this map is injective at the level of sets of solutions (from QME to Maurer-Cartan equation) and hence can use the modes Lie algebra to identify deformations of the BV theory.
\end{rmk}

\subsection{Recollections on equivariant BV quantization} \label{sec equiv bv 1}

In this section we discuss what it means for a Lie algebra $\fh$ to act on a quantum field theory.
To be more precise, we review the formalism developed in \cite{CG2}, notably for the factorization Noether theorems  (see chapters 11 and 12).
A key idea is to make $\clies(\fh)$ the base ring over which the field theory is defined, rather than the complex numbers $\CC$.
Under the dictionary discussed in Section \ref{sec DDT}, this approach should encode how the Lie algebra $\fh$ acts on the theory.
We have already seen this idea deployed for the classical field theory, 
by interpreting the local functional $I^{\rm W}$ of Lemma \ref{Noether} as a Maurer-Cartan element.

Recall that in the BV formalism, as developed in \cite{CosBook,CG2}, 
a quantum BV theory consists of a space of fields and an effective action functional $\{S[L]\}_{L \in (0,\infty)}$,
which is a family of non-local functionals on the fields that are parametrized by a length scale $L$ 
and satisfy
\begin{enumerate}[(a)]
\item an exact renormalization group (RG) flow equation,
\item the scale $L$ quantum master equation (QME) at every length scale $L$,~and
\item as $L \to 0$, the functional $S[L]$ has an asymptotic expansion that is local.
\end{enumerate}
The first condition ensures that the scale $L$ action functional $S[L]$ determines the functional at every other scale.
The second can be interpreted as saying that we have a proper path integral measure at scale $L$ 
(i.e., the QME can be seen as a definition of the measure).
The third condition implies that the effective action is a quantization of a classical field theory,
since a defining property of a classical theory is that its action functional is local.
(A full definition is available in Section 8.2 of~\cite{CG2}.)

\begin{rmk}
The length scale is associated with a choice of Riemannian metric on the underlying manifold,
but the formalism of \cite{CosBook} keeps track of how the space of quantum BV theories depends upon such a choice 
(and other choices that might go into issues like renormalization).
Hence, when the choices should not be essential --- such as with a topological field theory --- one can typically show rigorously that different choices give equivalent answers.
The length scale is also connected with the use of heat kernels in \cite{CosBook},
but one can work with more general parametrices (and hence more general notions of ``scale''),
as explained in Chapter 8 of \cite{CG2}.
We use a natural length scale in this section; 
when it becomes relevant, in the context of factorization algebras, we switch to general parametrices.
\end{rmk}

If we start with an $\fh$-equivariant classical BV theory $\cE$ with action functional $S$ --- so that $\fh$ has an $\L8$ action on the fields that preserves the pairing and the action functional $S$ --- then we can encode the action of $\fh$ as a Maurer-Cartan element $I^\fh$ in $\clies(\fh) \otimes {\rm C}^*_{\rm loc}(\cE)$.
(For the formal $\beta\gamma$ system, we did this in Lemma \ref{Noether}.)
We then view the sum $S + I^\fh$ as the \emph{equivariant} action functional:
the operator $\d_{\clies(\fh)} + \{S+ I^\fh,-\}$ is the twisted differential on $\clies(\fh) \otimes {\rm C}^*_{\rm loc}(\cE)$ with $I^\fh$ as the twisting cocycle,
and this operator is square-zero because $\d_{\clies(\fh)}(S + I^\fh) + \{S + I^\fh, S+ I^\fh\}$ is a ``constant'' (i.e., lives in $\clies(\fh)$ and hence is annihilated by the BV bracket).

This perspective suggests the following definition of an equivariant quantum BV theory.
The starting data is two-fold:
an $\fh$-equivariant classical BV theory with equivariant action functional $S+I^\fh$, 
and a BV quantization $\{S[L]\}$ of the non-equivariant action functional $S$.
Following Costello, it is convenient to write $S$ as $S_{\text{free}} + I$, 
where the first ``free'' term is a quadratic functional and the second ``interaction'' term is cubic and higher.
In this situation, the effective action $S[L] = S_{\text{free}} + I[L]$, 
i.e., only the interaction changes with the length scale.

\begin{dfn} \label{eqQFT} 
An {\em $\fh$-equivariant BV quantization} is a collection of effective interactions $\{I^\fh[L]\}_{L \in (0\infty)}$
satisfying
\begin{enumerate}[(a)]
\item the RG flow equation
\[
W(P_\epsilon^L, I[\epsilon]+I^\fh[\epsilon]) = I[L] + I^\fh[L]
\]
for all $0 < \epsilon < L$,
\item the equivariant scale $L$ quantum master equation, which is that
\[
Q(I[L]+I^\fh[L]) + \d_{\clies(\fh)} I^\fh[L] + \frac{1}{2}\{I[L]+I^\fh[L],I[L]+I^\fh[L]\}_{L} + \hbar \Delta_L(I[L]+I^\fh[L])
\]
lives in $\clies(\fh)$ for every scale $L$, and
\item the locality axiom, with the additional condition that as $L \to 0$, we recover the equivariant classical action functional $S+ I^\fh$ modulo $\hbar$.
\end{enumerate}
\end{dfn}

In other words, we simply follow the constructions of \cite{CosBook,CG2} working over the base ring $\clies(\fh)$.
A careful reading of those texts shows that the freedom to work over interesting dg commutative algebras is built into the formalism.
Note that our situation is particularly simple since the non-equivariant classical field theory is free and hence admits a very simple quantization,
with $I[L] = 0$ for all $L$.

\begin{rmk}
Equivariant quantization is essentially a version of the background field method in QFT.
One treats elements of $\fh$ as background fields and 
the interaction terms $I^\fh[L]$ encode the variation of the path integral measure with respect to these background fields.
(Solving the QME is our definition of well-posedness of the measure.)
\end{rmk}

\subsection{The pre-theory}
\label{sec prequant}

We will follow an approach directly parallel to the non-equivariant construction of a BV quantization of the curved $\beta\gamma$ system in \cite{WG2}.
Our first step is to try to construct an equivariant effective pre-theory 
(i.e., effective actions satisfying the locality and RG flow conditions but not necessarily the QME condition)
for the $\Vect$-equivariant formal $\beta\gamma$ system.
Essentially, we try to run the RG flow from the classical theory by naively guessing
\[
I^\fh[L] = \lim_{\epsilon \to 0} W(P^L_\epsilon,I^\fh)
\]
and then adding counterterms to deal with singularities that prevent this limit from existing.
(One of the main theorems of \cite{CosBook} guarantees that we can construct such a pre-theory.)
In the next subsection, we will examine the failure of this action to satisfy the equivariant QME.

To construct the pre-quantization explicitly, we need to specify certain data, such as the heat kernels and propagators with which we will work.
As we are working on the Riemann surface $S = \CC$, it is natural to work with the standard Euclidean metric
and to take advantage of the compatibility between its Laplacian and the operators $\partial$ and $\dbar$.
The {\em analytic heat kernel} we use is 
\ben
K_t(z,w) = \frac{1}{4 \pi t} e^{-|z-w|/ 4 t} \cdot (\d z - \d w) \wedge (\d \Bar{z} - \d \Bar{w}) .
\een
Thanks to the decomposition $\fgn^\CC = \Omega^{0,*}(\CC) \otimes \fgn$ (i.e., the vector bundle is trivialized),
the heat kernel for $\DD\fgn^\CC$ factors into an analytic part and an algebraic part
\ben
K_t = K^{an}_t \otimes ({\rm Id}_{\fg} + {\rm Id}_{\fg^\vee}) .
\een 
The \emph{propagator} $P_{\epsilon<L}$ likewise factors as $P^{an}_{\epsilon<L} \otimes ({\rm Id}_{\fg} + {\rm Id}_{\fg^\vee})$ 
where
\ben
P^{an}_{\epsilon<L} = \int_{t=\epsilon}^L (\dbar^* \otimes 1) K^{an}_t\, \d t .
\een
The analytic part of the propagator is only nontrivial on ``mixed inputs,'' 
i.e., where one side of the edge is labeled by a $\gamma$ and the other side is a $\beta$.
(This property is, of course, a direct consequence of the shifted pairing on fields.)
Thus, one can view the propagator as ``directed'' from $\gamma$ to $\beta$.
Figure \ref{fig:prop} shows how we draw the edge labeled by a propagator.

The vertices of Feynman diagrams are also highly constrained, 
since every term in the interaction $I^{\rm W}$ is linear in $\beta$.
Figure \ref{fig:vertex} shows the vertex where the target is $\hD^1$ 
and the formal vector field is $t^n \partial_t$. 
As with the propagator, we view $\gamma$ and $\beta$ legs as oriented, 
and there is only ever one $\beta$ leg.

There are strong consequences for Feynman diagrams
due to this directedness and the linearity in $\beta$:
the only nontrivial connected Feynman diagrams that can appear have zero or one loops.
A connected graph of genus zero will be a tree with one leaf labeled by $\beta$ and 
the other leaves labeled by $\gamma$ or a formal vector field $X \in \Vect$.
(and hence will encode a functional that has weight one for the scaling action).
Now consider the simplest kind of one-loop graph: a wheel with $k$ vertices.
Since the edges of the loop are labeled by the propagator --- and 
so the $\beta$ legs of the vertices are used up on the loop --- the leaves can only take $\gamma$ or $X$ as input.
A general one-loop graph will be a wheel with trees attached.
See Figure \ref{fig:wheel} for a simple example.

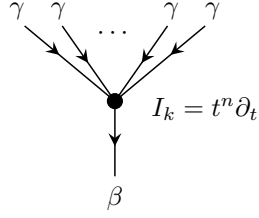
\begin{figure}
\begin{tikzpicture}[decoration={markings,mark=at position 0.6cm with {\arrow[black,line width=.4mm]{stealth}}}];
\draw[postaction=decorate, line width=.2mm] (-1.2,1) -- (0,0);
\draw[postaction=decorate, line width=.2mm] (-0.7, 1) -- (0,0);
\draw[postaction=decorate, line width=.2mm] (0.7,1) -- (0,0);
\draw[postaction=decorate, line width=.2mm] (1.2,1) -- (0,0);
\draw[postaction=decorate, line width=.2mm] (0,0) -- (0,-1);
\filldraw[color=black]  (0,0) circle (.1);
\draw (1.2,-0.1) node {$I_k = t^n \partial_t$};
\draw (-1.3,1.2) node {$\gamma$};
\draw (-0.75,1.2) node {$\gamma$};
\draw (0.75,1.2) node {$\gamma$};
\draw (1.3,1.2) node {$\gamma$};
\draw (0,0.9) node {$\cdots$};
\draw (0,-1.3) node {$\beta$};
\end{tikzpicture}
\caption{The vertex with $n$ incoming $\gamma$ legs and one outgoing $\beta$ leg}
\label{fig:vertex}
\end{figure}

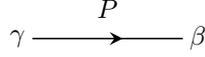
\begin{figure}
\begin{tikzpicture}[decoration={markings,
   mark=at position 1.2cm with {\arrow[black,line width=.4mm]{stealth}}}];
\draw[postaction=decorate, line width=.2mm] (-1,0) -- (1,0);
\draw (-1.2,0) node {$\gamma$};
\draw (1.2,0) node {$\beta$};
\draw (0, 0.4) node {$P$};
\end{tikzpicture}
\caption{The propagator as directed from $\gamma$ to $\beta$}
\label{fig:prop}
\end{figure}

\begin{figure}
\begin{center}
\begin{tikzpicture}[scale=0.8,decoration={markings,mark=at position 0.9cm with {\arrow[black,line width=.4mm]{stealth}}}];
\draw[postaction=decorate, line width=.2mm] (2,0) -- (0,0);
\draw (1, -0.4) node {$P$};
\draw[postaction=decorate, line width=.2mm] (0,0) -- (0,2);
\draw (-0.4, 1) node {$P$};
\draw[postaction=decorate, line width=.2mm] (0,2) -- (2,2);
\draw (1, 2.4) node {$P$};
\draw[postaction=decorate, line width=.2mm] (2,2) -- (2,0);
\draw (2.4, 1) node {$P$};
\draw[postaction=decorate, line width=.2mm] (-1.4,0) -- (0,0);
\draw (-1.6,0) node {$\gamma$};
\draw[postaction=decorate, line width=.2mm] (0,-1.4) -- (0,0);
\draw (0,-1.6) node {$\gamma$};
\draw[postaction=decorate, line width=.2mm] (-1.12,-1.12) -- (0,0);
\draw (-1.3,-1.3) node {$\gamma$};
\draw[postaction=decorate, line width=.2mm] (-1.12,3.12) -- (0,2);
\draw (-1.3,3.3) node {$\gamma$};
\draw[postaction=decorate, line width=.2mm] (3.4,2.5) -- (2,2);
\draw (3.6,2.6) node {$\gamma$};
\draw[postaction=decorate, line width=.2mm] (2.5,3.4) -- (2,2);
\draw (2.6,3.6) node {$\gamma$};
\draw[postaction=decorate, line width=.2mm] (3.4,-0.5) -- (2,0);
\draw (3.6,-.6) node {$\gamma$};
\draw[postaction=decorate, line width=.2mm] (2.5,-1.4) -- (2,0);
\draw (2.6,-1.6) node {$\gamma$};
\filldraw[color=black]  (0,0) circle (.1);
\filldraw[color=black]  (0,2) circle (.1);
\filldraw[color=black]  (2,0) circle (.1);
\filldraw[color=black]  (2,2) circle (.1);
\end{tikzpicture}
\caption{A wheel with four vertices}
\label{fig:wheel}
\end{center}
\end{figure}
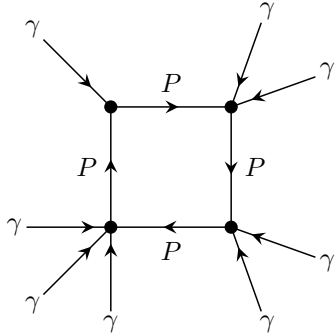

\begin{prop} 
For a connected genus one graph $\Gamma$, the limit $\lim_{\epsilon \to 0} W_{\Gamma}(P_{\epsilon<L}, I^{\rm W})$ exists. 
\end{prop}

In this proposition we remove the factors of $2$ and $\pi$ in the
definition of the heat kernel for shortness of exposition. These
factors will clearly not affect the existence of the limit. 

\begin{proof}
The graph weight $W_\Gamma$ will be a function of $\Vect$ and $\Omega^{0,*}(\CC) \otimes \fg_n$. 
The graph weight, like the propagator, factors as
\ben
W^{an}_\Gamma(P_{\epsilon<L}, I^{\rm W}) W^{\fg}_\Gamma
\een
where the analytic factor $W^{an}_\Gamma$ is a functional on the space $\Omega^{0,*}_c \oplus \Vect [1]$ and 
the algebraic factor $W^{\fg}_\Gamma$ is a functional on the space $\fg [1] \oplus \Vect [1]$. The algebraic factor does not depend on the regularization; it is independent of $\epsilon$ and $L$. 
Thus, to show that the limit exists it suffices to consider the analytic weight.
It also suffices to assume that $\Gamma$ is a wheel,
since the singularities arise from the wheel and not from any trees attached to the wheel. 

Suppose $\Gamma$ has $k$ vertices and choose a labeling of the vertices $v = (v_1,\ldots, v_k)$.
Let vertex $v_i$ correspond to the functional
\[
(\gamma,\beta,X) \mapsto a_{j, \bm} \int_S  \gamma^{\wedge m_1}_1 \wedge \cdots \wedge \gamma^{\wedge m_n}_n \wedge \beta_j
\]
where 
\[
X = \sum_{j = 1}^n \sum_{\bm = (m_1,\ldots,m_n) \in \NN^n} a_{j, \bm} t_1^{m_1} \cdots t_n^{m_n} \partial_j.
\] 
In other words, this functional only cares about the coefficient of $t_1^{m_1}\cdots t_n^{m_n} \partial_j$ in the vector field $X$
and uses it to produce a functional on $\beta$ and $\gamma$ of polynomial degree $m_1 + \cdots + m_n +1$.
The vertex $v_i$ thus has valence $1 + \nu_i = m_1 + \cdots + m_n +2$, 
where $\beta$ and $X$ each contribute one leg and the remaining legs are $\gamma$.
From hereon we will ignore the coefficient from $X$, as it does not affect convergence (only changing an overall constant)
and cease to discuss the leg associated to $X$.
Hence we will view $v_i$ as having valence $\nu_i$.

When we form the wheel, the $\beta$ leg of $v_i$ is paired with a $\gamma$ leg of $v_{i+1}$ by a propagator.
Thus there are $N = \sum_i (\nu_i - 1)$ external legs.
We now view the functional as a function of $N$ distinct inputs $\gamma_1,\ldots,\gamma_N$ of $\Omega^{0,*}(\CC)$,
which makes it easier to examine convergence.

Fix functions $\{f_{i,j_i}\} \in \Omega^{0,0}_c(\CC) = C^\infty_c(\CC)$ 
where $i = 1, \ldots k$ and $j_i = 1,\ldots,\nu_i - 1$. 
The analytic weight is 
\ben
W_\Gamma^{an}(P^{an}_{\epsilon<L}, I^{\rm W}(\{f_{i,i_j}\})) = 
\int_{{\bf z} \in \CC^k} \left(\prod_{i=1}^k \d^2 z_i \right)  \prod_{i = 1}^k \left(\left( \prod_{j_i =1}^{\nu_i - 1} f_{i,j_i} (z,\zbar) \right) P^{an}_{\epsilon<L}(z_j,z_{j+1}) \right). 
\een
In the product $z_{N+1}$ is identified with $z_1$, and $\d^2 z$ denotes $\d z \,\d \zbar$.
When $k =1$ this integral vanishes because the propagator vanishes
along the diagonal. Hence consider $k \geq 2$.
We want to show that the $\epsilon \to 0$ limit of the above integral exists
for any choices of $f_{i,j_i}$.

Before delving into analysis, we make some remarks that simplify notation.
First, we replace the product function $\prod_{i = 1}^k \prod_{j_i =1}^{\nu_i - 1} f_{i,j_i} (z,\zbar)$ 
by an arbitrary smooth function $\phi$ on $\CC^k$ with compact support,
as the functional above defines a distribution on $\CC^k$.
We thus need to show the integral vanishes for all such $\phi$.
Second, we repress from our notation obvious factors such as $\d^2 z_i$,
which can be reinserted by looking at the domain of integration (which is always a vector space).
Finally, we make a linear change of coordinates: $w_i := z_{i+1} - z_i$ for $1 \leq i < k$ and $w_k = z_k$. 
Note then that
\[
z_k - z_1 = \sum_{j = 1}^{k-1} w_j.
\]
Up to a constant factor independent of $\epsilon$ (i.e., the Jacobian of this change of coordinates),
the weight is
\bestar
\int_{w \in \CC^{k}} 
\phi(w,\Bar{w})
\left(
\int_{(t_1,\ldots,t_k) \in [\epsilon, L]^k} 
\frac{1}{t_1 \cdots t_k}
\left( \prod_{i=1}^{k-1} \frac{\Bar{w}_i}{t_i}   e^{-|w_i|^2/t_i} \right) 
\left(\sum_{j =1}^{k-1} \frac{\Bar{w}_j}{t_k} \right) e^{- |\sum_j w_j|^2/t_k }
\right) 
\eestar
The factor in parentheses is an explicit expression for the analytic propagators. 
We rewrite this expression as 
\begin{equation}\label{integral before}
\int_{w \in \CC^{k}} 
\phi(w,\Bar{w})
\left(
\int_{(t_1,\ldots,t_k) \in [\epsilon, L]^k} 
\frac{1}{t_1 \cdots t_k}
\left(\prod_{i=1}^{k-1} \frac{\Bar{w}_i}{t_i} \right)
\left(\sum_{j =1}^{k-1}\frac{\Bar{w}_j}{t_k} \right) 
e^{- \sum_i |w_i|^2 / t_i - |\sum_j w_j|^2/t_k }
\right). 
\end{equation}
Because
\[
\frac{\partial}{\partial w_i} e^{-|w_i|^2/t_i} = -\frac{\Bar{w}_i}{t_i }e^{-|w_i|^2/t_i},
\]
one can use integration by parts  to trade powers of $t_i^{-1} \Bar{w}_i$ for derivatives of $\phi$.
This is our next step in proving convergence.

Define the differential operator 
\[
\sigma(t) = \frac{1}{t_1 + \cdots + t_k}\sum_{j=1}^{k-1} t_j \partial_{w_j},
\]
which is a differential operator on functions on 
$\CC^{k-1} = \CC_{w_1} \times \cdots \times \CC_{w_{k-1}}$ (i.e., functions of the variables $w_1,\ldots,w_{k-1}$) 
whose coefficients are functions of the variables $(t_1,\ldots,t_k) \in [\epsilon,L]^k$.
Define the first-order differential operator  
\ben
D_m(t) :=\partial_{w_m} - \frac{1}{t_1 + \cdots + t_k}\sum_{j=1}^{k-1} t_j \partial_{w_j} = \partial_{w_m} - \sigma(t),
\een
with $1 \leq m < k$.
We now explain the utility of these operators.

Set
\[
E = e^{- \sum_i |w_i|^2 / t_i - |\sum_i w_i|^2/t_k }.
\]
Then
\begin{align*}
\sigma(t)E 
&= -\left( 
\frac{1}{t_1 + \cdots + t_k} \sum_{j =1}^{k-1}\left( \Bar{w_j} + \frac{t_j}{t_k}\sum_{i =1}^{k-1} \Bar{w_i} \right) 
\right) 
E \\
&= -\left(\frac{1}{t_1 + \cdots + t_k} \left(1 + \frac{\sum_{j=1}^{k-1} t_j}{t_k}\right)\left(\sum_{j =1}^{k-1} \Bar{w_j}\right) 
\right) 
E \\
&= -\frac{1}{t_k}\left(\sum_{j =1}^{k-1} \Bar{w_j}\right) E,
\end{align*}
and so we find
\[
D_m(t) E = - \frac{\Bar{w_m}}{t_m} E
\]
for any $m$.
In consequence, for example,
\begin{align*}
D_1 (t) &
\left(
\phi(w,\Bar{w}) 
\prod_{i=2}^{k-1} \frac{\Bar{w}_i}{t_i}
\left(\sum_{j =1}^{k-1} \frac{\Bar{w}_j}{t_k} \right) 
e^{- \sum_i |w_i|^2 / t_i - |\sum_i w_i|^2/t_k }
\right) \\
 & = 
 \left( 
 - \phi(w,\Bar{w}) \frac{\Bar{w}_1}{t_1} + (D_1(t) \phi)(w, \Bar{w}) 
 \right) 
\prod_{i=2}^{k-1} \frac{\Bar{w}_i}{t_i}
\left(\sum_{j =1}^{k-1} \frac{\Bar{w}_j}{t_k} \right) e^{- \sum_i |w_i|^2 / t_i - |\sum_i w_i|^2/t_k }.
\end{align*}
Note that the left hand side is a total derivative and hence integrates over $w \in \CC^k$ to zero.
The first summand on the right hand side is our integrand from the integral (\ref{integral before}), 
up to a sign and the factor $(t_1\cdots t_{k-1})^{-1}$.
Hence, we find that the integral (\ref{integral before}) equals
\[
\int_{w \in \CC^{k}} 
\int_{(t_1,\ldots,t_k) \in [\epsilon, L]^k}
\frac{1}{t_1 \cdots t_k}
(D_1(t) \phi)(w, \Bar{w}) 
\prod_{i=2}^{k-1} \frac{\Bar{w}_i}{t_i}
\left(\sum_{j =1}^{k-1} \frac{\Bar{w}_j}{t_k} \right) e^{- \sum_i |w_i|^2 / t_i - |\sum_i w_i|^2/t_k }.
\]
Analogous arguments apply, of course, for any $D_m$, due to the symmetry of the integrand.

Hence, applying the $D_m(t)$ in order and using a variant of the preceding argument, 
we find that the integral (\ref{integral before}) equals
\[
\int_{w \in \CC^{k}} 
\int_{(t_1,\ldots,t_k) \in [\epsilon, L]^k}
\frac{1}{t_1 \cdots t_k}
(D_{k-1}(t) \cdots D_1(t) \phi)(w, \Bar{w}) 
\left(\sum_{j =1}^{k-1} \frac{\Bar{w}_j}{t_k} \right) e^{- \sum_i |w_i|^2 / t_i - |\sum_i w_i|^2/t_k }.
\]
We apply the same argument with $\sigma(t)$ to show that the integral (\ref{integral before}) equals
\[
\int_{w \in \CC^{k}} 
\int_{(t_1,\ldots,t_k) \in [\epsilon, L]^k}
\frac{1}{t_1 \cdots t_k}
(\sigma(t)D_{k-1}(t) \cdots D_1(t) \phi)(w, \Bar{w}) 
e^{- \sum_i |w_i|^2 / t_i - |\sum_i w_i|^2/t_k }.
\]
This integral depends on $\epsilon$ through both the domain of integration
and the dependence of the operators $D_m(t)$ and $\sigma(t)$ on $t$.
We first eliminate the second kind of dependence.

Observe that for any choice of allowed $t$, we have 
\[
|\sigma(t) f| \leq \sum_{j =1}^{k-1} |\partial_{w_j} f|,
\]
since $t_j/\sum_i t_i < 1$ for every $j$.
Hence, we may replace $\sigma(t)D_{k-1}(t) \cdots D_1(t) \phi$ in the integrand
by a compactly supported function $\psi(w,\Bar{w})$.
That is, to show convergence of integral (\ref{integral before}) as $\epsilon \to 0$,
it suffices to show convergence of
\[
\int_{w \in \CC^{k}} 
\int_{(t_1,\ldots,t_k) \in [\epsilon, L]^k}
\frac{1}{t_1 \cdots t_k}
\psi(w, \Bar{w})e^{- \sum_i |w_i|^2 / t_i - |\sum_i w_i|^2/t_k }
\]
for any compactly supported~$\psi(w,\Bar{w})$.
We may suppose that $\psi$ factors as $f(w_1,\ldots,w_{k-1})g(w_k)$ 
and focus only on integrating over the variables $w_1,\ldots,w_{k-1}$.

In this integral, there is no problem with integrating over the $w$ variables,
since the integrand is compactly supported in $w$. 
The possible problems arise from the factor $(t_1 \cdots t_k)^{-1}$,
which is not integrable over the domain $[0,L]^k$.
We need to show that the integral over $w$ contributes positive powers of the $t_i$
so that the integral over $t$ has an $\epsilon \to 0$ limit.

Note that, due to our arguments above, 
integration by parts allows us to trade a power of $\Bar{w}_j$ for a $1/t_j$.
Hence if we give a partial Taylor expansion of $\psi$ around the origin,
the integrals against nonconstant terms (which possess powers of $\Bar{w}_j$) 
are more convergent than the constant term of $\psi$.
In other words, it suffices to show that there exists the $\epsilon \to 0$ limit of
\begin{equation}
\label{core comp}
\int_{(w_1,\ldots,w_{k-1}) \in \CC^{k-1}} 
\int_{(t_1,\ldots,t_k) \in [\epsilon, L]^k}
\frac{1}{t_1 \cdots t_k}
e^{- \sum_i |w_i|^2 / t_i - |\sum_i w_i|^2/t_k }.
\end{equation}
Performing a Gaussian integral on the variables $w_1,\ldots,w_{k-1}$, 
we see that expression (\ref{core comp}) is proportional to 
\ben
\int_{(t_1,\ldots,t_k) \in [\epsilon,L]^k}
\left(\sum_{i=1}^k t_i\right)^{-1} \leq C \cdot \prod_{i=1}^k \int_{t_i=\epsilon}^L \frac{1}{t_i^{1/k}}  
= C' \prod_{i=1}^k (L^{(k-1)/k} - \epsilon^{(k-1)/k}),
\een
with $C$ and $C'$ constants.
For $k \geq 2$ the right hand side is finite as $\epsilon \to 0$. 
\end{proof}

Thanks to this proposition we have a well-defined equivariant prequantization.

\begin{dfn}
For $L > 0$, let
\ben
I^{\rm W}[L] := \lim_{\epsilon \to 0} W(P_{\epsilon < L}, I^{\rm W}) 
= \lim_{\epsilon \to 0} \sum_{\Gamma } \frac{\hbar^{g(\Gamma)}}{|{\rm Aut}(\Gamma)|} W_\Gamma(P_\epsilon^L, I^{\rm W}) . 
\een 
Here the sum is over all isomorphism classes of stabled connected graphs, but only graphs of genus $\leq 1$ contribute nontrivially. 
By construction, the collection satisfies the RG flow equation and its tree-level $L \to 0$ limit is manifestly $I^{\rm W}$.
Hence $\{I^{\rm W}[L]\}_{L \in (0,\infty)}$ is a \emph{$\Vect$-equivariant prequantization} of the $\Vect$-equivariant classical formal $\beta\gamma$ system.
\end{dfn}

Organizing the sums by genus of the graphs, we write the interaction as a sum $I^{\rm W}[L] = I^{{\rm W},0}[L] + \hbar I^{{\rm W},1}[L]$ where 
\bestar
I^{{\rm W},0}[L] & = & \sum_{\Gamma \in \; {\rm Trees}} \frac{1}{|{\rm Aut}(\Gamma)|} W_{\Gamma}(P_{\epsilon < L}, I^{\rm W}),\\
I^{{\rm W},1}[L] & = & \sum_{\Gamma \in \; {\rm 1-loop}} \frac{1}{|{\rm Aut}(\Gamma)|} W_{\Gamma}(P_{\epsilon < L}, I^{\rm W}).
\eestar
We now turn to studying the obstruction to satisfying the equivariant quantum master equation. 

\subsection{The obstruction}

With the pre-theory in hand, we ask whether it satisfies the QME.
The main result of this subsection provides a direct link between the topology of manifolds and the analysis of Feynman diagrams, where a characteristic class yields a local functional via the map~$J^{\rm W}$.

\begin{prop}\label{obsprop} 
There is an obstruction to a $\Vect$-equivariant quantization of the formal $\beta\gamma$ system 
that preserves the $\CC^\times \times \Aff(\CC)$ action by scaling and affine transformations. 
It is represented by a non-trivial cocycle of degree one
\ben
\Theta \in (\Def^{\rm W}_n)^{\CC^\times \times {\rm Aff}(\CC)} 
\een
such that 
\[
\Theta  = a J^{\rm W}(\ch^{\GF}_2(\hT_n))
\]
for some non-zero number $a$, where $J^{\rm W}$ is the quasi-isomorphism of Proposition \ref{eqdef}
and $\ch^{\GF}_2(\hT_n)$ is the component of the Gelfand-Fuks Chern character living in~$\clie^2(\Vect; \hOmega_{n,cl}^2).$ 
\end{prop}

This claim will follow from the series of definitions and lemmas that follows below.

By definition the scale $L$ {\em obstruction cocycle} $\Theta[L]$ is 
the failure for the interaction $I^{\rm W}[L]$ to satisfy the scale $L$ equivariant quantum master equation. 
Explicitly, one has
\ben
\hbar \Theta[L] = (\d_{\Vect} + Q)I^{\rm W}[L] + \hbar \Delta_L I^{\rm W}[L] + \{I^{\rm W}[L], I^{\rm W}[L]\}_L,
\een
where the right hand side is divisible by $\hbar$ since $I^{{\rm W},0}$ satisfies the classical master equation
so that the $\hbar^0$ component vanishes.
Moreover, the right hand side has no components weighted by $\hbar^2$ or higher powers,
because the BV Laplacian $\Delta_L$ vanishes on $I^{{\rm W},1}[L]$ as it is only a function of $\gamma$ and a vector field $X$.
Thus, we have
\ben
\hbar \Theta[L] = (\d_{\Vect} + Q)I^{{\rm W},1}[L] + \hbar \Delta_L I^{{\rm W},0}[L] + 2\{I^{{\rm W},0}[L], I^{{\rm W},1}[L]\}_L,
\een
and so $\Theta[L]$ only depends on $\gamma$ and hence is a degree one
element of $\clie^*(\Vect ; \clies(\fg_n^\CC))$. 

\begin{lemma}[Corollary 16.0.5 of \cite{WG2}]\label{obslemma}
The limit $\Theta := \lim_{L \to 0} \Theta[L]$ exists and 
is an element of degree one in $\clie^*(\Vect,\Cloc^*(\fg_n^\CC))$. 
Moreover, it is given by
\ben
\lim_{\epsilon \to 0} \sum_{\substack{\Gamma \in \text{\rm 2-vertex wheels}\\ e \in {\rm Edge}(\Gamma)}} W_{\Gamma,e}(P_{\epsilon<1}, K_\epsilon,
I^{\rm W}[\epsilon]),
\een
where the sum is over wheels $\Gamma$ with two vertices and a distinguished inner edge $e$.
\end{lemma}

In the lemma above, the notation $W_{\Gamma, e}(P_{\epsilon < 1},K_\epsilon, I^{\rm W}[\epsilon])$ 
denotes a variation on the usual weight associated to a graph. 
As usual, we attach the interaction term $I^{\rm W}[\epsilon]$ to each vertex. 
To the distinguished internal edge labeled $e$, we attach the heat kernel $K_\epsilon$, 
but we attach the propagator $P_{\epsilon < 1}$ to every other internal edge. 

We now turn to the proof of Proposition \ref{obsprop}. Let us be clear on what we need to accomplish, as the computations are lengthy and explicit. We must construct the obstruction cocycle $\Theta$ by the techniques of perturbative field theory. In the end, we want to recognize it as the local functional $J^{\rm W}(\ch^{\GF}_2(\hT_n))$. We can describe that local functional already, thanks to our description of $J^{\rm W}$.

\begin{lemma}
\label{lem: form of J(ch2)}
Let $X = a^i \partial_i$ and $Y = b^j \partial_j$ be formal vector fields in $\Vect$ where the coefficients $a^i, b^j$ live in $\hO_n$.
For simplicity, suppose all the $a^i$ are homogeneous of degree $k$ and the $b^j$ are homogeneous of degree $l$.
Then 
\[
J^{\rm W}(\ch^{\GF}_2(\hT_n))(X,Y,\gamma) = \int_S \<(\partial_j a^i)^S(\gamma), \partial\left( (\partial_i
b^j)^S(\gamma)\right)\>_{\fgn},
\]
with surface $S = \CC$ and using the notation $f^S$ from Section~\ref{defining J}.
\end{lemma}

In particular, to focus on the analytic component, suppose $n =1$ so $\gamma \in \Omega^{0,*}(\CC)$ as the target is one-dimensional.
Moreover, we can restrict to $a(t) = t^k$ and $b(t) = t^l$.
Then
\begin{align}
J^{\rm W}(\ch^{\GF}_2(\hT_n))(t^k \partial_t,t^l \partial_t,\gamma) &= \int_\CC k \gamma^{\wedge k-1} \wedge \partial_z( l \gamma^{\wedge l-1}) \d z \\
&= k l (l-1) \int_\CC \gamma^{\wedge k+l-2} \wedge \partial_z( \gamma) \d z. \label{J when n=1}
\end{align}
This expression will appear as the analytic component of our Feynman diagrams.

\begin{proof}
We first observe that
\[
J^{\rm W}_{\omega}(X,Y,\gamma) = J_{\omega(X,Y)}(\gamma)
\]
since the map $J$ is $\Vect$-equivariant.
Moreover, since $J_{\d_{dR} \theta} = \Tilde{J}_\theta$, we deduce that
\[
J^{\rm W}_{\d_{dR} \theta}(X,Y,\gamma) = J_{\theta([X,Y])} (\gamma).
\]
Hence it is convenient to recognize that 
\[
\ch^{\GF}_2(\hT_n) = \d_{dR}(\alpha)
\]
where $\alpha \in \clies(\Vect,\hOmega^1_n)$ satisfies
\[
\alpha(X,Y) = \alpha(a^i \partial_i, b^j \partial_j) = - (\partial_j a^i) \d_{dR}(\partial_i b^j).
\]
Note that if the $a^i$ are homogeneous of degree $k$ and the $b^j$ are homogeneous of degree $l$,
then $\alpha(X,Y)$ is a one-form whose coefficients are homogeneous of degree~$k+l-3$.

Lemma \ref{easy} then implies
\ben
\Tilde{J}_{\alpha(X,Y)} (\gamma_1,\ldots,\gamma_{k-1}, \gamma'_1,\ldots,\gamma_{l-1}') 
= \int_S \<(\partial_j a^i)^S(\gamma_1,\ldots, \gamma_{k-1}), \partial\left( (\partial_i
b^j)^S(\gamma_1',\ldots,\gamma_{l-1}')\right)\>_{\fgn},
\een
where $S = \CC$ for us. 
(Here we are describing the local functional as a tensor with $k+l-2$ inputs to be maximally explicit.)
\end{proof}

Now we turn to producing a simple, explicit expression for the obstruction.
The limit in Lemma \ref{obslemma} can be moved inside the summation, 
i.e., the weight for each 2-vertex wheel $\Gamma$ with edge $e$ has an $\epsilon \to 0$ limit.
We denote this summand by
\ben
\Theta_{\Gamma,e} = \lim_{\epsilon \to 0} W_{\Gamma,e}(P_\epsilon^1,K_\epsilon,I^{\rm W}[\epsilon]) .
\een
By the nature of the graph, this functional is of the form
\ben
\Theta_{\Gamma,e} : {\rm W}_n^{\tensor 2} \tensor \Sym(\Omega^{0,*}_c
\tensor \fg_n [1] ) \to \CC .
\een
Given two formal vector fields $X,Y$, let $\Theta_{\Gamma,e}(X,Y)$ denote the associated local functional in~$\cloc^*(\fg_n^S)$. 

Due to linear dependence on the vector fields, it suffices to assume that $X,Y$ are of the form $X = a^i \partial_i$ and $Y = b^j \partial_j$ where the coefficients $a^i,b^i \in \hO_n$ are homogeneous of degrees $k$ and $l$, respectively. In this case, there is only one graph $\Gamma$ whose functional $\Theta_{\Gamma,e}(X,Y)$ is nonzero: this graph has a vertex of valency $k+1$ and a vertex of valency $l+1$, namely 
\begin{center}
\begin{tikzpicture}[decoration={markings,mark=at position 1.7cm with {\arrow[black,line width=.4mm]{stealth}}}];

\filldraw (-1.5,0) circle (.1);
\draw (-1.5, .5) node {$I_X^{\rm W}$};
\draw[postaction=decorate, line width=.2mm] (-3,0.5) -- (-1.5,0);
\draw (-3.3,0.6) node {$\gamma$};
\draw (-2.8,0.1) node {$\vdots$};
\draw[postaction=decorate, line width=.2mm] (-3,-0.5) -- (-1.5,0);
\draw (-3.3,-0.65) node {$\gamma$};

\filldraw (1.5,0) circle (.1);
\draw (1.5, .5) node {$I_Y^{\rm W}$};
\draw[postaction=decorate, line width=.2mm] (3,0.5) -- (1.5,0);
\draw (3.3,0.6) node {$\gamma$};
\draw (2.8,0.1) node {$\vdots$};
\draw[postaction=decorate, line width=.2mm] (3,-0.5) -- (1.5,0);
\draw (3.3,-0.65) node {$\gamma$};

\draw[postaction=decorate, line width=.2mm] (-1.5,0) .. controls (0,.75) .. (1.5,0);
\draw (0, 1) node {$P_{\epsilon<1}$};
\draw[postaction=decorate, line width=.2mm] (1.5,0) .. controls (0,-.75) .. (-1.5,0);
\draw (0, -1) node {$K_\epsilon$};
\end{tikzpicture}
\end{center}
For this graph, the functional $\Theta_{\Gamma,e}(X,Y)$ is homogeneous of degree $k+l-2$:
\ben
\Theta_{\Gamma,e}(X,Y) : \Sym^{k+l -2} (\Omega^{0,*}_c(\CC) \tensor \fg_n [1]) \to \CC .
\een
By describing this functional explicitly, 
we will complete the proof of Proposition \ref{obsprop},
as it will agree on the nose with~$J^{\rm W}(\ch_2^{\GF}(\hT_n))$.

\begin{prop}
Let $X = a^i \partial_i$ be homogeneous of degree $k$
and $Y = b^j \partial_j$ homogeneous of degree $l$. Let $\Gamma$ be the two-vertex wheel with vertices of valencies $k+1$ and $l+1$ and mark one internal edge as distinguished. Then, we have an identification $\Theta_{\Gamma,e}(X,Y) = a J^{\rm W}(\ch_2^\GF (\hT_n))(X,Y)$ for some nonzero number $a$. 
\end{prop}

\[
\Theta_{\Gamma,e}(X,Y)(\gamma) = \int_\CC \<(\partial_j a^i)^S(\gamma), \partial\left( (\partial_i
b^j)^S(\gamma)\right)\>_{\fgn}.
\]
 
\begin{proof}
We simplify further by setting
\[
X = t_1^{k_1} \cdots t_n^{k_n} \partial_i \quad\text{and}\quad Y = t_1^{l_1} \cdots t_n^{l_n} \partial_j
\]
with $k =\sum k_m$ and $l = \sum l_m$.
Ignoring the analytic factors momentarily, 
we observe that in computing the weight of the graph $\Gamma$,
we contract $\beta$ legs with $\gamma$ legs.
In our case, the $X$-vertex contributes a $\beta_i$ leg,
which then contracts with the $k_i$ different $\gamma$ legs from the $Y$-vertex.
Likewise, the $Y$-vertex contributes a $\beta_j$ leg,
which then contracts with the $k_j$ different $\gamma$ legs from the $X$-vertex.
These contractions explain the terms $(\partial_j a^i)^S(\gamma)$ and $(\partial_i b^j)^S(\gamma)$ in the integrand.

We now turn to comparing the analytic factors. It suffices here to consider the situation $n =1$,
since we have already taken care of the dependence on the target coordinates.
To clarify the notation, we use $f_1, \ldots, f_{k-1}$ to label the inputs to the remaining legs of the $X$-vertex.
We use $g_1,\ldots, g_{l-1}$ to label the inputs to the remaining legs of the $Y$-vertex.

The following diagram encodes the weight that we must compute:
\begin{center}
\begin{tikzpicture}[decoration={markings,mark=at position 1.7cm with {\arrow[black,line width=.4mm]{stealth}}}];

\filldraw (-1.5,0) circle (.1);
\draw[postaction=decorate, line width=.2mm] (-3,0.75) -- (-1.5,0);
\draw (-3.45,0.9) node {$f_1\,\d \zbar$};
\draw[postaction=decorate, line width=.2mm] (-3,0.35) -- (-1.5,0);
\draw (-3.3,0.4) node {$f_2$};
\draw (-2.8,0) node {$\vdots$};
\draw[postaction=decorate, line width=.2mm] (-3,-0.7) -- (-1.5,0);
\draw (-3.3,-0.9) node {$f_{k-1}$};

\filldraw (1.5,0) circle (.1);
\draw[postaction=decorate, line width=.2mm] (3,0.65) -- (1.5,0);
\draw (3.3,0.7) node {$g_1$};
\draw (2.8,0.1) node {$\vdots$};
\draw[postaction=decorate, line width=.2mm] (3,-0.6) -- (1.5,0);
\draw (3.4,-0.7) node {$g_{l-1}$};

\draw[postaction=decorate, line width=.2mm] (-1.5,0) .. controls (0,.75) .. (1.5,0);
\draw (0, 1) node {$P_{\epsilon<1}$};
\draw[postaction=decorate, line width=.2mm] (1.5,0) .. controls (0,-.75) .. (-1.5,0);
\draw (0, -1) node {$K_\epsilon$};
\end{tikzpicture}
\end{center}
We wish to take the $\epsilon \to 0$ limit of the associated integral.
Thus, we have
\begin{align*}
\Theta_{\Gamma,e}(X,Y) (f_1 \d \zbar, f_2,\cdots,g_{l-1}) 
& = \lim_{\epsilon \to 0} \int_{\CC^2} 
\left(\prod_{i=1}^{k-1} f_i(z_1) \right) 
\left(\prod_{j=1}^{l-1} g_i(z_2)\right) \d \zbar_1 
\wedge K^{an}_\epsilon(z_1,z_2) \wedge P^{an}_{\epsilon < 1} (z_1,z_2) \\
& =\lim_{\epsilon \to 0} \int_{\CC^2}
  \left(\prod_{i=1}^{k-1} f_i (z_1) \right) 
  \left(\prod_{j=1}^{l-1} g_i(z_2) \right) 
  \int_{t = \epsilon}^L \frac{1}{(4 \pi)^2 \epsilon t}
  e^{-|z_1-z_2|^2/4 \epsilon}  \frac{\partial}{\partial z_1}
  e^{-|z_1-z_2|^2/ 4 t} \, \d t.
\end{align*}
Now, $\partial_{z_1} e^{-|z_1-z_2|^2/4 t} = - \frac{1}{4t} (\zbar_1 - \zbar_2)
e^{-|z_1-z_2|^2/t}$. We make the change of coordinates $w_1 = z_2 -
z_1$ and $w_2 = z_2$. The integral over $w_1,w_2$ can be written as
\be\label{integral1}
- \int_{w_1,w_2\in \CC} \left(\prod_{i=1}^{k-1} f_i \right) \d^2 w_1 \d^2 w_2 \left(\prod_{j=1}^{l-1}
  g_i \right) \Bar{w}_1 \frac{1}{4 (4 \pi)^2 \epsilon t^2} \exp\left(-\frac{1}{4}(t^{-1} +
\epsilon^{-1} ) |w_1|^2\right) .
\ee
Using the same trick as in the proof that the theory involves no
counterterms, we introduce the differential operator
\ben
D(t) = \left(1 - \frac{t}{t+ \epsilon} \right) 
\frac{\partial}{\partial
  w_1} = \frac{\epsilon}{t+\epsilon} \frac{\partial}{\partial w_1} . 
\een 
Then
\begin{align*}
D_1(t) & \left(\prod_{i=1}^{k-1} f_i \prod_{j=1}^k
  g_i \frac{1}{\epsilon t} \exp\left(-\frac{1}{4}(t^{-1} +
\epsilon^{-1} ) |w_1|^2\right) \right) \\ & = \left(-
                                            \frac{\Bar{w}_1}{t} \prod_{i=1}^{k-1} f_i \prod_{j=1}^{l-1}
  g_i \Bar{w}_1 + D_1(t) \left( \prod_{i=1}^{k-1} f_i \prod_{j=1}^{l-1}
  g_i \right) \right) \frac{1}{4\epsilon t} \exp\left(-(t^{-1} +
\epsilon^{-1} ) |w_1|^2\right) 
\end{align*}
The left hand side is a total derivative, hence the integal in (\ref{integral1}) can
be written as 
\ben
- \int_{w_1,w_2} \frac{\partial}{\partial w_1} \left(\prod f_i \prod 
  g_i\right) \frac{1}{4 (4 \pi)^2 t(\epsilon + t)} \exp\left(-\frac{1}{4}(t^{-1} +
\epsilon^{-1} ) |w_1|^2\right) .
\een
In the $\epsilon \to 0$ limit only the the first term in the Wick
expansion for integrating $w_1$ will be nonzero. This term is
\ben
\frac{1}{(4 \pi)^2}\int_{w_2} \d^2 w_2 \frac{\partial}{\partial w_1}
\left(\prod f_i \prod 
  g_i\right)(w_1 = 0) \frac{\epsilon}{(t + \epsilon)^2} .
\een
Note that the condition $w_1 = 0$ implies that $z_1 = z_2$ in our
original parametrization. Thus 
\ben
\frac{\partial}{\partial w_1} \left(\prod f_i \prod
  g_i\right)(w_1 = 0) = \left(\frac{\partial}{\partial z}  \prod
  f_i(z) \right) \prod g_j (z)
\een
where $z = z_1 = z_2$. Finally, we compute the $\epsilon \to 0$
limit of the $t$-integral
\ben
\frac{1}{(4 \pi)^2}\lim_{\epsilon \to 0} \int_{\epsilon}^1 \frac{\epsilon}{(t + \epsilon)^2} \d t = \frac{1}{2}  .
\een 
Integrating by parts (to get rid of the $(-)$ sign) we see that the total weight is
\ben
\frac{1}{2 (4 \pi)^2}\int_{z \in \CC} \left( \prod 
  f_i\right) \frac{\partial}{\partial z} \left( \prod g_j \right) \d^2 z
\een
as desired. Setting $f_i = g_j$ we see that this coincides with the analytic part of $J^{\rm W}(\ch_2(\hT_n))(X,Y,f_i = g_j)$ written above in (\ref{J when n=1}). 
\end{proof}

\begin{rmk} 
Note that when restricted to {\em linear} vector fields $\gl_n \hookrightarrow \Vect$, 
the entire obstruction $\Theta$ vanishes. 
This vanishing means that there is no obstruction to quantizing equivariantly for the Lie algebra $\gl_n$. 
This result is just the Lie algebra-level version of an earlier observation: 
the action of the group $\GL_n$ lifts $\hbar$-linearly to an action on the quantization.
\end{rmk}

\subsection{The extended theory} 
\label{extendedtheory}

We have just seen that there is an obstruction to the existence of a $\Vect$-equivariant quantization of the formal $\beta\gamma$-system. 
As is common in physics, we use that obstruction to extend the Lie algebra and obtain an equivariant quantization for the extended Lie algebra.
Indeed, we have already seen that the second Gelfand-Fuks-Chern character defines the extension
\ben
\xymatrix{
0 \ar[r] & \hOmega^2_{n,cl} \ar[r] & \TVect \ar[r]^-{p} & \Vect \ar[r] & 0 
}
\een 
in Section \ref{sec gk descent}. 
We will now construct a classical theory that is equivariant for $\TVect$ and 
show that it admits a natural equivariant BV quantization.

\subsubsection{}

The action of $\Vect$ on the classical formal $\beta\gamma$ system
is given by a map of $L_\infty$ algebras $I^{\rm W} : \Vect \l8to \Cloc^*(\DD \fg_n^\CC)[-1]$. 
By composing with the projection $p : \TVect \to \Vect$,
we get an $L_\infty$ map
\ben
\Tilde{I}^{W} := p^* I^{\rm W} : \TVect \l8to \Cloc^*(\DD \fg_n^\CC)[-1] .
\een
Equivalently, $p^*I^{\rm W}$ determines a Maurer-Cartan element in the dg Lie algebra $\clie^*(\TVect ; \Cloc^*(\DD \fg_n^\CC))$
and hence a $\TVect$-equivariant classical field theory. 

As in the non-extended case, there is a $\TVect$-equivariant obstruction-deformation complex $\TDef_n^{\rm W}$, 
which is the graded vector space $\hsym(\TVect^\vee[-1]) \tensor \Cloc^\sharp(\DD \fg_n^S)$
equipped with the differential $\d_{\TVect} + \dbar + \{\Tilde{I}^{W}, -\}$,
where $\d_{\TVect}$ denotes the differential on $\clie^*(\TVect)$. Note that we can write
\ben
\TDef_n^{\rm W} \cong \clie^*(\TVect) \otimes_{\clie^*(\Vect)}
\Def_n^{\rm W} .
\een
Proposition \ref{eqdef}, which concerns the unextended deformation complex, 
then implies that the $\CC^\times \times \Aff(\CC)$-invariant piece of the extended deformation complex satisfies
\ben
\left(\TDef_n^{\rm W}\right)^{\CC^\times \times \Aff(\CC)} \simeq \clie^*(\TVect;p^* \hOmega^2_{n,cl}[1]).
\een
Here, $p^* \hOmega_{n,cl}^2$ is the $\TVect$-module given by pulling back the natural $\Vect$-module structure on closed two-forms along~$p$. 

\subsubsection{The extended pre-theory}

Our goal is to describe quantizations for this extended $\TVect$-equivariant field theory. 
Let $\{I^{\rm W}[L]\}$ be the prequantization for the $\Vect$-equivariant classical field theory, as above. 
For each $L > 0$, we define the functional
\ben
\Tilde{I}^{W} [L] := p^* I^{\rm W} [L] \in \clie^\sharp(\TVect) \otimes \clie^\sharp(\DD\fgn^\CC) \llbracket \hbar \rrbracket .
\een

\begin{lemma} 
The collection $\{\Tilde{I}^{W}[L]\}$ defines a pre-quantization for
the $\TVect$-equivariant classical field theory. Moreover, the obstruction to satisfying the $\TVect$-equivariant QME at scale $L$ is $\Tilde{\Theta}[L] = p^* \Theta [L]$.
In particular $\Tilde{\Theta} := \lim_{L \to 0} \Tilde{\Theta}[L]$ exists and is equal to $p^* \Theta$. 
\end{lemma}
\begin{proof} This follows from the fact that for any graph $\Gamma$
  we have $W_\Gamma(P_{\epsilon < L}, p^*I^{\rm W}) =
  p^*W_{\Gamma}(P_{\epsilon < L}, I^{\rm W})$. 
\end{proof}

Just as in the non-extended case there is the possibility that the
pre-quantization does not define an equivariant quantization. The
above lemma identifies this obstruction cocycle which we will go on to
show is cohomologically trivial. 

The quasi-isomorphism $J^{\rm W} : \clie^*(\Vect ; \hOmega^2_{n,cl} [1]) \to
(\Def_n^{\rm W})^{\CC^{\times} \times {\rm Aff}(\CC)}$ from
Proposition \ref{eqdef} extends to a quasi-isomorphism
\begin{equation}\label{tildeJ}
J^{\Tilde{\rm W}} : \clie^*(\TVect ; \hOmega^2_{n,cl} [1]) \xto{\simeq} (\Tilde{\Def}_n^{\rm
  W})^{\CC^{\times} \times {\rm Aff}(\CC)}
\end{equation}
by tensoring $\clie^*(\TVect)$ over the ring $\clie^*(\Vect)$. The
lemma implies that the obstruction $\Tilde{\Theta}$ is identified with
the cocycle $p^*({\rm ch}^\GF_2(\hT_n))$ under the map $J^{\Tilde{\rm W}}$. 

\subsubsection{Quantum correction}

Let $\fh$ be a Lie algebra and $V$ a module for $\fh$. Moreover,
suppose $\alpha \in \clie^2(\fh ; V)$ is a 2-cocycle. Then, we can
form the extension
\ben
0 \to V \to \Tilde{\fh} \xto{p} \fh \to 0 .
\een
The bracket between $x,y \in \fh$ is defined by $[x,y]_{\Tilde{\fh}}
:= [x,y]_{\fh} + \alpha(x,y)$ where $[-,-]_{\fh}$ is the bracket in
the original Lie algebra. The bracket between $x \in \fh$ and $v \in
V$ is $[x,v]_{\Tilde{\fh}} = x \cdot v$. We can
pull back the cocycle $p^* \alpha \in \clie^*(\Tilde{\fh} ; V)$. In
this situation, this pullback cocycle is automatically trivial. An
explicit trivializing element is $\id_V : V \to V$ viewed as an
element of the Chevalley-Eilenberg complex $\clie^*(\Tilde{\fh} ;
V)$. 

In our situation this says that the cocycle $p^*(\ch_2^{\rm
  GF}(\hT_n))$ is trivializable and hence so is the obstruction
$\{\Tilde{\Theta}[L]\}$. To define a quantum theory we need this
trivialization at the level of functionals on fields. Indeed,
according to the above Lemma, which uses standard facts about Feynman
diagrams, it suffices to trivialize the local functional $\Tilde{\Theta}$
encoding the obstruction. 

\begin{lemma}[Lemma 3.33 of \cite{LiLi}]
\label{genlem}
If $I^{qc}$ and $O_1 \in \Def_n$ satisfy
\ben
Q I^{qc} + \{I, I^{qc}\} = O_1,
\een 
then, for each $L$, the functional
\ben
I^{qc} [L] = \lim_{\epsilon \to 0} \sum_{\substack{\Gamma \in
    \text{\rm Trees}\\ v \in V(\Gamma)}} W_{\Gamma, v}(P_{\epsilon <
  L}, I, I^{qc})
\een
satisfies 
\be\label{treetriv}
Q I^{qc} [L] + \{I^{(0)}[L], I^{qc}[L]\}_L = O_1[L] .
\ee
\end{lemma}
\begin{proof} 
For the non-equivariant case, see the referenced Lemma in \cite{LiLi}. The equivariant case is an immediate consequence.
\end{proof}

As a corollary of this general fact we see that if $I^{qc} \in
\Tilde{\Def}_n^{\rm W}$ trivializes the obstruction cocycle $\Tilde{\Theta}$,
then the effective family $I[L] + \hbar I^{qc}[L]$ satisfies the
$\TVect$-equivariant quantum master equation. In fact, we have an
obvious choice for the local functional
$I^{qc}$. The map $J : \hOmega^2_{n,cl} \to \Def_n$ determines an
element in $\clie^1(\hOmega^2_{n,cl} ; \Def_n) \subset \clie^1(\TVect
; \Def_n)$ and hence an element the equivariant deformation complex $\Tilde{\Def}_n^{\rm
  W}$. We will use $I^{qc} = J$. 

\begin{prop}\label{obstriv} 
The local functional $J$ trivializes $\Tilde{\Theta}$ in the equivariant deformation complex. That is, 
\be\label{classtriv}
(\dbar + \d_{\TVect}) J + \{\Tilde{I}^{\rm W}, J\} =
\Tilde{\Theta} .
\ee
\end{prop}
\begin{proof} 
The functional $J$ is the image of $\id_{\Omega^2}$ under
the map $J^{\Tilde{\rm W}}$ from Equation (\ref{tildeJ}). By construction $J^{\Tilde{W}}$ determines a map of complexes $\clie^*(\TVect ;
\hOmega^2_{n,cl}) \to \Tilde{\Def}_n^{\rm W}$ and hence commutes with
the differentials on both sides. That is,
\ben
J^{\Tilde{\rm W}}(\d_{\TVect} \varphi) = \dbar J^{\Tilde{\rm W}}
(\varphi) + \{\Tilde{I}^{\rm W},J^{\Tilde{\rm W}}(\varphi)\} 
\een
for all $\varphi \in \clie^*(\TVect ; \hOmega^2_{n,cl})$. In
particular, for $\varphi = \id_{\Omega^2}$ we have
\ben
J^{\Tilde{\rm W}}(p^*\ch_2^{\rm GF}(\hT_n)) = \dbar J +
\{\Tilde{I}^{\rm W}, J\} .
\een
We have already seen that the image of $p^*\ch_2^{\rm GF}(\hT_n)$
under $J^{\Tilde{\rm W}}$ is the obstruction cocycle $\Tilde{\Theta}$,
and this is what we wanted to show. 
\end{proof}

Finally, we arrive at the main result concerning the extended
equivariant BV theory. 

\begin{thm} \label{QME1}
The effective family $\{\Tilde{I}^{\rm W}[L] + \hbar J[L] \}_{L > 0}$ satisfies both RG flow {\em and} the $\TVect$- equivariant quantum master equation
\ben
(\d_{\Tilde{{\rm W}}} + Q) (\Tilde{I}^{\rm W} [L] + \hbar J[L]) + \frac{1}{2}
\{\Tilde{I}^{\rm W}[L] + \hbar J[L] , \Tilde{I}^{\rm W}[L] + \hbar
J[L] \}_L + \hbar \Delta_L (\Tilde{I}^{\rm W}[L] + \hbar J[L] ) = 0 .
\een 
Hence it provides a $\TVect$-equivariant quantization of the classical
theory $\Tilde{I}^{\rm W}$ based on a length scale regularization. Moreover,
this quantization is unique up to homotopy.
\end{thm}
\begin{proof} The first part follows from Proposition \ref{obstriv} and Lemma
  \ref{genlem} above. Uniqueness follows from the fact
  that $H^1(\TVect ; \hOmega^2_{n,cl}) = 0$.
\end{proof}

\begin{rmk}
Here, as in Lemma \ref{genlem}, the term $I^{qc}[L]$ arises naturally by naively applying RG flow to $\Tilde{I}^{\rm W} + \hbar I^{qc}$ and 
asking only for the sum of the terms in which at least one vertex is labeled by $I^{qc}$.
(The terms involving just $\Tilde{I}^{\rm W}$ have singularities, but we've already resolved them.)
Note that a stable connected graph containing $I^{qc}$ as a vertex has nonzero weight only if it is a tree, 
because $I^{qc}$ only has inputs from $\gamma$.
Moreover, only one copy of $I^{qc}$ can appear.
\end{rmk}

\subsection{The conformal anomaly} \label{sec conformal anomaly}

In Section \ref{sec hol vf} we discussed how the classical theory of the formal $\beta\gamma$ system is equivariant for the action of holomorphic vector fields on the source $\cT^S$.
Indeed, we have described the local functional $I^\cT \in \Def_n^{\cT}$ that encodes this action. 
In this section we address the problem of the quantizing this symmetry compatibly with the action of formal vector fields $\Vect$ on the target $n$-disk. 

\begin{prop} There is an obstruction to a $\cT^\CC \times \Vect$-equivariant
  quantization of the formal $\beta\gamma$ system. It is represented
  by a non-trivial cocycle 
\ben
2n \omega^\GF + \Theta^{\rm W} + \Theta^\cT \in \clie^*\left(\Vect ; \Def_n^{\cT}\right).
\een
Here, $\omega^\GF \in \Def_n^\cT$ is the local Gelfand-Fuks cocycle representing the generator of $H^3_{\rm Lie}({\rm W}_1)$. Moreover, 
\ben
J^{\rm W} (\ch_2^\GF(\hT_n)) = a \Theta^{\rm W} \;\; , \;\; K^{\rm W}(c_1^{\GF}(\hT_n)) = b \Theta^{\cT} .
\een
for some constants $a,b$.
\end{prop}

\begin{rmk} 
This proposition says that there are {\em three} independent obstructions to finding a $\cT^\CC \times \Vect$-equivariant quantization of the formal $\beta\gamma$ system. The obstruction $\Theta^{\rm W}$ coincides with the $\Vect$-equivariant obstruction computed in the previous sections and is independent of $\cT^\CC$. The obstruction $\Theta^\cT$ is new, and we will show that it reflects the fact that chiral differential operators on a complex manifold $X$ admit a global conformal structure if and only if $c_1(T_X) = 0$. The obstruction $\omega^{\GF}$ only depends on the background fields $\cT^\CC$ and hence is independent of the fields of the $\beta\gamma$ system. It reflects that fact that even when $c_1(T_X) = 0$, one needs to centrally extend holomorphic vector fields to get a global action. We will see that this obstruction constitutes the central charge of resulting Virasoro symmetry.
\end{rmk}

\begin{proof}
The obstruction is computed in a manner similar to the obstruction just for $\Vect$. Indeed, a version of Lemma \ref{obslemma} still holds, with the interaction $I^{\rm W}$ replaced by $I^{\rm W} + I^\cT$. That is, the obstruction to a $\cT^\CC \times \Vect$ equivariant quantization can be written as a graph expansion
\ben
\lim_{\epsilon \to 0} \sum_{\substack{\Gamma \in \text{\rm 2-vertex wheels}\\ e \in {\rm Edge}(\Gamma)}} W_{\Gamma,e}(P_{\epsilon<1}, K_\epsilon,
I^{\rm W}[\epsilon] + I^\cT[\epsilon]).
\een
This obstruction is an element of $\clie^*(\Vect ; \Def_n^\cT)$ and splits up into a sum of three linear pieces: 
\begin{enumerate}
\item a factor that does not depend on $\Vect$, i.e., lives in $\Def_n^\cT \subset \clie^*(\Vect ; \Def_n^\cT)$ ; 
\item a factor $\Theta^{\rm W}$ that does not depend on $\cT^\CC$ and is an cocycle in $\clie^*(\Vect ; \Def_n)$; and 
\item a factor $\Theta^{\cT}$ that is linear in both $\cT_S$ and $\Vect$ and is a cocycle in $\clie^1(\Vect ; \Def_n^\cT)$.
\end{enumerate}
We now describe these terms explicitly.\\

\noindent{\em The first term.} 
The term in $\Def_n^\cT$ has the form
\ben
\lim_{\epsilon \to 0} \sum_{\substack{\Gamma \in \text{\rm 2-vertex wheels}\\ e \in {\rm Edge}(\Gamma)}} W_{\Gamma,e}(P_{\epsilon<1}, K_\epsilon, I^\cT[\epsilon]) .
\een
The calculation of this obstruction was performed in Section 7 of \cite{bw_vir} and was shown to be equal to the local functional $2n \omega^\GF \in \Def_n^\cT$ where we have defined $\omega^{\GF}$ in Section \ref{sec hol vf}.\\ 

\noindent{\em The second term.} 
The term independent of $\cT^\CC$ has a graph expansion of the form
\ben
\Theta^{\rm W} = \lim_{\epsilon \to 0} \sum_{\substack{\Gamma \in \text{\rm 2-vertex wheels}\\ e \in {\rm Edge}(\Gamma)}} W_{\Gamma,e}(P_{\epsilon<1}, K_\epsilon,
I^{\rm W}[\epsilon]).
\een
This term is precisely the local functional $\Theta \in \clie^*(\Vect ; \Def_n)$ representing the obstruction to a $\Vect$-equivariant quantization. Thus $J^{\rm W} (\ch^\GF_2(\hT_n)) = \Theta^{\rm W}$, as desired.\\ 

\noindent {\em The third term.} 
We aim to show that there is an identification $K^{\rm W} (c_1^\GF(\hT_n)) = b \Theta^\cT$. Since we only consider the graph expansion over two-vertex wheels, the cocycle representing the third piece of the obstruction $\Theta^\cT$ is given by the weight of the $\epsilon \to 0$ limit of the following diagram
\begin{center}
\begin{tikzpicture}[decoration={markings,mark=at position 1.7cm with {\arrow[black,line width=.4mm]{stealth}}}];

\filldraw (-1.5,0) circle (.1);
\draw (-1.5,.5) node {$I^{\cT}$};
\draw[postaction=decorate, line width=.2mm] (-3,0) -- (-1.5,0);
\draw (-3.15,0) node {$\xi$};

\filldraw (1.5,0) circle (.1);
\draw (1.5,.5) node {$I^{\rm W}$};
\draw[postaction=decorate, line width=.2mm] (3,0.5) -- (1.5,0);
\draw (3.3,0.5) node {$\gamma$};
\draw (2.8,0.1) node {$\vdots$};
\draw[postaction=decorate, line width=.2mm] (3,-0.5) -- (1.5,0);
\draw (3.3,-0.5) node {$\gamma$};

\draw[postaction=decorate, line width=.2mm] (-1.5,0) .. controls (0,.75) .. (1.5,0);
\draw (0, 1) node {$P_{\epsilon<1}$};
\draw[postaction=decorate, line width=.2mm] (1.5,0) .. controls (0,-.75) .. (-1.5,0);
\draw (0, -1) node {$K_\epsilon$};
\end{tikzpicture}
\end{center}
where $\xi$ labels a holomorphic vector field in $\cT^\CC$ and $\gamma \in \Omega^{0,*}(\CC ; \fg_n)$.\\

For fixed $\xi \in \cT^\CC$ we have the functional $K_\xi := K^{\rm W}(c_1^{\GF}(\hT_n))(\xi, -)$, which is an element of the $\Vect$-equivariant deformation complex $\Def_n^{\rm W}$. For simplicity, we consider the case that $\xi = \xi^0 \partial_z \in \Omega^{0}(\CC ; T \CC)$. The Gelfand-Fuks-Chern character evaluated on a vector field $X = a^i \partial_i$ is $c_1^{\GF}(\hT_n)(a^i \partial_i) = \frac{1}{2\pi i} \partial_i a^i$. Thus, we have the explicit formula for $K_\xi$
\ben
K_\xi(X, \gamma) = \int_S \partial_z \xi^0 \<(\partial_i a^i)^S(\gamma), \partial \gamma\>_{\fg_n} .
\een
It suffices to show that for each $X \in \Vect$ the obstruction satisfies $\Theta^{\cT}(\xi, X, -) = b K_\xi(X)$, for some nonzero constant $b$, as elements of $\Def_n$. 

As we did in the calculation of the obstruction in the previous sections, it suffices to assume that the formal vector field is homogeneous of the form $X = t_1^{k_1} \cdots t_n^{k_n} \partial_i$ where $k_1 + \cdots + k_n = k$. Then, both $\Theta^{\cT}(\xi, X, -)$ and $K_\xi(X)$ are of homogeneous degree $k-1$:
\ben
\Sym^{k-1}(\Omega^{0,*}(\CC) \tensor \fg_n) \to \CC .
\een 
Ignoring the analytic factors momentarily, we observe that in computing the weight of the graph $\Gamma$,
we contract $\beta$ legs with $\gamma$ legs. In our case, the $X$-vertex contributes a $\beta_i$ leg,
which then contracts with the $k_i$ different $\gamma$ legs from the vertex labeled by the holomorphic vector field $\xi$. These contractions explain the term $(\partial_i a^i)^S(\gamma)$.

We now compare the analytic factors. Since the dimension of the target formal disk was only relavent for the algebraic piece, it suffices to set $n = 1$. The analytic weight we must compute is represented by the $\epsilon \to 0$ limit of the diagram
\begin{center}
\begin{tikzpicture}[decoration={markings,mark=at position 1.7cm with {\arrow[black,line width=.4mm]{stealth}}}];

\filldraw (-1.5,0) circle (.1);
\draw[postaction=decorate, line width=.2mm] (-3,0) -- (-1.5,0);
\draw (-3.3,0) node {$\xi^0 \partial_z$};

\filldraw (1.5,0) circle (.1);
\draw[line width=.2mm] (3,1) -- (1.5,0);
\draw[postaction=decorate, line width=.2mm] (3,0.5) -- (1.5,0);
\draw (3.3, 1) node {$f_1 \d \zbar$};
\draw (3.3,0.5) node {$f_2 $};
\draw (2.8,0.1) node {$\vdots$};
\draw[postaction=decorate, line width=.2mm] (3,-0.5) -- (1.5,0);
\draw (3.3,-0.5) node {$f_{k-1}$};

\draw[postaction=decorate, line width=.2mm] (-1.5,0) .. controls (0,.75) .. (1.5,0);
\draw (0, 1) node {$P_{\epsilon<1}$};
\draw[postaction=decorate, line width=.2mm] (1.5,0) .. controls (0,-.75) .. (-1.5,0);
\draw (0, -1) node {$K_\epsilon$};
\end{tikzpicture}
\end{center}
The weight of this diagram is given by
\ben
\int_{\CC^2} \left(\xi^0 \partial_{z_1} P_{\epsilon < 1}(z_1,z_2) \right) \wedge \left(\prod_{i=1}^{k-1} f_i(z_2) \right) \d \zbar_2 \wedge K_\epsilon(z_1,z_2) .
\een 
We compute the $z_1$-derivative of the propagator as
\ben
\frac{\partial}{\partial z_1} P_{\epsilon < 1} (z_1,z_2) = \int_{t = \epsilon}^L \frac{1}{16 (4 \pi) t^3} (\zbar_1 - \zbar_2)^2 e^{-|z_1-z_2|^2/4t} \d t \left(\d z_1 - \d z_2\right).
\een
Making the standard change of coordinates $w_1 = z_2 - z_1$ and $w_2 = z_2$ we find that the weight can be expressed as 
\ben
\int_{w_1,w_2} \xi^0 \Bar{w}_1^2 \left(\prod_{i=1}^{k-1} f_i \right) \d^2 w_1 \d^2 w_2 \int_{t=\epsilon}^1 \frac{1}{16 (4 \pi)^2 \epsilon t^3} \exp\left(-\frac{1}{4} (t^{-1} + \epsilon^{-1}) |w_1|^2\right) .
\een 
The only term in the Wick expansion of the integral above that contributes is a nonzero multiple of
\ben
\int_{z} \left(\partial_z^2 \xi^0\right)(z) \left(\prod_{i=1}^{k-1} f_i (z) \right) \d^2 z \int_{t=\epsilon}^1 \frac{\epsilon^3}{(\epsilon + t)^3} \d t .
\een
A simple evaluation of the $t$-integral yields a finite limit as $\epsilon \to 0$. Furthermore, we can integrate the above $z$-integral by parts to put the analytic part of the obstruction $\Theta^{\cT}(\xi, X, f_1 \d \zbar, f_2,\ldots,f_{k-1})$ in the form that is proportional to
\ben
\int_z \left(\partial_z \xi^0\right) \partial_z \left(\prod_{i=1}^{k-1} f_i \right) \d^2 z .
\een
This is precisely the analytic form of the functional $K_\xi(X)$, as desired.
\end{proof}

\section{The partition function of the equivariant theory}

In this section we analyze the scale $\infty$ effective interaction on an elliptic curve coming from the quantization constructed above. 
It defines a natural element of the Gelfand-Fuks cohomology $\clie^*(\Vect ; \Omega^{-*}_n)$ 
that deserves to be called the $n$-dimensional {\it formal Witten class}. 
We show that under Gelfand-Kazhdan descent, this formal cocycle maps to the Witten class of the complex manifold. 

\begin{rmk}
The arguments here are borrowed from \cite{WG2}, notably Section 17,
where Costello identifies the Witten class of the target $X$ 
as part of the quantized action functional of the curved $\beta\gamma$ system.
We simply observe that his approach applies equally well with the formal disk as target,
so long as one uses Gelfand-Fuks cohomology.
In \cite{WG2} Costello also provides an interpretation of the Witten class as a kind of ``projective volume form''
on the derived mapping space from the universal elliptic curve to $X$.
We do not discuss that here, but his interpretation applies to our approach as well.
\end{rmk}

\subsection{The formal Witten class}

Let $\cV$ be a formal vector bundle, i.e., an object of the category $\VB_{(\Vect, \GL_n)}$. 
We have constructed the Gelfand-Fuks-Chern characters
\ben
\ch^{\rm GF}_k(\cV) \in \clie^k(\Vect, \GL_n ; \hOmega^k_{n,cl}) .
\een
Let $\hOmega^{-*}_n = \bigoplus_{k} \hOmega^k_n [k]$ denote the formal de Rham forms 
arranged in opposite degrees from usual (i.e., with $k$-forms beginning in degree $-k$ rather than $k$).
Note that we do {\em not} include the exterior derivative as part of the total differential
(for degree reasons this is not possible, but it is not relevant to our setting either).
Each cocycle $\ch^{\rm GF}_k(\cV)$ then provides a cocycle of degree \emph{zero} in $\hOmega^{-*}_n$.
Thus, any interesting formal combination of such characters --- like the Witten class defined below --- naturally sits in degree zero.

\begin{dfn} 
Let $E$ be an elliptic curve equipped with a holomorphic volume form $\omega \in \Omega^{1,0}(E)$. 
The $n$-dimensional {\em logarithmic formal Witten class} evaluated at $(E,\omega)$ is the formal sum
\ben
\log \Wit_n (E, \omega) := \sum_{k \geq 2} \frac{(2k-1)!}{(2\pi i)^{2k}}\, E_{2k}(E,\omega)\, \ch_{2k}^{\rm GF} (\hT_n) .
\een 
\end{dfn}

The Eisenstein series $E_{2k}$ is given by the formula
\ben
E_{2k}(\CC/\Lambda,\d z) = \sum_{\lambda \in \Lambda - \{0\}} \lambda^{-2k}
\een
where we witness the elliptic curve given as a quotient of the complex plane by the lattice $\Lambda$.
For an arbitrary $E$, we find a lattice $\Lambda$ such that 
$\d z$ identifies with $\omega$ under the isomorphism $\CC / \Lambda \cong E$. 
If the lattice is spanned by two elements
\ben
a + i b, c + i d \in \CC .
\een
Then the sum can be written as
\ben
\sum_{\lambda \neq \{0\}} \lambda^{-2k} = \sum_{(m,n) \in \ZZ^2 - \{(0,0)\}} (m a + i m b + n c + i n d)^{-2k} .
\een

Our main result in this section is that the scale $\infty$ effective interaction of the formal $\beta\gamma$ system, in the presence of the background $\TVect$ fields, is equivalent to the formal Witten class $\Wit_n(E,\omega)$ plus a term proportional to $\ch_2^{\rm GF}(\hT_n)$. 
To state the result, recall the extension $p : \TVect \to \Vect$ of Lie algebras determined by the formal second Chern character.

\begin{prop} \label{prop: formalwitten}
As a function on the harmonic forms $\cH(E)$,
the one-loop part of the scale $\infty$ effective quantization
$\Tilde{I}^{{\rm W}, (1)}[\infty]$ is
\ben
\frac{1}{32 \pi^4} \Bar{E}_2(E,\omega)\, p^*\ch_2^{\rm GF}(\hT_n) + 
\sum_{k \geq 2} \frac{(2k-1)!}{(4 \pi^2)^{2k}} E_{2k}(E, \omega)\, p^*\ch_{2k}^{\rm GF} (\hT_n),
\een
as a cocycle in $\clie^*(\TVect ; \hOmega^{-*}_{n}) $.
In particular, the one-loop effective quantization is cohomologous to $p^*\log \Wit_n(E, \omega)$.
\end{prop}

\begin{rmk}
The series $\Bar{E}_2(E)$ is the ``modular completion" of the second Eisenstein series.
In terms of the modular parameter $\tau$, it is defined by
\[
\Bar{E}_2 (\tau,\Bar{\tau}) = 1 - 24 \sum_{n=1}^\infty \frac{n q^n}{1-q^n} - \frac{3}{\pi {\rm Im}(\tau)}
\]
where $q = e^{2\pi i \tau}$. 
It has the property that it is modular, but not holomorphic. 
\end{rmk}

The term proportional to $\ch_2^{\rm GF}(\hT_n)$ arises from the term in the effective interaction on $E$ given by the weight of a wheel with two vertices. 
There is some delicate analysis involved in computing the precise contribution of this weight, but we see that when restricted to the extended Lie algebra $\TVect$ it is cohomologous to zero, by construction so we may disregard it. 

\subsection{The theory on an elliptic curve}

The quantization we have constructed above is invariant 
for the group $\Aff(\CC)$ of affine symmetries of the complex plane. 
Thus, for any elliptic curve $E = \CC/\Lambda$,
we can descend the quantization on $\CC$ along the quotient map $\CC \to E = \CC / \Lambda$.
The dg Lie algebra describing the theory on the elliptic curve $E$ is
\ben
\DD \fg_n^E = \Omega^{0,*}(E ; \fg_n) \oplus \Omega^{1,*}(E ; \fg_n^\vee[-2]) .
\een 
There is a simplification we can make in this setting. 
The choice of a holomorphic volume form $\omega$ determines an isomorphism of dg Lie algebras
\[
\begin{array}{ccc}
\Omega^{0,*}(E; \fg_n \oplus \fg_n^\vee[-2]) &\cong & \DD \fg_n^E\\
\gamma \tensor (\xi, \tau) & \leftrightarrow & (\gamma \tensor \xi,(\gamma \wedge \omega) \tensor \tau)
\end{array}.
\]
This isomorphism is naturally $(\Vect, \GL_n)$-equivariant. 

Note that there is an element $\omega^\vee \in \Omega^{0,1}(E)$  such that $\int \omega \wedge \omega^\vee = 1$.
At the level of cohomology, $[\omega^\vee]$ spans $H^1(E,\cO)$, by Serre duality.
We are free to choose $\omega^\vee$ to be {\em harmonic}, 
meaning it is annihilated by both $\partial$ and $\dbar$.
If $E = \CC/\Lambda$, then there is a constant
\[
v(E) = \int_{E} \d z \, \d\zbar
\]
and $\omega^\vee = v(E)^{-1} \d\zbar$.
In general, let $\delta$ denote~$(i \pi)^{-1}\omega^\vee$.

Let $\cH(E) \subset \DD \fg_n^E$ denote the sub dg Lie algebra of harmonic forms 
(that is, those forms that are in the kernel of $\partial$ and $\dbar$).
We have an isomorphism 
\ben
\cH(E) \cong \CC [\delta] \tensor (\fg_n \oplus \fg_n^\vee[-2])
\een
of dg Lie algebras, thanks to our choices above.

In anticipation of this section's main result, note that 
\[
\CC[\delta] \tensor \fg_n[-1] \cong \fgn \ltimes \fgn, 
\]
the natural extension of $\fgn$ by the shifted adjoint representation $\fgn[-1]$.
Hence,
\ben
\clie^*\left(\Vect ; {\rm C}_{\rm Lie,red}^*(\CC[\delta] \tensor \fg_n)\right)
\cong \clie^*(\Vect ; \hOmega^{-*}_n),
\een
where $\hOmega^{-*}_n$ is the regraded formal de Rham complex. 
We now explain why the scale $\infty$ effective action for the equivariant BV theory
produces a cocycle in this cochain complex.

The harmonic subspace $\cH(E)$ describes the solutions on $E$ to the equations of motion for the formal $\beta\gamma$ system.
If we restrict the scale $\infty$ effective interaction to this subspace, 
it provides an $\hbar$-dependent cocycle in the Lie algebra cochains:
\ben
\left(\Tilde{I}^{\rm W}[\infty] + \hbar J[\infty]\right) \Big|_{\cH(E)} \in
\clie^*\left(\TVect ; {\rm C}_{\rm Lie,red}^*(\CC[\delta] \tensor (\fg_n \oplus \fg_n^\vee[-2])) \right)  [\hbar].
\een
Note that the one-loop term of the effective interaction $\Tilde{I}^{\rm W,(1)}[\infty] + \hbar J [\infty]$ 
is only a functional of $\CC [\delta] \tensor \fg_n$ and
does not depend on $\CC [\delta] \tensor \fg_n^\vee[-2]$. 

In fact, at scale $\infty$, things become even simpler. 

\begin{lemma} 
The functional $J[\infty]$ vanishes on the subspace of harmonic forms:
\ben
J[\infty] \big|_{\cH(E)} = 0 .
\een
Thus, the scale $\infty$ effective interaction lies in the image of
$I^{\rm W}[\infty] |_{\cH}$ under the pullback map
\ben
p^* : \clie^*\left(\Vect ; {\rm C}_{\rm Lie,red}^*(\CC[\delta] \tensor \fg_n)\right) 
\to \clie^*(\TVect ; {\rm C}_{\rm Lie,red}^*\left(\CC[\delta] \tensor \fg_n)\right),
\een
where $p : \TVect \to \Vect$ is the extension of $\Vect$ by closed two-forms 
determined by~$\ch_2^{\rm GF}(\cT_n)$.
\end{lemma}

\begin{proof}
Recall, for fixed closed two-form $\omega$ the local functional $J_\omega$ is defined to be a functional on the
space $\Omega^{0,*}(E ; \fg_n)$. The definition of $J_\omega$ invovles a
single holomorphic
derivative acting on one of the input fields. When we restrict
to harmonic forms $\CC[\delta] \tensor \fg_n \hookrightarrow
\Omega^{0,*}(E ; \fg_n)$ the holomorphic derivative acts by zero and
hence $J_\omega|_\cH$ vanishes for all $\omega$. Thus $J|_{\cH}$ is
identically zero. Since the scale $\infty$
action $J[\infty]$ involves at least one vertex labeled by $J$ we see
that its restriction also vanishes. 
\end{proof}

In particular, the one-loop scale $\infty$ interaction comes as an element in  
\[
\clie^*(\Vect ; {\rm C}_{\rm Lie, red}^*(\CC [\delta] \tensor \fg_n)) = \clie^*(\Vect ; \hOmega^{-*}_n). 
\]
We wish to explicitly compute this element.

First, we make a remark about where the functional $I^{\rm W}$ lives
when our spacetime is an elliptic curve and we restrict to harmonic
forms. This restriction can be viewed as a functional
\ben
I^{\rm W} \big|_{\cH(E)} : \Vect \oplus \CC [\delta] \tensor (\fg_n[1]
\oplus \fg_n^\vee[-1]) \to \CC .
\een
Since $I^{\rm W}$ is linear in the $\fg_n^\vee$ and $\delta \fg_n$
component, we can view this restriction as an element in space
\ben
 \clie^*\left(\Vect ; {\rm C}_{\rm Lie,red}^*(\fg_n) \tensor \delta
 \fg^\vee \tensor \fg_n \right) .
\een 
Let
\ben
\d_{dR} : {\rm C}_{\rm Lie,red}^*(\fg_n)) \to \clie^*(\fg_n,\fg^\vee) \cong \hOmega^1_n
\een
be the de Rham differential. Then the element 
\ben
(\d_{dR} \tensor 1)I^{\rm W} \in  \clie^*(\Vect ; {\rm C}_{\rm
  Lie,red}^*(\fg_n) ) \tensor \delta \fg^\vee \tensor \End(\fg_n)
\cong \clie^*(\Vect ; \hOmega^1_n \tensor \End(\fg_n))
\een
is precisely the Atiyah class $\At^{\rm GF}(\hT_n)$ as shown in
Section \ref{sec atiyah 2}. 

Now we can move on to the main result of this section.

\subsection{Proof of Proposition \ref{prop: formalwitten}}

We recall the general approach for computing the renormalized effective action on the elliptic curve $E$. 
The procedure splits into the following steps:
\begin{itemize}
\item[(1)] truncate the propagator $P_{\epsilon < L}$ at both the lower and upper bounds;
\item[(2)] compute the Feynman graph weight $W(P_{\epsilon<L}, I) = \sum_{\Gamma} W_{\Gamma} (P_{\epsilon<L}, I)$ as integrals over the elliptic curve $W_{\Gamma}(P_{\epsilon<L} , I) = \int_{E} w_{\Gamma}(P_{\epsilon < L}, I)$;
\item[(3)] take the limit $\lim_{\substack{\epsilon \to 0 \\ L \to 0}} W_{\Gamma}(P_{\epsilon<L} , I)$.
\end{itemize}
Our analysis so far has shown that the limit in step (3) exists.
However, in the proof of the proposition we will compute the limit in a direct way by exchanging the limit with the integration over $E$:
\[
\lim_{\substack{\epsilon \to 0 \\ L \to 0}}\int_{E} w_{\Gamma}(P_{\epsilon < L}, I) \stackrel{?}{=} \int_E \lim_{\substack{\epsilon \to 0 \\ L \to 0}}w_{\Gamma}(P_{\epsilon < L}, I) .
\]
The issue is that even though the limit in (3) exists, the expression above is not valid in general. 
Feynman integrals of this type were studied extensively in the work \cite{LiFeynman}.
A similar analysis as done there, and one we can check by direct calculation, implies that the limits above can be interchanged when the graph $\Gamma$ has no wheels with fewer than $3$ vertices.
The term proportional to $\ch_2^{\rm GF}(\hT_n)$ in the proposition corresponds to the wheel with two vertices, and so the proper regularization scheme outlined above must be performed. 

The weight expression for $\Tilde{I}^{{\rm W},(1)}$ is given by
\[
\sum_{\Gamma \in \text{\rm Wheels}} \frac{1}{|\Aut(\Gamma)|} \lim_{\substack{\epsilon \to 0 \\ L \to 0}} W_\Gamma(P_{\epsilon < L}, \Tilde{I}^{\rm W}) .
\]
Just as in the effective action on $\CC$, the tadpole diagram is identically zero.
When the number of vertices of the wheel is two, a similar calculation as in Lemma 2.2 of \cite{LiFeynman} shows that the weight of the graph is $\frac{1}{32 \pi^4} \Bar{E}_2(E,\omega)\, p^*\ch_2^{\rm GF}(\hT_n)$.
This is the first term in the expression of the effective action. 

It remains to compute the weights of wheels with number of vertices strictly bigger than two. 
By the remarks above, we can interchange the limits to obtain
\ben
\sum_{\Gamma \in \text{\rm Wheels}_{>2}} \frac{1}{|\Aut(\Gamma)|} \lim_{\substack{\epsilon \to 0 \\ L \to 0}} W_\Gamma(P_{\epsilon < L}, \Tilde{I}^{\rm W})  = \sum_{\Gamma \in \text{\rm Wheels}_{>2}} \frac{1}{|\Aut(\Gamma)|} W_\Gamma(P_{0 < \infty}, \Tilde{I}^{\rm W}) 
\een
where the sum is over $\Gamma \in \text{\rm Wheels}_{>2}$ all wheels with number of vertices bigger than two. 

We are computing the restriction of this to the subspace
\ben
\cSym(\Vect[1]^\vee) \tensor \cSym\left((\left(\CC[\delta] \tensor
  \fg_n[1]\right)^\vee)\right) .
\een

Each vertex of the wheel is labeled by the interaction $I^{\rm W}$. 
We now write down the propagators for which we are contracting.

We identify
\ben
\begin{array}{ccc}
\Omega^{0,*}(E) & \cong & C^\infty(E) \tensor \CC[\delta] \\
\mu^{-1} \d \zbar & \leftrightarrow & \delta
\end{array}
\een
where $\mu = i \pi \int_E \d z \,\d \zbar$. 

The scale $\infty$ propagator is
\ben
P_{0 < \infty}(z,w) = \int_{0}^\infty (\dbar^* \tensor 1) K_t (z,w) \d t .
\een 
When we descend to the elliptic curve, the heat kernel is 
\ben
K_t (z,w) = \left(\frac{1}{4 \pi t} e^{-|z-w|^2/4t}\right) (\delta \tensor 1 - 1
  \tensor \delta) \tensor (\id_{\fg_n} + \id_{\fg_n^\vee})
\een 
where $(4 \pi t)^{-1} e^{-|z-w|^2/4t} \in C^\infty(E \times E)$ is the
scalar heat kernel for the Laplacian on functions. The adjoint $\dbar^*$ satisfies
\ben
\dbar^*(f(z,\zbar) \d \zbar) = \frac{\partial f}{\partial z} 
\een
and hence we identify
\ben
P_{0<\infty}(z,w) = \int_{0}^\infty \mu^{-1} \frac{\partial}{\partial z}
\left(\frac{1}{4 \pi t} e^{-|z-w|^2/4t} \right) \tensor (\id_{\fg_n} + \id_{\fg_n^\vee})\, \d t .
\een 
In turn, we can think of this formula as the integral kernel for the operator
\ben
\mu^{-1} \frac{\partial}{\partial z} (2 \dbar \dbar^*)^{-1} \tensor \id : C^\infty(E) \tensor \fg_n \to C^\infty(E) \tensor \fg_n
\een
where $2 \dbar \dbar^*$ is the scalar Laplacian acting on functions. 

Using the identity $\d_{dR} I^{\rm W} = \At(\hT_n)$, 
we see that the sum of weights for diagrams of exactly $k$ vertices is
\ben
\frac{1}{k} \Tr \left( \left(p^*\At^{\rm GF} (\hT_n) \tensor
    \mu^{-1} \frac{\partial}{\partial z} (2\dbar \dbar^*)^{-1}
  \right)^{k} \right) \in \clie^*\left(\TVect ; \clie^*(\CC[\delta]
\tensor \fg_n ) \right) \cong \clie^*(\TVect ; \hOmega^{-*}_n) .
\een 
(This calculation recapitulates that of the obstruction.)
We know that the algebraic piece simplifies to $\Tr(\At^{\rm GF}(\hT_n)) = k! (2\pi i)^k \ch_{2k}^{\rm GF}(\hT_n)$. 
For odd $k$, the analytic factor vanishes. 
Finally, a direct computation shows that
\ben
\Tr\left(\left(\mu^{-1} \frac{\partial}{\partial z} (2\dbar \dbar^*)^{-1} \right)^{2k} \right) = \frac{1}{(4 \pi^2)^{2k}} E_{2k} 
\een
for $k > 2$.
One simply picks a natural Fourier basis for smooth functions on an elliptic curve $E = \CC/\Lambda$,
on which basis the operator  $\frac{\partial}{\partial z} (2\dbar \dbar^*)^{-1}$ is easy to describe.
(See, for instance, section 17.8 in \cite{WG2}.) This fact completes the proof. 

\section{The factorization algebras of equivariant
  observables}\label{sec equiv obs}

So far we have constructed classical and quantum BV theories for the formal $\beta\gamma$ system.
We now turn to analyzing the observables of these theories,
using the machinery of \cite{CG2}, which intertwines the BV formalism with factorization algebras.
As we show in Part 3, the factorization algebras that we construct here 
provide a refinement of the vertex algebras $\Gr\,\hCDO_n$ and $\hCDO_n$ from Part 1.

In brief a factorization algebra is a local-to-global object on a manifold --- in that sense, it is like a sheaf --- that encodes how to combine sections living on disjoint opens --- and hence, like an algebra.
In \cite{CG2} it is shown that every field theory in the BV formalism has an associated factorization algebra of observables.
For a classical field theory, the observables $\Obs^\cl$ assign to an open $U$, 
the commutative dg algebra of functions on the space of fields on $U$.
Thus classical observables form a commutative factorization algebra.
A BV quantization amounts, in essence, to deform the differential from $\{S^\cl,-\}$ to $\{S^\q,-\} + \hbar \Delta$, 
where $S^\cl$ is the classical action functional and $S^\q$ is its quantization.
The quantum observables are thus a deformation of the commutative factorization algebra $\obs^\cl$ to a factorization algebra (which has no commutative structure).

Our work in Section \ref{sec equiv bv} thus provides factorization algebras 
for the equivariant and non-equivariant formal $\beta\gamma$ system.
Before we spell out those objects in detail, though, 
we must give the definition of a factorization algebra and discuss the relevant functional analysis.

\begin{rmk}
Although we attempt to describe all the relevant ideas and definitions here,
we rely extensively on results and arguments in \cite{CG2},
which contains a lengthy treatment of the formalism we deploy.
For further details, motivation, and context, we refer the reader there.
\end{rmk}

\subsection{An overview of factorization algebras}

Let $X$ be a topological space and $\cC^\otimes$ a symmetric monoidal category.
For us $X$ will be a Riemann surface, typically $\CC$, and 
$\cC^\otimes$ will be a category of cochain complexes of vector spaces with $\otimes$ the tensor product.
(In Section \ref{sec functional analysis} we discuss the type of vector spaces and tensor product that we use,
since issues of functional analysis appear.)
Here we give the general definition and refer to \cite{CG1,CG2} for more detail and motivation.

\begin{dfn}
A \emph{prefactorization algebra} $\cF$ on $X$ with values in $\cC^\otimes$ 
assigns to each open $U$ in $X$, an object $\cF(U)$ in $\cC$, and 
assigns to each finite collection $\{U_1,\ldots,U_n\}$ of pairwise disjoint opens in $X$
where each $U_i \subset V$, a morphism
\[
\cF_V^{U_1,\ldots,U_n}: \cF(U_1) \otimes \cdots \otimes \cF(U_n) \to \cF(V).
\]
These assignments satisfy
\begin{enumerate}
\item the morphisms compose, so that
\[
\cF_V^{U_1,\ldots,U_n} \circ \bigotimes_{i = 1}^n \cF_{U_i}^{T_{i,1},\ldots,T_{i,{m_i}}} = \cF_V^{T_{1,1},\ldots,T_{n,m_n}}
\]
for any choice of pairwise disjoints open $\{T_{i_1},\ldots,T_{i_{m_i}}\}$ inside $U_i$ for each $i$, and
\item the morphisms are equivariant under rearrangement of labels, 
so that for any permutation $\sigma \in S_n$, the composite
\[
\cF(U_{\sigma(1)}) \otimes \cdots \otimes \cF(U_{\sigma(n)}) \xto{\cong} \cF(U_1) \otimes \cdots \otimes \cF(U_n) \xto{\cF_V^{U_1,\ldots,U_n}} \cF(V)\to \cF(V)
\] 
equals $\cF_V^{U_{\sigma(1)}) \otimes \cdots \otimes \cF(U_{\sigma(n)}}$.
\end{enumerate}
\end{dfn}

This structure encodes a kind of algebra parametrized by the geometry of $X$.
The data of $\cF$ explains how to ``multiply'' elements living on opens $U_i$ into an element on $V$.

An associative algebra $A$ provides an example living on $X = \RR$.
To each open interval $I$, one assigns $A$, and 
to a union of disjoint intervals $\sqcup_{j \in J} I_j$, one assigns the tensor product $\bigotimes_{j \in J} A$.
Each structure map is determined by the multiplication in $A$.

Another example, central to our work here, is the following. 
Let $E \to X$ be a vector bundle on a smooth manifold.
Let $\sE_c$ denote the precosheaf of compactly supported sections of $E$:
to each open $U$, we assign $\sE_c(U) = \Gamma_c(U,E)$, and 
there is a natural extension-by-zero $\sE_c(U) \to \sE_c(V)$ whenever $U \subset V$.
This precosheaf satisfies that
\[
\sE_c( U_1 \sqcup U_2) \cong \sE_c(U_1) \oplus \sE_c(U_2)
\]
for any disjoint union of opens. 
Using the appropriate notion of tensor product, discussed below, 
one then sees that
\[
\Sym(\sE_c( U_1 \sqcup U_2)) \cong \Sym(\sE_c(U_1) \oplus \sE_c(U_2)) \cong \Sym(\sE_c(U_1)) \otimes \Sym(\sE_c(U_2)),
\]
which provides a natural map 
\[
\Sym(\sE_c(U_1)) \otimes \Sym(\sE_c(U_2)) \to\Sym(\sE_c(V))
\]
for any $V \supset U_1 \sqcup U_2$.
In this way, one shows that $\Sym(\sE_c)$ forms a prefactorization algebra.

A factorization algebra is a prefactorization algebra satisfying a local-to-global condition,
just as a sheaf is a presheaf satisfying one.
The primary difference in the conditions is that the notion of cover changes.

\begin{dfn}
A \emph{Weiss cover} $\{U_i\}_{i \in I}$ of an open $V$ is a collection 
of opens $U_i \subset V$ such that for any finite set of points $\{x_1,\ldots,x_n\} \subset V$,
there is some $U_i \supset \{x_1,\ldots,x_n\}$.
\end{dfn}

We will restrict our attention now to $\cC$ that are categories of cochain complexes of vector spaces.
(More generally, the definition below is well-behaved for cochain complexes on a Grothendieck abelian category.
See Appendix C of \cite{CG1}.)
Since we view quasi-isomorphic cochain complexes as equivalent 
(i.e., we are interested in the higher category arising from quasi-isomorphism as the notion of weak equivalence),
the local-to-global condition is a cochain refinement of the usual notion.

\begin{dfn}
A \emph{factorization algebra} is a prefactorization algebra $\cF$ such that
for any open $V$ and any Weiss cover $\{U_i\}_{i \in I}$ of $V$,
the natural map
\[
\check{C}(\{U_i\}_{i \in I},\cF) \to \cF(V)
\]
is a quasi-isomorphism.
(Here the left hand side denotes the \v{C}ech complex of $\cF$ on the cover.)
\end{dfn}

\subsection{A comment on functional analysis}
\label{sec functional analysis}

We are working throughout with infinite-dimensional vector spaces 
such as the space of smooth functions $C^\infty(X)$ on a smooth manifold.
Thus we need to be careful about issues such as tensor products and duals,
since the setting of plain vector spaces is not appropriate or adequate for our constructions.
Appendix B of \cite{CG1} describes a category of {\em differentiable vector spaces} well-suited to our setting, 
and it explains its relationship with other natural choices, 
such as locally convex topological vector spaces, bornological vector spaces, or convenient vector spaces.
The reader wishing for a discussion about the subtleties of constructing factorization algebras in such settings should look in~\cite{CG1}.

Here we simply state explicitly what we mean by duals and tensor product for the vector spaces with which we work.
These definitions are natural for both differential geometry and functional analysis.

Let $E \to X$ be a finite-rank vector bundle on a smooth manifold.
We use the following notations:
\begin{enumerate}
\item[(1)] the smooth sections are $\sE = \Gamma(X,E)$,
\item[(2)] the compactly supported smooth sections are $\sE_c = \Gamma_c(X,E)$,
\item[(3)] the distributional sections are $\Bar{\sE}$, and
\item[(4)] the compactly supported distributional sections are $\Bar{\sE}_c$.
\end{enumerate}
Let $E^! = E^\vee \otimes {\rm Dens}_X$ denote the vector bundle given by 
the tensor product of the fiberwise linear dual $E^\vee$ with the density line ${\rm Dens}_X$.
Then we write
\begin{enumerate}
\item[(1)] the smooth sections as $\sE^! = \Gamma(X,E^!)$,
\item[(2)] the compactly supported smooth sections as $\sE_c^! = \Gamma_c(X,E^!)$,
\item[(3)] the distributional sections are $\Bar{\sE}^!$, and
\item[(4)] the compactly supported distributional sections are $\Bar{\sE}_c^!$.
\end{enumerate}
Note that the vector bundle map $\ev: E^\vee \otimes E \to \underline{\CC}$ given by the fiberwise evaluation pairing 
induces a vector bundle map $\langle-,-\rangle_{fib}: E^! \otimes E \to {\rm Dens}_X$.
This pairing then extends a natural bilinear pairing 
\[
\begin{array}{cccc}
\langle-,-\rangle: &\Bar{\sE}^!_c \times \sE & \to & \CC \\
& (\lambda,f) & \mapsto & \int_X \langle \lambda, f\rangle_{fib}
\end{array}.
\]
There are clearly also versions for $\Bar{\sE}^! \times \sE_c$ or with distributional sections of $\sE$ and so on.

\begin{dfn}
We write $\sE^\vee$ for $\Bar{\sE}^!_c$ and call it the {\em dual} of $\sE$. 
We use $\langle-,-\rangle$ for the {\em evaluation pairing} $\ev: \sE^\vee \times \sE \to \CC$.
Similarly, we write $\sE_c^\vee$ for $\Bar{\sE}^!$, $(\Bar{\sE})^\vee$  for $\sE^!_c$, and $(\Bar{\sE}_c)^\vee$  for $\sE^!$.
\end{dfn}

Given $E \to X$ and $F \to Y$ finite-rank vector bundles on smooth manifolds,
let $E \boxtimes F \to X \times Y$ denote $\pi_X^*E \otimes \pi_Y^* F$,
i.e., the tensor product of the vector bundles pulled back along the projection maps 
$\pi_X: X \times Y \to X$ and $\pi_Y: X \times Y \to Y$.

\begin{dfn}
We write $\sE \otimes \sF$ for the smooth sections of $E \boxtimes F$ and call it the \emph{tensor product}.
Similarly, we write 
$\sE_c \otimes \sF_c$ for the compactly supported smooth sections of $E \boxtimes F$,
$\Bar{\sE} \otimes \Bar{\sF}$ for the distributional sections of $E \boxtimes F$, and
$\Bar{\sE}_c \otimes \Bar{\sF}_c$ for the compactly supported distributional sections of $E \boxtimes F$.
\end{dfn}

It makes sense to ask for sections of $E \boxtimes F$ that are distributional in the $X$-direction but smooth in the $Y$-direction,
and we write $\Bar{\sE} \otimes \sF$ for this space.

\begin{dfn}
For a $\ZZ$-graded vector bundle $E \to X$, the {\em algebra of functions} on $\sE$ is
\[
\Sym(\sE^\vee) := \bigoplus_{n \geq 0} ((\Bar{\sE}_c^!)^{\otimes n})_{S_n}.
\]
The {\em completed} algebra of functions on $\sE$ is
\[
\cSym(\sE^\vee) := \prod_{n \geq 0} ((\Bar{\sE}_c^!)^{\otimes n})_{S_n}.
\]
\end{dfn}

In particular, an element $f$ of the $n$th symmetric power $\Sym^n(\sE)$ 
can be identified with a compactly supported distributional section of $\Gamma(X^n, (E^!)^{\boxtimes n})$
that is invariant under the natural permutation action of~$S_n$.

Note that these definitions make it straightforward to express the Chevalley-Eilenberg cochains
of dg Lie algebras like $\DD\fgn^S$, whose underlying graded vector spaces are of the type described here.

\subsection{The non-equivariant classical observables}

We begin by defining the classical observables on a fixed source.

\begin{dfn}
The \emph{classical observables} $\Obs^{\cl}_n(S)$ 
for the rank $n$ formal $\beta\gamma$ system on the Riemann surface $S$ 
is the completed algebra of functions on the space of fields
\[
\Omega^{0,*}(S)^{\oplus n} \oplus \Omega^{1,*}(S)^{\oplus n} = (\DD \fgn^S)[1]
\]
equipped with the differential given by extending $\dbar$ as a derivation.
Hence
\[
\Obs^{\cl}_n(S) = \clies(\DD \fgn^S),
\]
where the Chevalley-Eilenberg cochains are constructed using the appropriate versions of dual and tensor product.
\end{dfn}

Explicitly, the underlying graded algebra is
\[
\cSym(\Bar{\Omega}^{1,*}_c(S)^{\oplus n}[1] \oplus \Bar{\Omega}^{0,*}_c(S)^{\oplus n}[1]).
\]
The differential can be understood explicitly as follows.
For some $n$-fold tensor product of linear functionals on the fields
\[
a = \alpha_1 \otimes \cdots \otimes \alpha_n,
\]
we have
\[
\dbar(a) = (\dbar \alpha_1) \otimes \cdots \otimes \alpha_n \pm \alpha_1 \otimes (\dbar \alpha_1) \otimes \cdots \otimes \alpha_n + \cdots \pm\alpha_1 \otimes \cdots \otimes (\dbar \alpha_n).
\]
This differential is equivariant with respect to the permutation action of the symmetric group $S_n$ and 
hence induces a differential on the $n$th symmetric power.

It is manifest that these observables are natural with respect to holomorphic embeddings.
That is, given a holomorphic embedding $i: S \hookrightarrow S'$,
there is a natural extension map
\[
i_*: \Obs^\cl_n(S) \to \Obs^\cl_n(S')
\]
that is naturally induced by the restriction map of fields
\[
i^*: \DD\fgn^{S'} \to \DD\fgn^{S}.
\]
Indeed, we have a factorization algebra on any Riemann surface by Theorem 5.2.1 of \cite{CG1}.
For the purpose of extracting the vertex algebra, it will suffice to focus on $S = \CC$ and not consider all Riemann surfaces at the same time.

\begin{dfn}
Let $\Obs^\cl_n$ denote the factorization algebra on $\CC$ of classical observables for the rank $n$ formal $\beta\gamma$ system.
\end{dfn}

We remark that as $\GL_n(\CC)$ acts naturally on $\CC^n \cong \fgn[1]$,
it also acts naturally on $\Obs^\cl_n(S)$ for any Riemann surface $S$.
This action manifestly respects the differential $\dbar$, 
which only depends on the source $S$ and not on the target $\hD^n$.

\subsection{The non-equivariant quantum observables}

The BV formalism suggests that the quantum observables on $S$ should arise by 
\begin{enumerate}
\item[(a)] tensoring the underlying graded vector space of $\Obs^\cl_n$ with $\CC[[\hbar]]$ and
\item[(b)] modifying the differential to $\dbar +\hbar \Delta$, where $\Delta$ is the BV Laplacian.
\end{enumerate}
This suggestion does not work because $\Delta$ is not defined on all of the observables;
the naive formula involves an ill-defined pairing of distributions.
There are two ways to circumvent this difficulty. 
First, one can work with a smaller class of observables --- such as those arising from smooth functionals, not distributional ones --- and this approach is developed in detail for the free $\beta\gamma$ system in Chapter 5, Section 3 of \cite{CG1}.
(We discuss this approach in Section \ref{noneqsec}, where we also show the two approaches agree.) 
Second, one can mollify $\Delta$ instead.
This approach is developed in a very broad context in Chapter 9 of \cite{CG2},
and we have encountered it already in the scale $L$ BV Laplacians $\Delta_L$.
These two approaches provide quasi-isomorphic factorization algebras, 
as we show in Proposition~\ref{equiv of noneq}. The second approach is what we will explain here, as it is the one that extends to the equivariant setting.

Before delving into the machinery necessary to define a factorization algebra of quantum observables,
let us note that we have a working description of the global observables on~$\CC$.

\begin{dfn}
The \emph{global scale $L$ quantum observables for the rank $n$ formal $\beta\gamma$ system} has
underlying graded vector space
\[
\cSym(\Bar{\Omega}^{1,*}_c(\CC)^{\oplus n}[1] \oplus \Bar{\Omega}^{0,*}_c(\CC)^{\oplus n}[1])[[\hbar]]
\]
with differential $\dbar + \hbar \Delta_L$. 
We denote it $\Obs^\q_n[L](\CC)$.
\end{dfn}

The quantum observables are isomorphic for any choice of length scale.
In fact, the RG flow provides an explicit isomorphism $W_\epsilon^L: \Obs^\q_n[\epsilon](\CC) \to \Obs^\q_n[L](\CC)$ as follows: 
given an observable $f$ at scale $\epsilon$, let $W_\epsilon^L(f)$ denote the observable such that
\[
W(P_\epsilon^L, \delta f) = \delta W_\epsilon^L(f)
\]
where $W(P_\epsilon^L, -)$ is the RG flow operation defined earlier on action functionals, 
$\delta^2=0$,  and $|\delta| = - |f|$. 
(In other words, this map arises by taking the ``derivative of RG flow'' $W(P_\epsilon^L, -)$.)

The basic approach used here works in general, 
except that we will need to work with a more flexible notion of ``length scale'':
we need to allow arbitrary parametrices for $\dbar$.
(Indeed, when interactions are included, the RG flow map $W_\epsilon^L$ does not preserve support conditions on observables,
because the heat kernel has support everywhere.) 
After reviewing this machinery, we use it to define the factorization algebra of quantum observables.

\subsubsection{Recap of parametrices and BV Laplacians}
\label{parametrices}

We give here the specialization to our situation of the general definition from Chapter 8, Section 2.4 of \cite{CG2}.
Recall that we are working on $\CC$ with its standard, Euclidean metric.
Let $\dbar^*$ denote the Hodge dual operator to $\dbar$ with respect
to this metric. It is our choice of ``gauge-fix,'' in the terminology of \cite{CosBook}.
Let $\triangle = [\dbar,\dbar^*]$ denote the Hodge Laplacian on Dolbeault forms.

We remark on our convention for integral kernels.
Given a continuous linear endomorphism $P$ of $\Omega^{0,*}(\CC)$, 
we use $K_P$ to denote the \emph{integral kernel} for $P$:
it is the section of $\Omega^{0,*}(\CC) \widehat{\otimes}_\pi \Bar{\Omega}^{1,*}_c(\CC)$ such that
\[
(P\alpha)(z) = \int_{w \in \CC} \langle K_P(z,w), \alpha(w)\rangle_w,
\]
where the BV pairing is along the $w$-direction.
(We remark that Schwarz's kernel theorem tells us where the integral kernel lives depending upon the domain and range of the continuous linear operator.  
Here $\Bar{\Omega}^{1,*}_c(\CC)$ appears because the domain is the continuous linear dual space~$\Omega^{0,*}(\CC)$.)

\begin{dfn}
A {\em parametrix} for $\dbar$ on $\Omega^{0,*}(\CC)$ is a distributional section $\Phi$ of
$\Omega^{1,*}(\CC \times \CC)$ such that
\begin{enumerate}
\item $\Phi$ has cohomological degree one,
\item $\Phi$ is symmetric with respect to the $S_2$ action,
\item $\Phi$ has proper support with respect to the two projection maps from $\CC^2$ to $\CC$,~and
\item $(\triangle \otimes \id) \Phi - K_{\id}$ is a smooth section of $\Omega^{1,*}(\CC \times \CC)$,
where $K_{\id}$ is the integral kernel for the identity operator with respect to the BV pairing.
\end{enumerate}
Let ${\rm Param}$ denote the set of parametrices.
\end{dfn}

There is a natural partial ordering on ${\rm Param}$ by support:
$\Psi \leq \Phi$ if $\supp(\Psi) \subset \supp(\Phi)$.

We remark that the integral kernel $\Psi = \int_0^L K_t^{an}$, using the analytic heat kernel from Section \ref{sec prequant},
satisfies all these conditions except proper support. 
It is, in fact, supported everywhere on $\CC^2$.
(It is thus an ``almost-parametrix.'')
One can easily obtain a parametrix from $\Psi$ as follows: 
pick a smooth function $f$ on $\CC^2$ that is 1 in a neighborhood of the diagonal and vanishes sufficiently far from the diagonal,
and consider $f \Psi$.
This construction will allow us to translate between results written in terms of heat kernels (i.e., length scale)
and those written in terms of parametrices.

\begin{rmk}
Above we only define parametrices for $\dbar$. 
Each $\Phi \in {\rm Param}$ automatically determine a parametrix for the rank $n$ formal $\beta\gamma$ system
by taking $\Phi \otimes (\id_{\fgn} + \id_{\fgn^\vee})$.
Given this relationship, we will not overload the notation and use $\Phi$, 
with the implicit understanding that the algebraic factor is included in the rank $n$ case.
\end{rmk}

We now define versions of the propagator and BV Laplacian for each parametrix $\Phi$,
analogous to $P_\epsilon^L$ and $\Delta_L$ from earlier.
Note that for the rank $n$ formal $\beta\gamma$ system, 

\begin{dfn}
Given a parametrix $\Phi$, the \emph{$\Phi$-propagator} is the integral kernel
\[
P_{\Phi}=\frac{1}{2} (\dbar^* \otimes \id + \id \otimes \dbar^*)\Phi.
\]
Let $\kappa_\Phi$ denote the integral kernel $K_\id - (\dbar \otimes \id + \id \otimes \dbar)P_\Phi$.
\end{dfn}

The crucial point here is that the kernel $\kappa_\Phi$ is smooth. Moreover, it is the analog of the analytic heat kernel $K_t$ from earlier.

We are now in a position to define a mollified BV Laplacian.

\begin{dfn}
The \emph{$\Phi$-BV Laplacian} $\Delta_\Phi$ is the operator $\partial_{\kappa_\Phi}$.
That is, it is the endomorphism of $(\Obs^\cl_n(\CC))^\sharp$ --- the underlying graded algebra of observables --- given by contracting with $\kappa_\Phi$. 
\end{dfn}

For clarity's sake, let us describe this operator explicitly.
Given $a$ in the $n$th symmetric power of the observables, pick a lift $\tilde{a}$ to the $n$th tensor power.
Then 
\[
(\partial_{\kappa_\Phi}a)(x) = \tilde{a}(\kappa_\phi \otimes x \otimes \cdots \otimes x),
\]
where we insert $n-2$ copies of $x$ on the right hand side.

These definitions allow one to define effective field theories, but with length scale replaced by a choice of parametrix.
For a full treatment, see Section 8.2.9 of \cite{CG2}.
The essential changes are that 
\begin{itemize}
\item RG flow from $\Phi$ to $\Psi$ is given by $W(P_\Phi - P_\Psi,-)$, using the same Feynman diagram expansion, and
\item the same local functional should be recovered in the limit as the support of the parametrices goes to the small diagonal.
\end{itemize} 
We can obtain such an effective field theory from the length scale version 
by RG flow to any parametrix from a fixed almost-parametrix for some length scale.

\subsubsection{Observables}

We can now mimic the scale $L$ definition of global observables.

\begin{dfn}
For a parametrix $\Phi$, the  \emph{global $\Phi$-quantum observables for the rank $n$ formal $\beta\gamma$ system} 
has underlying graded vector space  
\[
\cSym(\Bar{\Omega}_c^{1,*}(\CC)^{\oplus n}[1] \oplus \Bar{\Omega}^{0,*}_c(\CC)^{\oplus n}[1])[[\hbar]]
\]
with differential $\dbar + \hbar \Delta_\Phi$. 
We denote it $\Obs^\q_n[\Phi](\CC)$.
\end{dfn}

Again, Lemma 9.3.1.2 of \cite{CG2} shows that the quantum observables are isomorphic for any choice of parametrix.
The isomorphism is explicitly given by the ``derivative of the RG flow'':
for $f \in \Obs^\q_n[\Phi](\CC)$, set
\[
W_\Phi^\Psi(f) = \frac{\partial}{\partial \delta}(W(P_\Psi - P_\Phi, \delta f)) 
\]
in $\Obs^\q_n[\Psi](\CC)$.
Because our parametrices have proper support, though, they only expand the support of an observable $f$ by a controlled amount 
(essentially determined by how the size of the parametrix's support).

\begin{dfn}
Let $\Obs^\q_n(\CC)$ denote the cochain complex of quantum observables up to isomorphism.
That is, an \emph{observable} $f \in \Obs^\q_n(\CC)$ is a family of elements $\{f[\Phi]\}_{\Phi \in {\rm Param}}$ for every parametrix $\Phi$ 
such that the $f[\Psi] = W_\Phi^\Psi(f[\Phi])$ for any pair of parametrices $\Psi$ and $\Phi$.
\end{dfn}

\begin{dfn}
Let $f$ be an observable in $\Obs^\q_n(\CC)$ and denote its Taylor expansion by
\[
f = \sum_{j,k \geq 0} \hbar^j f_{j,k},
\]
with $f_{j,k}$ in the $k$th symmetric power. 
We say that $f$ has \emph{support in $U \subset \CC$} if 
for every $(j,k)$, there is some compact subset $C \subset U^k$ and some parametrix $\Phi$
such that $\supp(f_{j,k}[\Psi]) \subset C$ for all $\Psi \leq \Phi$.
\end{dfn}

By Lemma 9.4.0.2 of \cite{CG2}, the graded vector space $\Obs^\q_n(U)$ of observables with support in $U$
is preserved by the differential $\dbar + \hbar \Delta$ and hence provides a sub-complex of $\Obs^\q_n(\CC)$.
The remainder of Chapter 9 of \cite{CG2} shows that these naturally form a factorization algebra.

\begin{dfn}\label{noneq quantum obs}
Let $\Obs^\q_n$ denote the factorization algebra on $\CC$ of quantum observables for the rank $n$ formal $\beta\gamma$ system.
\end{dfn}

We remark again that $\GL_n(\CC)$ acts naturally on $\CC^n \cong \fgn[1]$
and on its linear dual so as to preserve the evaluation pairing.
Hence $\GL_n(\CC)$ also acts naturally on $\Obs^\q_n(U)$ for any open $U \subset \CC$.
This action respects the differential $\dbar + \hbar \Delta_\Phi$, 
since $\dbar$ only depends on the source $U$ and not on the target $\hD^n$
and $\Delta_\Phi$ depends on the target only through the evaluation pairing.

\subsection{The $\Vect$-equivariant classical observables}

We discussed in Section 3 \ref{classicalvectaction} that 
diffeomorphisms on the target of the curved $\beta\gamma$ system 
naturally act on the fields by post-composition.
In Section 3.3 we gave an efficient description of this action for the formal $\beta\gamma$ system
via an $\L8$-action of $\Vect$ on $\DD\fgn^U$ for any open $U \subset \CC$.
This action then determines an $\L8$-action of $\Vect$ on $\Obs^\cl_n(U) = \clies(\DD\fgn^U)$
and hence a cochain complex $\clies(\Vect,\Obs^\cl_n)$.
In other words, by the yoga of Koszul duality, 
this action can be encoded as a modification of the differential on the tensor product $\clies(\Vect) \otimes \Obs^\cl_n$.
By Lemma \ref{Noether} we know that $\{I^{\rm W},-\}$ provides this twisting of the differential.
Since $I^{\rm W}$ is a local functional, this modified differential is still local in the source manifold $\CC$
and thus respects the structure maps of the factorization algebra.
The following definition gathers together these observations.

\begin{dfn}
The \emph{factorization algebra of $\Vect$-equivariant classical observables} on $\CC$ is
\[
\eqObs^\cl_n = \clies(\Vect, \Obs^\cl_n).
\]
The underlying graded vector space is
\[
\cSym(\Vect^\vee[-1]) \otimes \cSym((\Bar{\Omega}^{1,*}_c)^{\oplus n}[1] \oplus (\Bar{\Omega}^{0,*}_c)^{\oplus n}[1])
\]
with differential $\d_{\clies(\Vect)} + \dbar + \{I^{\rm W},-\}$.
\end{dfn}

Note that in $I^{\rm W}$, the dependence on the vector field $X \in \Vect$ is linear. 
Hence $\Obs^\cl_n$ has a strict Lie algebra action of $\Vect$, not a complicated $\L8$-action.
In light of the remarks following Definition \ref{noneq quantum obs}, we see the following,
which we record as a lemma for use when applying Gelfand-Kazhdan descent in Section \label{sec ss GK descent}.

\begin{lemma}
The classical observables $\Obs^\cl_n$ are a representation of the Harish-Chandra pair~$(\Vect,\GL_n)$.
In particular $\GL_n$ acts by (strict) automorphisms of the factorization algebra, and
$\Vect$ acts by (strict) derivations of the factorization algebra. Via
restriction along $p : (\TVect, \GL_n) \to (\Vect, \GL_n)$ the
classical observables $\Obs^{\cl}_n$ are also a representation for the pair
$(\TVect, \GL_n)$. 
\end{lemma}

\subsection{The $\TVect$-equivariant quantum observables}

The construction of the $\TVect$-equivariant quantum observables is straightforward,
given the work we did in Section \ref{sec equiv bv}.
The logic is analogous to the case of classical observables:
we encode the $\L8$-action of $\TVect$ on observables in the differential.

\begin{dfn}
The \emph{factorization algebra of $\TVect$-equivariant quantum observables} on $\CC$ is  
\[
\eqObs^\q_n = \clies(\TVect, \Obs^\cl_n)[[\hbar]].
\]
The underlying graded vector space is
\[
\cSym(\TVect^\vee[-1]) \otimes \cSym((\Bar{\Omega}^{1,*}_c)^{\oplus n}[1] \oplus (\Bar{\Omega}^{0,*}_c)^{\oplus n}[1])[[\hbar]]
\]
with differential 
\[
\d_{\clies(\TVect)} + \dbar + \{I^{\rm W,0},-\} + \hbar \Delta + \hbar \{I^{\rm W,1}+J,-\},
\]
to give an explicit description.
\end{dfn}

For clarity's sake let us point out that this means that for each parametrix $\Phi$, we have global observables
\[
\cSym(\TVect^\vee[-1]) \otimes \cSym(\Bar{\Omega}^{1,*}_c(\CC)^{\oplus n}[1] \oplus \Bar{\Omega}^{0,*}_c(\CC)^{\oplus n}[1])
\]
with differential 
\[
\d_{\clies(\TVect)} + \dbar + \{I^{{\rm W},0}[\Phi],-\}_\Phi + \hbar \Delta_\Phi + \hbar \{I^{{\rm W},1}[\Phi]+J[\Phi],-\}_\Phi.
\]
These observables are isomorphic for all choices of parametrix,
so that our notation in the definition should be unambiguous.
Moreover, we find that the notion of support for an observable is well-behaved 
and so we can talk about the observables with support in a fixed open $U$,
thus obtaining a factorization algebra.

\begin{rmk}
Working over the base ring $\clies(\TVect)$ amounts to a version of the background field method,
where we view the quantum action functional (encoded in the differential) as depending on a choice of vector field and closed 2-form,
i.e., an element of $\TVect$.
In a sense we see that after quantizing, we obtain extended symmetries
of the theory which we have already seen coincide with those of the
physical curved $\beta\gamma$ system.
\end{rmk}

By contrast to the classical case, the quantum observables $\Obs^\q_n$ do not have a strict Lie action of $\TVect$.
The $\hbar$-term $I^{{\rm W},1}$ is not linear in $\Vect$ and has
contributions of every even power.
Thus we cannot apply strict Gelfand-Kazhdan descent for
$(\TVect,\GL_n)$. We have already observed that when restricted to
linear vector fields $\gl_n \hookrightarrow \TVect$ that the anomaly
vanishes. Thus, the quantum observables $\Obs^\q_n$ admit a {\it strict} action by~$\gl_n$.

\subsection{An aside on the two versions of non-equivariant observables}
\label{noneqsec}

As mentioned earlier, there is another approach to constructing the non-equivariant factorization algebra of observables
for the formal $\beta\gamma$ system, which is developed in \cite{CG1}.
We sketch it briefly here and prove that it is quasi-isomorphic to the observables described above.

Thus, the key idea is to work with observables built out of smooth or smeared distributions.
By contrast, the observables already introduced live in a completed symmetric algebra of distributions 
(more precisely, the distributions dual to Dolbeault forms),
and the need for parametrices is due to inability to apply the BV Laplacian to such distributions,
since distributions do not always pair.

Here is a concrete example of replacing distributions with smeared versions.
Consider the delta-function 
\[
\delta_0: \gamma \mapsto \gamma(0).
\]
Now pick a compactly-supported smooth function $f: (0,1) \to \RR$ such that $\int_\RR f(t) \, \d t = 1$.
Then a smeared version is
\[
\tilde{\delta}_0: \gamma \mapsto \frac{1}{2\pi i}\int_{r = 0}^1 \int_{|z|=r} \frac{\gamma(z)}{z} \d z\, f(r) \,\d r,
\]
which agrees with $\delta_0$ if $\gamma$ is holomorphic, by Cauchy's theorem.
In particular, in the cochain complex $\Bar{\Omega}^{1,*}_c(\CC)[1]$, 
these distributions $\delta_0$ and $\tilde{\delta}_0$ are cohomologous $0$-cocycles.

\def\fr{{\rm fr}}

\begin{dfn}
The \emph{smeared quantum observables for the rank $n$ formal $\beta\gamma$ system} 
with support in the open $U \subset \CC$ has underlying graded vector space
\[
\cSym({\Omega}^{1,*}_c(U)^{\oplus n}[1] \oplus {\Omega}^{0,*}_c(U)^{\oplus n}[1])[[\hbar]]
\]
with differential $\dbar + \hbar \Delta$. 
We denote it $\Obs^{\q,\fr}_n(U)$.
\end{dfn}

As the observables are built out of smooth sections, the ``naive'' BV Laplacian $\Delta = \partial_{K_{id}}$ is well-defined.
We view this operator as the BV Laplacian ``at scale zero," since $K_{id}$ is the distributional limit of the $K_L$.
Moreover, since $\Delta$ is fully local, these smeared observables automatically form a factorization algebra,
with no need to discuss support issues.

This construction raises the question of how the smeared observables compare to the observables from Definition \ref{noneq quantum obs}.
They are, in fact, quasi-isomorphic factorization algebras, but the quasi-isomorphism is built in two steps.
First, on smeared observables, the RG flow operator makes sense from ``scale zero" to an arbitrary parametrix $\Phi$:
\[
W_0^\Phi: \Obs^{\q,\fr}_n(\CC) \to \Obs^{\q,\fr}_n[\Phi](\CC)
\]
where the target $\Phi$-observables consists of the same graded vector space of smeared observables 
but with differential $\dbar +\hbar \Delta_\Phi$.
This map is an isomorphism of cochain complexes with inverse $W_\Phi^0$.
(It does affect support of observables, but we say an observable $f \in \Obs^{\q,\fr}_n[\Phi](\CC)$ is supported in an open set $U$
if $W_\Phi^0(f)$ is supported in $U$.)
Second, consider the inclusion
\[
i[\Phi]: \Obs^{\q,\fr}_n[\Phi](\CC) \hookrightarrow \Obs^{\q}_n[\Phi](\CC),
\]
arising from the inclusion of smooth sections into distribution sectionals. 
This map is a quasi-isomorphism: the spectral sequence arising from the $\hbar$-filtration is an isomorphism on the first page.
The composite $i \circ W_0^\Phi$ thus defines a quasi-isomorphism of cochain complexes,
and it intertwines support conditions, thus extending to a map $i: \Obs^{\q,\fr}_n \to \Obs^{\q}_n$.
Hence we have proved the following.

\begin{prop}
\label{equiv of noneq}
The map $i: \Obs^{\q,\fr}_n \to \Obs^{\q}_n$ is a quasi-isomorphism of factorization algebras.
\end{prop}

\section{Semi-strict Gelfand-Kazhdan descent}
\label{sec ss GK descent}

In Section \ref{sec formal}, we have seen that there is a dg Lie algebra $\DD \fgn^S$ 
encoding the $\beta\gamma$ system with target the formal disk $\hD^n$, 
and we have seen that this dg Lie algebra has a natural action of $\GL_n$ and has a natural $\L8$-action of $\Vect$. 
We might hope that the curved $\beta\gamma$ system with target a complex $n$-manifold $X$ could be obtained 
by applying Gelfand-Kazhdan descent to this dg Lie algebra.
This hope is not misplaced, as we'll see, 
but it requires generalizing the formalism of Harish-Chandra descent to allow for $\L8$-actions of the Lie algebra.

In this section we develop this formalism along the lines of our treatment of descent in Part I,
but we develop the minimum necessary to realize our primary goal and hence leave untreated many interesting questions
(such as allowing Harish-Chandra pairs in which the Lie algebra is replaced by an $\L8$ algebra).
Nonetheless, our techniques should apply to a broad collection of situations, 
notably to constructing the perturbative part of a nonlinear $\sigma$-model using BV quantization.
Indeed, much of what we do is a re-articulation of the methods of Kontsevich, Cattaneo-Felder, and many others, 
that is compatible with the machinery of \cite{CG1, CG2}.
We finish by explaining how our methods recover Costello's approach to the curved $\beta\gamma$ system in \cite{WG2}.
(His use of $\L8$-spaces, however, allows for more exotic targets than just complex manifolds, though.)

\begin{rmk}
As our particular examples are explicit, we are able to get away with 
a modest and quite limited generalization of Gelfand-Kazhdan descent for derived objects.
There should be a full-fledged derived version. 
(Parts of \cite{CPTVV} can be seen as a giant step in that direction.)
\end{rmk}

\subsection{Semi-strict modules}

We continue to work with Harish-Chandra pairs $(\fg,K)$, as in Part 1,
so $\fg$ is a Lie algebra and $K$ is a Lie group 
along with an action $\rho$ of $K$ on $\fg$ and an inclusion of Lie algebras $i: \Lie(K) \hookrightarrow \fg$
so that the Lie algebra action determined by $i$ agrees with the differential of the group action $\rho$.

\begin{dfn} 
\label{dfn ss HC mod}
A {\em semi-strict Harish-Chandra module} for the pair $(\fg,K)$ is a dg vector space $(V,\d_V)$ equipped with
\begin{itemize}
\item[(i)] a strict group action $\rho^K_V$ of $K$, meaning a group map 
\ben
\rho^K_{V^d} : K \to \GL(V^d)
\een 
for each degree $d$ such that the product map $\prod_d \rho^K_{V^d}: K \to \prod_d \GL(V^d)$ commutes with the differential~$\d_V$;
\item[(ii)] an $L_\infty$-action of $\fg$ on $V$, i.e., a map of $L_\infty$-algebras $\rho^\fg_V : \fg \l8to \End(V)$,
such that the composite 
$$\cliels(\rho^\fg_V) \circ \cliel_*(i): \cliels(\Lie(K)) \to \cliels(\End(V))$$
equals the map 
$$\cliel_*(D\rho^K_V): \cliels(\Lie(K)) \to \cliels(\End(V)).$$ 
\end{itemize}
Here $D \rho^K_V : \Lie(K) \to \End(V)$ is the differential of the strict $K$-action and $i: \Lie(K) \to \fg$ is part of the data of the Harish-Chandra pair $(\fg,K)$.
\end{dfn}

We call this \emph{semi-strict} because we allow an $\L8$-action of $\fg$,
but our other conditions are quite strict.
This definition, while {\it ad hoc}, is nonetheless well-suited to our situation.

\begin{dfn}
A {\em map of semi-strict Harish-Chandra modules} 
\[
f: (V,\rho^K_V,\rho^\fg_V) \to (W,\rho^K_W,\rho^\fg_W)
\]
consists of
\begin{itemize}
\item[(i)] a cochain map $f^K: V \to W$ that (strictly) intertwines the $K$-actions and
\item[(ii)] a map of $\cliel_*(\fg)$-comodules
\[
f^\fg: \cliel_*(\fg,V) \to \cliel_*(\fg,W),
\]
such that the composites
\[
\cliels(D\rho^K_W) \circ \cliels(f^K): \cliels(\Lie(K),V) \to \cliel_*(\fg,W)
\]
and
\[
f^\fg \circ \cliels(D\rho^K_V): \cliels(\Lie(K),V) \to \cliel_*(\fg,W)
\]
are identical.
\end{itemize}
\end{dfn}

\subsection{Semi-strict descent}

Fix a $(\fg,K)$-bundle $P$ with flat connection, 
so that there is a Maurer-Cartan element $\omega$ in the dg Lie algebra $\Omega^*(P) \otimes \fg$. 
Equivalently, there is a map of commutative dg algebras
\[
\omega^*: \clies(\fg) \to \Omega^*(P)
\]
determined by extending to an algebra map, 
the map $\omega^*: \fg^\vee[-1] \to \Omega^*(P)$ on generators encoded by~$\omega$.

Let $V$ be a semi-strict module for the pair $(\fg,K)$. 
Hence there is a map of commutative dg algebras
\[
\rho^{\fg*}_V: \clies(\End(V)) \to \clies(\fg),
\]
which is the linear dual of the coalgebra map $\rho^\fg_V: \fg \l8to \End(V)$. 
By composing, we obtain a map of commutative dg algebras
\[
\rho^{\fg*}_V \circ \omega^*: \clies(\End (V)) \to \Omega^*(P),
\]
which then corresponds to a Maurer-Cartan element 
\[
\omega_V \in \Omega^*(P) \otimes \End (V).
\]
The operator
\ben
\nabla^{P,V} := \d_{dR} + \omega_V
\een
then defines a flat ``super-connection'' on the trivial bundle $P \times V \to P$ over $P$.
(Here ``super'' simply means that some terms of $\omega_V$ may contain higher forms, and not just one-forms.)

The following results straightforwardly from the definitions.

\begin{lemma} 
The operator $\nabla^{P,V}$ has the following properties:
\begin{itemize}
\item[(1)] It preserves the sub-algebra of basic forms. 
\item[(2)] If $f \in \cO(X) \cong \cO(P)^K$ and $\alpha \in  (\Omega^k(P) \otimes V)_{bas} \cong \Omega^k(X; V_X)$,
then
\ben
\nabla^{P,V}(f \cdot \alpha) = (\d_{dR} f) \otimes \alpha + f \otimes \nabla^V \alpha .
\een
\item[(3)] It is square-zero. 
\end{itemize}
\end{lemma}

Using this lemma, we define the cochain complex
\be\label{ddesc}
\ddesc((P\to X, \omega), V) := \left( (\Omega^*(P) \otimes V)_{bas}, \nabla^{P,V}\right) .
\ee
It is a dg module over the commutative dg algebra $\Omega^*(X)$. 

\begin{dfn}
The \emph{semi-strict descent functor}
\ben
\ddesc : \Loc_{(\fg,K)}^{\rm op} \times \Mod^{fin}_{(\fg,K)} \to
\Mod_{\Omega^*(X)} 
\een 
is given by the construction just described.
\end{dfn}

Note that if $V$ is strict HC-module, then (\ref{ddesc}) is just the de Rham complex of the flat vector bundle 
$\desc((P \to X,\omega), V) = (V_X, \nabla^{P,V})$ from Definition~\ref{def desc}.

\emph{Semi-strict Gelfand-Kazhdan descent} is simply semi-strict Harish-Chandra descent
applied to the pairs $(\Vect,\GL_n)$ or $(\TVect, \GL_n)$  along $X^{coor}$ or $\Tilde{X}^{coor}_\alpha$, respectively. 
Everything is parallel to what we did in Part~1. In particular, it is lax monoidal, via the argument from Lemma \ref{prop lax}.

We now note an important relationship between strict and semi-strict descent,
which follows from a standard fact about $\L8$-representations:
given an ordinary  Lie algebra $\fg$ (i.e., concentrated in degree zero) and an $\L8$ representation $V$ of $\fg$,
the cohomology $H^*(V)$ is a strict representation of $\fg$.
Hence we observe the following.

\begin{lemma}
\label{lem strict vs semistrict}
If $(V, \rho^K_V, \rho^\fg_V)$ is a semi-strict module for the pair $(\fg,K)$,
then $H^*(V)$ naturally becomes a strict module for $(\fg,K)$ with
$\rho^K_{H^*(V)}$ the induced action of $K$ on $H^*(V)$ (since it respects the differential on $V$)
and $\rho^\fg_{H^*(V)}$ the induced strict action of $\fg$ on $H^*(V)$.
\end{lemma}

\subsection{Descent of the equivariant observables}

We record the following immediate consequences of our work in Section
\ref{sec equiv obs}.

\begin{prop}
For each open $U \subset \CC$,
\begin{enumerate}
\item the classical observables $\Obs^\cl_n(U)$ is a strict module over $(\Vect,\GL_n)$, and
\item the quantum observables $\Obs^\q_n(U)$ is a semi-strict module over $(\TVect,\GL_n)$.
\end{enumerate}
The structure maps of $\Obs^\cl_n$ are strictly equivariant map for $(\Vect,\GL_n)$, 
i.e., maps of strict $(\Vect,\GL_n)$-modules.
The structure maps of $\Obs^\q_n$ are maps of semi-strict $(\TVect,\GL_n)$-modules.
\end{prop}

These assertions follow by reinterpreting, via Koszul duality, 
our descriptions of the equivariant observables as dg modules 
over $\clies(\Vect)$ (in the classical case) or $\clies(\TVect)$ (in the quantum case).

We can thus apply semi-strict Gelfand-Kazhdan descent and obtain the following result.
Note that we are working in the category of dg modules over the commutative dg algebra $\Omega^*(X)$ 
in the category of differentiable vector spaces discussed in Section~\ref{sec functional analysis}.

\begin{cor}
The strict Gelfand-Kazhdan descent of $\Obs^\cl_n$ on an $n$-dimensional complex manifold $X$ 
is a commutative factorization algebra in dg modules over $\Omega^*(X)$.
It depends on a choice of Gelfand-Kazhdan structure $(\Fr\, X, \sigma)$,
but every choice produces a naturally isomorphic factorization algebra.

If the $n$-dimensional complex manifold $X$ has vanishing $\ch_2(T_X)
\in H^2(X, \Omega^{2,hol}_{cl})$, then each extended Gelfand-Kazhdan structure $(X, \alpha, \sigma,
\sigma_{\Omega^2})$ where $\alpha$ is a choice of trivialization of
$\ch_2(T_X)$ and $\sigma,\sigma_{\Omega^2}$ are the auxiliary sections
needed to define descent (whose choices do not change the descent object up to homotopy)
the semi-strict Gelfand-Kazhdan descent of $\Obs^\q_n$ is a factorization algebra in dg modules over $\Omega^*(X)$.
\end{cor}

\begin{dfn}
For $X$ a complex $n$-manifold, let $\Obs^\cl_X$ denote the commutative factorization algebra on $\CC$ 
produced by strict Gelfand-Kazhdan descent of $\Obs^\cl_n$.

For $X$ a complex $n$-manifold and $\alpha \in \Omega^2_{cl}(X)$ a trivialization of $\ch_2(T_X)$,
let $\Obs^\q_{X,\alpha}$ denote the factorization algebra on $\CC$
produced by semi-strict Gelfand-Kazhdan descent of $\Obs^\q_n$.
\end{dfn}

By Lemma \ref{lem strict vs semistrict} we know that for each open $U \subset \CC$,
\begin{enumerate}
\item the cohomological classical observables $H^* \Obs^\cl_n(U)$ is a strict module over $(\Vect,\GL_n)$, and
\item the cohomological quantum observables $H^*\Obs^\q_n(U)$ is a strict module over $(\TVect,\GL_n)$.
\end{enumerate}
Moreover, the structure maps of $H^*\Obs^\cl_n$ are strictly equivariant map for $(\Vect,\GL_n)$, 
i.e., maps of strict $(\Vect,\GL_n)$-modules.
Likewise, the structure maps of $H^*\Obs^\q_n$ are maps of strict $(\TVect,\GL_n)$-modules.
Thus we can also apply strict Gelfand-Kazhdan descent to the cohomology of observables.

\begin{cor}
The strict Gelfand-Kazhdan descent of $H^*\Obs^\cl_n$ on an $n$-dimensional complex manifold $X$ 
is a commutative factorization algebra.
It depends on a choice of Gelfand-Kazhdan structure $(\Fr\, X, \sigma)$,
but every choice produces a naturally isomorphic factorization algebra.

If the $n$-dimensional complex manifold $X$ has vanishing $\ch_2(T_X) \in H^2(X, \Omega^{2,hol}_{cl})$,
then for each choice of trivialization $\alpha$ of $\ch_2(T_X)$ and each extended Gelfand-Kazhdan structure
the strict Gelfand-Kazhdan descent of $H^*\Obs^\q_n$ is a factorization algebra.
\end{cor}

In Section \ref{sec concrete} below, we provide a description of these
factorization algebras that is humanly understandable,
but first we will swiftly relate our work to Costello's approach in \cite{WG2}.

\subsection{Comparison with Costello's work} \label{sec comparison}

In \cite{WG2} Costello provided a BV quantization of the curved $\beta\gamma$ system with target a complex manifold $X$,
and it was clear that the associated factorization algebra ought to be chiral differential operators,
based on the work in \cite{WittenCDO,Nek}. 
Our work grew out of attempts to verify that expectation.
Here we explain how Gelfand-Kazhdan descent recovers the $\L8$ spaces that Costello uses
and why descent of our equivariant quantization recovers the relevant cases of Costello's quantizations.
These results are independent of the rest of the text, and hence the disinterested reader should skip this section.

Our construction starts by encoding the formal $n$-disk as an $L_\infty$ algebra $\fgn$
and the formal $\beta\gamma$ system as $\DD\fgn^S$.
Costello's approach is to write down a global analogue:
for each complex manifold $X$, he constructs
a curved $L_\infty$ algebra $\fg_X$ in dg modules over the de Rham complex $\Omega^*(X)$.
His version of the classical curved $\beta\gamma$ system is encoded in $\DD\fg_X^S$,
whose Maurer-Cartan equation recover the equations of motion.
(More precisely, this Maurer-Cartan equations describes formal deformations of constant maps to holomorphic maps.)
The factorization algebra of classical observables assigns to an open set $U \subset \CC$,
the cochain complex~$\clies(\DD\fg_X^U)$.
The quantum observables are a deformation thereof.

Let us explain his construction of $\fg_X$.
Consider the $\infty$-jet bundle $J^{hol}_X$ for holomorphic functions,
which has a canonical flat connection.
The sheaf of horizontal sections for this flat connection is exactly the sheaf $\cO_X$ of holomorphic functions on $X$.
In fact, the de Rham complex of $J^{hol}_X$ is quasi-isomorphic to $\cO_X$, 
where the quasi-isomorphism sends a holomorphic function to its $\infty$-jet.
By definition, $\fg_X$ is the curved $\L8$ algebra encoded under Koszul duality by the commutative dg algebra
\[
\clies(\fg_X) = \Omega^*(X, \cSym(T^{1,0*}_X)) \cong \Omega^*(X, J_X^{hol}).
\]
(Everything here is in modules over $\Omega^*(X)$.)
The differential on the left hand side is pulled back along an isomorphism of pro-vector bundles 
$\sigma: \cSym(T^{1,0*}_X) \xto{\cong} J^{hol}_X$.
This isomorphism $\sigma$ is constructed by fixing a connection on the tangent bundle $T_X$
and using its associated exponential map at each point $x$ 
to identify the formal neighborhood of $x$ in $X$ with the formal neighborhood of the origin in $T_x X$.
In this way, the $\infty$-jet of a function at $x$ is identified with a formal power series in $T_x^* X$,
which is the desired isomorphism $\sigma$.

But this procedure is precisely how Gelfand-Kazhdan descent works! 
Once we fix a formal exponential on the frame bundle of $X$ --- typically via a choice of connection --- we have an isomorphism $\sigma$.
Moreover, the descent of $\clies(\fgn)$ using this data is exactly $\Omega^*(X, \cSym(T^{1,0*}_X))$
equipped with the pullback of the Grothendieck connection along $\sigma$.
In other words, Gelfand-Kazhdan descent recovers Costello's curved $\L8$ algebra,
once one applies the Koszul duality.

A parallel argument applies to $\DD\fg_X^S$.
After fixing the isomorphism $\sigma$, Gelfand-Kazhdan descent of $\DD\fgn^S$ 
produces $\clies(\DD\fg_X^S)$ on the nose.
Hence, under Koszul duality, we recover Costello's classical BV theory 
as encoded in the curved $\L8$ algebra~$\DD\fg_X^S$.

A careful reading of \cite{WG2} will show that his Feynman diagrammatic work is the global version of ours: 
our analysis of the obstructions to quantization and constructions of quantizations given a trivialized obstruction
is directly parallel and descends to his.

Our discussion can be summarized as follows.

\begin{prop}
Under Gelfand-Kazhdan descent on a complex manifold $X$, 
the formal $\beta\gamma$ system recovers the classical BV theory associated to $X$ in \cite{WG2}.
Moreover, the obstruction-deformation complex descends to that in \cite{WG2},
so that the obstruction to BV quantization recovers the obstruction identified in \cite{WG2}.
Finally, given a trivialization of this obstruction, descent recovers the quantized action functional in~\cite{WG2}.
\end{prop}

The primary corollary of this is that the factorization algebra associated to Costello's quantization of the curved $\beta\gamma$ system with target $X$ is isomorphic to the factorization algebra $\Obs^\q_{X,\alpha}$ we have constructed via semi-strict Gelfand-Kazhdan descent of the formal $\beta\gamma$ system with target~$\hD^n$. 

Moreover, we recover the Witten class as originally obtained by Costello. That
is, we have already identified the equivariant scale $\infty$ interaction
$\Tilde{I}^{\rm W}[\infty]$ on an elliptic curve $E$ with the Witten genus
\ben
p^* \log \Wit_n (E, \omega) \in \clie^*(\TVect , \GL_n ; \hOmega^{-*}_n) .
\een 
Consider the characteristic map defined by extended Gelfand-Kazhdan
descent determined by a trivialization $\alpha$ of the second Chern
character of $X$. It is given by
\ben
\Tilde{\ch}_\alpha : H_{\rm Lie}^*(\TVect ,\GL_n ; \hOmega^{-*}_n) \to
H^*(X ; \Omega^{-*}_X) .
\een
The image of $p^* \Wit_n(E, \omega)$ under this map is the logarithmic
Witten genus of the complex manifold~$X$ 
\ben
\log \Wit(X,E, \omega) = \sum_{k \geq 2} \frac{(2k-1)!}{(2 \pi i)^{2k}}
E_{2k}(E, \omega) \ch_{2k}(T_X) ,
\een 
described using a holomorphic volume element on the elliptic curve~$E$.

\section{A concrete description of the observables}
\label{sec concrete}

In this Section we examine the factorization algebras $\Obs^\cl_X$ and $\Obs^\q_{X,\alpha}$ 
produced by descent of the equivariant observables for the formal $\beta\gamma$ system.
Our goal is to extract information from them that is easy to interpret,
particularly from the physical point of view.
For instance, we will give an explicit description of observables with support at a point in the source $\CC$ --- and hence also for observables supported at finitely many points --- which is a bridge to Part III, where we show that the cohomological factorization algebras
$H^*\Obs^\cl_X$ and $H^*\Obs^\q_{X,\alpha}$ recover the vertex algebras $\Gr\CDO_X$ and $\CDO_{X,\alpha}$, respectively.
In short, we show that these point observables admit explicit expressions in terms of natural geometric objects on the target manifold, 
notably tensor bundles.

Throughout we use the tensor product and symmetric powers described in
Section~\ref{sec functional analysis}.

\subsection{Polynomials, power series, and the $(\Vect,\GL_n)$-decomposition of observables}

We will provide a characterization of the formal tensor fields 
that constitute the observables for the formal $\beta\gamma$ system.
We will use this characterization in the next section
for the non-formal observables.

Before talking about the full algebras of observables, 
it is useful to understand the space of linear observables,
which are simply the dual space to the fields.
It will help to bear in mind some simple facts about smooth functions.

For any disk $D_r(0) \subset \CC$ centered at the origin, 
there is a natural linear map
\[
j: C^\infty(D_r(0)) \to \CC[[z,\zbar]]
\]
sending a function $f$ to its Taylor series $j(f)$ at the origin.
(We use the coordinates $z$ and $\zbar$ since we will eventually focus on holomorphic functions.)
Borel's lemma tells us this map is surjective.
There is also an inclusion
\[
\CC[z,\zbar] \hookrightarrow C^\infty(D_r(0))
\]
obtained by viewing a polynomial as a function on the disk,
and the composite with $j$ is the inclusion of polynomials into power series.

The $\dbar$ operator makes sense on both polynomials and power series.
Let $\Omega^{0,*}_{poly}$ denote the cochain complex
\[
\CC[z,\zbar] \xto{\dbar} \CC[z,\zbar]\d\zbar
\]
and let $\Omega^{0,*}_{pow}$ denote the version with power series.
We have the following relationship.

\begin{lemma}
There is a commuting diagram
\[
\begin{tikzcd}
\CC[z] \ar[hook]{r} \ar[hook]{d}{\simeq} & \cO(D_r(0)) \ar[hook]{r} \ar[hook]{d}{\simeq} & \CC[[z]] \ar[hook]{d}{\simeq} \\
\Omega^{0,*}_{poly} \ar[hook]{r} & \Omega^{0,*}(D_r(0)) \ar[two heads]{r}{j} & \Omega^{0,*}_{pow}
\end{tikzcd}
\]
where the vertical maps are the inclusion of the cohomology, which is concentrated in degree zero.
\end{lemma}

Note that at the level of fields --- rather, in terms of the dg Lie algebras encoding the formal $\beta\gamma$ system --- this result tells us that we have
\[
\begin{tikzcd}
\fgn[z] \oplus (\fgn^\vee[z]\d z)[-2] \ar[hook]{r} \ar[hook]{d}{\simeq} & \fgn \otimes \cO(D_r(0)) \oplus (\fgn^\vee \otimes \Omega^{1,hol}(D_r(0)))[-2]  \ar[hook]{r} \ar[hook]{d}{\simeq} & \fgn[[z]] \oplus (\fgn^\vee[[z]]\d z)[-2] \ar[hook]{d}{\simeq} \\
\DD\fgn^{poly} \ar[hook]{r} & \DD\fgn^{D_r(0)} \ar[two heads]{r}{j} & \DD\fgn^{pow}
\end{tikzcd}
\]
where, for example, $\DD\fgn^{poly}$ means the dg Lie algebra 
\[
\Omega^{0,*}_{poly} \otimes \fgn \oplus \Omega^{1,*}_{poly} \otimes \fgn^\vee[-2].
\]
This relationship is convenient for analyzing observables.

\begin{lemma}
The classical observables $\Obs^\cl_n(D_r(0))$ sit inside the commuting diagram
\[
\begin{tikzcd}
\clies(\DD\fgn^{pow}) \ar{r} \ar{d} & \Obs^\cl_n(D_r(0)) \ar{r} \ar{d} & \clies(\DD\fgn^{poly}) \ar{d}\\
\clies(\fgn[[z]] \oplus (\fgn^\vee[[z]]\d z)[-2]) \ar[hook]{r} & H^* \Obs^\cl_n(D_r(0))  \ar[hook]{r} & \clies(\fgn[z] \oplus (\fgn^\vee[z]\d z)[-2])
\end{tikzcd}
\]
by applying the functor $\clies$.
\end{lemma}

These maps naturally intertwine the Harish-Chandra action of $(\Vect,\GL_n)$, 
so that we obtain an analogous commuting diagram after Gelfand-Kazhdan descent.
One must verify that the vertical maps are quasi-isomorphisms, 
which we do below in Proposition \ref{cohomology of point obs}.
But first let us analyze in more detail which $(\Vect,\GL_n)$-representations appear in the observables.

Consider the case of $\clies(\fgn[[z]] \oplus (\fgn^\vee[[z]]\d z)[-2])$,
since it sits inside all the other examples.
Recall that $\clies(\fgn) \cong \hO_n = \CC[[t_1,\ldots,t_n]]$.
The Lie algebra $\fgn[[z]] \oplus (\fgn^\vee[[z]]\d z)[-2]$ can be viewed 
as an extension of $\fgn \simeq \fgn \cdot z^0$ by the representation 
\[
M = \fgn[[z]]z \oplus (\fgn^\vee[[z]]\d z)[-2],
\]
and hence
\[
\clies(\fgn[[z]] \oplus (\fgn^\vee[[z]]\d z)[-2]) \cong \clies(\fgn, \cSym(M^\vee[-1])).
\]
We now show that this vector space (as it all sits in degree zero) is a direct product of tensor fields.

Some notation will simplify the discussion. 
The appropriate linear dual of $\CC[[z]]$ is the direct sum $\bigoplus_{k \geq 0 } \CC \, \zeta_k$,
where $\zeta_k$ is the dual element to $z^k$.
Let $\zeta_k \d z^\vee$ denote the dual to $z^k\d z$.
Then
\[
M^\vee[-1] = \bigoplus_{0<k} (\fgn^\vee \otimes \zeta_k)[-1] \oplus \bigoplus_{0\leq l} (\fgn \otimes \zeta_l \d z^\vee)[1].
\]
More succinctly, we have
\[
M^\vee[-1] \cong \bigoplus_{0<k} \fg^\vee[-1] \oplus \bigoplus_{0\leq l} \fgn[1].
\]
Let $\widehat{\otimes}$ denotes the completed tensor product, 
so that $\cSym(V \oplus W) \simeq \cSym(V)
\,\widehat{\otimes}\,\cSym(W)$ for any pair of vector
spaces. 

Then 
\begin{align*}
\cSym(M^\vee[-1]) 
& \cong \underset{{0<k}}{\widehat{\bigotimes}}\,\cSym(\fgn^\vee[-1]) \,\widehat{\otimes}\, \underset{{0\leq l}}{\widehat{\bigotimes}}\,\cSym(\fgn[1]) \\
&= \underset{K,L \to \infty}{\text{colim}} \underset{{0<k<K}}{\widehat{\bigotimes}}\,\cSym(\fgn^\vee[-1]) \,\widehat{\otimes}\, \underset{{0\leq l<L}}{\widehat{\bigotimes}}\,\cSym(\fgn[1]).
\end{align*}
where the $(k,l)$th tensor term is associated to $\zeta_k$ and $\zeta_l \d z^\vee$.
(Recall that in the infinite tensor product of unital algebras, 
a term $a_1 \otimes a_2 \otimes \cdots$ has $a_j =1$ for all but finitely many $j$.)
In summary we have the following.

\begin{lemma}
As a $(\Vect,\GL_n)$-modules, the commutative algebra $\clies(\fgn[[z]] \oplus (\fgn^\vee[[z]]\d z)[-2])$ 
decomposes as the infinite tensor product of formal tensor fields,
\[
\underset{{0<k}}{\widehat{\bigotimes}}\,\cSym(\hT^*_n)\, \widehat{\otimes}\, \underset{{0\leq l}}{\widehat{\bigotimes}}\,\cSym(\hT_n),
\]
where $\hT_n$ denotes the formal vector fields viewed as an adjoint representation of $\Vect$ and 
$\hT^*_n$ denotes the formal one-forms viewed as the coadjoint representation.
\end{lemma}

Since Gelfand-Kazhdan descent is monoidal, we obtain a useful corollary.

\begin{cor}
\label{lem g[[z]]}
For $X$ a complex $n$-manifold, the Gelfand-Kazhdan descent of $\clies(\fgn[[z]] \oplus (\fgn^\vee[[z]]\d z)[-2])$
is isomorphic to the $\cO_X$-module
\[
\underset{{0<k}}{\widehat{\bigotimes}}\,\cSym_{\cO_X}(\cT^*_X)\, \widehat{\otimes}\, \underset{{0\leq l}}{\widehat{\bigotimes}}\,\cSym_{\cO_X}(\cT_X),
\]
where $\cT_X$ denotes the sheaf of holomorphic vector fields and $\cT^*_X$ denotes the sheaf of holomorphic one-forms.
\end{cor}

Analogous results can be formulated for the polynomial situation 
and also for $\clies(\DD\fgn^{pow})$ and $\clies(\DD\fgn^{poly})$,
which now involve powers of $\zbar$ and $\d \zbar^\vee$ as well.
By the same reasoning as we just use, we find the following.
(Observe that due to having monomials of the form $z^k \zbar^{k'}$,
the indexing now is doubled.)

\begin{lemma}
For the Lie algebra $\fgn[[z,\zbar]] \oplus \fgn^\vee[[z,\zbar]][-2]$,
the commutative algebra $\clies(\fgn[[z,\zbar]] \oplus \fgn^\vee[[z,\zbar]][-2])$ decomposes as a $(\Vect,\GL_n)$-module
into
\[
\underset{{(k,k') \in \NN^2 \{(0,0)\}}}{\widehat{\bigotimes}}\,\cSym(\hT_n^*)\, \widehat{\otimes}\, \underset{{(l,l') \in \NN^2}}{\widehat{\bigotimes}}\,\cSym(\hT_n),
\]
where $(k,k')$ indexes $z^k \zbar^{k'}$ and likewise for $(l,l')$.
\end{lemma}

A little more work provides us with this result.

\begin{lemma}
\label{lem DDfgn^pow}
The underlying graded vector space of commutative algebra $\clies(\DD\fgn^{pow})$ decomposes as a $(\Vect,\GL_n)$-module into
\[
\underset{{(k,k') \in \NN^2 \{(0,0)\}}}{\widehat{\bigotimes}}\,\cSym(\hT_n^*)\, \widehat{\otimes}\,
\underset{{(l,l') \in \NN^2}}{\widehat{\bigotimes}}\,\cSym(\hT_n) \, \widehat{\otimes}\, 
\underset{{(m,m') \in \NN^2}}{\widehat{\bigotimes}}\,\cSym(\hT_n^*[1])\, \widehat{\otimes}\, 
\underset{{(n,n') \in \NN^2}}{\widehat{\bigotimes}}\,\cSym(\hT_n[1]),
\]
where $(k,k')$ indexes $z^k \zbar^{k'}$ and likewise for the other double indices.
\end{lemma}

\begin{proof}
The decomposition of the underlying graded $(\Vect,\GL_n)$-module of $\clies(\DD\fgn^{pow})$ 
is also straightforward, given our work above, but it involves some bookkeeping.
The underlying graded vector space of $\DD\fgn^{pow}$ is
\[
\begin{array}{ccc}
\text{degree:} & \underline{1} & \underline{2} \\
\text{vector space:} & \fgn[[z,\zbar]] \oplus \fgn^\vee[[z,\zbar]][-2] & (\fgn[[z,\zbar]] \oplus \fgn^\vee[[z,\zbar]][-2]) \otimes \d\zbar
\end{array}
\]
Thus, $\clies(\DD\fgn^{pow})$ is concentrated in nonpositive degrees, 
and the preceding lemma gives us the degree zero component,
which we denote $\cA$.
The new contribution $\cB$ comes from the degree two component of $\DD\fgn^{pow}$.
It generates an algebra by linear dual of this component placed in degree~-1:
\[
\cB := \cSym\left( ((\fgn[[z,\zbar]] \oplus \fgn^\vee[[z,\zbar]][-2]) \otimes \d\zbar)^\vee[1] \right).
\]
We  have
\[
\clies(\DD\fgn^{pow})^\sharp \cong \cA \otimes \cB
\]
as graded algebras.
We thus need to have a succinct way to describe the algebra~$\cB$.

Let $\zeta_{m,n} \d \zbar^\vee$ denote the dual element to $z^m \zbar^n \d\zbar$.
Then 
\begin{align*}
\cB &= \cSym \left( \bigoplus_{(k,k') \in \NN^2} \fgn^\vee \, \zeta_{m,n} \d \zbar^\vee \oplus \bigoplus_{(l,l') \in \NN^2} \fgn[2] \, \zeta_{m,n} \d \zbar^\vee\right)\\
&= \underset{{(k,k') \in \NN^2}}{\widehat{\bigotimes}}\,\cSym(\hT_n^*[1])\, \widehat{\otimes}\, \underset{{(l,l') \in \NN^2}}{\widehat{\bigotimes}}\,\cSym(\hT_n[1]).
\end{align*}
Hence we obtain the claim.
\end{proof}

\subsection{The classical observables supported at a point}

The observables $\clies(\DD\fgn^{pow})$ have a natural field-theoretic interpretation:
they are the observables supported at the origin in the source manifold $\CC$.
As we will explain below, these observables map to the observables $\Obs^\cl_n(U)$
supported on any open $U \subset \CC$ containing the origin,
and so they provide a rich source of easily-understood measurements.

Recall that the distributions (i.e., continuous linear functionals on smooth functions) 
supported at the origin in $\CC$ consist of finite linear combinations of the the delta function $\delta_0$ 
supported at the origin and its partial derivatives.
In other words, it consists of linear functionals that read off the Taylor coefficients of a smooth function.
We introduce the notation $\zeta_{m,n}$ for the distribution $\partial_z^m \partial_{\zbar}^n \delta_0$.
Hence, the smooth distributions with support at the origin are the vector space
\[
\mathcal{D}_0 := \bigoplus_{(m,n) \in \NN^2} \CC\, \zeta_{m,n} = (\CC[[z,\zbar]])^\vee.
\]
By our work above, we see that the linear dual of the power series Dolbeault complex $\Omega^{0,*}_{pow}$
is the cochain complex 
\[
\begin{array}{ccc}
\mathcal{D}_0 \,\d\zbar^\vee & \xto{\dbar^\vee} & \mathcal{D}_0 \\
\zeta_{m,n} \,\d\zbar^\vee & \mapsto & \zeta_{m,n+1}
\end{array},
\]
which is concentrated in degrees 0 and -1.
We denote it by $(\Omega^{0,*})^\vee_0$.
An analogous complex $(\Omega^{1,*})^\vee_0$ encodes 
the distributional dual of the Dolbeault complex of 1-forms with support at the origin.

\begin{rmk}
A nice feature of working with a holomorphic field theory, like the $\beta\gamma$ system,
is that using the linear observables supported at the origin,
one can fully identify any solution to the equations of motion.
This fact is the linear dual to the fact that $\cO(\CC) \hookrightarrow \CC[[z]]$,
i.e., every holomorphic function is determined by its power series expansion.
\end{rmk}

The linear observables --- supported at the origin --- of the rank $n$ formal $\beta\gamma$ system are then
\[
(\Omega^{0,*})^\vee_0 \otimes \fgn^\vee[-1] \oplus (\Omega^{1,*})^\vee_0 \otimes \fgn[1].
\]
For us, the algebra of classical observables is the completed symmetric algebra on these linear observables.
Let $(\Obs^\cl_n)_0$ denote the cochain complex of observables 
with support at the origin on the rank $n$ formal $\beta\gamma$ system.
Explicitly, we have
\[
(\Obs^\cl_n)_0 = \cSym((\Omega^{0,*})^\vee_0 \otimes \fgn^\vee[-1] \oplus (\Omega^{1,*})^\vee_0 \otimes \fgn[1]) 
\]
with the differential by extending as a derivation the differential on the linear generators.

We record an immediate consequence of the fact the distributions with compact support extend from smaller to larger open sets.

\begin{lemma}
For any open set $U \subset \CC$ containing the origin, there is a cochain map
\[
(\Obs^\cl_n)_0 \hookrightarrow \Obs^\cl_n(U)
\]
extending the inclusion of the distributions supported at the origin to the distributions with support in $U$.
\end{lemma}

This map is manifestly equivariant with respect to the $(\Vect,\GL_n)$ action and hence descends.

\begin{cor}
For any open set $U \subset \CC$ containing the origin, there is a map of dg $\Omega^*(X)$-modules
\[
(\Obs^\cl_X)_0 \hookrightarrow \Obs^\cl_X(U)
\]
extending the inclusion of the distributions supported at the origin to the distributions with support in $U$.
\end{cor}

\begin{rmk}
Formulating a version of this statement for the quantum observables would be more delicate,
as one must work with parametrices and RG flow.
As we are working with a free theory here, however, one can instead use the ``smoothed observables.''
See below for a discussion of quantum observables.
\end{rmk}

Note that the underlying graded vector space of $(\Obs^\cl_X)_0$ is
\[
\cSym_{\Omega^\sharp(X)}(\Omega^\sharp(X,(\Omega^{0,*})^\vee_0 \otimes T^*_X) \oplus \Omega^\sharp(X,(\Omega^{1,*})^\vee_0 \otimes T_X)) 
\]
as we are working over the base algebra $\Omega^*(X)$.
Equivalently, one can express it as the de Rham complex of a (gigantic!) dg vector bundle 
\[
\Omega^\sharp(X, \cSym((\Omega^{0,*})^\vee_0 \otimes T^*_X \oplus (\Omega^{1,*})^\vee_0 \otimes T_X)).
\]
The differential involves both $\dbar$ for the source manifold $\CC$ and a connection $\nabla$ along the target $X$.

We now turn to determining the cohomology of $(\Obs^\cl_X)_0$, 
which encodes the measurements one can make at the origin of $\CC$ 
of the fields $\gamma$ and $\beta$ with target $X$.

\begin{prop}\label{cohomology of point obs}
There is a natural isomorphism 
\[
H^*(\Obs^\cl_X)_0 \cong 
H^*\left(X,\underset{{0<k}}{\widehat{\bigotimes}}\,\cSym_{\cO_X}(\cT^*_X)\, \widehat{\otimes}\, \underset{{0\leq l}}{\widehat{\bigotimes}}\,\cSym_{\cO_X}(\cT_X)\right),
\]
identifying the cohomological observables supported at the origin with $\clies(\fgn[[z]] \oplus (\fgn^\vee[[z]]\d z)[-2])$.
\end{prop}

\begin{proof}
Before embarking on a spectral sequence, we note that the arguments
for Lemma \ref{lem DDfgn^pow} tell us that we obtain an infinite (completed) tensor product
of tensor bundles from Gelfand-Kazhdan descent,
via the identification of $(\Obs^\cl_X)_0$ with $\clies(\DD\fgn^{pow})$,
so that we view $\DD\fgn^{pow}$ as the ``fields'' (more accurately, jets of fields at the origin).

The differential on $(\Obs^\cl_n)_0$ has the form $\nabla + \dbar$, 
where $\dbar$ denotes the extension of the differential on $(\Omega^{0,*})^\vee_0$ and $(\Omega^{1,*})^\vee_0$
and $\nabla$ denotes the connection along $X$ arising from Gelfand-Kazhdan descent.
As $\nabla$ is a connection, it increases the de Rham form degree in the $X$-direction, 
whereas $\dbar$ preserves this de Rham form degree in the $X$-direction, since it only cares about the $\CC$-direction.
Consider then the filtration on $(\Obs^\cl_n)_0$ induced by the filtration $\Omega^{\geq *}(X)$ on $\Omega^*(X)$.
The first page of the spectral sequence is the cohomology with respect to $\dbar$:
\[
\Omega^\sharp(X, \cSym(H^*(\Omega^{0,*})^\vee_0 \otimes T^*_X \oplus H^*(\Omega^{1,*})^\vee_0 \otimes T_X)).
\]
These groups $H^*(\Omega^{0,*})^\vee_0$ and $H^*(\Omega^{1,*})^\vee_0$
are concentrated in cohomological degree 0 and are spanned by the linear functionals $\{\zeta_{n,0}\}$, which give the holomorphic Taylor coefficients.
They do not vary along $X$, so
\[
H^*(\Omega^{0,*})^\vee_0 \otimes T^*_X \cong \bigoplus_{n \in \NN} T^*_X \quad \text{and} \quad H^*(\Omega^{1,*})^\vee_0 \otimes T_X \cong \bigoplus_{n \in \NN} T_X.
\]
The induced differential on the first page of the spectral sequence is the induced connection $\Tilde{\nabla}$ on the bundle
\[
\cSym(H^*(\Omega^{0,*})^\vee_0 \otimes T^*_X \oplus H^*(\Omega^{1,*})^\vee_0 \otimes T_X),
\]
so we need to unravel what this bundle means from the perspective of Gelfand-Kazhdan descent.
By Corollary \ref{lem g[[z]]} we know it is identified with the $\infty$-jet bundle 
\[
\underset{{0<k}}{\widehat{\bigotimes}}\,\cSym(T^{1,0*}_X)\, \widehat{\otimes}\, \underset{{0\leq l}}{\widehat{\bigotimes}}\,\cSym(T^{1,0}_X).
\]
(More precisely, it is the jet bundle encoding holomorphic sections.)
The induced connection is the many-fold tensor product of the Grothendieck connection under this identification.
This induced connection $\Tilde{\nabla}$ is the differential on the first page of the spectral sequence, which collapses on the second page.
Thus, we see that the second page is
\[
H^*\left(X,\underset{{0<k}}{\widehat{\bigotimes}}\,\cSym_{\cO_X}(\cT^*_X)\, \widehat{\otimes}\, \underset{{0\leq l}}{\widehat{\bigotimes}}\,\cSym_{\cO_X}(\cT_X)\right),
\]
the cohomology of the $\cO$-module from Lemma \ref{lem g[[z]]}.
\end{proof}

\subsection{The decomposition by conformal dimension: the rotation action on the source}

The rotation action of $U(1)$ on $\CC$ --- or any disk $D_r(0)$ --- induces a rotation action on 
the fields of the curved $\beta\gamma$ system and on the observables $\Obs^\cl_X(D_r(0))$.
It is easy to identify the subspaces of given conformal dimension in light of our work above. 
In particular, we know that polynomials decompose according to conformal dimension as
\[
\CC[z,\zbar] = \bigoplus_{w \in \ZZ} \CC[z,\zbar]_w \quad\text{where}\quad \CC[z,\zbar]_w = \bigoplus_{m - n = w} \CC\, z^m \zbar^n,
\]
and so power series decompose as
\[
\CC[z,\zbar] = \prod_{w \in \ZZ} \CC[z,\zbar]_w.
\]
Polynomials are dense in smooth functions, so we see that $C^\infty(\CC)$ has the same subspaces,
albeit it is some completion of the direct sum of these subspaces.
These dimensional decompositions directly apply to the ``fields'' $\DD\fgn^{poly}$ and~$\DD\fgn^{pow}$.

As observables are symmetric algebras on the linear observables, 
the dimensional decomposition of the fields allows us to identify the observables' dimensional decomposition.
A thorough description is straightforward but involves substantial notation,
so we will state the result only for the cohomological observables $H^*(\Obs^\cl_X)_0$ 
supported at the origin, as this result is the only one we explicitly need.

\begin{lemma}
The conformal dimension $N$ component of $H^*(\Obs^\cl_X)_0$ is
\[
\bigoplus_{\substack{a,b \in \NN^\NN \text{ such that}\\ \sum_{0 < k} a_k k + \sum_{0 \leq l} b_l l = N}} \bigotimes_{0 <k } \Sym^{a_k}_{\cO_X}(\cT^*_X) \otimes \bigotimes_{0 \leq l } \Sym^{b_l}_{\cO_X}(\cT^*_X).
\]
\end{lemma}

\begin{proof}
For example, the conformal dimension of an element of the symmetric power $\Sym^a_{\cO_X}(\cT_X)$ associated to the monomial $z^k$ is $ak$.
Hence the direct sum consists of conformal dimension $N$ components, by summing all the relevant conformal dimensions.
In the other direction, note that an element of the infinite tensor product composing $H^*(\Obs^\cl_X)_0$ 
must be the identity for all but finitely many of the indices $k$ and $l$.
A term in this element must have a polynomial contribution from any given index, due to the bound of $N$ on the total conformal dimension.
\end{proof}

\subsection{The quantum observables}

The quantum observables exhibit the same behavior as the classical observables
with respect to the rotation action and the $(\TVect,\GL_n)$ action.
The arguments must be modified, however, to deal with the BV Laplacian, for example.
Instead of working with the observables supported at the origin, 
it is more convenient to work with the global observables $\Obs^\q_n(\CC)$ or~$\Obs^\q_X(\CC)$.

Let us begin by discussing~$\Obs^\q_n(\CC)$.
There are two versions, depending on whether one works with parametrices or the smoothed observables.
We restrict our attention to parametrices associated with the heat kernel, 
as these are manifestly invariant for the $U(1)$ and $(\TVect,\GL_n)$ actions:
\[
\Obs^\q_n(\CC)[L] = (\cSym(\Bar{\Omega}^{1,*}_c(\CC) \otimes \fgn^\vee[-1] \oplus \Bar{\Omega}^{0,*}_c(\CC) \otimes \fgn[1])[[\hbar]], \dbar + \hbar \Delta_L).
\]
The {\em smoothed observables} are
\[
{\Obs}^{\q,\fr}_n(\CC) = (\cSym(\Omega^{1,*}_c(\CC) \otimes \fgn^\vee[-1] \oplus \Omega^{0,*}_c(\CC) \otimes \fgn[1])[[\hbar]], \dbar + \hbar \Delta),
\]
so that we take the symmetric algebra on the smooth distributions,
such as $\Omega^{1,*}_c(\CC) \subset \Bar{\Omega}^{1,*}_c(\CC)$, inside all distributions. 
The naive BV Laplacian is well-defined on the smoothed observables. 
(See \cite{CG1} for a discussion.)
Ignoring the differentials momentarily, one sees that the same decompositions from above apply;
one can still use the monomials $z^m \zbar^n$ to organize one's
thinking.

The new term in the differential is a BV Laplacian, either $\Delta_L$ or $\Delta$.
Recall that these are built out of the evaluation pairing between $\fgn$ and $\fgn^\vee$
and the wedge-and-integrate pairing between $\Omega^{0,*}_c(\CC)$ and $\Omega^{1,*}_c(\CC)$.
(At scale $L$ the pairing on Dolbeault forms is modified by a mollifying function that is $U(1)$-invariant.)
Both pairings are equivariant for the actions by $U(1)$ and $(\TVect,\GL_n)$,
and we have constructed equivariant quantizations, 
so that BV Laplacians manifestly intertwine with these actions.
In particular, the dimensional decompositions are preserved by the differential.
 
Moreover, when one wants to focus on cohomology, one can exploit the $\hbar$-filtration 
\[
\Obs^\q_n \supset \hbar\, \Obs^\q_n \supset \hbar^2 \Obs^\q_n \supset \cdots
\]
to good effect. 
For instance, there is a spectral sequence associated to this filtration, 
and it collapses on the first page, since the cohomology with respect to $\dbar$ is concentrated in degree zero.
Hence, as vector spaces,
\[
H^* \Obs^\q_n(\CC) \cong H^* \Obs^\cl_n(\CC)[[\hbar]].
\]
By our discussion above, we see that we thus already know the decompositions
of $H^* \Obs^\q_n(\CC)$ with respect to the $U(1)$ or $(\TVect,\GL_n)$ actions.

\begin{prop}
The conformal dimension $N$ component of $H^*(\Obs^\q_n(\CC))$ is
\[
\CC[[\hbar]] \otimes \bigoplus_{\substack{a,b \in \NN^\NN \text{ such that}\\ \sum_{0 < k} a_k k + \sum_{0 \leq l} b_l l = N}} \bigotimes_{0 <k } \Sym^{a_k}(\hT_n^*) \otimes \bigotimes_{0 \leq l } \Sym^{b_l}(\hT_n).
\]
\end{prop}

We obtain an immediate corollary by Gelfand-Kazhdan descent.

\begin{cor}
The conformal dimension $N$ component of $H^*(\Obs^\q_X(\CC))$ is
\[
\CC[[\hbar]] \otimes \bigoplus_{\substack{a,b \in \NN^\NN \text{ such that}\\ \sum_{0 < k} a_k k + \sum_{0 \leq l} b_l l = N}} \bigotimes_{0 <k } \Sym^{a_k}_{\cO_X}(\cT^*_X) \otimes \bigotimes_{0 \leq l } \Sym^{b_l}_{\cO_X}(\cT_X).
\]
\end{cor}

\section{Conformal structure on observables}\label{sec conformal obs}

We have already discussed how a trivialization $\alpha$ of the second component of the Chern character of the complex manifold $X$ determines a sheaf of factorization algebras $\Obs^\q_{X,\alpha}$,
via semi-strict Gelfand-Kazhdan descent of the $\TVect$-equivariant factorization algebra $\Obs^\q_n$. 
Equivalently, in Section \ref{sec comparison}, we showed that this is the factorization algebra of quantum observables of the curved $\beta\gamma$ system with target~$X$ associated to the trivialization~$\alpha$. 

In Section \ref{sec hol vf} we showed how the Lie algebra of holomorphic vector fields on the source acts on the classical formal $\beta\gamma$ system. 
Indeed, we constructed a Maurer-Cartan element $I^\cT \in \cloc^*(\cT^S) \tensor \cloc^*(\DD \fg_n^S)$ implementing this symmetry. 
By Koszul duality, this element thus defines a map of Lie algebras 
\[
I^\cT: \cT^S = \Omega^{0,*}(S; T^{1,0}_S) \to \cloc^*(\DD \fg_n^S)[-1], 
\]
where the the BV bracket $\{-,-\}$ provides the Lie bracket on local functionals.
A local functional can be interpreted as an observable, at least when $S$ is compact,
and so our goal is to refine this map to a map of factorization algebras.

First, we need to replace $\cT^S$ by a local-to-global object.
Holomorphic vector fields on $S$ admit a natural enhancement to a {\em sheaf} of dg Lie algebras: to an open set $U \subset S$, we assign $\cT^U := \Omega^{0,*}(U, T^{1,0}U)$.
But a sheaf is contravariant in opens on $S$, whereas $\Obs^{\cl}_S$ is covariant in opens on $S$.
There is an easy fix: take holomorphic vector fields with {\em compact support}.
Let $\cT^S_c$ denote this precosheaf of dg Lie algebras, which is also a cosheaf of dg vector spaces. 
By Chapter 11 of \cite{CG2}, the map on global sections refines to a map of precosheaves 
\ben
\Psi^{\rm cl}_n : \cT^S_c \to \Obs^{\cl}_n [-1] 
\een
of dg Lie algebras on $S$.
Since $\cT^S_c$ is a trivial $\Vect$-module, we see that applying Gelfand-Kazhdan descent yields a map of sheaves on~$X$
\[
\Psi^{\cl}_X : {\ul \cT^S_c} \to \Obs^{\cl}_X[-1]
\]
of precosheaves of dg Lie algebras on $S$, where $\Obs^{\cl}_X$ is the classical observables of the curved $\beta\gamma$ system with target $X$. The underline means that it is a constant sheaf in the $X$-direction. 

This map $\Psi^{\cl}_X$ extends to a map of factorization algebras
\be\label{classical vir}
\Psi^{\cl}_X : \ul{\Sym_* (\cT^S_c[1])} \to \Obs^{\cl}_X[-1], 
\ee 
since $\Sym(\fg[1])$ is the enveloping $P_0$ algebra of a Lie algebra $\fg$ and $\Obs^{\cl}$ has a natural $P_0$ algebra structure.
(We note that the symmetric algebra of a cosheaf has a natural factorization algebra structure.)
This map should be viewed as a factorization algebra refinement of the Noether theorem in classical physics: a symmetry determines an operator (i.e., a current) in the observables of the classical theory. 
We now wish to study the quantum counterpart to this map of factorization algebras. 

The quantum version of the symmetry of holomorphic vector fields is a factorization algebra depending on a central charge $c$ that we call the {\em Virasoro factorization algebra with charge $c$} and denote $\sVir_{c}$. On $\CC$, this holomorphic factorization algebra is related to the Virasoro vertex algebra in a natural way, as shown in \cite{bw_vir}. 
We already know that the factorization algebra of quantum observables $\Obs^{\q}_n$ carries an action of the extended Lie algebra $\TVect$ and hence determines a sheaf $\Obs^{\q}_{X,\alpha}$ on any complex manifold $X$ with trivial second Chern character. 

The natural question is, then, how to construct the quantum version of the map (\ref{classical vir}). This question can also be understood as a problem of equivariant quantization, by the Koszul-type duality between solutions of the equivariant quantum master equation and maps of BD algebras. We have already computed the obstruction to quantizing the symmetry of holomorphic vector fields on $S = \CC$ in a way compatible with the action of formal vector fields on the target $\Vect$. Hence, the main result is the following.

\begin{prop} \label{prop quantum vir} 
Let $\alpha$ be a trivialization of $\ch_2(T_X)$, and let $\Obs^\q_{X,\alpha}$ be the resulting factorization algebra on $\CC$ of observables for the curved $\beta\gamma$ system with target $X$. If $c_1(T_X) = 0$,  then there is a map of sheaves on $X$ of holomorphic factorization algebras on $\CC$ 
\be\label{quantum vir}
\Psi^\q_X : {\ul \sVir}_{c=2n} \to \Obs^\q_{X,\alpha}
\ee
that, modulo $\hbar$, agrees with the classical map of factorization algebras $\Psi^{\rm cl}$ in Equation~(\ref{classical vir}). 
\end{prop}

This claim will follow from obstruction calculations we have already done when combined with the following quantum version of the Noether theorem for factorization algebras. 

\begin{thm}[\cite{CG2}, Theorem 12.1.0.2] \label{thm quantum noether} 
Let a local Lie algebra $\cE$ defined on a manifold $S$ describe a classical BV theory and suppose $\cE$ has an action of a local Lie algebra $\cL$. Fix a $\cL$-dependent quantization $\{I^\cL[L]\}_{L > 0}$ as described in Section \ref{sec equiv bv 1} such that the obstruction to solving the equivariant QME vanishes modulo functionals depending solely on $\cE$. There is then an $\hbar$-dependent cocycle $\eta \in \cloc^*(\cL)[[\hbar]]$ of degree one, and a map of factorization algebras
\ben
\Psi^{\q} : {\rm C}^{{\rm Lie}, \eta}_*(\cL_c) \to \Obs^\q,
\een
where $ {\rm C}^{{\rm Lie}, \eta}_*(\cL_c)$ is the factorization algebra of $\eta$-twisted Chevalley-Eilenberg chains of~$\cL_c$.
\end{thm}

Explicitly, this cocycle $\eta$ determines a central extension
\ben
0 \to \CC[-1] \to \Tilde{\cL}_c(U) \to \cL_{c}(U) \to 0 
\een
for each open set $U \subset S$.
By definition, we set 
\[
{\rm C}^{{\rm Lie}, \eta}_*(\cL_c) (U) = {\rm C}_*^{\rm Lie}(\Tilde{\cL}_c(U)),
\] 
which is a factorization algebra as shown in \cite{CG2}.
We now proceed to prove Proposition~\ref{prop quantum vir}. 

\begin{proof} We consider the Gelfand-Kazhdan descent of the classical formal $\beta\gamma$ system. As  discussed in Section \ref{sec comparison}, the $L_\infty$ algebra $\DD \fg_n^S$ becomes the curved $L_\infty$ algebra $\DD \fg_X^S$ defined over $\Omega^*_X$. The dg Lie algebra of holomorphic vector fields $\cT^S$ is classically a trivial $\Vect$-module, thus the action of $\cT^S$ on the formal $\beta\gamma$ system descends to an action on the curved $\beta\gamma$ system described by $\DD \fg_X^S$. As usual, we work on $S = \CC$. 

The obstruction calculation of Proposition \ref{sec conformal anomaly} for the $\Vect \times \cT^\CC$ equivariant quantization implies that the $\DD \fg_X^\CC$-dependent obstruction vanishes provided we choose a trivialization $\alpha$ for $\ch_2(T_X)$ and a trivialization $\beta$ for $c_1(T_X)$. Given such a quantization, we see that the cocycle as in Theorem \ref{thm quantum noether} is precisely given by $\eta = 2 n \omega^{\GF}$. That is, the part of the obstruction that is independent of the fields $\DD \fg_X^\CC$. 

The resulting factorization algebra for the curved $\beta\gamma$ system with choice of trivialization $\alpha$ is given by $\Obs^\q_{X,\alpha}$. Finally, the Virasoro factorization algebra of central charge $2n$ is precisely the factorization algebra $\sVir_{c=2n} := {\rm C}^{{\rm Lie}, 2 n \omega^{\GF}}_* (\cT^\CC_c)$. The proposition follows.
\end{proof}

%% file: part3.tex
\part*{Part III: Comparison of the constructions}

\section{Overview}

In this part, we finally relate the two stories we have told: 
we show that the Batalin-Vilkovisky quantization of the curved $\beta\gamma$ system from Part II 
produces the chiral differential operators constructed in Part I. 
The key technical tool is a functor $\Vert$ that extracts a vertex algebra 
from a factorization algebra on $\CC$ satisfying a set of natural conditions.
This tool was introduced in \cite{CG1},
where it was already shown that the formal $\beta\gamma$ system recovers the correct vertex algebra
and an isomorphism was given from $\Vert(\Obs^\q_n)$ to $\hCDO_n$.
But it is more subtle to identify that the BV quantization recovers the correct {\em equivariant} vertex algebra.
To show this, we develop some general arguments that relate factorization algebra derivations with vertex algebra derivations.
From these arguments we swiftly verify that $\Vert(\Obs^\q_n)$ is naturally isomorphic to $\hCDO_n$ 
as a $(\TVect,\GL_n)$-equivariant vertex algebra.

Thanks to the machinery of Gelfand-Kazhdan descent, we then deduce our main result.

\begin{thm} 
Let $X$ be a complex $n$-manifold together with a trivialization $\alpha$ of $\ch_2(T_X) \in H^2(X ; \Omega^{2,hol}_{cl})$. Then the factorization algebra $\Obs^{\q}_{X,\alpha}$ obtained by Gelfand-Kazhdan descent 
determines a sheaf of vertex algebras $\Vert(\Obs^{\q}_{X,\alpha})$ on $X$. 
Moreover, there is an isomorphism of sheaves of vertex algebras on $X$
\ben
\Phi : \CDO_{X,\alpha} \xto{\cong} \Vert(\Obs^{\q}_{X,\alpha})
\een
that is natural in the choice of trivialization~$\alpha$.
\end{thm}

Another goal of this paper is to show how physical arguments about the curved $\beta\gamma$ system
are transformed into vertex algebra arguments.
Thus, as a short coda, we review the treatments by Witten \cite{WittenCDO} and Nekrasov \cite{Nek}, 
and we indicate how their approaches are related to our methods.

\section{From factorization to vertex algebras}

Our central challenge now is to relate the vertex algebra produced in Part I with the factorization algebra produced in Part II.
Although factorization algebras are more flexible and general than vertex algebras
--- appearing in every dimension, for instance, and not just on Riemann surfaces --- 
there are recognition criteria that guarantee when a factorization algebra on $\CC$ recovers a vertex algebra.
In essence, the vector space of the vertex algebra is determined by the value of the factorization algebra on a disk,
and the vertex operators are determined by the structure map for two disjoint disks sitting inside a larger disk (i.e., by a flattened pair of pants).
Chapter 5 of \cite{CG1} is devoted to a careful treatment of this relationship 
and constructs a functor $\Vert$ from a certain category of ``holomorphic'' factorization algebras on $\CC$ to the category of vertex algebras.
This chapter also includes a detailed examination of the free $\beta\gamma$ system and its associated vertex algebra.
Here we will overview the main theorem relating factorization and vertex algebras, 
which requires us to introduce some terminology and machinery we need for our main goal.

Two kinds of technical issues appear in formulating the theorem:
\begin{itemize}
\item describing how the structure maps can ``vary holomorphically'' and
\item pinning down various functional analytic aspects.
\end{itemize}
The first involves ideas essential to the goal of this paper, so we dwell a bit on it.
The second is resolved essentially automatically, given our context and the results from \cite{CG1},
but we discuss it briefly.

\subsection{Translation and derivations}

We need to be able to talk about the structure maps in families in order to say that they vary holomorphically.
Our earlier definition of factorization algebras, however, works with the collection of opens in $\CC$ as a set, 
with no topological---much less complex-analytic---structure.
It is straightforward to introduce variations of the definitions with such structure and 
that manifestly contain the examples we've constructed here.

\begin{dfn}
For $U \subset \CC$ and $z \in \CC$, let 
\[
\tau_z U = \{ w \in \CC \,:\, w - z \in U \}
\]
denote the translation of $U$ by $z$. 
Then a factorization algebra $\cF$ on $\CC$ is (discretely) {\em translation-invariant} if we have an isomorphism
$\tau_z: \cF(U) \cong \cF(\tau_z U)$ for every open $U$ and every $z \in \CC$ satisfying
\begin{enumerate}
\item[(i)] for any $z, z'$, $\tau_{z} \circ \tau_{z'} = \tau_{z+z'}$ and
\item[(ii)] for any disjoint open subsets $U_1,\dots, U_k$ in $V$, the diagram
\[
\xymatrix{
\cF(U_1) \otimes \dots \otimes \cF(U_k)  \ar[r]^-{\tau_z} \ar[d] &
\cF(\tau_z U_1) \otimes \dots \otimes \cF(\tau_z U_k)  \ar[d] \\
\cF(V) \ar[r]^{\tau_x} & \cF(\tau_z V)
}
\]
commutes. (Here the vertical arrows are the structure maps of the factorization algebra.) 
\end{enumerate}
\end{dfn}

Note that the sheaf of holomorphic functions $\cO$ on $\CC$ satisfies the sheaf-theoretic version of this definition,
as does the Dolbeault complex. 
In consequence, the factorization algebras $\Obs^\cl_n$ and $\Obs^\q_n$ are translation-invariant.

We now turn to talking about families.
Let $\cl(U)$ denote the closure of an open set $U \subset \CC$.
Given $U_1,\ldots,U_n$ disjoint opens in $V$, let
\[
{\rm Conf}(U_1,\ldots,U_n \,|\, V) = \{ (z_1,\ldots,z_n) \in \CC^n \,:\,\forall i \neq j ,\, \cl(\tau_{z_i} U_i) \cap  \cl(\tau_{z_j} U_j) =\emptyset \text{ and }\forall i,\, \cl(\tau_{z_i} U_i) \subset V \}.
\]
In other words, this open subset of $\CC^n$ parametrizes all the translations of the $U_i$ that keep them in $V$ 
and keep their closures disjoint. (It will suffice to focus on collections $U_i$ whose closures are disjoint.)
This space ${\rm Conf}(U_1,\ldots,U_n \,|\, V)$ inherits the structure of a complex manifold from $\CC^n$. 

Now let $\cF$ be a discretely translation-invariant factorization algebra. 
We can use the isomorphisms to replace any appearance of $\cF(\tau_z U)$ with $\cF(U)$.
Hence for each point $(z_1,\ldots,z_n) \in {\rm Conf}(U_1,\ldots,U_n \,|\, V)$, we have a structure map
\[
m_{(z_1,\ldots,z_n)}: \cF(U_1) \otimes \cdots \otimes \cF(U_n) \to \cF(V)
\]
by the composite
\[
\cF(U_1) \otimes \cdots \otimes \cF(U_n) \to \cF(\tau_{z_1} U_1) \otimes \cdots \otimes \cF(\tau_{z_n} U_n) \to \cF(V),
\]
where the first map is the tensor product of translation maps $\tau_{z_i}$ and the second map is the structure map of $\cF$.

To talk about these structure maps varying smoothly over ${\rm Conf}(U_1,\ldots,U_n \,|\, V)$, 
we need the factorization algebra to take values in vector spaces (or cochain complexes thereof) 
in which one can talk about smooth families.
We will work in the context described in Section \ref{sec functional analysis}, 
where the topology (or bornology) provides a precise notion of smooth families of linear maps.

\begin{dfn}
A translation-invariant factorization algebra $\cF$ on $\CC$ is {\em smoothly} translation-invariant 
if: 
\begin{enumerate}
\item[(i)] For any collection of opens $U_1,\ldots,U_n$ in $V$ whose closures are pairwise disjoint, 
the maps $m_{z_1,\dots,z_n}$ depend smoothly on $(z_1,\dots,z_k) \in {\rm Conf}(U_1,\ldots,U_n \,|\, V)$.
\item[(ii)] The factorization algebra $\cF$ is equipped with an action by derivations of the abelian Lie algebra $\RR^2$ of translations. 
For $v \in \RR^2$ and open $U \subset \CC$, we will denote the corresponding derivation by $\d/\d v : \cF(U) \to \cF(U)$.
This Lie algebra action is viewed as an infinitesimal version of the global translation invariance.
\item[(iii)] This infinitesimal action is compatible with the global translation invariance in the following sense. 
For $v \in \RR^2$, let $v_i \in (\RR^2)^n$ denote the vector $(0,\ldots,v,\ldots,0)$, with $v$ placed in the $i$-slot and $0$ in the other $n-1$ slots.  If $\alpha_i \in \cF(U_i)$, then we require that
$$
\frac{\d}{\d v_i} m_{z_1,\dots,z_n} (\alpha_1,\dots,\alpha_n) = m_{z_1,\dots,z_n}\left(\alpha_1,\dots,\frac{\d}{\d v} \alpha_i, \dots, \alpha_n\right).
$$
\end{enumerate}  
\end{dfn}

The translation Lie algebra is real. 
As $\cF$ is defined over $\CC$, we can extend the action to the complexified translation Lie algebra $\RR^{2} \otimes_{\RR} \CC$.  
We will denote by $\partial_z$ and $\partial_{\zbar}$ the derivations on $\cF$ associated to the obvious vector fields on $\CC$.
To be {\em holomorphic}, we want the vector field $\partial_{\zbar}$ to act homotopically trivially on $\cF$.

\begin{dfn}\label{holomorphically_translation_invariant_definition}
A translation-invariant prefactorization algebra $\cF$ on $\CC$ is \emph{holomorphically translation-invariant} 
if it is equipped with a derivation $\eta: \cF \to \cF$ of cohomological degree $-1$ such that
\[ 
\d \eta = \partial_{\zbar}, \; [\eta,\eta] = 0,\,\text{and}\,  \left[ \eta,  \partial_{\zbar}  \right] = 0. 
\]
Here $\d$ refers to the differential on the dg Lie algebra ${\rm Der}(\cF)$.
\end{dfn}

\subsection{Rotation and decomposition}

In practice, we are interested in $\cF$ where the action by translation extends to an action of orientation-preserving Euclidean transformations of $\CC$.

\begin{dfn}
A \emph{holomorphically translation-invariant prefactorization algebra $\cF$ on $\CC$ with a compatible $U(1)$ action} 
is a smoothly $U(1) \ltimes \RR^2$-invariant prefactorization algebra $\cF$ 
together with an extension of the action of the complex Lie algebra 
$$
{\rm Lie}_{\CC}(U(1) \ltimes \RR^2) = \CC\left\{ \partial_\theta, \partial_z, \partial_{\zbar} \right\},
$$
where $\partial_\theta$ is a basis of ${\rm Lie}_\CC(U(1))$, to an action of the dg Lie algebra
$$
\CC\left\{ \partial_\theta, \partial_z, \partial_{\zbar} \right\} \oplus \CC\{\eta\} ,
$$
where $\eta$ is of cohomological degree $-1$ and the differential is $\d \eta = \partial_{\zbar}$.
In this dg Lie algebra, all commutators involving $\eta$ vanish except for $[\partial_\theta, \eta] =~-\eta$.
\end{dfn}

The theorem on vertex algebras requires a technical hypothesis regarding the $U(1)$-action on the factorization algebra $\cF$:
we need this action to be {\em tame}, in the following sense. 

For any compact Lie group $G$, the space $\mathcal{D}(G)$ of distributions on $G$ is an algebra under convolution.  
The convolution product $\ast$ is smooth in the sense that it varies nicely in families, as per our approach to functional analysis,
so that the algebra structure is smooth in families.
There is a natural map $\delta: G \to \mathcal{D}(G)$
sending an element $g$ to the $\delta$-function at $g$.  
It is a smooth map and a homomorphism of monoids.  

\begin{dfn} 
A \emph{tame action} of $G$ on a vector space $V$ of the type discussed in Section \ref{sec functional analysis} (e.g., convenient) is a smooth action of the algebra $\mathcal{D}(G)$ on $V$.  
(Note that this means $G$ acts on $V$ via composition $G \to \mathcal{D}(G)^{\times}$ sending $g$ to $\delta_g$.)
For $V$ a cochain complex of such vector spaces, a tame action commutes with the differential on $E$. 
\end{dfn}

The case $G = U(1)$ is the only one relevant for us here. 
For each integer $k$, the function $\rho_k: e^{i\theta} \mapsto e^{ik\theta}$ encodes an irreducible representation of $U(1)$.
It determines a distribution $\rho_k \, \d \theta$ on $U(1)$ that we will abusively call $\rho_k$ as well.
In $\mathcal{D}(U(1))$, the element $\rho_k$  is an idempotent.

\begin{dfn}
Let $V$ be equipped with a tame action of $U(1)$, which we will denote by $\ast$. 
Let $V_k$ denote the {\em weight $k$ eigenspace} for the $U(1)$-action on $V$.
The map $\rho_k \ast - : V \to V$ defines a projection from $V$ onto~$V_k$. 
\end{dfn}

\subsection{The theorem about $\Vert$}

We now turn to the main theorem from Chapter 5 of \cite{CG1}, 
which provides a functor from a certain category of factorization algebras on $\CC$ to the category of vertex algebras.

\begin{dfn}
Let $\cF$ be a tamely $U(1)$-equivariant holomorphically translation-invariant factorization algebra on $\CC$.
Let $\cF_k(D_r(0))$ denote the subcomplex of weight $k$ eigenspaces in $\cF(D_r(0))$, 
the value of $\cF$ on a radius $r$ disk around the origin.
Then $\cF$ is {\em amenably holomorphic} if it satisfies the following conditions:
\begin{enumerate}
\item 
For every pair of radii $r < r'$, the structure map
$$
\cF_k(D(0,r)) \to \cF_k(D(0,r'))
$$
is a quasi-isomorphism.
\item For $k \gg 0$, the vector space $H^\ast (\cF_k(D(0,r))$ is zero.  
\item For each $k$ and $r$, we require that $H^\ast(\cF_k(D(0,r) )$ is isomorphic to a countable sequential colimit of finite-dimensional graded vector spaces. 
\end{enumerate}
\end{dfn}

Observe that for an amenably holomorphic factorization algebra $\cF$, 
the vector space $H^* (\cF_k(D(0,r))$ is independent of $r$ by assumption.
Let 
\[
V_\cF= \bigoplus_{k \in \ZZ} H^* (\cF_k(D(0,r))
\]
be the direct sum of the weight spaces, 
and let 
\[
\overline{V}_\cF = \prod_{k \in \ZZ} H^* (\cF_k(D(0,r))
\]
be the direct product of the weight spaces.
Note that for any disk $D_r(0)$, there is a map $V_\cF \to H^*\cF(D_r(0))$ by the inclusion of the weight spaces.
Likewise, there is map $H^*\cF(D_r(0)) \to \overline{V}_\cF$.
The structure maps of $\cF$ thus determine a family of maps $m_z: V_\cF \otimes V_\cF \to \overline{V}_\cF$ 
given by the composition
\[
V_\cF \otimes V_\cF \to H^*\cF(D_r(0)) \otimes H^* \cF(D_r(z)) \to \H^*\cF(D_R(0)) \to \overline{V}_\cF,
\]
where the middle map is the structure map of $H^* \cF$ with radii $r$ and $R$ such that $2r < |z|$ and $|z| + r < R$.
By construction, these maps $m_z$ vary holomorphically in the parameter $z~\in~\CC-\{0\}$.

Let $\FA^{am}$ denote the category of amenably holomorphic factorization algebras on $\CC$,
where morphisms are maps of prefactorization algebras intertwining the actions of $\CC$ by translation and $U(1)$ by rotation.
Let $\VA$ denote the category of vertex algebras.

\begin{thm}[Theorem 5.3.3, \cite{CG1}] 
\label{theorem_vertex_algebra}
There is a functor $\Vert:\FA^{am} \to \VA$. 
For $\cF$ amenably holomorphic, the underlying vector space of the vertex algebra is $V_\cF$,
and the vertex operator $Y_{\Vert(\cF)}$ is determined by the maps $m_z$ arising from the structure maps of~$\cF$.
\end{thm}

\begin{rmk}
\label{rmk on completions}
As remarked in \cite{CG1} immediately following the theorem,  
this construction makes sense on a factorization algebra $\cF$ that is the inverse limit $\lim_k \cF_k$ of amenably holomorphic factorization algebras,
so that $\Vert(\cF)$ is the inverse limit $\lim_k \Vert(\cF_k)$ of the associated vertex algebras.
This variant is needed by use to recover the completed CDOs~$\hCDO_n$ from the $\beta\gamma$ observables where the target is the formal $n$-disk.
\end{rmk}

Immediate consequences of this theorem include the following.

\begin{lemma}\label{factder}
Let $X \in \Der(\cF)$ be a derivation of $\cF$ as an amenably holomorphic factorization algebra. In particular, we require that $X$ commutes with translation, $[X, \partial_z] = 0$, $[X, \partial_{\zbar}] = 0$, and $[X, \partial_\theta] = 0$. Then $X$ induces a vertex algebra derivation $V_X$ on $\Vert(\cF)$.
\end{lemma}

We call such derivations {\em amenably holomorphic}.

\begin{proof}
Let $D_r$ denote the disk of radius $r$ centered at the origin. 
Let $\cF_k(D_r)$ denote the weight $k$ subspace of $\cF(D_r)$ with respect to the rotation action of $S^1$. 
By hypothesis, $X$ preserves the weight spaces: $\partial_\theta$ acts on $\cF_k(D_r)$ by multiplication by $k$, 
and since $X$ commutes with $\partial_\theta$, it preserves each weight space. 
Hence $X$ induces a linear map $V_X$ on $V_\cF = \bigoplus_k H^*(\cF_k(D_r))$. 

As $X$ is a derivation, it intertwines with the structure maps of $\cF$. In particular, for any two small disjoint disks $D_1$ and $D_2$ included into a larger disk $D_{big}$, we see that 
\[
X_{D_{big}}(m^{D_1,D_2}_{D_{big}}(v_1, v_2)) = m^{D_1,D_2}_{D_{big}}(X_{D_1} v_1, v_2) \pm m^{D_1,D_2}_{D_{big}}(v_1, X_{D_2} v_2).
\]
Since the action of $X$ equivariant with respect to affine transformations of $\CC$, we see that this derivation property holds for the one-parameter family of ``multiplication" operations 
\[
m_{z,0}: V_\cF \otimes V_\cF \to \overline{V}_\cF[[z,z^{-1}]].
\]
Hence the action of $V_X$ on $V_\cF$ is a derivation of the vertex operator map $Y$, which is given by the Laurent expansion of $m_{z,0}$. 
\end{proof}

By a similar but easier argument, we find the following.

\begin{lemma}
Let $\phi: \cF \to \cF$ be an automorphism of an amenably holomorphical factorization algebra $\cF$. Then $\phi$ induces an automorphism of vertex algebras $V_\phi: V_\cF \xto{\cong} V_\cF$.
\end{lemma}

Note that in both these situations we require that the derivation or
automorphism commutes on the nose with all the equivariant structure on $\cF$. 
Such a strict situation is adequate for our purposes here. 
(The homotopical versions of these statements should hold but will not be pursued.)

We now wish to apply these lemmas to the case where the factorization algebra has an action of a pair $(\fg,K)$. 
The data of a semi-strict $(\fg,K)$-structure on $\cF$ involves a group homomorphism $\rho^K : K \to \Aut(\cF)$
together with an $L_\infty$-homomorphism $\rho^\fg : \fg \to \Der(\cF)$
that satisfy the compatibilities in Definition~\ref{dfn ss HC mod}. 

\begin{cor}\label{vertpair} 
Let $\cF$ be an amenably holomorphic factorization algebra 
together with the structure of a semi-strict $(\fg,K)$-module. where
$\fg$ acts by derivations and $K$ acts by automorphisms of the amenably holomorphic factorization
algebra. 
Then the $\ZZ_{\geq 0}$-graded vertex algebra $\Vert(\cF)$ has a
strict action of the pair $(\fg,K)$. 
\end{cor}

\section{Observables for the formal $\beta\gamma$ system}

As an example of the relationship encoded in $\Vert$,
the free $\beta\gamma$ system is examined in depth in Chapter 5 of \cite{CG1}.
It is shown there that the factorization algebra $\Obs^{\fr,\q}_n$ of quantum observables for the free $\beta\gamma$ system is amenably holomorphic,
and it is also shown that the associated vertex algebra is precisely the usual $\beta\gamma$ vertex algebra.

\begin{thm}[Theorem 5.3.3.2, \cite{CG1}]
The factorization algebra $\Obs^{\q,\fr}_n$ is amenably holomorphic. 
Moreover, 
there is an isomorphism of $\ZZ_{\geq 0}$-graded vertex algebras 
\[
\Phi^{\fr}_n : \CDO_n \xto{\cong} \Vert(\Obs^{\q,\fr}_n)|_{\hbar= 2 \pi i}
\] 
after specializing $\hbar=2 \pi i$.
The map is $\GL_n$-equivariant.
\end{thm}

As noted in Remark~\ref{rmk on completions}, 
these results immediately imply the analogous results about the completed case, 
where the free $\beta\gamma$ system is replaced by having the target be the formal $n$-disk.
Hence, in combination with Proposition~\ref{equiv of noneq},
we obtain the following result.

\begin{cor}
The factorization algebra $\Obs^\q_n$ is amenably holomorphic. 
Moreover, there is a $\GL_n$-equivariant isomorphism of $\ZZ_{\geq 0}$-graded vertex algebras 
$\Phi_n : \hCDO_n \xto{\cong} \Vert(\Obs^\q_n)|_{\hbar= 2 \pi i}$.
\end{cor}

\begin{rmk}
The reader might (correctly) object that $\Obs^\q_n$ lives in modules over $\CC[[\hbar]]$,  
and so one cannot specialize $\hbar$ to a nonzero value.
We note, however, that the differential on $\Obs^\q_n$ has the form $D_0 + \hbar D_1$, 
so that only a single power of $\hbar$ appears.
Hence the construction of the quantum observables is well-defined in modules over $\CC[\hbar]$,
where one can specialize $\hbar$ to a nonzero value,
and we are using that version of $\Obs^\q_n$ here.
\end{rmk}

Section \ref{sec loc sym} tackles the much more subtle challenge of showing the isomorphism $\Phi_n$ is $\TVect$-equivariant,
and hence that we get an isomorphism of $(\TVect,\GL_n)$-equivariant vertex algebras.
This property is crucial for applying Gelfand-Kazhdan descent 
and hence recognizing that the BV quantization truly does recover chiral differential operators.

In the remainder of this section, we review some aspects of the $\beta\gamma$ system's factorization
and vertex algebras that are useful for our central goal.

\begin{rmk}
An extensive and expository treatment of these aspects appears in Chapter 6 of~\cite{GwThesis}.
\end{rmk}

\subsection{Some useful identifications}
\label{identifications}

It will be useful to understand explicitly how to identify a representative in the factorization algebra for an element in the vertex algebra.
To be more precise, the construction $\Vert$ ensures that given $v \in \hCDO_n$, there is some cohomology class $[O_v]$ in $H^* \Obs^\q_n(D_r(0)$.
We would like to have a {\em cochain} representative $O_v$ in the disk observables $\Obs^\q_n(D_r(0))$ as well.
Similarly, given a Fourier mode $v_{(n)}$, 
we would like to know a cochain representative $O_{v_{(n)}}$ in the annular observables $\Obs^\q_n(A_{r<R}(0))$.
Although the functor $\Vert$ ensures these wishes can be fulfilled,
the formulas may be quite complicated.

We now examine this issue, starting with the classical observables, where the situation is simpler,
before turning to the quantum observables.
To minimize the number of indices, we restrict to $n = 1$; 
hence, we have elements $b_n$ and $c_m$ with no upper index. 
The extension to arbitrary $n$ is straightforward:
just reinsert the superscripts, e.g., use $c^j_m$ and not just $c_m$.

For the classical observables and $\Gr\, \hCDO_n$,
Cauchy's formula provides explicit integral expressions for the most important linear observables
(i.e., distributions on the fields $\gamma$ and $\beta$).
For example, set 
\[
O_{c_{-m}}(\gamma,\beta) = \frac{m!}{2 \pi i} \int_{|z| = 1} \frac{\gamma(z)}{z^{m+1}} \d z.
\]
This linear observable simply reads off the coefficient of $z^m$ in the power series expansion of a holomorphic $\gamma$.
The support of this distribution is the unit circle, 
so that we can view $O_{c_n}$ as a cocycle in $\Obs^{\cl}_1(A_{r<R}(0))$ for any annulus with $r < 1 < R$.
But it also provides a cocycle in $\Obs^{\cl}_1(D_R(0))$, 
and this cocycle is a representative of the element $c_n$ in $\Gr\, \hCDO_1$.
Similarly, a cochain representative of $b_n$ is
\[
O_{b_{-l}}(\gamma,\beta) = \frac{(l-1)!}{2 \pi i} \int_{|z| = 1} \frac{\beta(z)}{z^{l}},
\]
which reads off the coefficient of $z^{l+1}$ in the power series expansion of a holomorphic one-form~$\beta$.

It is thus easy to provide explicit representatives for monomials like $b_{i_1} \cdots b_{i_l} c_{j_1} \cdots c_{j_m}$.
One simply takes the obvious product --- in the symmetric algebra of distributions --- of the representatives just given.

It is also straightforward to produce smeared versions of these observables, if one wants 
(and we will want it shortly).
Fix a bump function $f(r)$ on some interval $(a,b)$, with $0 < a$, 
such that $\int_a^b f(r) \, \d r = 1$.
Then
\[
O'_{c_{-m}}(\gamma,\beta) = \frac{m!}{2 \pi i}\int_a^b \int_{|z| = r} \frac{\gamma(z)}{z^{m+1}} \d z \, f(r) \d r
\]
is a smeared representative of $c_m$.

Note that if one is working with an annulus rather than disk, then negative powers of $n$ are allowed in the denominator.
In this setting the cocycles $O_{c_m}$ and $O_{b_l}$ read off Laurent coefficients.
As observables on an annulus, they correspond to Fourier modes from the point of view of vertex algebras.
To be explicit, the zeroth Fourier mode $(c_m)_{(0)}$ is represented by $O_{c_m}$ viewed as an observable on an annulus.
Let us explain why.

The vertex operator on $\Gr\,\CDO$ admits a concrete interpretation in terms of ``observing'' coefficients of expansions.
The element $Y(c_m;w)$ should be viewed as an observable on the annulus:
given $\gamma$ a holomorphic function on the annulus and $w$ a point on that annulus, 
$Y(c_m;w)$ measures the coefficient of $(z-w)^m$ in the power series expansion of $\gamma$ around $w$.
If we know the Laurent expansion of $\gamma$ around $0$, then we can provide an expression for this coefficient.

For instance, if we know a Laurent expansion
\[
\gamma(w) = \sum_{m \in \ZZ} c_{-m} w^m,
\]
then
\[
Y(c_0;w)(\gamma) = \sum_m w^m c_{-m}(\gamma).
\]
In the first line, we view the $c_m$ as numbers, providing the Laurent coefficients of $\gamma$,
but in the second line, we view the $c_m$ as operators, providing the Laurent coefficients of~$\gamma$.
In consequence, we see that the $m$th Fourier mode $(c_0)_{(-m)}$ has 
\[
O_{c_{-m}}(\gamma,\beta) = \frac{m!}{2 \pi i} \int_{|z| = 1} \frac{\gamma(z)}{z^{m+1}} \d z
\]
as an explicit representative.

From the factorization algebra point of view, 
the vertex operator amounts to saying that the observable
\[
O_{c_0,w}(\gamma,\beta) = \frac{1}{2 \pi i} \int_{|z-w| = \epsilon} \frac{\gamma(z)}{z-w} \d z,
\]
which measures the value of $\gamma$ at $w$, is cohomologous to the observable
\[
\sum_{m \in \ZZ} w^m O_{c_{-m}}(\gamma,\beta)
\]
in $\Obs^{\cl}(A_{r<R}(0))$ with for $r < |w| < R$ and $\epsilon$ sufficiently small. 

We now turn to providing a tool for understanding how the quantum observables
and $\hCDO_n$ relate.

\subsection{Quantizing observables}
\label{sec quant map}

For the free $\beta\gamma$ system on $\CC$, 
there is a natural cochain isomorphism 
\[
\frak{q}_U: \Obs^{\cl,\fr}(U)[\hbar] \xto{\cong} \Obs^{\q,\fr}(U)
\]
between the classical and quantum observables on any fixed open $U$.
(Recall that the superscript $\fr$ means the smooth or smeared observables.
See Section \ref{noneqsec}.)
This isomorphism ``promotes'' a classical observable to a quantum observable.
It does not preserve, however, the structure maps of the factorization algebras,
and so we view the quantum observables as deforming the structure maps of $\Obs^{\cl,\fr}$ in an interesting, $\hbar$-dependent way:
for $U, U'$ disjoint opens in $V$, the ``quantized'' structure map sends observables 
$F \in \Obs^{\cl,\fr}(U)$ and $F' \in \Obs^{\cl,\fr}(U')$ to 
\[
F \star F' = \frak{q}_{V}^{-1}(\frak{q}_U(F) \cdot \frak{q}_{U'}(F')) \in \Obs^{\cl,\fr}(V)[\hbar],
\]
where $\cdot$ denotes the factorization product in $\Obs^{\q,\fr}$.
We use $\star$ to emphasize that we are ``deformation-quantizing'' the factorization product on the classical observables.

This description allows one to understand concretely how BV quantization affects the factorization algebra,
since the classical observables $\Obs^{\cl,\fr}(U)$ are very explicit and simply amount to algebraic functions on the space of holomorphic functions $\cO(U)$ and holomorphic one-forms~$\Omega^1_{hol}(U)$.
Details of this construction can be found in Section 6, Chapter 4 and Section 3, Chapter 5 of~\cite{CG1}.

To construct the map $\frak{q}$, we use the fact that on $\CC$, 
the operator $\dbar$ possesses a natural choice of propagator (or Green's function),
namely
\[
P(z,w) = \frac{1}{2\pi i}\frac{\d z + \d w}{z - w}.
\]
This distributional one-form on $\CC^2$ satisfies $(\dbar \otimes 1) P = \delta_{\Delta}$, 
where $\delta_{\Delta}$ is the delta-current supported along the diagonal and providing the integral kernel for the identity.
One can view this one-form as a distributional section of the fields $\gamma$ and $\beta$:
for example, for fixed $w$, the one-form $\d z/(z-w)$ is a $\beta$ field in the $z$-variable as it is a $(1,0)$-form,
and dually for the other term in~$P$.

This element $P$ also defines a second-order differential operator $\partial_P$ on the commutative algebra $\Obs^{\cl,\fr}(U)$.
Let us recall the general algebraic context.
For any symmetric algebra $\Sym(V^*)$, an element $v \in V$ defines a vector field $\partial_v$ via contraction:
given $f \in \Sym^{n+1}(V^*)$, we set 
\[
\partial_v f(x_1 \otimes \cdots \otimes x_n) = f(v \otimes x_1 \otimes \cdots \otimes x_n),
\]
by viewing $f$ as an $S_n$-invariant element of $(V^*)^{\otimes n}$.
Similarly, given $p \in V^{\otimes 2}$, we define a second-order differential operator $\partial_p$ by contraction.
Recall that the classical observables $\Obs^{\cl,\fr}(U)$ are a symmetric algebra,
and let $\partial_P$ be the operator obtained by contraction.

\begin{dfn}
Define the {\em promotion map}
\[
\begin{array}{cccc}
\frak{q}:& \Obs^{\cl,\fr}(U)[\hbar] &\to& \Obs^{\q,\fr}(U)\\
& F & \mapsto & \exp(\hbar \partial_P)F
\end{array}.
\]
In other words, one applies a version of Wick contraction to any classical observable $F$,
repeatedly contracting away two inputs with the propagator.
\end{dfn}

\begin{rmk}
In terms of the RG flow used in Part II, this map $\frak{q}$ encodes flowing to length scale $L = \infty$.
Because our theory is free and we restrict to smeared observables, this operation is well-defined.
Effectively, it describes the relations between observables after integrating out the nonzero modes.
\end{rmk}

\subsection{An example}
\label{circ ex}

Consider the classical observable on the annulus $A = \{1/2 < |z| < 3/2\}$ given by
\[
F(\gamma,\beta) = \frac{1}{2\pi i}\int_{|z| = 1} \gamma \wedge \beta,
\]
for $\gamma \in C^\infty(A)$ and $\beta \in \Omega^{1,0}(A)$.
(We say $F$ vanishes if $\gamma$ is a $(0,1)$-form or if $\beta$ is a $(1,1)$-form.)
Its cohomology class $[F]$ in $H^0 \Obs^{\cl,\fr}(A)$ encodes a function on $\cO(A)$ and $\Omega^1_{hol}(A)$
where for 
\[
\gamma = \sum_{n \in \ZZ} c_{-n} z^n \quad \text{and} \quad \beta = \sum_{n \in \ZZ} b_{-n} z^n \, \d z,
\]
we have
\[
[F](\gamma,\beta) =  \sum_{-m - n = -1} c_m b_n
\]
by Cauchy's integral formula.
As $F$ is not a smeared classical observable, 
we cannot immediately apply $\frak{q}$ but first must replace it by a cohomologous smeared observable $\widetilde{F}$.
(If one tries to evaluate $\partial_P F$, one finds it is ill-defined.)

Here is one approach to smearing among many.
Note that the functional
\[
H(\gamma,\beta) = \frac{1}{(2\pi i)^2} \int_{|z| = r} \int_{|w| = R} \frac{\gamma(z) \d z \wedge \beta(w)}{z-w},
\]
with $R > r$, is cohomologous to $F$.
(Simply plug in holomorphic $\gamma$ and $\beta$ and use Cauchy's theorems.)
This functional $H$, while distributional, is easier to ``smear''
by letting $r$ and $R$ vary.
Fix a compactly supported bump function $f(r,R)$ on $B = (1/2,1) \times (1,3/2)$ such that $\int_B f(r,R) \,\d r \, \d R = 1$.
Define
\[
\widetilde{F}(\gamma,\beta) = 
\frac{1}{(2\pi i)^2} \int_B  \int_{|z| = r} \int_{|w| = R} \frac{\gamma(z) \d z \wedge \beta(w)}{z-w}\, f(r,R) \,\d r \, \d R.
\]
Then $\frak{q}(\widetilde{F}) = \widetilde{F}$, since 
\[
\partial_P \widetilde{F} = \widetilde{F}(P(z,w)) = \frac{1}{(2\pi i)^2} \int_B  \int_{|z| = r} \int_{|w| = R} \frac{\d z \wedge \d w}{(z-w)^2}\, f(r,R) \,\d r \, \d R = 0.
\]
In fact, the smearing was not necessary here: the contraction $\partial_P H$ is already well-defined.

\begin{rmk}
\label{promotion}
This approach works well for classical observables with simple descriptions, like our $F$ above.
The initial formula might involve integrating some polynomial in $\gamma$ and $\beta$ around a circle,
but one can replace it, up to cohomology, by an integral over a collection of disjoint circles,
where each copy of $\gamma$ and $\beta$ has its own circle.
Our $H$ is constructed in such a fashion.
Once the supports of these circles are disjoint, 
one can apply $\frak{q}$ directly, without smearing.
\end{rmk}

\section{Local symmetries acting on observables}
\label{sec loc sym}

Our goal here is to articulate how local symmetries of a field theory like the $\beta\gamma$ system
produce derivations of the associated vertex algebras. 
The core construction makes sense for any BV theory 
but we will focus on a version applicable here.
(These manipulations are certainly well-known in the physics literature; 
our work just articulates them in the language of factorization algebras.)

\subsection{General arguments}

Every local functional $L$ in a field theory provides both a derivation of the observables 
and an observable itself.
We want to understand how these two manifestations of $L$ relate.

The derivations arise as ``Hamiltonian vector fields.''
Consider the map of dg Lie algebras
\[
\begin{array}{cccc}
{\rm Ham}: & \cO_{loc}[-1] & \to & \Der(\Obs_T^\cl) \\
& L & \mapsto & \{L,-\} 
\end{array}.
\]
(See Section 3, Chapter 5 of \cite{CosBook} for a discussion of this construction.)
In other words, a local functional can be viewed as a symmetry of the classical field theory.
Note that this map naturally extends to a map of graded Lie algebras into $\Der(\Obs_T^q)$,
but it does not intertwine the differentials,
which is an example of why classical symmetries might not quantize.

We would like to view some of these symmetries as ``inner,''
i.e., realized as the factorization product with an observable,
just as an inner derivation of an associative algebra means it is given by bracketing with an element of the algebra.
To compare derivations to factorization products, however, 
we need to be able to turn local functionals into observables.
A minor issue is that local functionals need not have compact support
and hence do not provide observables on fields with non-compact support.
This problem is easy to fix.

Let $L$ be a local functional,
and let $\cL$ denote the Lagrangian density such that $L = \int \cL$.
By this we mean that if $\gamma$ and $\beta$ are fields with compact support, 
then
\[
L(\gamma,\beta) = \int_\CC \cL(\gamma,\beta).
\]
Let $K \subset \CC$ be a compact subset whose boundary $\partial K$ is a smooth submanifold.
Set $L_K = \int_K \cL,$ so that one simply integrates over $K$ rather than all of $\CC$.
As $K$ is compact, we see that $L_K$ is a well-defined observable on all fields, not just those with compact support.

In short, for $U$ an open set containing the compact submanifold $K$, we have a cochain map
\[
\begin{array}{cccc}
(-)_K: & \cO_{loc} & \to & \Obs_n^\cl(U) \\
& L & \mapsto & L_K
\end{array}.
\]
This map extends to quantum observables but no longer respects the differentials.

A direct computation then gives a relationship between the factorization product and the derivation.

\begin{lemma}
Let $F$ be a cocycle in $\Obs_n^\q(U)$ and $K \subset U$ a compact submanifold.
Then
\[
\d^\q(L_K F) = \d^\q(L_K)F + \{ L_K, F\},
\]
where the notation $L_K F$, for instance, denotes the product in the completed symmetric algebra underlying~$\Obs^\q_n$.
Hence, if $\d^\q(L_K)$ has support in $U\setminus K$, then at the level of cohomology
\[
[\{ L_K, F\}] = - [ \d^\q(L_K) \cdot F],
\]
where $\cdot$ denotes the factorization product for the structure map 
$\Obs_T^\q(V \setminus U) \otimes \Obs_T^\q(U) \to \Obs_T^\q(V)$.
\end{lemma}

\begin{rmk}
This relationship between ``local symmetries'' (i.e., given by local functionals {\em aka} local currents) and the operator product of observables is reminiscent of Ward identities.
We will see below an explicit instantiation of this relationship,
but note here one simple consequence of the lemma:
An observable that is killed by $\{L_K,-\}$ classically --- and hence is fixed by that symmetry --- 
may not be killed at the quantum level.
\end{rmk}

We now restrict our attention to local functionals of a special form,
which admit a particularly useful application of this lemma.
(The arguments we develop here apply with minor changes 
to any free BV theory on $\CC$ whose action is given by~$\dbar.$)

Given a finite set of constant-coefficient holomorphic differential operators
\[
D_1,\ldots,D_m \in \CC[\partial/\partial z],
\]
consider the local functional
\[
L(\gamma,\beta) = \int_\CC D_1(\gamma) \wedge \cdots \wedge D_m(\gamma) \wedge \beta.
\]
Note that one must take some care to properly interpret such a functional, 
as with $I^{{\rm W}}$ or $J$ defined in Section \ref{sec obsdef}. 
This functional vanishes except when working over base dg algebras not concentrated in degree zero.

\begin{dfn}
A local functional $L$ is {\em constant-coefficient holomorphic, $\beta$-linear} if it is a sum of local functionals of the form above.
\end{dfn}

A nice property of such a local functional $L$ is that the derivation $\{L,-\}$ is manifestly amenably holomorphic 
because the integrand is translation-invariant and rotation-equivariant. Hence we see the following.

\begin{lemma}
The factorization algebra derivation $\{L,-\}$ induces a derivation $V_L$ on the vertex algebra~$\hCDO_n$.
\end{lemma}

We now wish to find an alternative description of that derivation.
Recall that the differential on $\Obs^\cl_n$ is denoted $\dbar$, 
as it is the extension of $\dbar$ on the linear observables to a derivation on the symmetric algebra.

\begin{lemma}
For a local functional
\[
L(\gamma,\beta) = \int_\CC D_1(\gamma) \wedge \cdots \wedge D_m(\gamma) \wedge \beta,
\]
with $D_1,\ldots,D_m \in \CC[\partial/\partial z]$, there is an equality
\[
(\dbar L_K)(\gamma,\beta) = \int_{\partial K} D_1(\gamma) \wedge \cdots \wedge D_m(\gamma) \wedge \beta,
\]
for any compact submanifold $K$ of~$\CC$.
\end{lemma}

\begin{proof}
This claim follows from Stokes' lemma.
Compute
\begin{align*}
(\dbar L_K)(\gamma,\beta) 
&= \int_K \dbar \left( D_1(\gamma) \wedge \cdots \wedge D_m(\gamma) \wedge \beta. \right)\\
&= \int_K (\d - \partial) \left( D_1(\gamma) \wedge \cdots \wedge D_m(\gamma) \wedge \beta \right)\\
&= \int_K \d \left( D_1(\gamma) \wedge \cdots \wedge D_m(\gamma) \wedge \beta \right)\\
&= \int_{\partial K} D_1(\gamma) \wedge \cdots \wedge D_m(\gamma) \wedge \beta.
\end{align*}
The reason $\partial$ annihilates the integrand is that $\beta$ contributes a $\d z$ term already.
\end{proof}

Consider as well a closely related situation.

\begin{dfn}
A local functional $L$ is {\em constant-coefficient holomorphic, $\beta$-free} if it is a sum of local functionals of the form
\[
\int_\CC D_1(\gamma) \wedge \cdots \wedge D_m(\gamma) \wedge \d z
\]
where the $D_j$ are constant-coefficient holomorphic differential operators.
\end{dfn}

\begin{lemma}
For a constant-coefficient holomorphic, $\beta$-free local functional $L = \int \cL$,
there is an equality
\[
\dbar L_K = \int_{\partial K} \cL,
\]
for any compact submanifold $K$ of~$\CC$.
\end{lemma}

We now want to promote this relationship of classical observables to one between quantum observables.
We do this in two steps.
First, recall that the smeared classical observables are quasi-isomorphic to the (distributional) classical observables,
by the Atiyah-Bott lemma. (See Appendix D of \cite{CG1} as well as Section 2, Chapter 4.)
Hence we replace $L_K$ by a smeared observable $\widetilde{L}_K$ that is cohomologous, and likewise for any classical observable.
Notationally, we will leave this replacement implicit and write simply $L_K$.
Second, every (smeared) classical observable $F$ can be promoted to a quantum observable $\frak{q}(F)$ 
by the cochain isomorphism $\frak{q}$, as discussed in Section~\ref{sec quant map}.

As we want to identify elements of the vertex algebra from observables,
we restrict our attention to the following situation.
Fix radii $0 < s < S < r <R$ and
consider the inclusion 
\[
A_{S<r}(0) \sqcup D_s(0) \hookrightarrow D_R(0)
\]
of an annulus and a small disk into a big disk.
All are centered at the origin. We will consider the factorization product 
\be\label{disk ann disk}
\Obs^{\q}_n(A_{S<r}(0)) \tensor \Obs^{\q}_n (D_s(0)) \to \Obs^{\q}_n(D_R(0)) .
\ee
At the level of vertex algebras, this map corresponds to the action of ``fields'' (which live on annulus and thus depend on $z$ and $z^{-1}$) on ``states'' (which live on a disk and hence in the state space).

\begin{dfn}
Let $L$ be a constant-coefficient local functional that is $\beta$-linear or $\beta$-free.
Its {\em disk observable} $L_{disk}$ is $L_{\{|z| \leq r\}}$. For the circle $S^1_r$ of radius $r$, its {\em circle observable} $L_{circ}$ is $\dbar L_{disk}$.
\end{dfn}

The circle observable $L_{circ}$ is an element of $\Obs^{\q}_{n}(A)$ where $A$ is any annulus containing the circle of radius $r$. Note that a circle observable should be identifiable with a Fourier mode of some field for the vertex~algebra.

\begin{lemma}
\label{der vs star}
Let $L$ be a constant-coefficient local functional that is $\beta$-linear or $\beta$-free.
For any cocycle $F \in \Obs^\q_n(D_s(0))$, we have
\[
\hbar \left[ \{\frak{q}(L) , F\} \right] = \left[\frak{q}(L_{circ}) \cdot F\right]
\]
at the level of cohomology, where $\cdot$ denotes the factorization product in (\ref{disk ann disk}), and where $L_{circ}$ is the circle observable supported on some circle contained in the annulus $A_{S<r}(0)$.
\end{lemma}

\begin{proof}
We compute
\begin{align*}
\d^\q(\frak{q}(L_{disk}) F) &= \d^\q(\frak{q}(L_{disk})) F \pm \frak{q}(L_{disk}) \d^\q F+ \hbar \{\frak{q}(L_{disk}),F\} \\
&=\frak{q}(L_{circ})F +\hbar \{\frak{q}(L_{disk}),F\},
\end{align*}
since $\frak{q}$ intertwines the classical differential $\dbar $ and quantum differential $\d^\q$ and
since the support of $F$ is contained in the disk $|z| < r$, 
on which $L$ and $L_{disk}$ are indistinguishable.

At the level of cohomology, we thus obtain
\[
[\frak{q}(L_{circ}) \cdot F] =\hbar \left[ \{\frak{q}(L) , F\} \right] 
\]
as claimed.
\end{proof}

\begin{ex} Consider the local functional
\[
L(\gamma,\beta) = \frac{1}{2\pi i}\int \gamma \wedge \beta. 
\]
Note that
\[
L_{circ} (\gamma,\beta) = \frac{1}{2\pi i}\int_{|z|=1} \gamma \wedge \beta,
\]
is precisely the functional $F$ from Section \ref{circ ex}.
We showed there that $L_{circ}$ is cohomologous to the functional
\[
H(\gamma,\beta) = \frac{1}{(2\pi i)^2} \int_{|z| = r} \int_{|w| = R} \frac{\gamma(z) \d z \wedge \beta(w)}{z-w},
\]
with $R > r$.
This functional $H$ is manifestly the zeroth Fourier mode of $c_0 b_{-1}$,
by the discussion in Section~\ref{identifications}.

The zeroth Fourier mode of $c_0 b_{-1}$ acts like a number operator or Euler vector field, in the sense that 
\[
(c_0 b_{-1})_{(0)} f = p f 
\]
when $f$ is a homogeneous polynomial of degree $p$ in the variables $c_{m},b_{l}$, with $m \leq 0$ and~$l < 0$.

On the other hand, note that $\frak{q}(L) = L + C$, where $C$ is a constant, since $L$ is quadratic and hence only admits at most one nontrivial contraction with the propagator $P$. Hence, the derivation $\{\frak{q}(L), -\}$ agrees with the derivation $\{L, -\}$.
Direct computation of this derivation shows that it also counts the number of incoming $\gamma$ and $\beta$ legs into any observable; it is the number operator. At the level of the vertex algebra, it thus recovers the zeroth Fourier mode of $c_0 b_{-1}$,
as claimed by the lemma.
\end{ex}

Let us build on this example to get a general statement.
As a matter of notation, if $A_{S<r}(0) \hookrightarrow D_R(0)$ denotes the inclusion of the annulus inside of the disk, 
denote the resulting structure map of the factorization algebra by $\iota : \Obs^{\q}_n(A_{S<r}(0)) \to \Obs^{\q}_n(D_R(0))$. 
In the language of vertex algebras, this map sends a field $A$ to the state~$A |0\rangle$.

For an arbitrary constant-coefficient local functional, 
we then have the following relation between the vertex algebra and in the factorization algebra.

\begin{lemma}\label{zero fourier mode} 
Let $L$ be a local functional that is $\beta$-linear or $\beta$-free, 
and let $\cdot$ be the factorization product (\ref{disk ann disk}). 
For each disk observable $O \in \Obs^{\q}_n(D_s(0))^{(k)}$, 
\ben
[\frak{q}(L_{circ}) \cdot O] = \hbar [\iota(\frak{q}(L_{circ}))]_{(0)} [O] .
\een
That is, $L$ determines a vertex algebra derivation that is {\em inner}.
\end{lemma}

\begin{proof}
For notational convenience, we assume $n = 1$. 
Moreover, it suffices to assume that $c$ is a linear observable, since the operations are derivations and one can thus apply the Leibniz rule.

Let $L$ be a functional of the form $\int \gamma^{\wedge p} \wedge \beta$.
It follows that the circular observable is given by 
\[
L_{circ} = \int_{|z| = 1} \gamma^{\wedge p} \wedge \beta.
\] 
In the vertex algebra $\hCDO_1$, the element $(c_{0})^p b_{-1}$ corresponds to the cohomology class of~$L_{circ}$. 

Recall that a linear observable on the disk is a linear combination of observables of the form
\begin{align*}
O_{c_{-m}}\gamma,\beta) &= \frac{m!}{2\pi i} \int_{|z| = } \frac{\gamma(z)}{z^{m+1}} \d z \\
O_{b_{-l}}(\gamma,\beta)  &= \frac{(l-1)!}{2\pi i} \int_{|z| = } \frac{\beta(z)}{z^l}
\end{align*}
where $m \geq 0$, $l > 0$. 
We will compute the cohomology class of $L_{circ} \cdot O_{c_{-m}}$ in $\Obs^\q(D_{R}(0))$ and demonstrate the claim explicitly  in this case. We leave the case of $O_{b_{-l}}$ for the reader, as it follows a parallel treatment. 

Note that $\frak{q}(L_{circ}) = L_{circ}$ because both $O_{c_{-m}}$ and $O_{b_{-l}}$ involve a factor of $\d z$. Also, $\frak{q}(O_{c_{-m}}) = O_{c_{-m}}$. Moreover, we have
\ben
\partial_P \left(L_{circ} \cdot O_{c_{-m}} \right) (\gamma,\beta) = \frac{m!}{2\pi i}  \int_{|z| = } \frac{\gamma(z)^{\wedge p}}{z^{m+1}}  .
\een
Thus $(L_{circ} \cdot O_{c_{-m}})(\gamma) =  \hbar \frac{m!}{2\pi i}  \int_{|z| = } \frac{\gamma(z)^{\wedge p}}{z^{m+1}}$. 
On the other hand, the zeroth Fourier mode of $c_0^pb_{-1}$ applied to $c_{-m}$ is computed as
\ben
(c_0^p b_{-1})_{(0)} (c_{-m}) = (c_0^p b_{-1})_{(0)} (T^m c_{0}) = T^m ((c_0^p b_{-1})_{(0)} c_0) = T^m(c_0^p),
\een
which is precisely the cohomology class of the observable above. 

The proof for local functionals that are $\beta$-free is completely analogous. 
\end{proof}
 
\subsection{The action of $\TVect$}

The preceding discussion was abstract but there are two local functionals that play an important role for us:
the local functionals $I^{{\rm W}}$ and $J$ produced by equivariant BV quantization.
As discussed in Section \ref{extendedtheory} in Part II, the local functionals encode how $\TVect$ acts on 
the rank $n$ formal $\beta\gamma$ system.
Specifically, we showed that this equivariant quantization equipped the the factorization algebra
$\Obs^{\q}_n$ with the structure of a semi-strict $(\TVect,\GL_n)$-module. 
Moreover, we showed that this semi-strict action induces a strict action of $(\TVect,\GL_n)$ on the cohomology $H^*\Obs^{\q}_n$.

Our goal now is to use the tools we just introduced to describe the strict action of $\TVect$ on $\hCDO_n$
determined by these local functionals. 
Recall that there exists a Lie algebra map $\rho : \TVect \to \Der_{{\rm VA}}(\hCDO_n)$,
by Theorem \ref{MSV1} from Section \ref{sec hc cdo}.
Explicitly, viewing a pair $(X, \omega) \in \Vect \times \hOmega^2_{n,cl}$ as an element of $\TVect$
as in Section \ref{sec hc cdo}, 
we have
\[
\rho(f(t) \partial_j,0) = (f(c) b_{-1}^j)_{(0)}
\]
and
\[
\rho(0, \d(f(t) \d t_j)) = \left(f(c) T (c_0^{j})\right)_{(0)}.
\]
(On a formal disk, every closed 2-form is exact, so it suffices to give the formula in terms of a 1-form $f(t) \d t_j$.)
These vertex algebra derivations are manifestly \emph{inner}, i.e., come from elements of the state space of~$\hCDO_n$.

\begin{lemma}
For $X \in \Vect$, the local functional $I_X^{\rm W}$ determines 
a derivation $\{\frak{q}(I_X^{\rm W}),-\}$ of the factorization algebra $\Obs^\q_n$ 
whose associated vertex algebra derivation is~$\rho(X,0)$.
\end{lemma}

\begin{proof}
Every formal vector field is a linear combination of vector fields with monomial coefficients,
so we simply consider  $X = t^{m_1}_1 \cdots t^{m_n}_n \partial_j$.
The associated local functional is
\ben
I^{\rm W}_X (\gamma, \beta) =\int_\CC\gamma_1^{m_1} \wedge \cdots \wedge \gamma_n^{m_n} \wedge \beta_j,
\een 
which is constant-coefficient holomorphic and $\beta$-linear.
By Lemma \ref{der vs star}, we know that we can understand the derivation $\{\frak{q}(I^{\rm W}_X),-\}$ 
through the factorization product with $(I^{\rm W}_X)_{circ}$, 
which should correspond to the Fourier mode of some element of~$\hCDO_n$.

By Lemma \ref{zero fourier mode} we find that for $X = t^{m_1}_1 \cdots t^{m_n}_n \partial_j$,
the factorization product by $(I^{\rm W}_X)_{circ}$ corresponds, at the level of the vertex algebra, to the zeroth Fourier mode of~$(c^1_0)^{m_1} \cdots (c^n_0)^{m_n} b^j_{-1}$, as desired. Thus, we recover precisely the formula for~$\rho$.
\end{proof}

Likewise, we have the following.

\begin{lemma}
For $\omega \in \hOmega^2_{n,cl}$, the local functional $J_\omega$ corresponds to the vertex algebra derivation $\rho(0,\omega)$.
\end{lemma}

\begin{proof}
Every closed 2-form $\omega$ on the formal disk is the exterior derivative $\d \theta$ of a 1-form $\theta$. 
Moreover, every 1-form is a linear combination of 1-forms with monomial coefficients,
so we simply consider  $\theta = t^{m_1}_1 \cdots t^{m_n}_n \d t_j$.
The associated local functional is
\ben
J_{\d \theta} (\gamma, \beta) = \int_\CC\gamma_1^{m_1} \wedge \cdots\wedge \gamma_n^{m_n} \wedge \partial_z \gamma_j \,\d z,
\een 
which is constant-coefficient holomorphic and $\beta$-free.
The derivation $\{\frak{q}(J_{\d \theta}),-\}$ corresponds to the factorization product with the circle observable
\[
(J_{\d \theta})_{circ} (\gamma, \beta) =\int_{|z|=1}\gamma_1^{m_1} \wedge \cdots\wedge \gamma_n^{m_n} \wedge \partial_z \gamma_j \,\d z.
\]
Moreover, this circle observable corresponds to the zeroth Fourier mode of $(c^1_0)^{m_1} \cdots (c^n_0)^{m_n} T(c^j_0)$, by Lemma \ref{zero fourier mode}, as desired.
\end{proof}

\section{The main result}

In light of our arguments in the preceding sections, we obtain the following.

\begin{thm} 
The isomorphism of $\ZZ_{\geq 0}$-graded vertex algebras
\ben
\Phi : \hCDO_n \xto{\cong} \Vert(\Obs^{\q}_n) 
\een
is equivariant with respect to the actions of $(\TVect, \GL_n)$. 
Moreover, it is compatible with the $\hO_n$-module structure. 
\end{thm}

\begin{proof}
We proved the equivariance assertion in the preceding sections. 
Thus it remains to discuss the $\hO_n$-module structure.
This aspect, however, is the focus of Section \ref{sec concrete},
where we describe how observables on a disk decompose according to conformal dimension
and then identify each subspace of fixed conformal dimension with some type of formal tensor fields.
At the level of cohomology --- which provides the decomposition for $\Vert(\Obs^{\q}_n)$ --- these match with~$\hCDO_n$.
\end{proof}

An immediate corollary, via Gelfand-Kazhdan descent and its variants, is our main result.

\begin{cor} 
Let $X$ be a complex $n$-manifold together with a trivialization $\alpha$ of $\ch_2(T_X) \in H^2(X ; \Omega^{2,hol}_{cl})$. Then the factorization algebra $\Obs^{\q}_{X,\alpha}$ obtained by Gelfand-Kazhdan descent 
determines a sheaf of vertex algebras $\Vert(\Obs^{\q}_{X,\alpha})$ on $X$. 
Moreover, there is an isomorphism of sheaves of vertex algebras on $X$
\ben
\Phi : \CDO_{X,\alpha} \xto{\cong} \Vert(\Obs^{\q}_{X,\alpha})
\een
that is natural in the choice of trivialization~$\alpha$.
\end{cor}

The naturality in the choice of trivialization can be phrased in a compelling way.
Recall from Section \ref{sec obsdef} that there is an obstruction-deformation complex for the curved $\beta\gamma$ system on $X$, which is a sheaf of dg vector spaces encoding important information about this BV theory.
For instance, a degree one cocycle encodes a first-order deformation of the classical action that satisfies the classical master equation and thus defines a well-posed classical BV theory.
This complex is obtained by Gelfand-Kazhdan descent, and it involves only local functionals that are invariant for the action of $\CC^\times \times {\rm Aff}(\CC)$ by scaling (of the $\beta$ field) and affine transformations.

In particular, the obstruction to quantization is a degree two cocycle in this cochain complex, 
which is identified with $\ch_2(T_X)$ by Lemma \ref{obsprop}.
Corollary \ref{gerbe of obsdef} tells us that this sheaf of dg vector spaces is quasi-isomorphic to the sheaf $\Omega^{2,hol}_{cl}$ on $X$. Hence we deduce the following.

\begin{cor}
The map $\Phi$ provides an isomorphism of gerbes from the gerbe of BV quantizations of the curved $\beta\gamma$ system (constructed via descent and invariant under $\CC^\times \times {\rm Aff}(\CC)$) to the gerbe of CDOs.
\end{cor}

\subsection{Remark on conformal structure}

With the identification of chiral differential operators with the observables of the $\beta\gamma$ system, 
our analysis in Part II immediately implies an observation about the conformal symmetry of this sheaf of vertex algebras. 

In Section \ref{sec conformal obs} we showed that after fixing a trivialization $\alpha$ of $\ch_2(T_X)$, 
there is a map of sheaves on~$X$ 
\ben
\Psi^\q : \ul{\sVir}_{c=2n} \to \Obs^\q_{X, \alpha}
\een 
of factorization algebras on $\CC$, {\em provided} that $c_1(T_X) = 0$. (In fact, we have such a map for every trivialization of~$c_1(T_X)$). 

The factorization algebra $\sVir_{c}$ is amenably holomorphic, and it is shown in \cite{bw_vir} that its associated vertex algebra $\Vert(\sVir_{c})$ is isomorphic to the Virasoro vertex algebra ${\rm Vir}_c$ of central charge $c$. 
By the functoriality of the functor $\Vert$, 
we obtain the following immediate corollary of the above analysis, which implies the aforementioned Proposition \ref{prop conformal cdo} from Part~I. 

\begin{cor} 
Let $\alpha$ be trivialization of $\ch_2(T_X)$ and let $\CDO_{X,\alpha}$ be the associated sheaf of CDOs. 
Then for each trivialization $\beta$ of $c_1(T_X)$, 
the map of holomorphic factorization algebras $\Psi^\q_\beta : \ul{\sVir}_{c=2n} \to \Obs^\q_{X,\alpha}$ determines a map of sheaves of vertex algebras $\Vert(\Psi^\q_\beta) : \ul{\rm Vir}_{c=2n} \to \CDO_{X,\alpha}$.
\end{cor}

\section{Discussion of some physics literature}

Our goal in this section is to relate our work to the perspectives offered by Witten and Nekrasov on the curved $\beta\gamma$ system.
Both \cite{WittenCDO} and \cite{Nek} undertake a similar analysis, but we will focus on Witten's.
The format of our comparison is to remind the reader about general aspects of $\sigma$-models,
to explain how Witten identifies the anomalies to quantization and how his method relates to ours,
and to indicate how Witten determines the patching rules for the chiral algebras and how this approach relates to ours.

\subsection{General comments about nonlinear $\sigma$-models}

\def\Maps{{\rm Maps}}

We begin by sketching a general perspective that informs the problem.

\begin{rmk}
This perspective assumes that the path integral exists and exhibits behavior analogous to finite-dimensional integrals.
In a sense, we run the argument sketched here {\it backwards} to construct the putative path integral measure.
\end{rmk}

Let $\Sigma$ denote a source manifold and $X$ a target manifold.
A nonlinear $\sigma$-model has, as its space of fields, the infinite-dimensional manifold $\Maps(\Sigma,X)$.
The equations of motion for the theory cut out a submanifold $Sol$ of this mapping space as the space of solutions,
and typically a component of this space of solutions is given by a copy of $X$ viewed as the constant maps from $\Sigma$ to $X$.
We will call this the {\em perturbative} sector.

In trying to compute the path integral, one expects that for $\hbar$ very small, 
the putative measure should be concentrated very close to $Sol$ inside $\Maps(\Sigma,X)$.
One might then try to approximate the path integral by simply integrating over a small tubular neighborhood around $Sol$.
The perturbative contribution would then be the integral over a small tubular neighborhood around $X$ inside $\Maps(\Sigma,X)$.
To organize the computation of this perturbative contribution, one can identify a tubular neighborhood with the normal bundle to $X$ inside $\Maps(\Sigma,X)$.
Hence, one obtains an infinite-dimensional vector bundle over $X$ whose fiber over $x \in X$ is 
\[
N_x \Maps(\Sigma,X) \cong T_x \Maps(\Sigma,X)/T_x X \cong \Maps(\Sigma, T_x X)/T_x X.
\]
One then computes the integral over the tubular neighborhood in two stages: 
first, fiberwise integration over the normal bundle, and then integration over $X$.
This fiberwise integral can be approached with Feynman diagrammatics,
with the base $X$ playing the role of a ``background field.''
In this sense, the perturbative sector is {\em local on the target~$X$}

A better approximation to the full path integral would involve the other components of $Sol$. 
They provide the ``instanton corrections'' to the perturbative computation.
As a nonconstant solution is not concentrated at a single point in $X$ --- by definition --- 
these corrections are not local on $X$ and require different techniques.

\subsection{Anomalies and obstructions}

Fix some method for perturbative computations.
At each point $x$ of $X$, we apply our method to integrate over the normal bundle $N_x$.
This integral ought to take values in a one-dimensional vector space, and 
hence the full fiberwise integral provides a section in a line bundle over $X$.
{\em A priori} we do not know which line bundle it is, 
since the perturbative constructions are done locally on $X$ and then patched together.

As Witten notes in Section 2.3 of \cite{WittenCDO}, the Chern class of this line bundle is a discrete invariant and 
hence should not depend on continuous parameters, 
such as the coupling constants of the fiberwise perturbative field theories.
Thus we can compute it by fiberwise quantizing the family of free theories over $X$,
scaling the interactions to zero.
In other words, one simply keeps the Hessian of the action functional at each point $x \in X \subset \Maps(\Sigma,X)$.
The free theory at each $x$ corresponds to some elliptic complex on $\Sigma$. 
Now, it is standard to identify the one-dimensional vector space at $x \in X$ with the determinant line of the cohomology of this elliptic complex.
(This identification can be recovered as a {\em consequence} of BV quantization, as shown in \cite{GwHaug}.)
Hence, one can use a families index theorem to compute the Chern class of the determinant line bundle.

In the case of the $\beta\gamma$ system, 
this amounts to considering the trivial fiber bundle $\pi: \Sigma \times X \to X$
and letting the elliptic complex at a point $x \in X$ be $\Omega^{0,*}(\Sigma) \otimes T_x X$.
In other words, we are considering $\cF = \cO_\Sigma \otimes \cT_X$ as a sheaf on $\Sigma \times X$.
We wish to understand the pushforward $\pi_* \cF$ on $X$ or, more accurately, the determinant line of its derived pushforward.

\def\Td{{\rm Td}}

We assume now that $\Sigma$ is closed.
The first Chern class of the determinant line agrees with the first Chern class of the derived pushforward.
The Grothendieck-Riemann-Roch theorem then implies that the Chern character of the derived pushforward 
is given by 
\begin{align*}
\pi_*( \ch(\cF) \Td(\cT_\pi)) &= \pi_*( (1 + c_1(\cF) + \frac{1}{2}(c_1(\cF)^2 - 2c_2(\cF)) + \cdots)(1 +\frac{1}{2}c_1(T_\Sigma)))\\
&= \pi_*( (1 + \pi^* c_1(\cT_X) + \pi^* \ch_2(\cT_X) + \cdots)(1 +\frac{1}{2}c_1(T_\Sigma)))
\end{align*}
and the first Chern class is the component of cohomological degree 2, namely
\[
(1-g) c_1(\cT_X) + \ch_2(\cT_X) 
\]
where $g$ denotes the genus of $\Sigma$.
In short, one finds that the determinant line is trivial if and only if 
both $\ch_2(\cT_X) = 0$ and either $c_1(\cT_X) = 0$ or $\Sigma$ is genus one.

These results correspond to ours,
although our approach is, on its face, rather different.
We choose to work with the formal $n$-disk as target 
and then apply Gelfand-Kazhdan descent to obtain (the perturbative sector of) the theory with any complex $n$-manifold $X$ as target.
Our obstruction cocycle (or anomaly) thus lives in Gelfand-Fuks cohomology
and maps to de Rham cohomology of some $X$ by descent.
As we have seen in \ref{sec equiv bv}, the obstruction to BV quantization descends to $\ch_2(\cT_X)$ when the source is $\CC$ 
and we require equivariance with respect to
\begin{itemize}
\item translation and dilation on the source (which ensures we can descend to genus one curves) and
\item holomorphic diffeomorphisms of the formal $n$-disk as target.
\end{itemize}
It is a consequence of the calculations of Section \ref{sec conformal anomaly} combined with Gelfand-Kazhdan descent that requiring equivariance under {\em all} holomorphic diffeomorphisms of a source disk requires $c_1(\cT_X) = 0$.
Such equivariance ensures one can descend to higher genus curves by considering them as quotients of the disk by Fuchsian groups.

\begin{rmk}
On the other hand, Witten's global argument using Grothendieck-Riemann-Roch applies in the BV context as well,
because BV quantization of the free $\beta\gamma$ system 
on a Riemann surface $\Sigma$ but twisted by $\cT_X$ recovers the same determinant line, 
up to some cohomological shift.
A benefit of our approach here is that we actually construct the BV quantization for the full $\beta\gamma$ system
rather than merely identifying the obstruction cohomologically.
\end{rmk}

As a further point of comparison, note that the obstruction-deformation complex we compute in Section \ref{sec eq def cplx} recovers precisely the same information that Witten and Nekrasov find.
For instance, they find that first-order deformations of the theory are given by $H^1(X, \Omega^2_{cl})$.
(See Section 2.2 of \cite{WittenCDO} or Section 2.6.3 of \cite{Nek}, although Nekrasov keeps track of deformations of complex structure of the target too.)
Similarly, one can see that the local symmetries of the theory are $H^0(X, \Omega^2_{cl})$, as Nekrasov notes in Section 2.6.6.
In other words, the BV formalism provides a systematic mechanism for answering the questions that Witten and Nekrasov address.

\subsection{Chiral algebras and observables}

In Section 3 on \cite{WittenCDO}, Witten explains how one can recover the sheaf of chiral differential operators by physical arguments.
His approach might be summarized as follows:
\begin{enumerate}
\item since the perturbative theory is local on the target $X$, we fix a good cover $\{U_i\}$ of $X$ and try to patch the quantizations;
\item a coordinatization $\phi_i: U_i \hookrightarrow \CC^n$ allows one to view the field theory as a restriction of the $\beta\gamma$ system with target $\CC^n$ to the open $\phi_i(U_i)$; 
\item one then constructs the chiral algebra of operators for the open $\phi_i(U_i)$ as target by restricting the chiral algebra for the free $\beta\gamma$ system of rank $n$;
\item one tries to patch the chiral algebras on overlaps $U_i \cap U_j$.
\end{enumerate}
The first two steps are built into this perturbative approach to general $\sigma$-models.
The third step depends on two things: 
first, knowing the chiral algebra of the free $\beta\gamma$ system 
(which is a standard computation in physics and which we formalized in Part I), 
and second, knowing the chiral algebra can be localized to smaller opens.
(We remark that {\em chiral} algebra here is synonymous with {\em vertex} algebra in the mathematics literature,
although physicists often (though not here) presume that a chiral algebra is invariant under holomorphic changes of coordinates 
and hence what a mathematician might call a vertex operator algebra.)
This second result is a computation done in \cite{MSV, GMS},
and it requires one to show that the OPE for the free $\beta\gamma$ system can be localized from polynomials in the coordinates on the target $\CC^n$ to holomorphic functions in those coordinates.

The final step is a bit involved, and Witten explains it in Section 3.4 of \cite{WittenCDO}.
He wants to patch the chiral algebras for small opens in $\CC^n$,
so he needs to identify the automorphism group of the chiral algebra.
In practice, he instead computes the Lie algebra of this automorphism group 
or, more accurately, the infinitesimal automorphisms arising from the chiral algebra itself.
Witten wants to find those elements of the chiral algebra whose zeroth Fourier mode 
acts on the chiral algebra as a derivation that preserves conformal dimension.
In his terminology, such an element is a ``dimension one current'' and 
its zeroth Fourier mode is called its ``charge,'' 
which is the integral of the current along a loop around the origin in the source manifold.
Witten uses $\fg$ to denote the Lie algebra given by integrals of dimension one currents modulo total derivatives.

Witten provides two natural types of symmetries.
A holomorphic vector field $V = V^i \partial_i$ on the target determines the current $J_V = -V^i \beta_i$,
viewed as an element of the chiral algebra.
Similarly, a holomorphic 1-form $B = B_i \d z^i$ determines a current $J_B = B_i \partial \gamma^i$.
The charge of a 1-form $B$ vanishes if and only if it is exact (i.e., $B = \partial H$),
so that the collection of such charges is isomorphic to closed 2-forms (as a vector space).
The charges of distinct vector field are, by contrast, distinct.
Let $\frak{v}$ denote the Lie algebra formed by the charges for vector fields $V$ and 
let $\frak{c}$ denote the Lie algebra formed by the charges for one-forms $B$.
Together they span $\fg$, according to Witten.

A direct computation with the charges (or using OPE with these currents) shows that there is an exact sequence of Lie algebras
\[
0 \to \frak{c} \to \fg \to \frak{v} \to 0,
\]
and this Lie algebra corresponds to the extension $\widetilde{W}_n$ that we construct.

In our setting, a current can be defined as a local functional $I$ such that $\{I,-\}$ is a cocycle in derivations of the observables.
As discussed in Section \ref{sec loc sym}, such a functional $I$ determines an element $I_{disk}$ in the observables on the disk and hence an element of the vertex algebra.
The associated charge is the element $I_{circ}$, which is an observable on the annulus and hence corresponds to the zeroth Fourier mode of the element in the vertex algebra.

Witten chooses to find the infinitesimal symmetries of the $\beta\gamma$ system 
by doing explicit computations in the chiral algebra after obtaining it from the free theory.
By contrast, we use the BV formalism to determine how to lift the classical symmetries to quantum symmetries
(at the cost of a Lie algebra extension)
and then extract the chiral algebra statements.
In a sense, we do path integral manipulations to recover chiral algebra, 
and Witten follows the reverse logic.
The results naturally agree.

\begin{rmk}
Recent work  \cite{LiQME} of Si Li provides another useful perspective on this relationship. 
He considers free holomorphic BV theories on the complex line,
and he produces an identification between the obstruction-deformation complex of the BV theory and the mode Lie algebra of its vertex algebra.
His result, extended to our equivariant context, then recovers Witten's computation. 
See Remark~\ref{rmk on si's work} for an extended discussion of this point.
\end{rmk}

After determining the appropriate symmetries of the chiral algebra, 
Witten tries to lift the patching of coordinates on the good cover to patching of the chiral algebras.
The coordinate patching can be seen as living in the Lie algebra $\frak{v}$ (or rather its assciated group),
so the challenge is to lift to $\fg$, which involves choices.
Any choice determines a \v{C}ech 2-cocycle with values in $\Omega^2_{cl}$,
and if the cocycle vanishes,
the choices determine a patching of the chiral algebras.
The work of \cite{GMS} showed that this cocycle is indeed $\ch_2(T_X)$.
Our method provides another perspective.

%% file: appx.tex
\part*{Appendix}

\section{The $\beta\gamma$ system as an infinite-volume limit}

\subsection{Introduction}

This appendix gives an explanation for why one might be interested in the curved $\beta\gamma$ system,
that is, how one might discover this action functional by studying a limit of a more familiar class of classical field theories.
The idea is to modify the usual two-dimensional sigma model with Hermitian target in two steps: 
\begin{enumerate}
\item[(1)] we scale the metric on the target manifold until it becomes ``infinitely big" (this drastically simplifies the problem, as we'll show), and 
\item[(2)] we show that this infinite-volume theory ``splits" into a holomorphic and antiholomorphic theory (physicists use ``chiral and antichiral splitting").
\end{enumerate}
The chiral part is the curved $\beta\gamma$ system.

The core aspects of this construction can be seen by having a complex vector space (or formal disk) as the target manifold. After introducing the ingredients of our theory, we rework the usual action functional into a form better suited to our purposes. This {\em first-order formulation} of the theory makes the infinite-volume limit easy to understand and motivate. Finally, we exploit a special property of the theory --- arising from the interplay between the differential geometry of the source 2-manifold and the target Hermitian manifold --- to obtain the splitting.

\begin{rmk}
This approach is well-known to physicists.
Essentially the same construction is given in \cite{Zeitlin,LMZ,Nek}, and a closely related argument for the half-twisted $\sigma$-model is given in \cite{KapCDR}.
Perhaps the main contribution here is the explicit discussion of how to understand various manipulations within the BV formalism.
The version presented here unpacks and elaborates upon on a lecture by Kevin Costello at the Northwestern CDO workshop in summer 2011.
\end{rmk}

\subsection{The ingredients}

The input data of our classical field theory is the following.

\begin{itemize}
\item Let $S$ be an oriented real 2-manifold with a metric $g$. (We will indicate as we go along why everything only depends on the conformal class of $g$.) We denote the associated volume form by $\dvol_g$ and the dual inner product on $\Omega^1_S$ by $g^\vee$.

\item Let $V$ be an even-dimensional real vector space, equipped with a complex structure by $J$ (so we can view $V$ as complex, when needed). It is equipped with a hermitian inner product $h$, also written $(-,-)_V$. (Recall this means that $h$ is a an ordinary inner product on the real vector space $V$ and that $J$ is an isometry.)

\item Let $V^\vee$ denote the dual real vector space. We denote its dual complex structure by $J^\vee$. There is a canonical evaluation pairing $ev: V \otimes V^\vee \to \RR$, and we have 
\[
ev(Jv,\lambda) = \lambda(Jv) = ev(v,J^\vee \lambda)
\]
by definition.

\item Let $\Omega^k_S(V)$ denote the $V$-valued $k$-forms, i.e., $\Omega^k_S \otimes_\RR V$.

\end{itemize}

Consider the Hodge star operator $\ast$ on $\Omega^1_S$ arising from $g$. A computation in local coordinates shows that 
\[
\int_S h \otimes g^\vee(\alpha, \beta) \dvol_g = \int_S h( \ast \alpha \wedge \beta),
\]
where the right hand side means ``apply $h$ to the $V$-component but simply wedge the 1-form components." 

\subsection{The first-order formulation of the sigma model}

Let $f: S \to V$ be a smooth map. The usual action functional for the sigma model is
\[
S_{SO}(f) =  \int_S h \otimes g^\vee(df, df) \dvol_g.
\]
The subscript $SO$ stands for ``second-order."

There is an equivalent description of the same classical field theory where the fields are $f \in \rm{Maps}(S,V)$ and $A \in \Omega^1_S(V^\vee) $ and the action functional is
\[
S_{FO}(f,A) = \int_S ev(df \wedge A) - \frac{1}{2} \int_S h^\vee(\ast A \wedge A).
\]
The subscript $FO$ stands for ``first-order." This first-order action functional motivates the action functional we finally work with.

\begin{lemma}
The equations of motion for $S_{FO}$ are
\[
df = \ast h^\vee A  \quad\text{and}\quad dA = 0,
\]
and so solutions are given by all $f$ such that $d(\ast df) = 0$. This space of solutions is exactly the same as solutions to the equation of motion 
\[
\triangle_g f = (\ast d \ast) df = 0
\]
for $S_{SO}$.
\end{lemma}

\begin{proof}
We obtain the equations of motion for $S_{SO}$ first. We have
\[
S_{SO}(f)  = \int_S h(\ast df \wedge df) = - \int_S h((d \ast d f) \wedge f) = - \int_S h(\triangle_g f, f) \dvol_g,
\]
where we use integration by parts in the second step and the fact that $\ast$ preserves inner products in the last step. The usual variational procedure then recovers the equation of motion.

Now we treat $S_{FO}$. We obtain the equation $dA = 0$ by considering a variation $f \to f + \delta f$. On the other hand, a variation $A \to A + \delta A$ has the following consequences for the second term,
\[
\frac{1}{2} \int_S h^\vee(\ast \delta A \wedge A) + \frac{1}{2} \int_S h^\vee(\ast A \wedge \delta A) = \int_S h^\vee(\ast A \wedge \delta A),
\]
and so we need $df - h^\vee \ast A= 0$.

Now observe that
\[
df = \ast h^\vee A  \Leftrightarrow \ast df = - h^\vee A,
\]
so we need
\[
d (\ast df) = 0,
\]
to satisfy the equations of motion for $S_{SO}$.
\end{proof}

\subsection{An involution on the space of fields}

We now explore a special property of the fields, arising from the fact that the source is 2-dimensional and the target is Hermitian. Because $\ast^2 = -1$, it provides a natural complex structure on $\Omega^1_S$. Thus, we obtain two involutions:
\begin{itemize}
\item on $\Omega^1_S(V)$, there is $\sigma := \ast \otimes J$, and
\item on $\Omega^1_S(V^\vee)$, there is $\sigma^\vee := \ast \otimes J^\vee$.
\end{itemize}
By using this polarization of the fields, we will obtain eventually the desired chiral decomposition.

\begin{lemma}
The operator $\sigma$ gives an eigenspace decomposition 
\[
\Omega^1_S(V) = \Omega^1_S(V)_+ \oplus \Omega^1_S(V)_-
\]
where $\Omega^1_S(V)_{\pm}$ denotes the $\pm 1$-eigenspace of $\sigma$, and likewise for $\Omega^1_S(V^\vee)$. 
\end{lemma}

\begin{proof}
Let $\Pi_{\pm}$ denote the endomorphism $\frac{1}{2}(1 \pm \sigma)$ on $\Omega^1_S(V)$. Then
\[
\Pi_+^2 = \frac{1}{4}(1 + 2 \sigma + \sigma^2) = \Pi_+,
\]
so $\Pi_+$ is a projection operator (and likewise for $\Pi_-$). As $1 = \Pi_+ + \Pi_-$, we obtain the decomposition.
\end{proof}

\begin{dfn}
We define $d_{\pm}: \Omega^0_S(V) \to \Omega^1_S(V)_{\pm}$ as $\Pi_{\pm} \circ d$. 
\end{dfn}

Consider the natural evaluation pairing
\[
\begin{array}{cccc}
ev_S: & \Omega^1_S(V) \otimes \Omega^1_S(V^\vee) & \to & \RR \\
& v \otimes \lambda & \mapsto & \int_S ev(v \wedge \lambda)
\end{array}.
\]
Observe that
\begin{align*}
ev_S(\sigma v, \sigma^\vee \lambda) &= \int_S ev(\ast J v \wedge  \ast J^\vee \lambda) \\
&= \int_S ev(Jv, J^\vee \lambda) \\
&= \int_S ev(J^2 v, \lambda) \\
&= -ev_S(v,\lambda),
\end{align*}
where in the second line we used the fact that $\ast \alpha \wedge \ast \beta = \alpha \wedge \beta$ for any $\alpha, \beta$ in $\Omega^1_S$. Thus we obtain the following.

\begin{lemma}
With respect to the pairing $ev_S$, $\Omega^1_S(V)_+$ is orthogonal to $\Omega^1_S(V^\vee)_+$, and $\Omega^1_S(V)_-$ is orthogonal to $\Omega^1_S(V^\vee)_-$.
\end{lemma}

\subsection{Replacing the first-order action functional}

We introduce a new theory whose fields are $f \in C^\infty_S(V)$ and $B \in \Omega^1_S(V^\vee)_-$. The action functional is
\[
S_+(f,B) = \int_S ev(d_+f \wedge B) - \frac{1}{2} \int_S h^\vee(\ast B \wedge B).
\]
It might seem like this action only sees {\em half} the information of $S_{SO}$ or $S_{FO}$, but it is actually equivalent. We begin with the heuristic argument before delving into a careful proof in the BV formalism.

\subsection{The heuristic argument} 

There is an illuminating ``completing the square" maneuver. Consider the following automorphism on the space of fields:
\[
f \mapsto f \quad \text{ and } \quad B \mapsto B + h (d_+ f).
\]  
(For $v \in V$, $hv$ denotes the element $h(v, -) \in V^\vee$.) When we apply $S_+$ after this transformation, our integrand is a sum of six terms:
\begin{multline*}
 ev(d_+ f \wedge B) + ev(d_+ f \wedge h (d_+ f))  - \frac{1}{2}  h^\vee(\ast B \wedge B) \\  - \frac{1}{2} \left(  h^\vee(\ast h(d_+ f) \wedge B) +  h^\vee(\ast B \wedge h(d_+f))\right)    - \frac{1}{2} h^\vee(\ast h(d_+f) \wedge h(d_+f)).
\end{multline*}
We can simplify this sum.

First, note that the fourth and fifth terms (which are grouped together already) are equivalent to
\[
- \frac{1}{2} \left( ev(\ast d_+ f \wedge B) +  ev(\ast B \wedge d_+f)\right) ,
\]
and thus together cancel the first term.

Second, note that the second term is equivalent to $h(d_+ f \wedge d_+ f)$. This term vanishes because for any one-form $\alpha$, $\alpha \wedge \alpha = 0$.

The last term is the most interesting: {\it the last term recovers the usual sigma model action.}

\begin{lemma}\label{SOvsPlus}
The last term
\[
- \frac{1}{2} h^\vee(\ast h(d_+f) \wedge h(d_+f))
\] 
is equivalent to $-h(\ast df \wedge df)/4$.
\end{lemma}

\begin{proof}
Recall $d_+ = \Pi_+ d = (1/2) (1 + \sigma) d$. Thus
\begin{align*}
4h(\ast d_+ f \wedge d_+f) &= h(\ast (1+\sigma)d f \wedge (1+\sigma)df)\\
&= h(\ast df \wedge df) + h(\ast \sigma df \wedge df) + h(\ast df \wedge \sigma df) + h(\ast \sigma df \wedge \sigma df) \\
&= h(\ast df \wedge df) -i h( df \wedge df) + ih(\ast df \wedge \ast df) + h( df \wedge \ast df) \\
&= 2 h(\ast df \wedge df).
\end{align*}
The initial term arises just by canceling out the excess copies of $h$ and $h^\vee$.
\end{proof}

All that remains to understand is the third term $ - \frac{1}{2}  h^\vee(\ast B \wedge B) $. From a heuristic perspective, it's irrelevant: for the classical theory, the only critical point is $B = 0$, and for the quantum theory, it contributes nothing of interest (just an extra space of fields equipped with a Gaussian measure centered at zero).

To summarize, we have made an ``upper-triangular" change of coordinates on the space of fields. At the classical level, we recover the same equations of motion. At the quantum level, the nonexistent Lebesgue measure is preserved and the weight $e^{-S_+}$  factors into $e^{-S_{SO}}$ times a Gaussian.

\subsection{The BV argument}

In fact, it is fairly straightforward to rephrase this heuristic argument into a rigorous statement in the BV formalism.
Our model throughout is the case of pure Yang-Mills theory (for which see Chapter 6, Section 3 of \cite{CosBook} or \cite{YMasBF}).

\subsubsection{The ingredients}

Our fields are $f \in C^\infty_S(V)$ and $B \in \Omega^1_S(V^\vee)_-$, so we introduce ``antifields" $f^\vee \in \Omega^2_S(V^\vee)$ and $B^\vee \in \Omega^1_S(V)_+$. As usual, the fields have cohomological degree 0 and the antifields have cohomological degree 1, as below.
\[
\begin{array}{cc}
\underline{0} & \underline{1} \\
C^\infty_S(V) & \Omega^2_S(V^\vee)\\
\oplus & \oplus \\
\Omega^1_S(V^\vee)_- & \Omega^1_S(V)_+\\
\text{(fields)} & \text{(antifields)}
\end{array}
\]
We equip this graded vector space $\sE$ with the following symplectic pairing of cohomological degree $-1$:
\begin{align*}
\langle f, f^\vee \rangle &= \int_S ev(f , f^\vee),\\
\langle f, f^\vee \rangle &= - \langle f^\vee, f \rangle, \\
\langle B, B^\vee \rangle &= -\int_S ev(B^\vee \wedge B),\\
\langle B, B^\vee \rangle &= -\langle B^\vee, B \rangle,
\end{align*}
with all other pairings automatically zero (e.g., $\langle f, B \rangle = 0$). This is simply the shifted antisymmetrization of $ev_S$.

We thus obtain a free BV theory (in the sense of Costello) as the following elliptic complex,
\[
\begin{array}{ccc}
C^\infty_S(V) & \overset{d_+}{\rightarrow} &\Omega^1_S(V)_+\\
& &  \\
\Omega^1_S(V^\vee)_-& \overset{d}{\rightarrow} & \Omega^2_S(V^\vee)
\end{array},
\]
where we simply extracted the quadratic part of $S_+$.

In particular, let $\Phi = (f, f^\vee, B, B^\vee)$ denote an element of $\sE$. Then the free BV theory has action functional
\begin{align*}
S_{free}(\Phi) &= -\frac{1}{2}\langle \Phi, Q \Phi\rangle \\
&= -\frac{1}{2}\langle (f, f^\vee, B, B^\vee), (0, dB, 0, d_+ f)\rangle \\
&= -\frac{1}{2} \left( \langle f, dB \rangle + \langle B,d_+f\rangle \right)\\
&= -\frac{1}{2} \left(\int_S ev(f, dB) - \int_S ev(d_+f \wedge B) \right)\\
&= \int_S ev(d_+f \wedge B).
\end{align*}
Thus we recover the free part of $S_+$.

In full, we have
\[
S_+(\Phi) = - \frac{1}{2}\langle \Phi, Q \Phi\rangle + \frac{1}{2} \langle B, h^\vee (\ast B) \rangle.
\]

\subsubsection{Equivalence at the classical level}

In the classical BV formalism, two different action functionals $S$ and $S'$ give equivalent classical theories if they are cohomologous in the cochain complex $(\sO_{loc}(\sE), \{S,-\})$. (Here we assume $S$ satisfies the classical master equation $\{S,S\} = 0$.) Using more geometric language, we say that $S$ and $S'$ live in the same orbit of the gauge group of symplectomorphisms acting on the space of fields $\sE$ (and hence on the space of action functionals $\sO_{loc}(\sE)$).

\begin{rmk} 
To relate these two assertions, note that the cochain complex, once shifted, is a dg Lie algebra that describes the formal neighborhood of $S$ in the moduli space of classical field theories on $\sE$. Thus, if they are cohomologous, we can construct a Hamiltonian flow moving from $S$ to $S'$.
\end{rmk}

In fact, this setting lets us dress up the heuristic picture, as follows. We replace the change of coordinates by modifying $S_+$ by a boundary in $(\sO_{loc}, \{S_+,-\})$.

\begin{lemma}
Let $H$ denote the local functional of cohomological degree $-1$ where
\[
H(\Phi) = \langle \ast h(d_+f),B^\vee \rangle.
\]
Then
\[
\{S_+,H\} = S_{SO} - S_{free},
\]
so $S_+$ is cohomologous to
\[
S_{SO} - \frac{1}{2} \langle B, h^\vee (\ast B) \rangle
\]
in $(\sO_{loc}, \{S_+,-\})$.
\end{lemma}

\begin{proof}
Observe
\[
\{S_{free},H\} = \pm \langle \ast h(d_+ f), d_+f\rangle = \pm S_{SO}.
\]
In the first equality, we use that the shifted Poisson bracket $\{-,-\}$ is dual to $\langle -, - \rangle$. In the second equality, we use lemma \ref{SOvsPlus}.

A parallel computation shows that $\{I,H\}$, where $I$ denotes the ``interaction term" of $S_+$, recovers $\pm S_{free}$.
\end{proof}

The action functional $S_{SO} \pm I$ thus defines a classical BV theory equivalent to $S_+$. Note, however, that this action functional {\em completely decouples} $f$ and $B$. The term $S_{SO}$ only depends on $f$, and the term $I$ only depends on $B$. Moreover, the critical point of $I$ is $\{B = 0\}$, so the equations of motion pick out the same solutions as $S_{SO}$ on its own.

\subsubsection{Equivalence at the quantum level}

In our setting of a linear target with linear metric, we have shown that the classical BV theory specified by $S_+$ is equivalent to a free BV theory, given by the elliptic complex
\[
\begin{array}{ccc}
C^\infty_S(V) & \overset{\triangle_g \dvol_g}{\longrightarrow} &\Omega^2(V)\\
& &  \\
\Omega^1_S(V^\vee)_-& \overset{Id}{\longrightarrow} & \Omega^1_S(V^\vee)_-
\end{array},
\]
once one writes down a suitable pairing $\langle-,-\rangle$. (We chose to change the pairing so that the elliptic complex is simple. Alternatively, we could have retained the same pairing but written a complicated-looking elliptic complex.) As the second line is acyclic, we see it is irrelevant to {\em both} the classical and quantum theories. In particular, as the theory is free, we can quantize immediately and show that the quantum observables are homotopy-equivalent to the quantum observables constructed  just from the first line.

This argument is another way of saying ``we can integrate out the $B$ fields and they do not affect any observables" (cf. the discussion of Yang-Mills theory in Chapter 6 of \cite{CosBook}).

\begin{rmk}
This argument is the only piece that does not port immediately to the case of a curved target. In that case, we need to verify we {\em can} construct a quantization. Nonetheless, it is plausible that we could quantize while maintaining the complete decoupling of the $f$ and $B$ fields, in which case we could work with just the subcomplex depending on the $f$ fields.
\end{rmk}

\subsection{The infinite volume limit}

As $S_+$ is equivalent to $S_{SO}$, we hereafter focus on $S_+$. Our goal is to study a degenerate limit of $S_+$ where the situation drastically simplifies.

The idea is quite simple: if we scale the metric $h^\vee$ to $t h^\vee$, then as $t$ goes to zero, we scale away the dependence on $h^\vee$ in $S_+$. The limit theory is then independent of the hermitian inner product on $V$. Note that on $V$, the limit $t \to 0$ is equivalent to scaling $h$ to $h/t$, so that the volume of any cube grows toward infinity.

\begin{dfn}
The {\em infinite volume limit} is the action functional
\[
S_{IVL}(f,B) = \int_S ev(d_+ f \wedge B),
\]
with $f \in C^\infty_S(V)$ and $B \in \Omega^1_S(V)_-$.
\end{dfn}

The equations of motion are $d_+ f = 0$ and $dB = 0$. 

\subsection{The chiral splitting}

The operator $d_+$ interacts nicely with the complexifications of our spaces of fields, and thus we will be able to massage our theory into another, appealing form.

Consider the decompositions
\[
\Omega^1_S \otimes_\RR \CC = \Omega^{1,0}_S \oplus \Omega^{0,1}_S
\]
and
\[
V \otimes_\RR \CC = V^{1,0} \oplus V^{0,1}.
\]
We have the respective projection operators
\begin{align*}
p^{1,0}_S &= \frac{1}{2}(1 - i \ast),\\
p^{0,1}_S &= \frac{1}{2}(1 + i \ast),\\
p^{1,0}_V &= \frac{1}{2}(1 - i J),\\
p^{1,0}_V &= \frac{1}{2}(1 + i J).
\end{align*}
By an explicit computation, we see
\[
p^{0,1}_S \otimes p^{1,0}_V  = \frac{1}{4} ( 1 +i \ast - i J + \ast J)
\]
and
\[
p^{1,0}_S \otimes p^{0,1}_V = \frac{1}{4} ( 1 - i \ast + i J + \ast J ),
\]
so
\[
p^{0,1}_S \otimes p^{1,0}_V + p^{1,0}_S \otimes p^{0,1}_V = \frac{1}{2} (1 + \ast J) = \Pi_+, 
\]
where we've extended scalars on $\Pi_+$ so that it works on the complexified $\Omega^1_S(V)^\CC$. 

The following result is an immediate consequence.

\begin{lemma}
On $\Omega^*_S(V)^\CC$, we have
\[
d_+ = \dbar_{V^{1,0}} + \partial_{V^{0,1}}.
\]
\end{lemma}

\begin{proof}
Note that
\begin{align*}
\Omega^1_S(V)^\CC &\cong (\Omega^1_S)^\CC \otimes_\CC V^\CC \\ 
&\cong \Omega^{1,0}(V^{1,0})\oplus \Omega^{1,0}(V^{0,1}) \oplus \Omega^{0,1}(V^{1,0}) \oplus \Omega^{0,1}(V^{0,1}).
\end{align*}
We thus need simply to unravel the relevant projections.

Recall that $\dbar$ means ``project the image of $d$ onto the $-i$-eigenspace of $(\Omega^1)^\CC$." Hence, as an example, $\dbar_{V^{1,0}}: \Omega^0_S(V^{1,0}) \to \Omega^{0,1}_S(V^{1,0})$ is precisely the operator 
\[
(p ^{0,1}_S \circ d) \otimes 1_{V^{1,0}}.
\]
Plugging in all the relevant operators, we obtain the desired result.
\end{proof}

We write the elliptic complex of fields, once the fields are complexified, using the decomposition of $d_+$ given above. This specifies a free BV theory:
\[
\begin{array}{ccc}
\Omega^{0,0}_S(V^{1,0}) & \overset{\dbar_{V^{1,0}}}{\longrightarrow} &\Omega^{0,1}_S(V^{1,0})\\
\oplus & & \oplus\\
\Omega^{0,0}_S(V^{0,1}) & \overset{\partial_{V^{0,1}}}{\longrightarrow} &\Omega^{1,0}_S(V^{0,1})\\
\oplus & & \oplus\\
\Omega^{1,0}_S(V^{\vee\,0,1}) & \overset{\dbar_{V^{\vee\,0,1}}}{\longrightarrow} & \Omega^{1,1}_S(V^{\vee\,0,1})\\
\oplus & & \oplus\\
\Omega^{0,1}_S(V^{\vee\,1,0}) & \overset{\partial_{V^{\vee\,1,0}}}{\longrightarrow} & \Omega^{1,1}_S(V^{\vee\,1,0})\\
\end{array}.
\]
We can separate this into a direct sum of two theories, one holomorphic (the pieces involving $\dbar$) and one antiholomorphic (the pieces involving $\partial$). Equivalently, we view this as working with one complex structure and its conjugate.

\begin{lemma}
On the complexified fields,
\[
S_{IVL}(f, \overline{w}{f}, B, \overline{w}{B}) = \int_S ev(\dbar f \wedge B) + \int_S ev(\partial \overline{w}{f} \wedge \overline{w}{B}).
\]
When restricted to the real points, it recovers the infinite volume limit action.
\end{lemma}